\documentclass[12pt, a4paper]{amsart}
\usepackage{amsmath}
\usepackage{geometry,amsthm,graphics,tabularx,amssymb,shapepar, eucal}
\usepackage{latexsym,amsfonts,amssymb,amsmath,amsxtra,bbm,color,mathrsfs}
\usepackage{amscd}
\usepackage{mathrsfs}
\usepackage{hyperref}
\usepackage{pdfsync}
\usepackage[all]{xypic}
\usepackage{verbatim}
\usepackage{braket}
\usepackage{rotating}
\usepackage{diagbox}

\usepackage[all,cmtip]{xy}
\newcommand{\BA}{{\mathbb{A}}}

\newcommand{\BC}{{\mathbb {C}}}

\newcommand{\BG}{{\mathbf {G}}}

\newcommand{\BN}{{\mathbb {N}}}

\newcommand{\BQ}{{\mathbb {Q}}}

\newcommand{\BZ}{{\mathbb {Z}}}

\newcommand{\CB}{{\mathcal {B}}}
\newcommand{\CC}{{\mathcal {C}}}

\newcommand{\CE}{{\mathcal {E}}}
\newcommand{\CF}{{\mathcal {F}}}
\newcommand{\CG}{{\mathcal {G}}}

\newcommand{\CK}{{\mathcal {K}}}
\newcommand{\CL}{{\mathcal {L}}}

\newcommand{\CO}{{\mathcal {O}}}
\newcommand{\CP}{{\mathcal {P}}}
\newcommand{\CQ}{{\mathcal {Q}}}

\newcommand{\CS}{{\mathcal {S}}}

\newcommand{\CW}{{\mathcal {W}}}
\newcommand{\CX}{{\mathcal {X}}}

\newcommand{\RA}{{\mathrm {A}}}
\newcommand{\RB}{{\mathrm {B}}}
\newcommand{\RC}{{\mathrm {C}}}
\newcommand{\RD}{{\mathrm {D}}}

\newcommand{\RG}{{\mathrm {G}}}
\newcommand{\RH}{{\mathrm {H}}}

\newcommand{\RL}{{\mathrm {L}}}
\newcommand{\RM}{{\mathrm {M}}}
\newcommand{\RN}{{\mathrm {N}}}
\newcommand{\RO}{{\mathrm {O}}}

\newcommand{\RQ}{{\mathrm {Q}}}

\newcommand{\RS}{{\mathrm {S}}}
\newcommand{\RT}{{\mathrm {T}}}
\newcommand{\RU}{{\mathrm {U}}}

\newcommand{\Ad}{{\mathrm{Ad}}}
\newcommand{\Aut}{{\mathrm{Aut}}}

\newcommand{\GL}{{\mathrm{GL}}}

\newcommand{\GSp}{{\mathrm{GSp}}}

\newcommand{\Hom}{{\mathrm{Hom}}}

\newcommand{\Ind}{{\mathrm{Ind}}}

\renewcommand{\Re}{{\mathrm{Re}}}

\newcommand{\Res}{{\mathrm{Res}}}

\newcommand{\rk}{{\mathrm{k}}}

\newcommand{\SL}{{\mathrm{SL}}}

\newcommand{\Sym}{{\mathrm{Sym}}}

\newcommand{\tr}{{\mathrm{tr}}}

\newcommand{\bs}{\backslash}

\newcommand{\con}{\textit{C}}

\newcommand{\diag}{\operatorname{diag}}

\newcommand{\od}{\operatorname{d}}

\newcommand{\oL}{\operatorname{L}}

\newcommand{\oH}{\operatorname{H}}

\newcommand{\g}{\mathfrak g}

\newcommand{\p}{\mathfrak p}

\renewcommand{\c}{\mathfrak c}
\newcommand{\n}{\mathfrak n}

\renewcommand{\l}{\mathfrak l}

\newcommand{\s}{\mathfrak s}

\newcommand{\z}{\mathfrak z}

\newcommand{\fN}{\mathfrak N}
\newcommand{\fZ}{\mathfrak Z}
\newcommand{\fG}{\mathfrak G}
\newcommand{\fP}{\mathfrak P}
\newcommand{\fL}{\mathfrak L}

\newcommand{\fV}{\mathfrak V}
\newcommand{\fS}{\mathfrak S}

\renewcommand{\rk}{\mathrm k}

\newcommand{\Z}{\mathbb{Z}}
\newcommand{\C}{\mathbb{C}}
\newcommand{\R}{\mathbb R}
\newcommand{\Q}{\mathbb Q}

\newcommand{\A}{\mathbb{A}}

\newcommand{\abs}[1]{\lvert#1\rvert}

\newcommand{\la}{\langle}
\newcommand{\ra}{\rangle}

\newcommand{\be}{\begin {equation}}
\newcommand{\ee}{\end {equation}}
\newcommand{\bee}{\begin {equation*}}
\newcommand{\eee}{\end {equation*}}

\newcommand{\cf}{\emph{cf.}~}

\theoremstyle{Theorem}

\theoremstyle{Theorem}
\newtheorem{lem}{Lemma}[section]
\newtheorem{corl}[lem]{Corollary}
\newtheorem{thml}[lem]{Theorem}
\newtheorem{leml}[lem]{Lemma}
\newtheorem{prpl}[lem]{Proposition}

\theoremstyle{Theorem}

\theoremstyle{Plain}

\newtheorem{remarkl}[lem]{Remark}

\theoremstyle{Remark}

\theoremstyle{remark}

\theoremstyle{remark}
\newtheorem{exl}[lem]{Example}

\theoremstyle{Definition}

\newtheorem{dfnl}[lem]{Definition}

\numberwithin{equation}{section}

\begin{document}

\title{Relative completed cohomologies and modular symbols}% period integrals}

%\author[J.-S. Li]{Jian-Shu Li} \address{Institute for Advanced Study in Mathematics, Zhejiang University\\  Hangzhou, 310058, China}\email{jianshu@zju.edu.cn}

\author[D. Liu]{Dongwen Liu}
\address{School of Mathematical Sciences,  Zhejiang University\\
  Hangzhou, 310058, China}\email{maliu@zju.edu.cn}

\author[B. Sun]{Binyong Sun}
\address{Institute for Advanced Study in Mathematics \& New Cornerstone Science Laboratory, Zhejiang University,  Hangzhou, 310058, China}\email{sunbinyong@zju.edu.cn}


%\thanks{Supported by NSFC Grant 11222101}

\begin{abstract}Generalizing Emerton's completed cohomologies, we define relative completed cohomologies of arithmetic manifolds. We also define modular symbols for them, and show that the relative completed cohomology spaces interpolate the ``nearly ordinary part" of the classical automorphic cohomologies, and the modular symbols defined for them interpolate the classical modular symbols.  As applications, we use these modular symbols to construct three families of nearly ordinary $p$-adic L-functions: (i)  Rankin-Selberg  $p$-adic L-functions for $\GL_n\times \GL_{n-1}$,  (ii)  Rankin-Selberg $p$-adic L-functions for $\RU_n\times \RU_{n-1}$, and (iii) Standard  $p$-adic  L-functions  of symplectic type for $\GL_{2n}$. We define and calculate explicitly the modifying   factors at $\infty$ and at $p$, and determine the exceptional zeros of the $p$-adic  L-functions for these examples. 
%Moreover,
The modifying  factors at $\infty$  are consistent with the conjectures given by Deligne  and Blasius, and the modifying   factors at $p$  are consistent with the conjecture given by Coates and Perrin-Riou.
%Moreover,

\end{abstract}

 \maketitle

\tableofcontents

\section{Introduction}\label{s1}

Period integrals of automorphic forms are ubiquitous in the study of complex L-functions. However, due to the lack of $p$-adic Haar measure, there is no obvious way to integrate $p$-adic functions on $p$-adic manifolds. This obstructs the study of $p$-adic L-functions by  ``$p$-adic period integrals".

On the other hand, in his seminal  work \cite{Em06a},
M. Emerton introduces the $p$-adic spaces of completed cohomologies of arithematic manifolds, which provide $p$-adic interpolations of  the classical automorphic cohomologies.  One may expect to define ``period integrals" on these completed cohomology spaces to get $p$-adic L-functions.  
Modular symbols, which serve as topological interpretations of period integrals, furnish a fundamental tool for the arithmetic exploration of L-functions.  In this paper, generalizing Emerton's completed cohomologies, we will define relative completed cohomologies and define modular symbols for them. Generalizing the fact that Emerton's spaces of completed cohomologies interpolate all the classical automorphic cohomologies, we will show that the relative completed cohomology spaces interpolate the ``nearly ordinary part" of the classical automorphic cohomology spaces. Moreover, the modular symbols defined for  the relative completed cohomologies interpolate the classical modular symbols. 

We will give three families of examples to illustrate that our modular symbols 
produce $p$-adic L-functions that interpolate special values of classical L-functions. We define and calculate explicitly the modifying   factors at $\infty$ and at $p$, and determine the exceptional zeros of the $p$-adic  L-functions for these examples. 
%Moreover,
The modifying  factors at $\infty$  are consistent with the conjectures given by Deligne \cite{Del79} and Blasius \cite{Bl97}.
The modifying   factors at $p$  are consistent with the conjecture given by Coates and Perrin-Riou \cite{CPR89, Co89}.

The construction of relative completed cohomologies and the explicit evaluation of modifying factors at $p$ in this paper have 
found important arithmetic applications in the recent works \cite{DZ24, Liu23, Liu24, LTX24}.

%With the point  of view that   $p$-adic valued functions   are not the natural objects to be integrated,  we will introduce certain ``nearly ordinary $p$-adic automorphic forms along a parabolic subgroup", and define period integrals for them. Similar to the classical case, these period integrals are  sources  of $p$-adic L-functions. 

\subsection{Some notations}\label{secnotation}
Let us first fix some notation that will be used throughout the paper. Let $p$ be a fixed rational prime.   Fix an algebraic closure $\overline \Q$ of the field $\Q$ of rational numbers. %, and an algebraic closure $\C_p$ of the field $\Q_p$ of $p$-adic  numbers.    
 Let $\abs{\,\cdot\,}_p$ denote the multiplicative non-archimedean norm on $\C_p$ that is normalized so that $\abs{p}_p=p^{-1}$. We  fix field embeddings
\[
\overline \Q \hookrightarrow \C_p\quad\textrm{and}\quad   \overline \Q \hookrightarrow \C
\]
so that $\overline \Q$ is viewed as a subfield of both $\C_p$ and $\C$.  Let $\mathsf E\subset \overline{\Q}$ be a subfield, and  $E\subset \C_p$ a closed subfield  containing $\sf E$. 
 Write $\A$ for the adele ring of $\Q$. It has the usual decompositions
  \[
   \A=\R\times \A^\infty = \R\times \Q_p \times \A^{\infty p}.
   %\quad\quad \A=\Q_p\times \A^p,  \quad \textrm{and}\quad  
 \] 
For every number field $\rk$, write $\rk_\infty:= \rk\otimes_\Q \R$, and for every ring $R$ with identity, write $R^\times$ for the group of its invertible elements. % A superscript $\,^\vee$ indicates the dual space in various contexts. For example, $\pi^\vee$ denotes the contragredient of $\pi$ when $\pi$ is a Casselman-Wallach representation of a real reductive group.  an admissible smooth representation of a totally disconnected locally compact Hausdorff  group Casselman-Wallach representation
A superscript $\,^\vee$ over a finite-dimensional vector space indicates the dual space. 
 
Given a locally compact   Hausdorff topological group  $\CG$, we write $\CG^\circ$ for the connected component of its identity element, and write
$\CG^\natural:=\CG/ \CG^\circ$. For every  locally compact Hausdorff homogeneous space $X$  of $\CG$, we let $\mathrm M(X)$ denote the space of all  $\CG$-invariant complex Borel measures on $X$. If $X$ is furthermore totally disconnected, let $ \RD(X)\subset \mathrm M(X)$ denote the subspace of the measures with respect to which all the open compact subsets have rational volumes.  %define a one-dimensional $\Q$-vector space \begin{eqnarray*}   \mathrm D(\CG)&:=&\{\textrm{left invariant $\Q$-valued distribution on $\CG$}\}\\  &=&\{\textrm{right invariant $\Q$-valued distribution on $\CG$}\}.\end{eqnarray*}Here a left invariant distribution is identified with a right invariant distribution if they have equal restrictions to all open compact subgroups of $\CG$. 
Specifically viewing $\CG$ as a homogeneous space under the right translation, we have the space $ \mathrm M( \CG)$, and the space $ \RD(\CG)$ in the case when $\CG$ is totally disconnected.    Push-forward of measures through the conjugation action yields a representation of $\CG$ on $\mathrm M(\CG)$. When the homogeneous space $X$ is a topological manifold,   we define the $\Q$-vector space 
\[
  \mathrm O(X):=\{\textrm{$\CG$-invariant sections of $\mathscr O_{X}$}\}, 
\]
where   $\mathscr O_X$ denotes the orientation local system  on $X$ with coefficients $\Q$. %Similar to the previous case, the conjugation action of $\CH$  yields a representation of the component group $\CH^\natural$ on $\mathrm O(\CG/\CH)$. 

When $\CG$ is totally disconnected, for every smooth representation $J$ of $\CG$ over a field of characteristic zero, write $$\widehat J:=\varprojlim_{\CK} J^\CK$$ for its formal completion, where $\CK$ runs over all open compact subgroups of $\CG$, and the implicit homomorphisms defining the inverse limit are the  averaging projections. See \cite[Section 6.3]{Be87} for more details. This is a (non-necessarily smooth) representation of  $\CG$ which contains $J$ as a subrepresentation.

 %Here $\mathrm{N}_{\CG}(K_\CG)$ denotes the normalizer of $K_\CG$ in $\CG$, and similar notation will be used without further explanation.

Let 
\be\label{five}
(\mathsf G,\, \dot{\sf G}, \, {\sf Z}, \, \imath : \dot{\sf G}\rightarrow \mathsf G, \, \jmath : \dot{\sf G}\rightarrow \mathsf Z)
\ee
be a $5$-tuple where $\mathsf G$ and $\dot{\sf G}$ are linear algebraic groups over $\Q$, ${\sf Z}$ is an algebraic torus over $\Q$, and $\imath$ and $\jmath$ are algebraic homomorphisms over $\Q$. Set 
\[
G:=\mathsf G(\Q_p),\quad  \dot G:=\dot{\mathsf G}(\Q_p),\quad Z:={\mathsf Z}(\Q_p),
\]
and write $\g$, $\dot \g$, $\z$ respectively for their Lie algebras. 

Let $\p$ be a Lie subalgebra of $\g$ that is transversal to $\dot \g$ in the sense that 
\be\label{pspherical}
\p+\imath(\dot \g)=\g.
\ee
Here and henceforth, we still use $\imath$ (or $\jmath$) to denote various maps induced by the homomorphism $\imath: \dot {\sf G}\rightarrow \sf G$ (or $\jmath: \dot{\sf G}\rightarrow \sf Z$). Write
\[
\dot \p:=\imath^{-1}(\p)\subset \dot \g\quad\  \textrm{so that 
 }\ \  \g/\p=\dot \g/\dot \p.
 \]
% Assume that $\jmath(\dot \p)\subset \z$ is algebraic in the sense that it is the Lie algebra of  an algebraic subtorus of $Z$ defined over $\Q_p$.
When no confusion is possible, we will not distinguish a linear algebraic group defined over $\Q_p$ with the $p$-adic Lie group of its $\Q_p$-points.  Let  $\s\subset \p$ be a Lie subalgebra and  $Z_0\subset Z$  an algebraic subtorus   such that 
\[
  \imath(\dot \p)\subset \s  \quad \textrm{and}\quad \jmath(\dot\p)\subset \z_0,\qquad  \textrm{where $\z_0$ denotes the Lie algebra of $Z_0$}.
\]
(The group $Z_0$ is trivial in the three families of examples of this paper.)

Throughout the paper, $\ell=\infty$ or a rational prime. % Let  $Z_0$ be a closed subgroup of $Z$ whose Lie algebra $\z_0$ contains $\jmath(\dot\p)$ 
Write 
\[
D^{\mathrm{max}}_{Z_0}:=\prod_{\ell\neq\infty} D^{\mathrm{max}}_\ell 
\]
 where $D^{\mathrm{max}}_p$ is the maximal compact subgroup of $Z_0$,  and $D^{\mathrm{max}}_\ell$ is the maximal compact subgroup of $\sf Z(\Q_\ell)$ for $\ell\neq \infty, p$.  
Fix a locally constant character 
\be\label{ram}
\varepsilon = \otimes_{\ell\neq\infty}  \varepsilon_{\ell}: D^{\mathrm{max}}_{Z_0}\rightarrow \overline{\Q}^\times,
\ee
to be called the ramification type. %which is a locally constant character and is trivial on $D^{\mathrm{max}}_{Z_0}\cap(\sf Z(\Q)\sf Z(\R)^\circ)$. 
We introduce the following sets  attached to $\mathsf Z$, $Z_0$ and $\varepsilon$:
\begin{itemize}
   \item 
$\mathcal X^{\mathrm{alg}}$ denotes the group of all algebraic characters ${\sf w}:  {\sf Z}_{\overline \Q}\rightarrow \GL(1)/_{\overline \Q}$;
\item    $\mathcal X^{\mathrm{aut}}$  denotes the group of all automorphic characters $\chi:\sf Z(\Q)\backslash \sf Z(\A)\rightarrow \C^\times$;
    %\item  $\mathcal X^{p-\mathrm{aut}}$  denotes the group of all continuous characters $\chi':\mathsf Z(\Q)\backslash \mathsf Z(\A)\rightarrow \C_p^\times$;
    \item 
    $\mathcal X(\varepsilon)$  denotes the set of all  characters $\chi\in \CX^{\rm aut}$ whose restriction to $D^\mathrm{max}_{Z_0}$ equals $\varepsilon$;
    \item $\mathcal X_\ell$  denotes the group  of all continuous characters $\chi_\ell: \mathsf Z(\Q_\ell)\rightarrow {\C}^\times$ ($\Q_\infty:=\R$);
        \item 
    for $\ell\neq\infty$, $\mathcal X(\varepsilon_\ell)$  denotes the set  of all  characters $\chi_\ell\in \CX_\ell$ whose restriction to $D^\mathrm{max}_\ell$ equals $\varepsilon_\ell$.
       %\item 
  %  $\mathcal X_\ell^\mathrm{ur}\subset \mathcal X_\ell$ denotes the subgroup  of the unramified characters; 
  %  \item $\mathcal X_\infty^\mathrm{tor}\subset \mathcal X_\infty$ denotes the subgroup of the finite order characters;
    \end{itemize}
Here $\sf Z_{\overline \Q}:=\sf Z\times_{\mathrm{Spec}( \Q)} {\mathrm{Spec}(\overline \Q)}$, and similar notation for base change will be used without further explanation.  Note that for each $\ell\neq \infty$, $\mathcal X(\varepsilon_\ell)$ is a scheme defined over  $\Q(\varepsilon_\ell)$. Here $\Q(\varepsilon_\ell)$ is the field extension of $\Q$ generated by the values of $\varepsilon_\ell$, and similar notations will be used without further explanation.  % For $\ell\neq \infty$, $\mathcal X_\ell$ is naturally  a commutative group scheme over $\Q$, and its identity connected component, which equals $\mathcal X_\ell^\mathrm{un}$, is a split torus. Moreover, $\mathcal X_p^{\z_0\mathrm{-ur}}$ is an open group subscheme of $\mathcal X_p$ over $\Q$. Also note that the group 

Let $\con(\mathsf Z, E)_{\z_0\mathrm{-sm}}$ denote the space of $E$-valued continuous functions on $\sf Z(\Q)\backslash \sf Z(\A)$ that are invariant under the translations of some open subgroups of $D^\mathrm{max}_{Z_0}$. This is naturally a locally convex topological vector space under the direct limit topology. Moreover, as a locally convex topological vector space over $\C_p$, 
\[
\con(\mathsf Z, \C_p)_{\z_0\mathrm{-sm}}=\bigoplus_{\varepsilon} \con(\mathsf Z,\C_p)(\varepsilon).
\]
Here $
\varepsilon$ runs over all characters as in 
\eqref{ram}, and whenever $\Q(\varepsilon) \subset {\sf E}$, 
$ \con(\mathsf Z, E)(\varepsilon)$ denotes  the space of all continuous functions $f: \mathsf Z(\Q)\backslash \mathsf Z(\A)\rightarrow E$ such that  
\[
f(xg)=\varepsilon(g)\cdot f(x)\quad\textrm{for all } x\in \mathsf Z(\A), \ g\in D^\mathrm{max}_{Z_0},
\]
which is a Banach space under the supremum norm.

For every $\sf w\in \mathcal X^{\mathrm{alg}}$ and $\ell=\infty$ or $p$, define a character 
\be\label{defwp}
  \sf w_\ell: \mathsf Z(\Q_\ell)\subset \mathsf Z(\C_\ell)\xrightarrow{\sf w} \C_\ell^\times\qquad (\C_\infty:=\C). 
\ee
When no confusion is possible, we still use ${\sf w}_\ell$ to denote the composition of 
\[
\mathsf Z(\A)\xrightarrow{\textrm{projection}} \mathsf Z(\Q_\ell)\xrightarrow{{\sf w}_\ell}\C_\ell^\times. 
\]

\begin{dfnl}
   (a) A character $\chi_\infty\in \mathcal X_\infty$ is said to be algebraic if there is an element $\sf w\in \mathcal X^{\mathrm{alg}}$ such that the character $\sf w_\infty\cdot\chi_\infty$ is locally constant. The algebraic character $\sf w^{-1}$ is called the weight of $\chi_\infty$.
   
   \noindent (b) A character $\chi\in \mathcal X^{\mathrm{aut}}$ is said to be algebraic if so is its archimedean component $\chi_\infty$. When this is the case the weight of  $\chi_\infty$ is also called the weight of $\chi$. 
\end{dfnl}

For every automorphic character $\chi\in \mathcal X^{\mathrm{aut}}$ that is algebraic of weight $\sf w^{-1}$, define a $p$-adic automorphic  character
\[ %\label{flatchi}
\chi^\flat:=(\mathsf w_\infty\cdot \chi)\cdot \mathsf w_p^{-1}: \mathsf Z(\Q)\backslash \mathsf Z(\A)\rightarrow \C_p^\times.
\]
%Moreover, if $\chi\in\CX(\varepsilon)$, $\Q(\chi)\subset {\sf E}$,  and ${\sf w}_p$ is trivial  on $Z_0$, then
%$\varepsilon$ extends to a character $\chi: \mathsf Z(\Q)\backslash \sf Z(\A)\rightarrow E^\times$ , we have an isometric isomorphism\[  %\label{isom}\con(\mathsf Z^{\flat,\z_0}, E)\rightarrow \con(\mathsf Z,E)(\varepsilon), \quad f\mapsto \chi^\flat\cdot f,\]where $\con(\mathsf Z^{\flat,\z_0}, E)$ denotes the Banach space of all $E$-valued continuous functions on $\mathsf Z^{\flat,\z_0}$.

%Note that for every  automorphic character $\chi\in \mathcal X^{\mathrm{aut}}$ that is algebraic, $\chi^\flat$ descends to a character of $\sf Z^{\flat,{\z_0}}$ if and only if it is unramified on ${\sf Z}(\A^{\infty p}) Z_0$. 
%\subsection{Modular symbols} \label{ssec1.1}

\subsection{The nearly ordinary part} \label{ssec1.1}
Throughout the Introduction we assume that $\mathsf G$ is connected and reductive, $\imath$ is injective so that $\dot{\sf G}$ is viewed as an algebraic subgroup of $\sf G$, and $\dot{\sf G}\cap \sf A$ is zero-dimensional. Here $\mathsf A$ denotes the largest central split torus in $\mathsf G$. % which may or may not be connected. %Denote by $A_{\sf G}$ the split component of the center of ${\sf G}(\R)$.   
 Fix a 
%maximal 
closed subgroup $K_\infty$ of $\mathsf G(\R)$ of the form
\[
K_\infty =\mathsf A(\R)^\circ\cdot K_\infty',
\]
where $K_\infty'$ is a maximal compact subgroup of ${\sf G}(\R)$. 
Put \be \label{Gnatural}
{\sf G}^\natural :=K_\infty^\natural \times \mathsf G(\A^\infty)
\ee 
%such that $K_\infty$ is compact-mod-center in the sense that $K_\infty/(K_\infty\cap Z_\infty)$ is compact, where $Z_\infty$ denotes the center of $\mathsf G(\R)^\circ$.
and 
\[
  \mathscr X:=\mathscr X_{\mathsf G, K_\infty}:=\mathsf G(\A)/K_\infty^\circ = (\mathsf G(\R)/K_\infty^\circ )\times \mathsf G(\A^\infty).
\]
For every open compact subgroup $K$ of $\mathsf G^\natural$, define the topological space 
\[
  S^{\mathsf G}_{K}:=\mathsf G(\Q)\backslash \mathscr X/ K.
\]

Let $\mathsf V$ be a finite-dimensional $\mathsf E$-vector space carrying a geometrically irreducible algebraic representation of 
$\mathsf G_\mathsf E$. Here %$\mathsf G_\mathsf E=\mathsf G\times_{\mathrm{Spec}(\mathbb Q)}\mathrm{Spec}( \mathsf E)$, and %similar notation will be used without further explanation; 
``geometrically irreducible" means that the representation of $\mathsf G(\overline \Q)$ on $\overline{\Q}\otimes_{\mathsf E}  \mathsf V$ is irreducible.
Specifically, $\mathsf V$ is a representation of $\mathsf G(\Q)\subset \mathsf G_\mathsf E(\mathsf E)$.
Define 
\[ %\label{spacehgv}
  \oH_\Phi^i(\mathsf G, \mathsf V):=\varinjlim_{K} \oH_\Phi^i(S^{\mathsf G}_{K}, \mathsf V_{[K]}),\qquad (i\in \Z),
\]
 where $K$ runs over all open compact subgroups of $\mathsf G^\natural$, 
and 
$\mathsf V_{[K]}$
is a sheaf of $\mathsf E$-vector spaces over $S^{\mathsf G}_{K}$ as defined in \eqref{sheafm}. Here the subscript ``$\Phi$" indicates a support condition for the cohomology group (see Section \ref{phicon}). 
Under the right translations, $\oH_\Phi^i(\mathsf G, \mathsf V)$ is a smooth representation of $\mathsf G^\natural$ over $\mathsf E$.
%(the representation is admissible because we assume that condition \eqref{conphi} holds).

Set 
\[ %\label{gnp}
{\sf G}^{\natural, p}: =K_\infty^\natural\times \mathsf G(\A^{ \infty p})\quad \textrm{so that}\quad {\sf G}^{\natural}={\sf G}^{\natural, p}\times G. 
\]
In this Introduction and in the next section we assume that $\p$ is a parabolic subalgebra of $\g$. Let $P\subset G$ denote the normalizer of $\p$, which is a parabolic subgroup of $G$ whose Lie algebra equals $\p$. Denote by $N$ the unipotent radical of $P$ and put $L:=P/N$. The Lie algebras of $N$ and $L$ are respectively denoted by $\n$ and $\l$. 

For every representation $J'$ of $\CG\times G$, where $\CG$ is a certain totally disconnected topological group, write
\be\label{psmooth}
J'_{\p\textrm{-sm}}
\ee
for the space of all vectors in $J'$ that are fixed by some groups of the form $\CK\times \fP$, where $\CK$ is an open subgroup of $\CG$, and $\fP$ is a closed subgroup of $G$ with Lie algebra $\p$. Similar notation with $\p$ replaced by other Lie subalgebras will be used without further explanation.  %({\sf G}^{\natural, p}\times P)$-smooth, namely fixed by some  open subgroups of $\mathsf G^{\natural, p}\times P$. 
For every smooth representation $J$ of $\mathsf G^\natural$, define 
\[
\CB_P(J):=\left( \left(\widehat J\right)_{\p\textrm{-sm}}\right)^N,
\]
which is a smooth representation of ${\sf G}^{\natural, p}\times L$. Here and henceforth, unless otherwise specified, we use a superscript group or Lie algebra to indicate the invariant space. %, and use a subscript group or Lie algebra to indicate the co-invariant space. 
By the second adjointness theorem of Bernstein and Casselman (see \cite[Theorems 6.4 and 0.1]{Be87}),  $\CB_P(J)$ is isomorphic to the Jacquet module of $J$ attached to a parabolic subgroup of $G$ opposite to $P$. % (\cf Lemma \ref{iden0}).  
%Since the Jacquet module of every admissible smooth representation is also an admissible smooth representation, we know that $\CB_P(J)$ is an admissilbe smooth representation of $\mathsf G^{\natural, p} \times L$, whenever $J$ is admissible as a smooth representation of $\sf G^{\natural}$. 

Denote by $T$  the largest central torus in $L$, and by $T^\dag$ the monid consisting of all $t\in T$ such that 
\[
   \abs{\alpha(t)}_p\geq 1
\]
for every character $\alpha: T\rightarrow  \C_p^\times$ that occurs in some irreducible subquotient of the adjoint representation of $P$ on  $ \C_p\otimes \n$ (these irreducible subquotients descend to  representations of $L$). Here and henceforth, 
we often omit the obvious base ring from the notation of a tensor product and hope this causes no confusion. For example, $ \C_p\otimes \n:=\C_p\otimes_{\Q_p} \n$. 

Set $V:=E\otimes \sf V$, which is naturally a representation of $G\subset \mathsf G_{\mathsf E}(E)$. Then $T$ acts on $V^\n$ through a character, to be denoted by $\alpha_V: T\rightarrow E^\times$. 
 %Here and thereafter, $\abs{\,\cdot\,}_p$ denotes the absolute value on $\C_p$ such that $|p|_p = p^{-1}$. 

Now we assume that $\Phi$ satisfies the usual condition \eqref{conphi} so that the smooth representation $\oH_{\Phi}^{i}({\mathsf G}, \mathsf V)$ (and hence $\CB_P(\oH_{\Phi}^{i}({\mathsf G}, \mathsf V))$) is admissible.  The following result is essentially due to Hida (\cf \cite{Hi93, Hi95}), and we will give a proof  in Proposition \ref{hida3} in a more general setting.

\begin{prpl}\label{hidaineq}
For every character $\chi\in \Hom(T, \C_p^\times)$ that occurs  as a subquotient in \[
\C_p\otimes \CB_P(\oH_{\Phi}^{i}({\mathsf G}, \mathsf V))\otimes \RD(\g/\p) ,
\]
 one has that
\be\label{fsch}
  \abs{\chi(t)}_p\leq \abs{\alpha_V(t)}_p\qquad \textrm{for all $t\in T^\dag$. }
\ee

\end{prpl}
%Throughout the paper, we say that a character of a monoid occurs in a module if the correpondence one-dimensional 

\begin{dfnl} \label{df:NVO}
Let  $\chi\in \Hom(T, \C_p^\times)$ be a character satisfying \eqref{fsch}. It is said to be nearly $V$-ordinary if the equality holds in \eqref{fsch} for all $t\in T^\dag$ (and hence for all $t\in T$). %Otherwise it is said to be $V$-supersingular.
\end{dfnl}

Let $\mathscr H$ be a ${\sf G}^{\natural, p}\times L$-subrepresentation of $\CB_P(\oH_{\Phi}^{i}({\mathsf G}, \mathsf V)) \otimes \RD(\g/\p)$ such that all characters of $T$ that occur in $\C_p\otimes \mathscr H$ as  subquotients are nearly $V$-ordinary. 
This is the nearly ordinary part of the classical automorphic cohomology space that has been mentioned before. 

\subsection{Relative cohomologies and relative completed cohomologies}\label{secrc}

In what follows we  briefly review the notions of relative cohomologies and relative completed cohomologies. See Section \ref{sec:RCC} for more details. Let $R$ be a commutative ring with identity, and $\BG$ an open subgroup of $\sf G^\natural$. For each monoid  $H^+$, write $R[H^+]$ for the monoid  algebra of $H^+$ with coefficients in $R$.  As an $R$-module, it is free with basis $H^+$. 

For every $R[\mathsf G(\Q)\times \BG]$-module $M$ and every  
compact subgroup $C$ of $\BG$, we define the formal cohomology group
\[
  \oH_\Phi^i(C,M):=\varprojlim_K \oH_\Phi^i(K,M),
\]
 where $K$ runs over open compact subgroups of $\BG$ containing $C$, the transition maps are the push-forward maps, and $\oH_\Phi^i(K,M)$ is the usual sheaf cohomology group that will be defined in \eqref{hkk}. Define the relative cohomology group
\[
  \oH_\Phi^i(\mathsf G,M)^{\la \p\ra}:=\varinjlim_{D\fP }\oH_\Phi^i(D\fP,M),
\]
where $D$ runs over open compact subgroups of ${\sf G}^{\natural,p}\cap \BG$, $\fP$ runs over compact subgroups of $ G\cap \BG$ with Lie algebra $\p$, and the transition maps are the pull-back maps. 

Write $\BN:=\{0,1,2,\dots\}$ as usual. %Let $\s$ be a Lie subalgebra  of $\p$. 
 Define the relative completed cohomology group
\be\label{defrcchom}
  \widetilde \oH_\Phi^i(\mathsf G,M)^{\la \p\supset \s\ra}:=\varinjlim_{D\fS }\varprojlim_{k\in \BN}\varinjlim_\fP \oH_\Phi^i(D\fP,M/p^k)\qquad ( M/p^k:=M/p^k M),
\ee
where $D$ runs over  open compact subgroups of ${\sf G}^{\natural,p}\cap \BG$, $\fS$ runs over compact subgroups of $ G\cap \BG$ with Lie algebra $\s$, and $\fP$ runs over compact subgroups of $ G\cap \BG$ with Lie algebra $\p$ that contain $\fS$. 

\begin{dfnl} \label{df:psmooth}
    Let $\CG$ be a topological group. An $R[\CG]$-module $M_0$ is said to be $p$-smooth if for every $k\in \BN$, some open subgroups of $\CG$ act trivially on $M_0/p^k$.
\end{dfnl}

It is crucial to note that, by using the trivial actions of small open compact subgroups of $G$,  the  relative completed cohomology group \eqref{defrcchom} is still defined when $M$ is replaced by an $R[\mathsf G(\Q)\times D\fS]$-module that is $p$-smooth as an $R[\fS]$-module, where $D$ is an arbitrary open compact subgroup  of ${\sf G}^{\natural,p}$ and $\fS$ is an arbitrary  compact subgroup  of $ G$ with Lie algebra $\s$.

 View $V$ as an $E[\mathsf G(\Q)\times {\mathsf G}^\natural] $-module with the given action of $G$ and the trivial action of $\mathsf G(\Q)\times {\mathsf G}^{\natural,p}$. Write 
 \[
 \CO:=\{x\in E\,:\, \abs{x}_p\leq 1\}
 \] for the ring of integers of  $E$. We define the integral relative cohomology space 
\[
{\oH}_\Phi^i({\mathsf G},  V)^{\la \p\ra, \circ}:=E\otimes {\oH}_\Phi^i({\mathsf G},  \fV)^{\la \p\ra},
\]
where $\fV$ is an $\CO$-lattice of $V$ (namely a free $\CO$-submodule whose rank equals $\dim V$). This is independent of $\fV$. Similarly, for each $(E\otimes \s)$-submodule $V_0$ of $V$ (which is necessarily stabilized by some  compact subgroups of $G$ with Lie algebra $\s$), we define the integral relative completed cohomology space 
\[
\widetilde{\oH}_\Phi^i({\mathsf G},  V_0)^{\la \p\supset \s\ra, \circ}:=E\otimes \widetilde{\oH}_\Phi^i({\mathsf G},  \fV_0)^{\la \p\supset \s \ra}, 
\]
where $\fV_0$ is an $\CO$-lattice of $V_0$. This is  independent of $\fV_0$.

We will show in Section \ref{sec:NOP} that the integral relative completed cohomology space interpolates  the nearly ordinary part $\mathscr H$. More precisely, we have a commutative diagram
\be\label{thediag00}
 \begin{CD}
\widetilde{\oH}_\Phi^i({\mathsf G},  V^\n)^{\la \p\supset \s\ra,\circ} @<\widetilde \xi<<\mathscr H@> \subset >>  \widehat{\oH_\Phi^i({\mathsf G}, \sf V)}_{{\p}\mathrm{-sm}}\otimes \RD(\g/\p)\\
      @VVV      @V V  V          @VV V\\
\widetilde{\oH}_\Phi^i({\mathsf G},  V)^{\la \p\supset \s\ra,\circ} @<<<{\oH}_\Phi^i({\mathsf G},  V)^{\la \p\ra, \circ}@>  >>     {\oH_\Phi^i({\mathsf G},  V)}^{\la \p\ra},\\
            \end{CD}
\ee
where all the arrows are canonically defined. If $P=G$, then 
$\n=\{0\}$ and
\[
\widehat{\oH_\Phi^i({\mathsf G}, \sf V)}_{{\p}\mathrm{-sm}}\otimes \RD(\g/\p)=\oH_\Phi^i({\mathsf G}, \sf V),
\]
and if furthermore $\s=\{0\}$, then 
\[
\widetilde{\oH}_\Phi^i({\mathsf G},  V^\n)^{\la \p\supset \s\ra,\circ}=\widetilde{\oH}_\Phi^i({\mathsf G},  E)^{\la \g\supset \{0\}\ra,\circ}\otimes V, 
\]
which agrees with Emerton's completed cohomology space, and the map $\widetilde \xi$ agrees with his interpolation map.  

\subsection{Modular symbols} \label{ssec:MS}

We now review some versions of modular symbols that are considered in this paper. See Section \ref{secpint} for details. %Let $\dot{ \mathsf G}$ be an algebraic subgroup of $\mathsf G$ defined over $\Q$ whose intersection with $\sf A$ is zero dimensional. %Fix a close subgroup of $\dot{\sf G}(\R)$ of  the form   $\dot K_\infty:=\dot{\mathsf A}(\R)^\circ \cdot \dot K_\infty\cap \dot{\mathsf G}(\R)$, where $\dot{\sf A}$ is a subtorus of $\sf A\cap \dot{\sf G}$, and $ \dot K_\infty$ is a compact subgroup of $K_\infty \cap \dot{\sf G}(\R)$. 

Set $\dot K_\infty:=K_\infty \cap \dot{\sf G}(\R)$.  
By \cite[Proposition 8.1]{Sun15}, it equals $ K_\infty'\cap \dot{\sf G}(\R)$ and is thus compact.  %$\dot{\mathsf A}(\R)^\circ \cdot \dot K_\infty'$, where $\dot{\mathsf A}$ is the identity connected component of    $\dot{\mathsf G} \cap \mathsf A$, and $\dot K_\infty'=\dot{\sf G}(\R)\cap K_\infty'$.  %Then we have a natural topological embedding \[\imath: \dot{\mathscr X}\to \mathscr X, \]where\[ \dot{\mathscr X}:=(\dot{\mathsf G}(\R)/\dot K_\infty^\circ)\times \dot{ \mathsf G}(\A^\infty).\]
With the fixed closed group $\dot K_\infty$, we define various cohomology groups for $\dot{\sf G}$, as we have defined for $\sf G$ in Section \ref{secrc}. Similarly these cohomology groups are also defined for ${\sf Z}$, with the fixed closed subgroup of 
$\sf Z(\R)$  taken to be the group 
\be\label{maximalkz}
 K_{\mathsf Z,\infty}:= \sf A_{\sf Z}(\R)\cdot (\textrm{the maximal compact subgroup of $\sf Z(\R)$}),
  \ee
  where $\sf A_{\sf Z}$ is the maximal split torus in $\sf Z$.  
 % K_{\mathsf Z,\infty$ itself. 
 In particular, we have an identification
\[
\widetilde{\oH}^0(\mathsf{Z},  E)^{\la \z\supset \z_0 \ra, \circ}=\con(\mathsf Z, E)_{\z_0\mathrm{-sm}}.
\]

%Let $\sf Z$ be an algebraic torus defined over $\Q$, with an algebraic homomorphism $\jmath: \dot{\sf G}\rightarrow \mathsf Z$ defined over $\Q$. % whose kernel contains $\dot{\sf A}$. 
For every algebraic character $\sf w \in \mathcal X^{\mathrm{alg}}$ that  is  defined over a subfield  $\sf E'$ of $\overline \Q$, write $\mathsf E'_{\sf w}:=\mathsf E'$ for the corresponding one-dimensional algebraic representation of $\sf Z_{\sf E'}$. Suppose that $\sf w \in \mathcal X^{\mathrm{alg}}$   is  defined over $\sf E$. Then the cohomology space $\oH^0({\sf Z}, {\sf E}_{\sf w})$ equals the space of all locally constant functions $f: \sf Z(\A)\rightarrow \sf E$ such that 
\[
  f(gx)=\mathsf w(g)\cdot f(x)\quad \textrm{for all }g\in \mathsf Z(\Q),\ x\in \mathsf Z(\A).
\]
Fix a functional $\lambda_{\sf V, w}\in \Hom_{\dot{\sf G}_{\sf E}}(\sf E_{\sf w}\otimes\sf V, \sf E)$. 
%Recall the Lie algebras $\z_0\subset \z$. % that contains $\z_0$.
%Likewise the relative completed cohomology space \[\widetilde{\oH}^0(\mathsf{Z},  E)^{\la \z\supset \r \ra, \circ}=\con(\mathsf Z;\C_p)_{\z_0\mathrm{-sm}}.\]

%equals the space of all continuous functions $\mathsf Z(\Q)\backslash \mathsf Z^\natural \rightarrow E$ that are invariant under some groups of the form $D_{\mathsf Z}\fR$, where $D_{\mathsf Z}$ is an open compact subgroup of $\mathsf Z(\A^{\infty p})$ and $\fR$ is a compact subgroup of $\mathsf Z(\Q_p)$ with Lie algebra $\r$. 
%, and $\con(\mathsf Z(\Q)\backslash \mathsf Z(\A)/D_{\mathsf Z}\fR,E)$ is  the Banach space of all $E$-valued continuous functions on $\mathsf Z(\Q)\backslash \mathsf Z(\A)/D_{\mathsf Z}\fR$.

Similar to \eqref{Gnatural}, write 
\[ %\label{Gnatural200}
{\dot{\sf G}}^\natural :=\dot K_\infty^\natural  \times \dot{\mathsf G}(\A^\infty).
\]
Put \be\label{rmd}
  \mathrm{D}(\dot{\mathsf G}):=\mathrm D(\dot{\mathsf G}^{\natural})\otimes \mathrm{O}(\dot{\mathsf G}(\R)/\dot K_\infty^\circ),
  \ee
  which is a one-dimensional $\Q$-vector space. Set 
  \[i_0:=\dim (\dot{\sf G}(\R)/\dot{K}_\infty^\circ).
  \]
 
%Then $\oH^0(\sf Z, \overline \Q_{\sf w})$ equals the space of all locally constant maps form 
Under a natural condition \eqref{precompact2} on $\Phi$, the classical modular symbol map is defined to be the composition of 
\begin{eqnarray}
  \label{CMS}   &&\oH^0({\sf Z}, {\sf E}_{\sf w})\times \left(\oH_\Phi^{i_0}(\mathsf{G},  \mathsf V)\otimes\RD(\dot{\mathsf G})\right)\\
     \nonumber &\xrightarrow{\textrm{pull-back}}&{\oH}^0(\dot{\mathsf G},  {\sf E}_{\sf w})\times \left(\oH_{\mathrm c}^{i_0}(\dot{\mathsf{G}},  \mathsf V)\otimes \RD(\dot{\mathsf G})\right)\\
  \nonumber  &\xrightarrow{\textrm{cup product}}& \oH_{\mathrm c}^{i_0}(\dot{\mathsf{G}},  \sf E_{\sf w}\otimes\mathsf V)\otimes  \RD(\dot{\mathsf G})\\
      \nonumber      &\xrightarrow{\lambda_{\sf V, w}}& {\oH_{\mathrm c}^{i_0}(\dot{\mathsf{G}},  \sf E)}\otimes \RD(\dot{\mathsf G})\\
     \nonumber &\xrightarrow{\textrm{pairing with the fundamental class}}& \sf E.
\end{eqnarray}
Here the subscript ``$\mathrm c$" indicates the cohomology with compact support. 

Write  
\[ %\label{decdg}
  \mathrm{D}(\dot{\mathsf G})=\mathrm{D}(\dot{\mathsf G},\dot \p)\otimes \RD(\g/\p),
\]
where  
\be\label{dgp}
\mathrm{D}(\dot{\mathsf G},\dot \p):=\mathrm D(\dot{\mathsf G}^{\natural,p})\otimes \RD(\dot \p)\otimes \mathrm{O}(\dot{\mathsf G}(\R)/\dot K_\infty^\circ),
\ee
and
\[
{\dot{\sf G}}^{\natural,p} :=\dot K_\infty^\natural  \times \dot{\mathsf G}(\A^{\infty p}).
\] 
Generalizing the classical modular symbol map, in Section \ref{secpint} we will define modular symbols for all the spaces occurring in the diagram \eqref{thediag00}. In particular, we have the stable modular symbol map
\be \label{stablems}
{\mathscr M}_{\lambda_{\sf V, w}} : \oH^0({\sf Z}, {\sf E}_{\sf w})\times \left(\widehat{\oH_\Phi^{i_0}(\mathsf{G},  \mathsf V)}_{\p-\mathrm{sm}}\otimes \RD(\dot{\mathsf G})\right)\rightarrow \sf E
\ee
which extends the classical modular symbol map \eqref{CMS}. Recall that $\s\supset \imath(\dot \p)$. Then we also have the relative completed  modular symbol map
\be\label{rcms0}  
\widetilde{\mathscr M}_{\lambda_0}: \widetilde{\oH}^0(\mathsf{Z},  E)^{\la \z\supset  \z_0 \ra, \circ} \times \left(\widetilde{\oH}_\Phi^{i_0}(\mathsf{G},  V^\n)^{\la \p\supset  \s  \ra, \circ} \otimes \RD(\dot{\mathsf G},\dot \p)\right)\rightarrow E,
\ee
where $\lambda_0\in \Hom_{E\otimes \dot \p}(V^\n,E)$. Both factors in  the domain of the bilinear map \eqref{rcms0} are  naturally   locally convex topological vector spaces over $E$ under the direct limit topologies. An important fact is that the bilinear map \eqref{rcms0} is separately continuous.

Assume that 
\be\label{w}
 \textrm{${\sf w}_p$ is trivial  on $Z_0$}.
\ee
%Here we still use $\sf w$ to denote the map $\z\rightarrow  E$ induced by $\sf w$, and similar convention will be used without further explanation. 
Then we have an embedding
\be\label{embpheck00}
  \oH^0({\mathsf Z}, {\sf E}_{\mathsf w})\rightarrow \widetilde{\oH}^0(\mathsf{Z},  E)^{\la \z\supset  \z_0 \ra, \circ}=\con(\mathsf Z, E)_{\z_0\mathrm{-sm}}, \qquad f\mapsto f\cdot \mathsf w_p^{-1}. 
\ee
We have natural inclusions
\be\label{emblambda}
\Hom_{\dot{\mathsf G}_{\mathsf E}}(\mathsf E_{\mathsf w}\otimes\mathsf V, \mathsf E)\subset \Hom_{\dot{ G}}(E_{\mathsf w}\otimes V,  E)\subset \Hom_{E\otimes \dot \p}( V,  E).
\ee
In particular, $\lambda_{\sf V, w}$ is also viewed as an element of $\Hom_{E\otimes \dot \p}( V,  E)$.

 The following theorem asserts that the relative completed modular symbols interpolate the stable modular symbols on the nearly ordinary part $\mathscr H$, for various weights $\sf w$ satisfying \eqref{w}. 
\begin{thml}[Theorem \ref{thmc}]  \label{thmcint}
Suppose that  $\sf w \in \mathcal X^{\mathrm{alg}}$   is  defined over $\sf E$, ${\sf w}_p$ is trivial on $Z_0$, and  $\lambda_{\sf V, w}|_{V^\n} = \lambda_0$.
Then the diagram
\be\label{mainthmcd}
 \begin{CD}
\oH^0(\mathsf Z, \mathsf E_{\mathsf w})\times \left(\mathscr H\otimes \RD(\dot{\mathsf G},\dot \p)\right)@> {\mathscr M}_{\lambda_{\sf V, w}}>>  \mathsf E\\
            @V \eqref{embpheck00}\textrm{ and }\widetilde{\xi} VV          @VV \subset V\\
\con(\mathsf Z, E)_{\z_0\mathrm{-sm}}\times \left(\widetilde{\oH}_\Phi^{i_0}(\mathsf{G},  V^\n)^{\la \p\supset \s \ra, \circ} \otimes \RD(\dot{\mathsf G},\dot \p)\right) @> \widetilde{\mathscr M}_{\lambda_0}  >>     E\\
            \end{CD}
\ee
commutes.
\end{thml}

% Note that the Lie subalgebra $\z_0\subset \z$ is algebraic in the sense that it equals the Lie algebra of an algebraic torus $Z_{\z_0}\subset \mathsf Z(\Q_p)$. 
For every $\phi\in \mathscr H\otimes \RD(\dot{\mathsf G},\dot \p)$, $\widetilde{\mathscr M}_{\lambda_0}(\,\cdot\, , \widetilde{\xi}(\phi))$ is a continuous linear functional on 
$\widetilde{\oH}^0(\mathsf{Z},  E)^{\la \z\supset  \z_0 \ra, \circ} = \con(\mathsf Z, E)_{\z_0\mathrm{-sm}}$, 
to be denoted by $\mathcal L_{ \lambda_0\otimes \phi}$. When $\phi$ is appropriately choosen, this often yields interesing $p$-adic L-functions. 

\begin{comment}
\begin{remarkl}
In view of the isometry \eqref{isom}, for an algebraic character $\chi\in \CX(\varepsilon)$ of weight ${\sf w}^{-1}$ with $\Q(\varepsilon)\subset {\sf E}$, the continuous linear functional
\[
\con({\sf Z}^{\flat, \z_0}, E)\to E, \quad f\mapsto \mathcal L_{ \lambda_0\otimes \phi}(\chi\cdot f)
\]
 defines an $E$-valued $p$-adic measure on  $\mathsf Z^{\flat,\z_0}$. 
\end{remarkl}
\end{comment}

\begin{remarkl}
It is also possible to use Theorem \ref{thmcint} to construct  multivariable $p$-adic L-functions for Hida families. We hope to carry out this construction in a future work. 
%Suppose that $\s\supset \n$ and set $\h:=\s/\n$. Let $L_\s$ denote the normalizer of $\s/\n$ in $L$. Then the space  \be\label{ninv}  E\otimes  \varinjlim_{D\fH} \varprojlim_{k\in \BN} \varinjlim_{\fL}{\oH}_{\Phi,P}^{i_0}(D \fL,  \CO/p^k) \ee is naturally a representation of $\mathsf G^{\natural, p}\times L_\s$ over $E$, where $D$ runs over all open compact subgroups of $\mathsf G^{\natural, p}$, $\fH$ runs over all compact subgroups of $L$ with Lie algebra $\h$, and $\fL$ runs over all open compact subgroups of $L$ containing $\fH$, and ${\oH}_{\Phi,P}^{i_0}$ indicates the parabolic cohomology that will be introduced in Definition \ref{defparaboliccoh}. %The space \eqref{ninv} has a natural homomorphism to   
 % \be\label{ninv2} \left(\widetilde{\oH}_\Phi^{i_0}(\mathsf{G},  V^\n)^{\la \p\supset \s \ra, \circ}\right)^N  \ee is naturally a representation of $\mathsf G^{\natural, p}\times L_\s$ over $E$. When $G$ is quasi-split and $P$ is a Borel subgroup of $G$, we may use the representation \eqref{ninv} to construct the ordinary part of the eigenvariety of  $\mathsf G$.  Theorem \ref{thmcint} can be applied to this  ordinary part to construct  multivariable $p$-adic L-functions for Hida families. To keep this paper to a reasonable length, we will leave the details of this construction to a future work. 
\end{remarkl}

This paper is organized as follows. In Section \ref{sec:CLPL} we outline the general formalism for the rationality of complex L-functions and construction of 
$p$-adic L-functions using relative completed cohomologies and modular symbols, and summarize the main results for three families of examples. In
Section \ref{sec:RCC} we develop the basic theory of relative cohomologies and relative completed cohomogoies, and in Section \ref{sec:PBI} we discuss the pull-back and integration of these cohomologies. In Section \ref{secpint} we compare various modular symbols and prove Theorem \ref{thmcint}. 
Sections \ref{sec:PC} and \ref{sec:NOP} are devoted to the theory of parabolic cohomologies and the nearly ordinary part $\mathscr{H}$ of the automorphic cohomology. 
In Sections \ref{sec:rs}--\ref{sec:rs3} we apply the main theory of this paper to construct the three families of nearly ordinary $p$-adic L-functions in Section \ref{sec:CLPL} in details, with explicit modifying factors and exceptional zeros and without any ramification restriction.

\section{Complex L-functions and $p$-adic L-functions} \label{sec:CLPL}

As examples of applications of Theorem \ref{thmcint}, we will consider three families of automorphic L-functions. In this section, we first develop the general formalism, and then outline the main results for these examples. See Section \ref{ssec:HR} for historical remarks and comparisons with previous results in the literature. 
The detailed construction will be given in Sections \ref{sec:rs}$-$\ref{sec:rs3}. 

Moreover,  for each of the three families,  we  supply the concrete example  of exceptional zeros of $p$-adic L-functions associated to symmetric powers of elliptic curves over $\Q$ without complex multiplication. %, which seem to be new to the best of our knowledge. 

\subsection{Period integrals and complex L-functions} \label{ssec:CLR}

Let $\sf H$ be an algebraic subgroup of $\sf G$, together with an automorphic character 
\[
\psi_{\sf H}=\otimes_\ell \psi_{\sf H,\ell}: \sf H(\Q)\backslash \sf H(\A)\rightarrow \C^\times.
\]
Set $\dot{\sf H}:=\sf H\cap \dot{\sf G}$. We assume that
\be\label{assh}
\begin{cases}
   % \dot{\sf H}\supset \dot{\sf A};&\\
    \psi_{\sf H} \textrm{ is trivial on $\dot{\sf H} (\A)$;}&\\
    \textrm{$\dot{\sf H}$ is contained in the kernel of $\jmath: \dot{\sf G}\rightarrow \mathsf Z$;}&\\
    \mathrm M(\dot{\mathsf H}(\A)\backslash \dot{\mathsf G}(\A))\neq \{0\}.
\end{cases}
\ee
Let 
$
\Pi= \otimes'_\ell \Pi_\ell
$
be an irreducible subrepresentation of the space of smooth automorphic forms on $\mathsf G(\Q)\backslash \mathsf G(\A)$. 
For the definition of the restricted tensor product, we realize  $\mathsf G$ as an algebraic subgroup of a general linear group $\GL(k)/_{\Q}$ ($k\in \BN$), and set 
\[
\mathsf H'(\Z_\ell):=\GL_k(\Z_\ell)\cap \mathsf H'(\Q_\ell)
\]
for every algebraic subgroup $\mathsf H'$ of $\mathsf G$ and every $\ell\neq \infty$. For all but finitely many $\ell\neq \infty$, we fix a nonzero vector $v_\ell^\circ\in \Pi_\ell^{\mathsf G(\Z_\ell)}$, and the restricted tensor products are defined with respect to the family $\{v_\ell^\circ \}_\ell$. 

Assume that 
\be\label{mulonemod}
 \dim \Hom_{{\mathsf H}(\A)}(\Pi, \psi_{\mathsf H})=1,
\ee
and 
the integrals  
\be\label{whit}
\lambda_{\sf H}: \Pi\rightarrow \C, \quad f\mapsto \int_{ {\mathsf H}(\Q)\backslash{\mathsf H}(\A)} f(x)\psi_{\mathsf H}^{-1}(x)\od\!x
\ee
are absolutely convergent and yield a generator of the one-dimensional space in \eqref{mulonemod}. 
Here $\od\! x$ is the right invariant Tamagawa measure. 
Then we have a decomposition 
\[
  \lambda_{\sf H}=\otimes_\ell \lambda_{\mathsf H, \ell},\qquad \lambda_{\mathsf H, \ell}\in \Hom_{{\mathsf H}(\Q_\ell)}(\Pi_\ell, \psi_{\mathsf H,\ell}),
\]
such that   $\lambda_{\mathsf H, \ell}(v_\ell^\circ)=1$ for all but finitely many $\ell \neq \infty$.

 Fix an ${\sf E}^\times$-valued character $\varepsilon= \otimes_{\ell\neq\infty}  \varepsilon_{\ell}$ as in \eqref{ram}. In this subsection and the next one we assume that $Z_0=Z$. 
Suppose that meromorphic continuations and normalizations of the integrals
\[
\int_{\dot{\mathsf H}(\Q_\ell)\backslash\dot{\mathsf G}(\Q_\ell)} \chi_\ell(\jmath(g)) \cdot\la \lambda_{\mathsf H,\ell},g.f\ra\od\!\mu( g),\qquad \chi_\ell\in \mathcal X_\ell, \, f\in \Pi_\ell, \, \mu\in \mathrm M(\dot{\mathsf H}(\Q_\ell)\backslash\dot{\mathsf G}(\Q_\ell))
\]
 yield the ``normalized zeta integrals"
\be\label{nz}
  {\mathcal P}_\ell^\circ : \mathcal X_\ell\times\left(\Pi_\ell \otimes \mathrm M(\dot{\mathsf H}(\Q_\ell)\backslash\dot{\mathsf G}(\Q_\ell))\right) \rightarrow \C
\ee
%for all $\ell$, where $\dot{\mathsf H}$ is a certain algebraic subgroup of $\dot{\mathsf G}$ such that $\mathrm M(\dot{\mathsf H}(\A)\backslash \dot{\mathsf G}(\A))\neq \{0\}$, $A_v:=\dot{\mathsf A}(\R)^\circ $ when $v=\infty$ and is the trivial group  when $v\neq \infty$, and $\mathrm M$ indicates the space of right invariant complex Borel measures. 
with the following properties:
\begin{itemize}
\item 
  if $\ell=\infty$, then ${\mathcal P}_\ell^\circ$ is holomorphic in the first variable, linear in the second variable, and continuous;
 \item 
  if $\ell\neq \infty$, then ${\mathcal P}_\ell^\circ$ is algebraic in the first variable and linear in the second variable;
\item it holds that \be\label{equip}
  {\mathcal P}_\ell^\circ(\chi_\ell, g.\phi)=\chi_\ell^{-1}(\jmath(g)) \cdot {\mathcal P}_\ell^\circ(\chi_\ell, \phi), 
\ee
for all $\chi_\ell\in \mathcal X_\ell$, $\phi\in \Pi_\ell \otimes \mathrm M(\dot{\mathsf H}(\Q_\ell)\backslash\dot{\mathsf G}(\Q_\ell))$, and $g\in \dot{\mathsf G}(\Q_\ell)$;
\item there is a family 
\[
\{\phi_\ell^\circ\in \Pi_\ell \otimes \RM(\dot{\mathsf H}(\Q_\ell)\backslash\dot{\mathsf G}(\Q_\ell))\}_{\ell\neq \infty}
\]
  such that   \[
\textrm{${\mathcal P}_\ell^\circ(\,\cdot\,,  \phi_\ell^\circ)$ takes a nonzero constant value $(\Omega_{\Pi_\ell}(\varepsilon_{\ell}))^{-1}$  on $\mathcal X(\varepsilon_{\ell})$,}
\]
where for all but finitely many $\ell\neq \infty$,  
\[
\Omega_{\Pi_\ell}(\varepsilon_{\ell})=1 \quad \text{and}\quad  \phi_\ell^\circ=v^\circ_\ell\otimes \mu_\ell^\circ.
\]
Here $\mu_\ell^\circ\in \mathrm M(\dot{\mathsf H}(\Q_\ell)\backslash\dot{\mathsf G}(\Q_\ell))$ denotes the measure with respect to which $\dot{\mathsf H}(\Z_\ell)\backslash \dot{\mathsf G}(\Z_\ell)$ has total volume $1$. 
\end{itemize}

Let $\chi=\otimes_\ell\chi_\ell\in \mathcal X(\varepsilon)$. By taking the tensor product we get a homomorphism 
\be\label{hominv}
\mathcal P^\circ_\chi:=\otimes \mathcal P_\ell^\circ(\chi_\ell,\,\cdot\,)\in \Hom_{\dot{\mathsf G}(\A)}(\Pi\otimes \mathrm M(\dot{\mathsf H}(\A)\backslash\dot{\mathsf G}(\A)), \chi^{-1}).
\ee
Here and henceforth, when no confusion is possible, we do not distinguish a complex character of a group with the vector space $\C$ carrying the corresponding group action.  
On the other hand, under certain growth assumptions, we get the global period integral  map
\[
\begin{array}{rcl}
   \mathcal P_\chi:  \Pi \otimes \mathrm M(\dot{\mathsf G}(\Q)\backslash\dot{\mathsf G}(\A))&\rightarrow  & \chi^{-1},\\
     f\otimes \mu &\mapsto &\int_{\dot{\mathsf G}(\Q)\backslash\dot{\mathsf G}(\A)}\chi(\jmath(g)) \cdot f(g)\od\!\mu(g),
\end{array}
  \quad 
\]
which belongs to the hom space in \eqref{hominv}. By fixing the right invariant Tamagawa measure on $\dot{\mathsf H}(\A)$, and the counting measure on $\dot{\mathsf G}(\Q)$  we have identifications
\be \label{IDmeas}
\mathrm M(\dot{\mathsf H}(\A)\backslash\dot{\mathsf G}(\A))=\mathrm M(\dot{\mathsf G}(\A))=\mathrm M(\dot{\mathsf G}(\Q)\backslash\dot{\mathsf G}(\A)). 
\ee
Suppose that certain theory of ``unfolding the global period integrals" yields a constant $\mathcal L_\Pi(\chi)\in \C$ such that
\[
\mathcal P_\chi=\mathcal L_\Pi(\chi)\cdot  \CP_\chi^\circ. 
\]
The function $\chi\mapsto \mathcal L_\Pi(\chi)$ on  $\mathcal X(\varepsilon)$ is the complex L-function that we are interested in.

\subsection{Rationality of special values} \label{ssec:RSV}
Now we consider the case when $\Pi$ is cohomological and its basic arithmetic properties have been understood. More precisely, we assume that
\begin{itemize}
    \item the  total relative Lie algebra cohomology 
    \[
    \oH^{\bullet}(\g_\C,K_\infty^\circ; \mathsf V\otimes \Pi_\infty)\qquad (\g_\C\textrm{ denotes the Lie algebra of $\mathsf G(\C)$})
        \]
        is nonzero (in this case we say that $\C\otimes \sf V^\vee$ is the coefficient system of $\Pi$);
        \item 
         there is a natural injective $\mathsf G^\natural$-homomorphism
        \be\label{embcoh}
         \oH^{\bullet}(\g_\C,K_\infty^\circ; \mathsf V\otimes \Pi)=\oH^{\bullet}(\g_\C,K_\infty^\circ; \mathsf V\otimes \Pi_\infty)\otimes (\otimes'_{\ell\neq \infty} \Pi_\ell)\rightarrow\C\otimes  \oH_\Phi^{\bullet}(\mathsf G,\mathsf V)
        \ee
        whose image is defined over $\sf E$;
        \item for each $\ell\neq \infty$, by certain local theory the functional $\lambda_{\sf H,\ell}$ yields an $\sf E$-form $\Pi_{\ell}(\sf E)$ of the representation $\Pi_\ell$,
        such that $\phi_\ell^\circ \in \Pi_{\ell}(\sf E)\otimes \RD(\dot{\mathsf H}(\Q_\ell)\backslash\dot{\mathsf G}(\Q_\ell))$ and 
        $\Omega_{\Pi_\ell}(\varepsilon_\ell)\cdot \CP^\circ_\ell$ is ${\rm Aut}(\C/{\sf E})$-equivariant when restricted to
        \[\mathcal X(\varepsilon_\ell)\times\left(\Pi_\ell \otimes \mathrm M(\dot{\mathsf H}(\Q_\ell)\backslash\dot{\mathsf G}(\Q_\ell))\right).\]
%is defined over ${\sf E}(\varepsilon_\ell)$ for all $\chi_\ell \in \CX(\varepsilon_\ell)$.
       \end{itemize}
Comparison of the $\sf E$-forms via the homomorphism \eqref{embcoh}  then yields an $\sf E$-form 
\[
\oH^{i}(\g_\C,K_\infty^\circ; \mathsf V\otimes \Pi_\infty)({\sf E})
\subset\oH^{i}(\g_\C,K_\infty^\circ; \mathsf V\otimes \Pi_\infty),\quad i \in \Z.
\]
Write 
\be \label{Piinf'}
\Pi_\infty':=\oH^{i_0}(\g_\C,K_\infty^\circ; \mathsf V\otimes \Pi_\infty)\quad \text{and} \quad \Pi_\infty'({\sf E}):=\oH^{i_0}(\g_\C,K_\infty^\circ; \mathsf V\otimes \Pi_\infty)({\sf E}).
\ee
Let $\mathcal X_\infty^\mathrm{tor}\subset \mathcal X_\infty$ denote the subgroup of the finite order characters. Suppose that we are given a ``normalized archimedean modular symbol" map 
\be \label{NAMS}
\widehat{\mathcal P}_\infty^\circ : \mathcal X^{\mathrm{tor}}_\infty\times\left(\Pi_\infty' \otimes \mathrm M 
 (\dot{\mathsf H}(\R))^\vee\otimes \mathrm{O}(\dot{\mathsf G}(\R)/\dot{K}_\infty^\circ)\right)\rightarrow \C,
\ee
that is linear in the second variable and  that
\[
  \widehat{\mathcal P}_\infty^\circ (\varepsilon_\infty, g.\phi)=\varepsilon_\infty(\jmath(g)) \cdot  \widehat{\mathcal P}_\infty^\circ (\varepsilon_\infty, \phi), 
\]
for all $\varepsilon_\infty \in \mathcal X_\infty^{\mathrm{tor}}$,  $\phi$ in the second factor of the domain, and $g\in \dot{K}_\infty^\natural$. 
Here  $\dot{K}_\infty^\natural$ acts trivially on $\mathrm M 
 (\dot{\mathsf H}(\R))^\vee$. We assume that the nonvanishing hypothesis holds, namely, for some vector $\phi$ is the second factor of the domain,   $\widehat{\mathcal P}_\infty^\circ(\,\cdot\, , \phi)$ is nowhere vanishing on $\mathcal X^{\mathrm{tor}}_\infty$.

%For every linear algebraic group $\dot{\sf H}'$ over $\Q$, we say that an element of  
We say that a right invariant measure in $\mathrm M 
 (\dot{\mathsf H}(\R))$ is rational if its product with every element in $\mathrm D(\dot{\mathsf H}(\A^\infty))$ is a rational multiple of the right invariant Tamagawa measure on $\dot{\mathsf H}(\A)$. All  such measures constitute  a  $\Q$-form of $\mathrm M 
 (\dot{\mathsf H}(\R))$, to be denoted by $\RD 
 (\dot{\mathsf H}(\R))$.
%Let $\RD   (\dot{\mathsf A}(\R)\backslash \dot{\mathsf H}(\R))$ denote the space of measures in $\mathrm M (\dot{\mathsf A}(\R)\backslash \dot{\mathsf H}(\R))$ that are quotients of rational right invariant measures on $\dot{\mathsf H}(\R)$ and $\dot{\mathsf A}(\R)$. 
 Fix an element \[
 \widehat \phi^\circ_\infty\in \Pi_\infty'({\sf E})\otimes  \RD  
 (\dot{\mathsf H}(\R))^\vee\otimes \mathrm{O}(\dot{\mathsf G}(\R)/\dot{K}_\infty^\circ)
 \]
such that 
\[
\widehat{\mathcal P}_\infty^\circ(\varepsilon_\infty,  \widehat \phi_\infty^\circ)\neq 0\quad \textrm{for all }\varepsilon_\infty\in \mathcal X^{\mathrm{tor}}_\infty.
\]

\begin{dfnl} \label{df:period}
    The constants \[
    \Omega_{\Pi}(\varepsilon_\infty):=\left(\widehat{\mathcal P}_\infty^\circ(\varepsilon_\infty,  \widehat \phi_\infty^\circ)\right)^{-1}\in\C^\times, \quad \varepsilon_\infty\in \mathcal X^{\mathrm{tor}}_\infty, 
    \]
    are called the  periods of $\mathcal L_\Pi$. 
\end{dfnl}

The $\sf V$-balanced condition in the following definition is necessary for the arithmetic study of the complex L-function  $\mathcal L_\Pi$ via modular symbols. %, we make the following definition.
\begin{dfnl}  (a) A character  $\sf w\in \mathcal X^{\mathrm{alg}}$ is  said to be $\sf V$-balanced if    \[   \dim \Hom_{\dot{\sf G}(\overline \Q)}(\overline{\Q}_{\sf w}\otimes {\sf V}, \overline{\Q})=1.    \]  

\noindent (b)
A character $\chi_\infty\in\mathcal X_\infty$ or $\chi\in\mathcal X^{\mathrm{aut}}$ is $\sf V$-balanced if it is algebraic and the inverse of its weight  is $\sf V$-balanced. 
\end{dfnl}

Suppose that $\sf w\in \mathcal X^{\mathrm{alg}}$ is $\sf V$-balanced and defined over $\sf E$, and that  $\lambda_{\sf V, w}\in \Hom_{\dot{\mathsf G}_{\mathsf E}}(\mathsf E_{\mathsf w}\otimes\mathsf V, \mathsf E)$ is a generator.  
Write $\C_{\sf w_\infty}:=\C$ for the representation of $\sf Z(\R)$ corresponding to the character ${\sf w}_\infty$. For every character $\chi_\infty\in \mathcal X_\infty$ that is algebraic of weight ${\sf w}^{-1}$, we define the archimedean modular symbol map 
%\[\widehat{\mathcal P}_\infty^\lambda(\chi_\infty,\,\cdot\,) : \oH^0(\z_\C,\mathsf Z(\R)^\circ; \C_{\mathsf w}\otimes \chi_\infty) )\times \oH^{i_0}(\g_\C,K_\infty^\circ; \mathsf V\otimes \Pi_\infty) \otimes \mathrm M (\dot{\mathsf H}(\R))^\vee\otimes \mathrm{O}(\dot{\mathsf G}(\R)/\dot{K}_\infty^\circ)\rightarrow \C\]
to be the composition of 
\begin{eqnarray}
   % &&\mathcal \oH^{i_0}(\g_\C,K_\infty^\circ; \mathsf V\otimes \Pi_\infty) \otimes \mathrm M  (\dot{\mathsf H}(\R))^\vee\otimes \mathrm{O}(\dot{\mathsf G}(\R)/\dot{K}_\infty^\circ)\\
 %&\xrightarrow{\mathsf w_\infty\cdot \chi_\infty}&
\label{AMS} &\widehat{\mathcal P}_\infty^{\lambda_{\sf V, w}} : &\oH^0(\z_\C,K_{\mathsf Z,\infty}^\circ; \C_{\mathsf w_\infty}\otimes \chi_\infty) \\
\nonumber &&\times\left(\Pi_\infty' \otimes \mathrm M 
 (\dot{\mathsf H}(\R))^\vee\otimes \mathrm{O}(\dot{\mathsf G}(\R)/\dot{K}_\infty^\circ)\right)\\
\nonumber &\xrightarrow{\textrm{pull-back}}&\oH^0(\dot \g_\C,\dot K_\infty^\circ; \C_{\mathsf w_\infty}\otimes \chi_\infty )\\
 \nonumber &&\times\left(\oH^{i_0}(\dot \g_\C,\dot K_\infty^\circ; \mathsf V\otimes \Pi_\infty) \otimes \mathrm M 
 (\dot{\mathsf H}(\R))^\vee\otimes \mathrm{O}(\dot{\mathsf G}(\R)/\dot{K}_\infty^\circ)\right)\\
\nonumber &\xrightarrow{\textrm{cup product}}&\oH^{i_0}(\dot \g_\C,\dot K_\infty^\circ; (\C_{\sf w_\infty}\otimes \mathsf V)\otimes (\chi_\infty \otimes\Pi_\infty))\\
\nonumber  && \otimes \mathrm M 
 (\dot{\mathsf H}(\R))^\vee\otimes \mathrm{O}(\dot{\mathsf G}(\R)/\dot{K}_\infty^\circ)\\
\nonumber &\xrightarrow{\lambda_{\sf V, w}\otimes \mathcal P_\infty^\circ }&\oH^{i_0}(\dot \g_\C,\dot K_\infty^\circ; \C) \otimes \mathrm M 
 (\dot{\mathsf G}(\R))^\vee\otimes \mathrm{O}(\dot{\mathsf G}(\R)/\dot{K}_\infty^\circ)\\
\nonumber &\xrightarrow{\textrm{push-forward of measures}}&\oH^{i_0}(\dot \g_\C,\dot K_\infty^\circ; \C) \otimes \mathrm M 
 (\dot{\mathsf G}(\R)/\dot K_\infty^\circ)^\vee\otimes \mathrm{O}(\dot{\mathsf G}(\R)/\dot{K}_\infty^\circ)\\
\nonumber &=&\C.
\end{eqnarray}
Here $\dot \g_\C$ denotes the Lie algebra of $\dot{\mathsf G}(\C)$. 
At least in the three families of examples we are concerned, we will prove the following archimedean period relations: when $\lambda_{\sf V, w}$ is nonzero and suitably normalized,
\[
\widehat{\mathcal P}_\infty^{\lambda_{\sf V, w}}(1,\,\cdot\,)=\Upsilon_{\Pi_\infty'}(\chi_\infty)\cdot \widehat{\mathcal P}_\infty^\circ(\mathsf w_\infty\chi_\infty,\,\cdot\,)
\]
for some  constant $\Upsilon_{\Pi_\infty'}(\chi_\infty)\in \C^\times$, where 
\[
1\in \C= \C_{\mathsf w_\infty}\otimes \chi_\infty = \oH^0(\z_\C,K_{\mathsf Z,\infty}^\circ; \C_{\mathsf w_\infty}\otimes \chi_\infty)
\]
and similar notation will be used without explanation.

\begin{dfnl}
    The constants 
    \[
    \Upsilon_{\Pi_\infty'}(\chi_\infty)\in \C^\times, \qquad \textrm{$\chi_\infty\in \mathcal X_\infty$ is $\sf V$-balanced,}
    \]
    are called the modifying factors at $\infty$ for $\Pi_\infty'$.
\end{dfnl}

Similar to \eqref{embpheck00} we have a linear embedding
\be\label{embpheck002}
  \oH^{0}(\z_\C,K_{\mathsf Z,\infty}^\circ;  \C_{\mathsf w_\infty}\otimes \chi)\rightarrow \C\otimes\oH^0({\mathsf Z}, {\sf E}_{\mathsf w}), \qquad 1\mapsto  \mathsf w_\infty\cdot \chi. 
\ee
Under certain growth conditions we have the following commutative diagram, which reflects the fact that  modular symbols interpret  period integrals: 
\be\label{cdm}
\begin{CD}
\oH^{0}(\z_\C,K_{\mathsf Z,\infty}^\circ;  \C_{\mathsf w_\infty}\otimes \chi)\times 
\oH^{i_0}(\g_\C,K_\infty^\circ;  \mathsf V\otimes \Pi)\otimes \RD(\dot{\mathsf G}) @>\widehat{\mathcal P}_\infty^{\lambda_{\sf V, w}}\otimes (\otimes_{\ell\neq \infty} {\mathcal P}_\ell^\circ)>>  \C\\
            @V \eqref{embpheck002}\textrm{ and }\eqref{embcoh} VV          @VV \mathcal L_{\Pi}(\chi) V\\
(\C\otimes \oH^0(\mathsf Z,  \mathsf E_{\mathsf w}))\times (\C\otimes\oH^{i_0}_\Phi(\mathsf G, \mathsf V)\otimes \RD(\dot{\mathsf G}))@> {\mathscr M}_{\lambda_{\sf V, w}}>>  \C.
            \end{CD}
\ee
Here $\chi\in\mathcal X^{\mathrm{aut}}$ is an algebraic automorphic  character of weight $\sf w^{-1}$, and  
the identification 
\[
\RD(\dot{\mathsf G})=\RD 
 (\dot{\mathsf H}(\R))^\vee\otimes \mathrm{O}(\dot{\mathsf G}(\R)/\dot{K}_\infty^\circ)\otimes \RD(\dot{\mathsf H}(\A^\infty)\backslash \dot{\mathsf G}(\A^\infty))
\]
is used for the definition of the top horizontal arrow. Implicitly we have used the identification $\mathrm D(\dot{\mathsf H}(\R))\otimes \mathrm D(\dot{\mathsf H}(\A^\infty))=\Q $ such that the right invariant Tamagawa measure is identified with $1\in \Q$.  In particular if $\chi\in \CX(\varepsilon)$, then applying the above commutative diagram to the test vector $\widehat\phi_\infty^\circ\otimes(\otimes_{\ell\neq \infty}\phi_\ell^\circ)$, we obtain the equality
\be\label{rational}
 \mathscr M_{\lambda_{\sf V, w}}(\mathsf w_\infty\chi, \widehat \phi_\infty^\circ\otimes(\otimes_{\ell\neq \infty} \phi_\ell^\circ)) = \frac{\Upsilon_{\Pi_\infty'}(\chi_\infty)\cdot\mathcal L_{\Pi}(\chi) } {(\prod_{\ell\neq\infty}\Omega_{\Pi_\ell}(\varepsilon_\ell)) \cdot\Omega_{\Pi}(\mathsf w_\infty\chi_\infty)}.
\ee
In this case the number in \eqref{rational} belongs to $\sf E$ as long as $ {\sf w}_\infty \chi$ is $\sf E^\times$-valued.

\subsection{$p$-adic L-functions} \label{ssec:PL}
%Now we assume that there is an irreducible representation \[\Pi=\Pi_\infty\otimes (\otimes_\ell' \Pi_\ell)\qquad (\textrm{$\ell$ runs over rational primes}) \]of $\mathsf G(\A)$ that is realized as a space of smooth automorphic forms on $\mathsf G(\Q)\backslash \mathsf G(\A)$.such that 
%$\mathsf G^\natural$-subrepresentation 
%\[
%\Pi^\natural=\mathscr H_\infty\otimes \left(\otimes'_{\ell} \Pi_\ell)\right)\subset \oH^{i_0}_\Phi(\mathsf G,\mathsf V)\qquad (\textrm{$\ell$ runs over rational primes})
%\]

The homomorphism \eqref{embcoh} in degree $i_0$ naturally  extends to a homomorphism 
        \begin{eqnarray*}
           \iota_\Pi&:& 
        \widehat{ \oH^{i_0}(\g_\C,K_\infty^\circ; \mathsf V\otimes \Pi)}_{\p\mathrm{-sm}}=\Pi_\infty'\otimes \left(\widehat{\Pi_p}\right)_{\p\mathrm{-sm}}\otimes(\otimes'_{\ell\neq \infty,p} \Pi_\ell)\\
        &\rightarrow&\C\otimes  \widehat{\oH_\Phi^{i_0}(\mathsf G,\mathsf V) }_{\p\mathrm{-sm}},
        \end{eqnarray*}
        and the commutative diagram \eqref{cdm} extends to a commutative diagram 
        \be\label{cdm22}
\begin{CD}
\oH^{0}(\z_\C,K_{\mathsf Z,\infty}^\circ;  \C_{\mathsf w_\infty}\otimes \chi)\times 
\widehat{ \oH^{i_0}(\g_\C,K_\infty^\circ; \mathsf V\otimes \Pi)}_{\p\mathrm{-sm}}\otimes \RD(\dot{\mathsf G}) @>\widehat{\mathcal P}_\infty^{\lambda_{\sf V, w}}\otimes (\otimes_{\ell\neq \infty} {\mathcal P}_\ell^\circ)>>  \C\\
            @V \eqref{embpheck002}\textrm{ and }\iota_{\Pi} VV          @VV \mathcal L_{\Pi}(\chi) V\\
(\C\otimes \oH^0(\mathsf Z,  \mathsf E_{\mathsf w}))\times (\C\otimes\widehat{\oH_\Phi^{i_0}(\mathsf G,\mathsf V) }_{\p\mathrm{-sm}}\otimes \RD(\dot{\mathsf G}))@> {\mathscr M}_{\lambda_{\sf V, w}}>>  \C.\\
            \end{CD}
\ee

The space \[
  \CB_P(\Pi_p):=\left( \left(\widehat{\Pi_p}\right)_{\p\textrm{-sm}}\right)^N
  \]
  is a smooth representation of $L$ of finite length, which is  defined over ${\sf E}$ 
  %(the rationality field of $\Pi_p$, see \cite[3.1]{Cl90}) 
  by the second adjointness theorem
  (see \cite[Theorems 6.4 and 0.1]{Be87}). 
  %and the $\sf E$-structure of $\Pi_p$ induces an $\sf E$ structure $\CB_P(\Pi_p)(\sf E)$ of it. 

\begin{dfnl}
    (a) An irreducible subrepresentation $\Pi_p'$ of $\CB_P(\Pi_p)$ defined over $\overline\Q$ is called a refinement of $\Pi_p$.  

   \noindent (b) A refinement $\Pi_p'$ of $\Pi_p$ is called nearly ordinary if the  (unique) character of $T$ that occurs in $\C_p\otimes \Pi_p'(\overline\Q)
   \otimes \RD(\g/\p)$ is nearly $V$-ordinary, where $\Pi_p'(\overline\Q)$ denotes the $\overline\Q$-form  of $\Pi_p'$.  
    \end{dfnl}
%:=\Pi_p^{\mathrm{ref}}\cap\left(\CB_P(\Pi_p)({\sf E}) \otimes \RD(\g/\p)\right)$, which is the $\sf E$-form of $\Pi_p^{\mathrm{ref}}$.

Let $\Pi_p'$ be a nearly ordinary refinement of $\Pi_p$  defined over ${\sf E}$, whose ${\sf E}$-form  is denoted by $\Pi_p'({\sf E})$. Now we assume that via the homomorphism $\iota_\Pi$, $\mathscr H$ is identified with an $\sf E$-subspace of 
 \[
 \widehat{\oH^{i_0}(\g_\C,K_\infty^\circ; \mathsf V\otimes \Pi)}_{\p\mathrm{-sm}}\otimes \RD(\g/\p)
 \]
of the form
 \[
\mathscr H=\otimes'_{\ell} \mathscr H_\ell,
\]
where 
\begin{itemize}
    \item $\mathscr H_\infty=\Pi_\infty'(\sf E)$;
    \item $\mathscr H_p=\Pi_p'(\sf E) \otimes \RD(\g/\p)$;
    \item $\mathscr H_\ell=\Pi_\ell(\sf E)$ for each $\ell\neq \infty, p$.
\end{itemize}

In the rest of this section we drop the assumption that $Z_0=Z$, %assume that $\z_0 = \jmath(\dot\p)$, 
and  fix an ${\sf E}^\times$-valued character $\varepsilon = \otimes_{\ell\neq\infty}\varepsilon_\ell$ as in 
\eqref{ram}. Set $\dot P:=\dot G\cap P$ whose Lie algebra equals $\dot \p$. For the definition of $p$-adic L-functions, suppose that we are given a ``normalized refined period" map 
\[
\widehat{\mathcal P}_p^\circ : \mathcal X_p\times \left(\Pi_p' 
 \otimes \mathrm M(\dot H\backslash \dot G)\right)\rightarrow  \C \qquad(\dot H:=\dot{\mathsf H}(\Q_p))
\]
%\[\widehat{\mathcal P}_p^\circ : \mathcal X_p\times \left( \mathscr H_p \otimes \RD(\dot \g/\dot{\p})^\vee  \otimes \mathrm M(\dot H\backslash \dot G)\right)\rightarrow  \C \qquad(\dot H:=\dot{\mathsf H}(\Q_p))\]
such that
\begin{itemize}
\item 
  it is algebraic in the first variable and linear in the second variable;
\item it holds that 
\[
  \widehat{\mathcal P}_p^\circ(\chi_p, g.\phi)=\chi_p^{-1}(\jmath(g)) \cdot \widehat{\mathcal P}_p^\circ(\chi_p, \phi), 
\]
for all $\chi_p\in \mathcal X_p$, $\phi$ in the second factor of the domain, and $g\in \dot P$;
\item there is an element 
\[
\widehat \phi_p^\circ\in \Pi_p'(\mathsf E) \otimes \RD(\dot H\backslash \dot G)
\]
such that $\widehat{\mathcal P}_p^\circ(\,\cdot\,,  \widehat \phi_p^\circ)$ takes a nonzero constant value $(\Omega_{\Pi_p'}(\varepsilon_{p}))^{-1}$  on $\mathcal X(\varepsilon_{p})$, and the map  $\Omega_{\Pi_p'}(\varepsilon_{p})\cdot \widehat{\mathcal P}_p^\circ$ is $\Aut(\C/\mathsf E)$-equivariant when restricted to
        \[\mathcal X_p(\varepsilon_p)\times\left( \Pi_p' 
 \otimes \mathrm M(\dot H\backslash \dot G)\right).\] %is the constant function $1$ on $\CX_p(\varepsilon_p)$. % $\mathcal X_p^{\z_0\mathrm{-ur}}$. 
\end{itemize}

Write 
\begin{eqnarray*}
\widehat \phi^\circ&:=&\widehat \phi_\infty^\circ\otimes \widehat \phi_p^\circ \otimes(\otimes_{\ell\neq \infty,p}\phi^\circ_\ell)\\
&\in& \mathscr H\otimes   \RD
 (\dot{\mathsf H}(\R))^\vee\otimes \mathrm{O}(\dot{\mathsf G}(\R)/\dot{K}_\infty^\circ)\otimes   \RD(\dot \g/\dot{\p})^\vee\otimes \RD(\dot{\mathsf H}(\A^\infty)\backslash\dot{\mathsf G}(\A^\infty)) \\
 &=&\mathscr H\otimes   \RD
 (\dot{\mathsf G}(\R))^\vee\otimes \mathrm{O}(\dot{\mathsf G}(\R)/\dot{K}_\infty^\circ)\otimes   \RD(\dot \g/\dot{\p})^\vee\otimes \RD(\dot{\mathsf H}(\A)\backslash\dot{\mathsf G}(\A)) \\
 &=& \mathscr H\otimes   \RD
 (\dot{\mathsf G}, \dot \p).
    \end{eqnarray*}

    \begin{dfnl} \label{df:padicL}
The continuous linear functional        
        \[
        \mathcal L_{\varepsilon \otimes  \mathscr H }:=(\mathcal L_{\lambda_0\otimes \widehat \phi^\circ})|_{\con({\sf Z}, E)(\varepsilon)}
        \]
        is called the $p$-adic L-function attached to $\varepsilon\otimes \mathscr H$.
          \end{dfnl}

Note that the $p$-adic L-function $\mathcal L_{\varepsilon \otimes  \mathscr H }$ defined above uniquely extends to a continuous linear functional $\mathcal L_{\varepsilon \otimes  \mathscr H }: \con({\sf Z}, \C_p)(\varepsilon)\rightarrow \C_p$. 

Define a compact $p$-adic Lie group
\[
\mathsf Z^{\flat,\z_0} := \varprojlim_{\fZ} {\sf Z}(\Q) \bs {\mathsf Z}(\A) / \left({\mathsf Z}(\R)^\circ\cdot D^{\mathrm{max}}_{Z_0} \cdot \fZ  \right), 
\]
where $\fZ$ runs over all open compact subgroups of $\mathsf Z(\Q_p)$. 
%Let  $\con(\mathsf Z^{\flat,\z_0}, E)$ denote the Banach space of all $E$-valued continuous functions on $\mathsf Z^{\flat,\z_0}$. 
If $\varepsilon$ is trivial, then as a Banach space $\con({\sf Z}, E)(\varepsilon)$ agrees with $\con(\mathsf Z^{\flat,\z_0}, E)$ (the space of all $E$-valued continuous functions on $\mathsf Z^{\flat,\z_0}$) and thus the $p$-adic L-fucntion $\mathcal L_{\varepsilon \otimes  \mathscr H }$ is an $E$-valued measure on $\mathsf Z^{\flat,\z_0}$. In general,  $\con({\sf Z}, E)(\varepsilon)$ is either a zero space or isometrically isomorphic to $\con(\mathsf Z^{\flat,\z_0}, E)$.  %Regarding  Definition \ref{df:padicL}, we remark that $\con({\sf Z}, E)(\varepsilon)$ can be interpreted as the space of global sections of an  $E$-line bundle over ${\sf Z}^{\flat, \z_0}$ associated to $\varepsilon$, which is the trivial bundle if $\varepsilon$ is trivial. 

    \begin{prpl}\label{defp}
If
\be\label{mulonep}
\dim \Hom_{\dot P}(\chi_p\otimes \mathscr H_p \otimes \RD(\dot \g/\dot{\p})^\vee\otimes \mathrm M(\dot H\backslash \dot G),\C)\leq 1
\ee
for all $\chi_p\in \CX_p(\varepsilon_p)$,  %\mathcal X_p^{\z_0\mathrm{-ur}}$, 
then $\mathcal L_{\varepsilon \otimes  \mathscr H }$ is independent of $\widehat \phi^\circ_p$. Similarly, if \[
\dim \Hom_{\dot{\mathsf G} (\Q_\ell)}(\chi_\ell \otimes \Pi_\ell \otimes \RD(\dot{\mathsf H}(\Q_\ell) \backslash \dot{\mathsf G}(\Q_\ell)),\C)\leq 1
\]
for all $\ell\neq \infty,\,p$ and all $\chi_\ell \in \mathcal X(\varepsilon_\ell)$, then 
 $\mathcal L_{\varepsilon \otimes  \mathscr H }$ is independent of the family 
 $\{\phi^\circ_\ell\}_{\ell\neq \infty,\,p}$. 
\end{prpl}

\begin{proof}
We only prove the first assertion. The proof of the second one is similar and will be omitted.  Without loss of generality we assume that $\sf E=\overline \Q$ so that $E=\C_p$. Suppose that $\chi = \otimes_\ell \chi_\ell\in \CX(\varepsilon)$ is a finite order character. Since all such characters are $\sf E^\times$-valued and span a dense space in  $\con(\mathsf Z,E)(\varepsilon)$, we only need to show that 
 $\mathcal L_{\varepsilon \otimes  \mathscr H }(\chi) $ is independent of $\widehat\phi_p^\circ$.

By definition, $\mathcal L_{\varepsilon \otimes  \mathscr H }(\chi)  = \widetilde{\mathscr M}_{\lambda_0}(\chi, \widetilde{\xi}(\widehat\phi^\circ))$, and by \eqref{mulonep}, there exists $c_\chi\in {\sf E}$ such that 
\[
\widetilde{\mathscr M}_{\lambda_0}\left(\chi, \widetilde{\xi}\left(\widehat\phi_\infty^\circ\otimes\varphi_p\otimes (\otimes_{\ell\neq \infty,p}\phi^\circ_\ell)\right)\right) = c_\chi\cdot \widehat\CP_p^\circ(\chi_p,\varphi_p)
\]
for all $\varphi_p \in \mathscr H_p \otimes \RD(\dot \g/\dot{\p})^\vee \otimes \RD(\dot H\backslash \dot G)$. Thus $\mathcal L_{\varepsilon \otimes  \mathscr H }(\chi)=c_\chi\cdot 
(\Omega_{\Pi_p'}(\varepsilon_p))^{-1}$ does not depend on the choice of $\widehat \phi^\circ_p$ satisfying that 
$\widehat{\mathcal P}_p^\circ(\,\cdot\,,  \widehat \phi_p^\circ)=(\Omega_{\Pi_p'}(\varepsilon_{p}))^{-1}$ on $\CX_p(\varepsilon_p)$. This proves the proposition. % assertion follows easily from this. 
\end{proof}

In conclusion, as least when the conditions of Proposition \ref{defp} are satisfied, we have defined a $p$-adic L-function for $\varepsilon\otimes \mathscr H$ which only depends on $\lambda_0$ and $\widehat\phi_\infty^\circ$. 

%\subsection{Modifying factors  and $p$-adic interpolation}

Since $\dot G$ is $P$-spherical (namely \eqref{pspherical} holds), the map \eqref{nz} for $\ell=p$ naturally extends to a map (the stable normalized zeta integral) 
\be\label{nz2}
\begin{array}{rcl}
  {\mathcal P}_p^\circ : \mathcal X_p\times\left(\left(\widehat{\Pi_p}\right)_{\p\mathrm{-sm}} \otimes \mathrm M(\dot{H}\backslash\dot{ G})\right) &\rightarrow &\C,\\
  (\chi_p, \phi\otimes \tau)&\mapsto & {\mathcal P}_p^\circ(\chi_p, \phi_\fG\otimes \tau),
  \end{array}
\ee
where $\fG$ is a sufficiently small open compact subgroup of $G$ and $\phi_\fG$ denotes the image of $\phi$ under the projection map $\widehat{\Pi_p}\rightarrow \Pi_p^\fG$. 
The map \eqref{nz2} is still   algebraic in the first variable and linear in the second variable,  and satisfies the analog of \eqref{equip} for $g\in \dot P$. Moreover, its multipliction by  $\Omega_{\Pi_p}(\varepsilon_p)$ is still $\Aut(\C/\sf E)$-equivariant. %Here the subscript ${\p\mathrm{-sm}}$ indicates the subspace of the vector that are fixed by some open subgroups of $P$. 

At least when \eqref{mulonep} holds, for each $\chi_p\in \CX_p(\varepsilon_p)$, there is a unique constant $\Upsilon_{\Pi_p'}(\chi_p)$ such that 
\be\label{defmodp}
{\mathcal P}_p^\circ(\chi_p, \, \cdot\,)|_{ \mathscr H_p \otimes \RD(\dot \g/\dot{\p})^\vee\otimes \mathrm M(\dot H\backslash \dot G)}=\Upsilon_{\Pi_p'}(\chi_p)\cdot \widehat{\mathcal P}_p^\circ(\chi_p, \, \cdot\,).
\ee

\begin{dfnl} \label{df:MFp}
    The constants 
    \[
      \Upsilon_{\Pi_p'}(\chi_p)\in \C,\qquad \chi_p\in \CX_p(\varepsilon_p), %\mathcal X_p^{\z_0\mathrm{-ur}},
    \]
    are called the modifying factors at $p$ for $\Pi_p'$.
\end{dfnl}

Applying the equality \eqref{defmodp} to the vector $\widehat \phi_p^\circ$, we know that $\Upsilon_{\Pi_p'}$ is algebraic as a function on $\CX_p(\varepsilon_p)$.         
We remark that the $p$-adic L-function and the modifying factors are defined even when  no $\sf w\in \mathcal X^{\mathrm{alg}}$ is $\sf V$-balanced. 
However, for all $\sf V$-balanced characters $\chi=\otimes_\ell \chi_\ell\in \mathcal X(\varepsilon)$, %such that $\chi^\flat$ descends to a character of $\sf Z^{\flat,\z_0}$,  
by applying the commutative 
diagram \eqref{mainthmcd} to the pair $(\sf w_\infty\cdot \chi,\widehat \phi^\circ)$ where ${\sf w}$ is the inverse of the weight of $\chi$, we get the equalities  (assuming  $\lambda_{\sf V, w}|_{V^\n}=\lambda_0$)
\begin{eqnarray}
 \label{pintp}   &&\mathcal L_{\varepsilon \otimes  \mathscr H }(\chi^\flat)\\
  \nonumber   &=&\widetilde{\mathscr M}_{\lambda_0}(\chi^\flat, \widetilde{\xi}(\widehat \phi^\circ))\\
 \nonumber   &=&
    {\mathscr M}_{\lambda_{\sf V, w}}(\sf w_\infty\cdot \chi,\widehat \phi^\circ)\qquad  \quad \textrm{by Theorem \ref{mainthmcd}}\\
  \nonumber  &=&\mathcal L_{\Pi}(\chi)\cdot \widehat{\mathcal P}_\infty^{\lambda_{\sf V, w}}(1,\widehat{\phi}_\infty^\circ)\cdot {\mathcal P}_p^\circ(\chi_p, \widehat{\phi}_p^\circ)\cdot \prod_{\ell\neq\infty, p}\CP^\circ_\ell(\chi_\ell, \phi_\ell^\circ)
  %\widehat{\mathcal P}_p^\lambda(\widehat{\phi}_p^\circ)
  \qquad  \quad \textrm{by \eqref{cdm22}}\\
    \nonumber &=&\frac{\Upsilon_{\Pi_\infty'}(\chi_\infty)\cdot \Upsilon_{\Pi_p'}(\chi_p)\cdot\mathcal L_{\Pi}(\chi)}{(\prod_{\ell\neq \infty, p}\Omega_{\Pi_\ell}(\varepsilon_\ell))\cdot \Omega_{\Pi_p'}(\varepsilon_{p})\cdot\Omega_{\Pi}(\sf w_\infty\chi_\infty)}. 
\end{eqnarray}

\begin{dfnl}
    A character $\chi': {\sf Z}(\Q)\bs {\sf Z}(\A)\rightarrow \C_p^\times$ is an exceptional zero of the $p$-adic L-function $\mathcal L_{\varepsilon \otimes  \mathscr H }$ if $\chi'=\chi^\flat$ for some ${\sf V}$-balanced character $\chi=\otimes_\ell\chi_\ell\in \CX(\varepsilon)$ such that $\Upsilon_{\Pi_p'}(\chi_p)=0$. 
\end{dfnl}

It is clear from \eqref{pintp} that all exceptional zeros of $\mathcal L_{\varepsilon \otimes  \mathscr H }$ are its zeros.

%In the rest of this subsection  we assume that $P$ is defined over $\Q_p\cap \sf E$, and \be\label{homvn}\dim \Hom_{E\otimes \dot \p}(V^\n,  E)=1. \ee Note that $V^\n$ is defined over $\sf E$ and hence the one-dimensional space in \eqref{homvn} is also defined over $\sf E$. 
 
 %by possibly e when $P$ is defined over $\Q_p\cap \sf E$. 

\begin{leml}\label{lambda01}
Suppose that $P$ is defined over $\Q_p\cap \sf E$ and 
\be\label{homvn}
\dim \Hom_{E\otimes \dot \p}(V^\n,  E)=1, 
\ee
so that  $V^\n$ is defined over $\sf E$ and the one-dimensional space in \eqref{homvn} is also defined over $\sf E$. Assume that $\lambda_0$ is a generator of  the one-dimensional space in \eqref{homvn} and is defined over $\sf E$. 
  Then for every $\sf V$-balanced algebraic character $\sf w$ of $\mathsf Z_{\sf E}$, there exists a unique element  $\lambda_{\sf V,\sf w}\in \Hom_{\dot{\sf G}_{\sf E}}(\sf E_{\sf w}\otimes {\sf V}, {\sf E})$ such  that 
 $\lambda_{\sf V, w}|_{V^\n}=\lambda_0$. 
\end{leml}

\begin{proof}
This is implied by the transversal condition. 
%The lemma follows easily from the fact that $\dot{\sf G}\cdot {\sf P}$ is Zariski dense in ${\sf G}$.  
\end{proof}

\subsection{Rankin-Selberg $p$-adic L-functions for $\GL_n \times\GL_{n-1}$}  \label{ssec:RSGL}

Let $\rk$ be a number field with adele ring $\A_\rk$.  Let $\Pi = \Pi_n\boxtimes\Pi_{n-1}$ ($n\geq 2$) be an irreducible automorphic representation
of $\GL_n(\A_\rk)\times \GL_{n-1}(\A_\rk)$ that is  regular algebraic in the sense of \cite{Cl90}. 
We assume that $\Pi_n$ is cuspidal, and  $\Pi_{n-1}$ is tamely isobaric in the sense of \cite[Section 6.2]{LLS24}. 
%isobaric, cohomological and globally generic (the last condition means that $\Pi_{n-1}$ is realized as a space of smooth automorphic forms whose Whittaker periods are not identically zero).   

Denote ${\sf G}_m:=\Res_{\rk/\Q}\GL_m$ for every positive integer $m$. Here and henceforth $\Res$ indicates the Weil restriction of scalars. 
Suppose that the 5-tuple in \eqref{five} is 
 \[
  \begin{cases}
     {\sf G}:={\sf G}_n\times {\sf G}_{n-1}; \ 
\dot {\sf G}:={\sf G}_{n-1}; \ 
{\sf Z}:={\sf G}_1;&\\
\imath: \dot{\sf G} \to {\sf G},\quad g\mapsto \left(\begin{bmatrix} g & 0 \\ 0 & 1\end{bmatrix}, g\right)\,\textrm{ is the diagnal embedding}; &\smallskip\\
 \jmath: \dot{\sf G} \to {\sf Z}\,\textrm{ is the determinant homomorphism}. &\\
 \end{cases}
 \]
%Let\[\imath: \dot{\sf G}:={\sf G}_{n-1} \to {\sf G}:={\sf G}_n\times {\sf G}_{n-1},\quad g\mapsto \left(\begin{bmatrix} g & 0 \\ 0 & 1\end{bmatrix}, g\right)\]be the diagonal embedding, and let \[\jmath: \dot{\sf G}\to {\sf Z}:={\sf G}_1\]be the determinant map. 
Then $\Pi$ is an irreducible automorphic representation of $\mathsf G(\A)$. Let $\sf P$ be a Borel subgroup of $\sf G$ as in \eqref{transP}, which is transversal to $\dot{\sf G}$ (in the sense that $\sf P\cdot \dot{\sf G}$ is open in $\sf G$). 
Suppose that $P={\sf P}(\Q_p)$. 

Suppose that $\Q(\Pi)\subset \sf E$ and $\Pi_p$  has a nearly ordinary refinement $\Pi_p'\subset \CB_P(\Pi_p)$ defined over $\sf E$. Here $\Q(\Pi)$ denotes the rationality field of $\Pi$, namely the subfield of $\C$ consisting of the elements that are fixed by all field automorphisms $\sigma$ of $\C$ such that $\Pi_\ell\otimes_{\C,\sigma} \C\cong \Pi_\ell$ for all $\ell\neq \infty$ (see \cite[3.1]{Cl90}). Similar notation for rationality fields applies to other automorphic representations.
%${\sf E} \supset \Q(\Pi, \Pi_p').$Here $\Q(\Pi, \Pi_p')$ denotes the composition of the rationality fields $\Q(\Pi)$ and $\Q(\Pi_p')$ (which are number fields), and similar notations will be used without explanation. 
The coefficient system of $\Pi$  is defined over ${\sf E}$ by \cite{Cl90} (see \cite[Theorem 2.14]{JST19}). %, and we suppose that ${\sf V}^\vee$ is its ${\sf E}$-form , so that ${\sf V}^\vee$ has the same infinitesimal character as that of $\Pi_\infty$.

The complex L-function in Section \ref{ssec:CLR} for this example is the Rankin-Selberg  L-function
$\CL_\Pi(\chi):=\oL(\frac{1}{2}, \Pi\times\chi)$, where $\chi: \rk^\times \bs\A_\rk^\times \to \C^\times$ is a Hecke character. In the previous work \cite{LLS24}, under a balanced condition on ${\sf V}$  we have proved the period relations for all the critical values of $\oL(s, \Pi\times \chi)$ when $\chi$ is of finite order. 
%The main novelty of the proof is \cite{LLSS23}, which gives a comparison between local Rankin-Selberg integrals and the period integrals  over the open orbit of the corresponding spherical subgroup acting on the flag variety. 
The same proof extends the period relations  %can be easily 
to the case that $\chi$ is an arbitrary ${\sf V}$-balanced Hecke character, and in particular the archimdean period relations give the modifying factors $\Upsilon_{\Pi_\infty'}(\chi_\infty)$ at $\infty$ in \eqref{MFinfrs}. This is the automorphic analog (due to Blasius \cite{Bl97}) of the celebrated conjecture of 
Deligne \cite{Del79}.

Suppose that $Z_0$ is trivial. %and $P$ is a Borel subgroup of $G$ such that $P\cap \dot G$ is trivial. 
Then  ${\sf Z}^{\flat,\{0\}}=\CC_\rk(p^\infty)$ is the idele class group of infinite level at $p$ defined by
\[
\CC_\rk(p^\infty):=\varprojlim_{m\in\BN\setminus\{0\}} \rk^\times \bs \A_\rk^\times / \left( (\rk_\infty^\times)^\circ \cdot \prod_{v\nmid \infty p} \CO_v^\times \cdot (1+p^m\CO_p)\right).
\]
Here and henceforth $\rk_v$ denotes the local field corresponding to a place $v$ of $\rk$, with ring of integers $\CO_v$ if $v$ is finite, and 
$\CO_p:=\CO_\rk\otimes_\Z \Z_p$ where $\CO_\rk$ is the ring of integers of $\rk$.
Denote by $\rk^{(\infty p)}$  the maximal abelian extension of $\rk$ that is  unramified outside $\infty p$. Then
$\CC_\rk(p^\infty)\cong   {\rm Gal}(\rk^{(\infty p)}/\rk)$ by class field theory.

 Applying the main theory of this paper as explained earlier, we will construct a $p$-adic L-function $\mathscr L_\Pi:=\mathcal L_{\varepsilon \otimes  \mathscr H }$, which is a continuous linear functional on $\con({\sf Z}, E)(\varepsilon)$. Here the ramification type $\varepsilon$ is assumed to be $\sf E^\times$-valued as before.   %following Definition \ref{df:padicL}.   
In particular when $\varepsilon$ is trivial,  $\mathscr L_\Pi$ is an $E$-valued $p$-adic measure on $\CC_\rk(p^\infty)$. 
We emphasize that the construction of $\mathscr L_\Pi$ does not require the existence of  ${\sf V}$-balanced characters.

%Based on these results and methods, 

In Sections \ref{ssec:OOI}--\ref{ssec:SI} we will explicitly calculate the modifying   factor  $\Upsilon_{\Pi_p'}(\chi_p)$ at $p$  as given by \eqref{MFinfrs}. %calculate the stable zeta integrals at $p$ for this example and obtain the explicit modifying   factor  $\Upsilon_{\Pi_p'}(\chi_p)$ at $p$ in full generality as given by \eqref{MFinfrs}.
This is consistent with the conjecture  given by Coates and Perrin-Riou in \cite{CPR89, Co89}.

As the first application of Theorem \ref{thmcint},  we obtain the following main result on the $p$-adic L-function for this example, which shows that 
%if there is a ${\sf V}$-balanced Hecke character then 
$\mathscr L_\Pi(\chi^\flat)$ interpolates $\oL(\frac{1}{2}, \Pi\times\chi)$ for all ${\sf V}$-balanced  Hecke characters 
$\chi\in \CX(\varepsilon)$  with explicit modifying factors at $p$, and the interpolation is consistent with the Principal Conjecture in \cite[Page 49]{CPR89}.

\begin{thml} [Theorem \ref{padicLrs}] \label{padicLrs0}
Let the notations and assumptions be as above. %If  there is a $\sf V$-balanced Hecke character in $\CX(\varepsilon)$, 
Then 
\[ % \label{padicL0}
\mathscr L_\Pi(\chi^\flat)  = \frac{\Upsilon_{\Pi_\infty'}(\chi_\infty)\cdot \Upsilon_{\Pi_p'}(\chi_p)\cdot \oL(\frac{1}{2}, \Pi\times \chi)}{ \mathscr G_\psi(\chi^{(p)})^{\frac{n(n-1)}{2}}\cdot\mathscr G_\psi(\chi_{\Pi_{n-1}}^{(p)}) 
\cdot  \Omega_{\Pi_p'}\cdot \Omega_{\Pi}({\sf w}_\infty \chi_\infty)   }
  \]
  for all ${\sf V}$-balanced Hecke characters $\chi=\otimes_\ell \chi_\ell \in \CX(\varepsilon)$, where  
  \begin{itemize}
  \item   ${\sf w}$ is the inverse of the weight of $\chi$;
\item $\Upsilon_{\Pi_\infty'}(\chi_\infty)$ is the modifying factor at $\infty$ in \eqref{MFinfrs};
\item $\Upsilon_{\Pi_p'}(\chi_p)$ is the modifying factor at $p$ in \eqref{MFprs};
\item $\Omega_{\Pi_p'}$ is in \eqref{omegapi'};
\item $\mathscr G_\psi(\chi^{(p)})$  and $\mathscr G_\psi(\chi_{\Pi_{n-1}}^{(p)})$ are   the Gauss sums outside $p$ as in \eqref{gaussp}, respectively of $\chi$ and the central character $\chi_{\Pi_{n-1}}$ of $\Pi_{n-1}$;
\item $\{\Omega_{\Pi}(\varepsilon_\infty)\in\C^\times\}_{\varepsilon_\infty \in \widehat{\rk_\infty^{\times, \natural}}}$ are the Whittaker periods of $\Pi$ in \eqref{whit-per}.  %(\cf \cite{LLS24})
  \end{itemize}
\end{thml}

\begin{remarkl}
The ${\sf V}$-balanced criterion is in \eqref{eq:bp}, and all ${\sf V}$-balanced Hecke characters are  critical for $\Pi$ (Definition \ref{critrs}). If there is a ${\sf V}$-balanced Hecke character in $\CX(\varepsilon)$, then $\mathscr L_\Pi$ is uniquely determined by the interpolation property
in Theorem \ref{padicLrs0}. If furthermore $\rk$ contains no CM subfield, then all critical Hecke characters for $\Pi$ are ${\sf V}$-balanced.
\end{remarkl}

 Here and henceforth,
 \begin{itemize}
     \item $\psi$ is the nontrivial additive character of $\rk\bs \A_\rk$  defined as the composition 
\be \label{psi}
\psi: \rk\bs \A_\rk \xrightarrow{{\rm Tr}_{\rk/\Q}} \Q\bs \A \to \Q\bs \A / \prod_{\ell\neq \infty}\Z_\ell = \R/\Z \xrightarrow{\psi_\R} \C^\times, 
\ee
where ${\rm Tr}_{\rk/\Q}$ is the trace map,  $\psi_\R(x):=e^{2\pi {\rm i} x}$, ${\rm i}:=\sqrt{-1}$;
\item 
 ``$\widehat{\phantom A}$" over an abelian topological group indicates the group of its characters;
 \item  notations similar to $\chi_{\Pi_{n-1}}$ for central characters will be used without explanation. 
\end{itemize}

Let us explain the modifying factors  and $\Omega_{\Pi_p'}$ in Theorem \ref{padicLrs0}, which are calculated  explicitly  based on \cite{LLSS23}.  Let  $\CE_\rk$ be the set of field embeddings $\iota:\rk\to\overline\Q$. Denote the highest weight of $\sf V^\vee$ by
\[
(\mu,\nu)= \left(\{\mu^\iota\}_{\iota\in\CE_\rk}, \{\nu^\iota\}_{\iota\in\CE_\rk}\right) \in (\Z^n)^{\CE_\rk}\times (\Z^{n-1})^{\CE_\rk},
\]
where every $\mu^\iota$ is a sequence of integers $\mu^\iota_1\geq \mu^\iota_2\geq\cdots \geq \mu^\iota_n$, and $\nu^\iota$ is similar. %Here and henceforth, a superscript $\,^\vee$ over a finite-dimensional representation indicates its contragredient. 
Throughout the paper, when ${\sf Z}={\sf G}_1$ we write ${\sf w} = \prod_{\iota\in \CE_\rk}\iota^{{\sf w}_\iota}$, where ${\sf w}_\iota\in\Z$. Then 
\be \label{MFinfrs}
\Upsilon_{\Pi_\infty'}(\chi_\infty) := {\rm i}^{-\sum_{\iota\in \CE_\rk}\sum^{n-1}_{i=1}(n-i)(\mu^\iota_i + \nu^\iota_i - {\sf w}_\iota)}
\cdot (-1)^{\sum_{\iota\in\CE_\rk}\sum_{i>j,\, i+j\leq n}(\mu^\iota_i+\nu^\iota_j-{\sf w}_\iota)}.
\ee

The refinement $\Pi_p'$ can be also viewed as a character of the torus $L=P/N$. Assume that conjugation of $P$ to the Borel subgroup of lower triangular matrices in $G$ takes $\Pi_p'$ to a character of the form 
\[ % \label{kappars}
\kappa=(\kappa_{1}, \kappa_{2}, \dots, \kappa_{n}, \kappa_{1}',  \kappa_{2}',\dots, \kappa_{n-1}')
\]
interpreted in the usual way. Here $\kappa_i = \otimes_{\wp\mid p}\kappa_{i, \wp}$ and  $\kappa_{i, \wp}$ is a character of $\rk_\wp^\times$ ($i=1,2,\dots, n$), and similar notations will be used for $\kappa_j'$ ($j=1,2,\dots, n-1$) and $\chi_p$.  Write $\psi=\otimes_v\psi_v$ where $v$ runs over all places of $\rk$. 
Then 
\be \label{MFprs}
\begin{aligned}
\Upsilon_{\Pi_p'}(\chi_p):  =\, &  \prod_{i>j, \, i+j\leq n}(\kappa_{i} \kappa_{j}'\chi_p)(-1)\\
& \cdot \prod_{\wp\mid p} \frac{\prod_{i+j\leq n}\gamma\left(n+1-i-j, 
\kappa_{i,\wp}^{-1}\kappa_{j,\wp}'^{-1}\chi_{\wp}^{-1}, \psi_\wp^{-1}\right)}{\oL(\frac{1}{2}, \Pi_\wp\times\chi_\wp)}.
\end{aligned}
\ee
Here and henceforth, for every place $v$ of $\rk$,  $\gamma(s, \omega, \psi_v')$ denotes the Tate $\gamma$-factor 
for  a character $\omega$ of $\rk_v^\times$ and a nontrivial additive character $\psi_v'$ of $\rk_v$ (see \cite{T79, Ku03}). Note that $\Upsilon_{\Pi_p'}$ is an algebraic function on $\mathcal X_p$ by Proposition \ref{prp:st}.

For every finite place $v$ of $\rk$,  the conductor of $\psi_v$ is the inverse different of $\rk_v$ given by 
\[
\frak{d}_v^{-1}:=\{ x\in \rk_v \,:\, {\rm Tr}_{\rk_v/\Q_\ell}(x\CO_v)\subset \mathbb{Z}_\ell\},
\] 
where $\ell$ is the residue characteristic of $\rk_v$. Then the volume of $\CO_v$ under the self-dual Haar measure on $\rk_v$ with respect to $\psi_v$ is equal to
\be \label{cond}
c_v:=[\CO_v: \frak{d}_v]^{-1/2}.
\ee 
Write
\[
c_{p}: = \prod_{\wp\mid p}c_{\wp} = \prod_{\wp\mid p}[\CO_\wp: \frak{d}_\wp]^{-1/2},
\]
and $\omega_{\psi_p}(\kappa):=\prod_{\wp\mid p}\omega_{\psi_\wp}(\kappa_\wp)$, where
\be \label{gaussrs}
 \omega_{\psi_\wp}(\kappa_\wp): = \prod^{n-1}_{i=1} \mathscr{G}_{\psi_\wp}(\kappa_{i, \wp})^{i-n}\cdot \prod^{n-2}_{j=1}\mathscr{G}_{\psi_\wp}(\kappa_{j,\wp}')^{j-n+1}
\ee
and the Gauss sums are as in \eqref{df:gauss}. Define
\be \label{omegapi'}
\Omega_{\Pi_p'}:=c_p^{n(n-1)/2}\cdot \omega_{\psi_p}(\kappa).
\ee

Using Theorem \ref{padicLggp0} and the formula of $\Upsilon_{\Pi_p'}(\chi_p)$, we determine the exceptional zeros of $\mathscr{L}_\Pi$ explicitly in Proposition \ref{excep0rs}. The existence of exceptional zeros of $p$-adic L-functions is of utmost arithmetic significance.  For instance see the discussion in \cite{MTT86} for an elliptic curve $\mathscr{E}$ over $\Q$. In this case, assuming that the ramification type $\varepsilon$ in \eqref{ram} is trivial,  an exceptional zero occurs if and only if $\mathscr{E}$ has split multiplicative reduction at $p$, which corresponds to the Steinberg representation of $\GL_2(\Q_p)$. 
As a direct application of Proposition \ref{excep0rs},  we have the following generalization.

\begin{exl} \label{symrs}
Let  $\mathscr{E}$ and $\mathscr{E}'$ be elliptic curves over $\Q$ without complex multiplication. By the modularity of elliptic curves over $\Q$ \cite{W95, BCDT01} and the  recent groundbreaking result on symmetric power functoriality in \cite{NT21a, NT21b},  ${\rm Sym}^n \mathscr{E}$ and ${\rm Sym}^{n-1}\mathscr{E}'$ ($n\geq 1$) correspond to irreducible regular algebraic cuspidal automorphic representations 
$\Pi_{n+1}$ and $\Pi_n$ of $\GL_{n+1}(\A)$ and $\GL_{n}(\A)$ respectively. Then for $\Pi := \Pi_{n+1}\boxtimes \Pi_{n}$ we have that
\[
\oL(s, \Pi) = \oL\left(s+n-\frac{1}{2}, {\rm Sym}^{n} \mathscr{E}\times {\rm Sym}^{n-1} \mathscr{E}'\right).
\]

Assume that $\mathscr{E}$ and $\mathscr{E}'$ are  $p$-ordinary, so that they have good or multiplicative reductions at $p$. When $\mathscr{E}$ has  good ordinary reduction at $p$, write $\alpha_p$ for one of the two eigenvalues in $\overline\Q^\times$ that is a $p$-adic unit, of the $p$-th power Frobenius endomorphism on the Tate module $T_\ell(\widetilde {\mathscr{E}}_{\mathbb{F}_p})$ 
 (see \cite[V.2]{Sil09}). Here $\widetilde{\mathscr{E}}$ is the N\'eron model of $\mathscr{E}$ and $\ell\neq p$ is a prime.  Similarly when $\mathscr{E}'$ has good ordinary reduction at $p$, write $\alpha_p'$ for the corresponding $p$-adic unit.  

Given a ramification type $\varepsilon$ in \eqref{ram}, which in this case is a finite order character of $\prod_{\ell\neq\infty, p}\Z_\ell^\times$, let $\mathscr L_\Pi$ be the corresponding $p$-adic L-function. Note that in this example (and Examples \ref{example2} and \ref{example3}),  since the representation  $\sf V$ is trivial, all exceptional zeros  of $\mathscr L_\Pi$ are of finite order. 

We first observe that by Proposition \ref{excep0rs}, if $\chi^\flat$ is an exceptional zero of $\mathscr L_\Pi$, where  $\chi=\otimes_\ell \chi_\ell \in \CX(\varepsilon)$ is a finite order character of   $\Q^\times\bs\A^\times$, then $\chi_p$ is unramified and $\chi_p(p)$  is a root of unity. 
 We list the values   $\chi_p(p)$ for all the  zeros $\chi_p$ of $\Upsilon_{\Pi_p'}$ given by Proposition \ref{excep0rs} in Table 1 (depending on different reduction types and ramification types),  where ``good", ``s-mult" and ``ns-mult" indicate good, split multiplicative and nonsplit multiplicative reductions at $p$ respectively, and $1\leq i\leq n$. However, since $|\alpha_p| = |\alpha_p'|=\sqrt{p}$ by Weil's conjecture for curves and $\chi_p(p)$ is a root of unity, after ruling out impossible cases we arrive at Table 2, where $d_{\chi_p}$ denotes the order of $\chi_p$ as a zero of $\Upsilon_{\Pi_p'}$.  \hfill\qed
 
\begin{table} 
	\caption{$\chi_p(p)$ for zeros $\chi_p$ of $\Upsilon_{\Pi_p'}$}
	\centering
 \begin{tabular}{|c|c|c|c|}
  \hline
  \diagbox{$\mathscr{E}$}{$\chi_p(p)$}{$\mathscr{E}'$} & good &  s-mult & ns-mult \\
  \hline
  good &  $\alpha_p^{2i-n-2} \alpha_p'^{n+1-2i}$ & $\alpha_p^{2i-n-2}$ &   $- \alpha_p^{2i-n-2}$ \\
  \hline
  s-mult  & $\alpha_p'^{n+1-2i}$ & 1 &   $-1$\\   
  \hline
  ns-mult  & $-\alpha_p'^{n+1-2i}$ & $-1$ & 1\\
  \hline
  \end{tabular}
  \end{table}
  
\begin{table} 
	\caption{Exceptional zeros for $\Sym^n\mathscr{E}\times\Sym^{n-1}\mathscr{E}'$}
	\centering
 \begin{tabular}{|c|c|c|c|}
  \hline
  \diagbox{$\mathscr{E}$}{$\chi_p(p)$}{$\mathscr{E}'$} & good &  s-mult & ns-mult \\
  \hline
  good &  NA & $1 \, (n\text{ even, }d_{\chi_p}=1)$ &   $-1 \, (n\text{ even, }d_{\chi_p}=1)$ \\
  \hline
  s-mult  & $1 \, (n\text{ odd, }d_{\chi_p}=1)$ & $1  \, (d_{\chi_p}=n)$ &   $-1  \, (d_{\chi_p}=n)$\\   
  \hline
  ns-mult  & $-1 \, (n\text{ odd, }d_{\chi_p}=1)$ & $-1  \, (d_{\chi_p}=n)$ & $1  \, (d_{\chi_p}=n)$\\
  \hline
  \end{tabular}
  \end{table}
\end{exl}

\subsection{Rankin-Selberg $p$-adic L-functions for  $\RU_n\times \RU_{n-1}$} \label{ssec:RSU}

Thanks to the recent advances \cite{BPLZZ21, BPCZ22} of the global Gan-Gross-Prasad conjecture \cite{GGP12} and Ichino-Ikeda conjecture \cite{II10, Ha14} for unitary groups, our results can be also applied to produce certain anticyclotomic $p$-adic L-functions.  The recent work \cite{Liu23} (see also \cite{D23}) gives a different approach by using the local Birch lemma as in \cite{KMS00, Jan11} etc.

Let $\rk' /\rk$ be a quadratic extension of number fields. 
Let $\tilde\pi = \tilde\pi_n \boxtimes\tilde\pi_{n-1}$ be an irreducible automorphic representation of $\GL_n(\A_{\rk'})\times \GL_{n-1}(\A_{\rk'})$ that is hermitian isobaric in the sense of \cite[Definition 1.5]{BPLZZ21} such that $\oL(\frac{1}{2}, \tilde\pi)\neq 0$. By the Gan-Gross-Prasad conjecture (Theorem \ref{ggp}), $\tilde\pi\cong {\rm BC}(\pi)$ is the weak base-change of an irreducible cuspidal automorphic representation 
$\pi = \pi_n \boxtimes \pi_{n-1}$ of $\RG_0(\A_\rk)$  with certain nonvanishing global period. Here $\RG_0:=\RU_n\times \RU_{n-1}$ is the product of two relevant unitary groups over $\rk$, so that
$\RU_{n-1}$ embeds into $\RG_0$ diagonally. Assuming that $\pi$ is everywhere tempered, the Ichino-Ikeda conjecture (Theorem \ref{iic}) makes a further refinement by relating the global and local periods through the following L-function $\CL_\Pi(\chi)$.

Denote by $\pi^\vee$ the contragredient of $\pi$, which is realized as the space of all complex conjugations of the automorphic forms in $\pi$. Let $\Pi:=\pi\boxtimes \pi^\vee$. For an automorphic character $\chi:
\RU_1(\rk)\bs \RU_1(\A_\rk)\to \C^\times$, define a meromorphic function 
\be  \label{cLggp}
\CL(s, \pi\times \chi):= \frac{\oL(s, \tilde\pi\times\tilde\chi)}{\oL(s+\frac{1}{2}, \pi, \Ad)}\cdot \prod^n_{i=1}\oL(s+i-\frac{1}{2}, \eta_{\rk'/\rk}^i),
\ee
which is always holomorphic at $s=\frac{1}{2}$, and set 
\[
\CL_\Pi(\chi): = |S_{\tilde\pi}|^{-1}\cdot \CL(\frac{1}{2}, \pi\times \chi).
\]
Here $\eta_{\rk'/\rk}$ is the quadratic character of $\rk^\times\bs \A_{\rk}^\times$ associated to $\rk' /\rk$, 
$\tilde\chi={\rm BC}(\chi)$, $\oL(s, \pi, \Ad)$ is the adjoint L-function of $\pi$,  and $S_{\tilde\pi}$ is a certain component group 
measuring  the Arthur packet of $\pi$. Here and henceforth, for every finite set $S$, denote by $|S|$ its cardinality.

Suppose that the 5-tuple in \eqref{five} is 
 \[
  \begin{cases}
     {\sf G}:=\Res_{\rk/\Q}(\RG_0\times \RG_0); \ \dot{\sf G}:=\Res_{\rk/\Q}(\RU_{n-1}\times \RU_{n-1});\ {\sf Z}:=\Res_{\rk/\Q}\RU_1;&\\
\imath: \dot{\sf G}\to {\sf G}\,\textrm{is the product of 
the diagonal embeddings};&\\
\jmath: \dot{\sf G}\to {\sf Z}\textrm{ is the composition of 
$\dot{\sf G} \xrightarrow{\det \times \det^{-1}} {\sf Z} \times {\sf Z} \xrightarrow{\textrm{multiplication}} {\sf Z}$}. &\\
 \end{cases}
 \]
% we put ${\sf G}:=\Res_{\rk/\Q}(\RG_0\times \RG_0)$, $\dot{\sf G}:=\Res_{\rk/\Q}(\RU_{n-1}\times \RU_{n-1})$ and ${\sf Z}:=\Res_{\rk/\Q}\RU_1$. Let $\imath: \dot{\sf G}\to {\sf G}$ be the product of the diagonal embeddings $\Res_{\rk/\Q}\RU_{n-1}\to \Res_{\rk/\Q}\RG_0$, and let $\jmath: \dot{\sf G}\to {\sf Z}$ be the composition of \[\dot{\sf G} \xrightarrow{\det \times \det^{-1}} {\sf Z} \times {\sf Z} \xrightarrow{\rm mult} {\sf Z},\]where the second arrow is the multiplication map. 
Then $\Pi$ is an automorphic representation of ${\sf G}(\A)$ as before.

For technical reasons, we assume that all places $v\mid \infty p$ of $\rk$ are split in $\rk'$. % In particular ${\sf G}$ is quasi-split over $\Q_p$. 
Write  
\[
{\sf E}_0:=\Q_p\cap {\sf E}\quad \textrm{ and }\quad \rk_{\sf E_0}:={\sf E}_0\otimes \rk.
\]
Suppose that
\begin{itemize}

\item  $\pi$ is regular algebraic,  $\Q(\pi)\subset {\sf E}$, and  the coefficient system of $\pi$ is defined over ${\sf E}$;

\item there is an isomorphism 
\be\label{splitrkp}
\rk_{\sf E_0}\otimes_\rk \rk'\cong \rk_{\sf E_0}\times \rk_{\sf E_0}\ \, \textrm{of  $\rk_{\sf E_0}$-algebras};
\ee
\item $\Pi_p$ has a nearly ordinary refinement $\Pi_p'\subset \CB_P(\Pi_p)$ defined over ${\sf E}$, where $P={\sf P}_{\sf E_0}(\Q_p)$ and ${\sf P}_{\sf E_0}$ is a Borel subgroup of ${\sf G}_{\sf E_0}$ as in \eqref{transP0} which is transversal to $\dot{\sf G}_{\sf E_0}$. 
\end{itemize}
Note that  $\Q(\pi)$  is a number field by \cite[1.4.2]{GL21}.

%\be \label{unitaryP}P = \prod_{\wp\mid p} \RB'_{\wp} \ee
%
Suppose that $Z_0$ is trivial. Then
\[
{\sf Z}^{\flat, \{0\}}= \CC_{\rk'}^-(p^\infty):= \ker(\RN_{\rk' /\rk}: \CC_{\rk'}(p^\infty) \to \CC_{\rk}(p^\infty)),
\]
where $\RN_{\rk' /\rk}$ is the norm map.
Applying the main theory of this paper when $\varepsilon$ in \eqref{ram} is trivial, 
%and assuming that $\Hom_{E\otimes\dot\p}(V^\n, E)\neq \{0\}$,  
we will construct an $E$-valued $p$-adic measure 
${\mathscr L}_\Pi:=\mathcal L_{\varepsilon \otimes \mathscr H}$ on $\CC_{\rk'}^-(p^\infty)$.
%By Proposition \ref{lem:cpt}, in this case hypothesis ($\star$) automatically holds. 
 As the second application of Theorem \ref{thmcint}, we have the following main result for this example, which shows that 
% if there is a ${\sf V}$-balanced automorphic character then 
$\mathscr L_\Pi(\chi^\flat)$ interpolates $\CL_\Pi(\chi)$ for all ${\sf V}$-balanced automorphic characters 
 $\chi$ (Definition \ref{critggp}) unramified outside $\infty p$.  The explicit modifying  factors at $p$ are again consistent with \cite{CPR89, Co89}.
 
 \begin{thml}[Theorem \ref{padicLggp}] \label{padicLggp0}
Let the notations and assumptions be as above. %If there is a ${\sf V}$-balanced automorphic character unramified outside $\infty p$, 
Then 
\[
\mathscr L_\Pi(\chi^\flat) = \frac{\Upsilon_{\Pi_p'}(\chi_p)\cdot \CL_\Pi(\chi)}{\Omega_{\Pi_p'}\cdot\Omega_\Pi({\sf w}_\infty \chi_\infty)}
\]
for all ${\sf V}$-balanced automorphic characters $\chi=\otimes_\ell \chi_\ell : \RU_1(\rk)\bs \RU_1(\BA_{\rk})\to\BC^\times$ unramified outside $\infty p$,  
where
\begin{itemize}
\item 
 ${\sf w}$ is the inverse of the weight of $\chi$; 
 %\item $\Upsilon_{\Pi_\infty'}$ is the modifying factor at $\infty$ in \eqref{sgninf};
\item $\Upsilon_{\Pi_p'}(\chi_p)$ is the modifying factor at $p$ in \eqref{MFpggp};
\item $\Omega_{\Pi_p'}$ is in \eqref{omegapi'ggp};
\item $\{\Omega_\Pi(\varepsilon_\infty)\in\C^\times\}_{\varepsilon_\infty\in \widehat{\RU_1(\rk_\infty)^\natural}}$ are the Bessel periods of $\Pi$  in \eqref{bessel-per}.
%where $\rk'^-_\infty=\RU_1(\rk_\infty)$ is the kernel of the norm map  $\rk'^\times_\infty \rightarrow \rk^\times_\infty$.
\end{itemize}
\end{thml}

\begin{remarkl}
Theorem \ref{padicLggp0} can be easily extended to the case of a general ramification type  $\varepsilon$, by twisting $\pi$ with a character $\chi\in \CX(\varepsilon)$.
The ${\sf V}$-balanced criterion is also as in \eqref{eq:bp}, where the weight $(\mu,\nu)$ is given in Section \ref{ssec:AMSMggp}.  If there is a ${\sf V}$-balanced automorhic character, then $\mathscr L_\Pi$ is uniquely determined by the interpolation property in Theorem \ref{padicLggp0}. 

By Lemma \ref{lem:crit} (b),  all algebraic automorphic characters are critical for $\pi$ (Definition \ref{critggp}). If $\rk'$ contains no CM subfield, then algebraic automorhic characters are of finite order. 
%In this case, ${\sf V}$-balanced automorphic characters are all critical for $\pi$, and the converse holds when $\rk$ contains no CM subfield (so that algebraic automorphic characters are of finite order).
\end{remarkl}
 
 %By standard arguments, these ingredients also give the period relation for $\CL(\frac{1}{2}, \pi_\chi)$ (Theorem \ref{ggp:pr}).

In this example, $\Upsilon_{\Pi_\infty'}(\chi_\infty)$ and $\Omega_{\Pi_\ell}(\varepsilon_\ell)$ ($\ell\neq\infty, p$)  in \eqref{pintp} are all equal to 1.
%In Theorem \ref{padicLggp0}, the modifying factor at $\infty$ and the factor that accounts for the nonarchimedean period relation are both constants in $\Q^\times$ that are independent of $\chi$ and have beennormalized to be 1. 
Let us explain $\Upsilon_{\Pi_p'}(\chi_p)$  and  $\Omega_{\Pi_p'}$ in Theorem \ref{padicLggp0}. 
%Suppose that ${\sf V} = {\sf V}_\pi\boxtimes {\sf V}_\pi^\vee$, where ${\sf V}_\pi={\sf F}_\mu^\vee\boxtimes {\sf F}_\nu^\vee$ and $(\mu, \nu)\in (\Z^n)^{\CE_\rk}\times (\Z^{n-1})^{\CE_\rk}$ is as in Section \ref{ssec:RSGL}. Then \be\label{sgninf} \Upsilon_{\Pi_\infty'}:=(-1)^{\sum_{\iota\in\CE_\rk}\sum^{n-1}_{i=1}(n-i)(\mu^\iota_i+\nu^\iota_i)}.\ee

Write $\rk_p:=\Q_p\otimes \rk$.  Fix an isomorphism \eqref{splitrkp} which indueces a $\rk_p$-algebra isomorphism 
\be\label{isokp}
\rk_p\otimes_\rk \rk' \cong \rk_p \times \rk_p.
\ee
By using the first factor of the product we have an identification 
\[
G =\GL_n(\rk_p)\times\GL_{n-1}(\rk_p)\times \GL_n(\rk_p)\times \GL_{n-1}(\rk_p) \quad (\textrm{with respect to certain bases}).
\]
% under suitable bases of the hermitian spaces. 

The refinement $\Pi_p'$  is also viewed as a character of the torus  $L=P/N$. By MVW involution (see \cite[Chapter 4, II.1]{MVW87}),  the conjugation of $P$ to the Borel subgroup of
lower triangular matrices in $G$   
takes 
$\Pi_p'$ to a character of the form 
\be \label{MVW}
\kappa = (\kappa_{1},\dots, \kappa_{n}, \kappa_{1}', \dots, \kappa_{n-1}' ; \kappa_{n}^{-1}, \dots, \kappa_{1}^{-1}, \kappa_{n-1}'^{-1},\dots, \kappa_{1}'^{-1}),
\ee
where $\kappa_{i} = \otimes_{\wp\mid p}\kappa_{i,\wp}$ and 
$\kappa_{i,\wp}$ is a character of $\rk_\wp^\times$ ($i=1,2,\dots, n$), and similar notation applies for $\kappa_j'$ ($j=1,2,\dots, n-1$).  
Then 
\be\label{MFpggp}
\begin{aligned}
\Upsilon_{\Pi_p'}(\chi_p): = 
 \prod_{\wp\mid p}&\left( \frac{\chi_{\pi_{n-1,\wp}}(-1)^n}{\CL(\frac{1}{2}, \pi_\wp\times\chi_\wp)}\prod_{i+j\leq n} \gamma\left(n+1-i-j, \kappa_{i,\wp}^{-1}\kappa_{j,\wp}'^{-1} \chi_\wp^{-1}, \psi_\wp^{-1}\right)\right.\\
&\left.\cdot\prod_{i+j>n} \gamma\left(i+j-n,\kappa_{i,\wp}\kappa_{j, \wp}' \chi_\wp,
\psi_\wp\right) \right),
\end{aligned}
\ee
where $\pi_{n-1,\wp}$ is the component of $\pi_{n-1}$ at $\wp$,  %$\chi_{\pi_{n-1,\wp}}$ is the central character of $\pi_{n-1,\wp}$, 
and  $\CL(\frac{1}{2}, \pi_\wp\times\chi_\wp)$ is the local factor of $\CL(\frac{1}{2}, \pi\times\chi)$ at $\wp$.  It is easily verified that \eqref{MFpggp} is independent of the isomorphism \eqref{isokp}. Note that  $\Upsilon_{\Pi_p'}$ is an algebraic function on $\mathcal X_p$ by Proposition \ref{ggp:st}. 

Define 
\be \label{omegapi'ggp}
\Omega_{\Pi_p'}:= \prod_{\wp\mid p}\omega_{\psi_\wp}(\kappa_\wp),
\ee
where
\[
\omega_{\psi_\wp}(\kappa_\wp):= \prod^{n-1}_{i=1}\left( \mathscr{G}_{\psi_\wp}(\kappa_{i,\wp}) \mathscr{G}_{\psi_\wp}(\kappa_{n+1-i,\wp}^{-1})\right)^{i-n}\cdot 
\prod^{n-2}_{j=1}\left( \mathscr{G}_{\psi_\wp}(\kappa_{j,\wp}') \mathscr{G}_{\psi_\wp}(\kappa_{n-j,\wp}'^{-1})\right)^{j-n+1}.
\]

The explicit calculation of  $\Upsilon_{\Pi_p'}(\chi_p)$  is again based on \cite{LLSS23}, together with a result in \cite{Zh14b} (Lemma \ref{lem:zhang}). 
In the special case  that % $\rk'/\rk$ is a CM extension,
$\pi_p$ is unramified and $\chi_p$ is ramified, the value $\Upsilon_{\Pi_p'}(\chi_p)$ agrees with \cite[Theorem 5.2]{Liu23} which is obtained by using the so-called local Birch lemma. 
Using the  formula of $\Upsilon_{\Pi_p'}(\chi_p)$, we also determine the exceptional zeros of $\mathscr L_\Pi$ (Proposition \ref{excep0ggp}). 

\begin{exl}\label{example2}
Let  $\mathscr{E}$ and $\mathscr{E}'$ be  elliptic curves over $\Q$ without complex multiplication,  and let $\rk'$ be a  quadratic number field. By \cite{AC89} and \cite{NT21a, NT21b}, $\Sym^n\mathscr{E}_{\rk'}$ and $\Sym^{n-1}\mathscr{E}'_{\rk'}$ ($n\geq 1$) are modular, which correspond to irreducible regular algebraic cuspidal automorphic representations $\tilde\pi_{n+1}$
and $\tilde\pi_n$ of $\GL_{n+1}(\A_{\rk'})$ and $\GL_n(\A_{\rk'})$ respectively.

Assume that $\mathscr{E}$ and $\mathscr{E}'$ are $p$-ordinary, $\rk'$ is real quadratic,  and $p$ splits in $\rk'$. 
If $\oL(\frac{1}{2}, \tilde\pi_{n+1}\times \tilde\pi_n)\neq 0$, then we have a $p$-adic measure  $\mathscr L_\Pi$  on $\CC^-_{\rk'}(p^\infty)$ associated 
to $\tilde\pi_{n+1}\boxtimes\tilde\pi_n$ as in Theorem \ref{padicLggp0}. By Proposition \ref{excep0ggp}, if $\chi^\flat$ is an exceptional zero of $\mathscr L_\Pi$, where $\chi$ is a finite order character of $\RU_1(\Q)\bs \RU_1(\A)$ unramified outside $\infty p$, then $\chi_p$ is unramified and the values of $\chi_p(p)$ (in $ \{\pm1\}$) are as in Table 2 of Example \ref{symrs}, with the order $d_{\chi_p}$ doubled.  \hfill\qed
\end{exl}

%In Theorem \ref{padicLggp0} we consider finite order automorphic characters $\chi$  only for simplicity. This theorem  can be extended to all critical algebraic automorphic characters.  

%It should be also possible to construct $p$-adic L-functions  for general $p$-adic places that are not necessarily split in $\rk'$. We hope to evaluate the stable integrals at nonsplit places and construct more general $p$-adic L-functions in future works.

\begin{comment}
\begin{remarkl} \label{rmk:ramanujan}
By \cite{Mok15, KMSW14, Ram18} (\cf \cite[Remark 1.1.4.1]{BPCZ22}), a weak base-change is automatically a strong base-change.
Since $\tilde\pi$ is regular algebraic and conjugate self-dual, if $\rk'$ is CM then the generalized Ramanujan conjecture holds for $\tilde\pi$ by \cite{Cl91, Sh11, Ca12, CH13} (see \cite[Theorem 1.2]{Ca12}). Thus when $\rk'$ is CM, the everywhere tempered assumption for Theorem \ref{padicLggp0} automatically holds.
\end{remarkl}
\end{comment}

\subsection{Standard $p$-adic L-functions of symplectic type for $\GL_{2n}$} \label{ssec:stL}
Let $\pi$ be an irreducible regular algebraic cuspidal automorphic representation of $\GL_{2n}(\A_\rk)$, where $\rk$ is a number field and $n\geq 1$. We assume that 
$\pi$ is of symplectic type, which is equivalent to that $\pi$ has a nonzero $(\eta,\psi)$-Shalika integral (see Section \ref{ssec:FJI} for the precise statement). 
Here $\eta: \rk^\times\bs \A_\rk^\times\to\C^\times$ is a Hecke character, which is necessarily algebraic. 

Recall from Section \ref{ssec:RSGL} that ${\sf G}_m :=\Res_{\rk/\Q}\GL_m$ for every positive integer $m$. Denote by $1_m$ the $m\times m$ identity matrix. 
 Suppose that the 5-tuple in \eqref{five} is 
 \[
  \begin{cases}
     {\sf G}:=({\sf G}_{2n}\times {\sf G}_1)/{\sf G}_1',\, \textrm{ where }\,
{\sf G}'_1:=\{ (a\cdot 1_{2n}, a^n) \,:\, a\in {\sf G}_1\};& \\
\dot {\sf G}:=({\sf G}_n \times {\sf G}_n)/\dot {\sf G}_1',\, \textrm{ where }\, \dot {\sf G}_1':=\{ (a\cdot 1_n, a\cdot 1_n) \,:\,  a\in {\sf G}_1\};&\\
{\sf Z}:={\sf G}_1;&\\
\imath: \dot{\sf G} \to {\sf G},\quad (g_1, g_2)\dot{\sf G}_1'\mapsto \left(\begin{bmatrix} g_1 & 0\\ 0 & g_2 \end{bmatrix}, \det g_1\right){\sf G}_1';&\smallskip\\
 \jmath: \dot{\sf G} \to {\sf Z}, \quad (g_1, g_2)\dot{\sf G}_1'\mapsto \frac{\det g_1}{\det g_2}.&\\
 \end{cases}
 \]
% ${\sf G}:=({\sf G}_{2n}\times {\sf G}_1)/{\sf G}_1'$, where ${\sf G}'_1:=\{ (a\cdot 1_{2n}, a^n) \,:\, a\in {\sf G}_1\}$, and  $\Pi:=\pi\boxtimes \eta^{-1}$. Then $\Pi$ is an automorphic representation of ${\sf G}(\A)$.  Let$\dot {\sf G}:=({\sf G}_n \times {\sf G}_n)/\dot {\sf G}_1'$, where $\dot {\sf G}_1':=\{ (a\cdot 1_n, a\cdot 1_n) \,:\,  a\in {\sf G}_1\}$,  and ${\sf Z}:={\sf G}_1$.  
%Define \[\begin{aligned}& \imath: \dot{\sf G} \to {\sf G},\quad (g_1, g_2)\dot{\sf G}_1'\mapsto \left(\begin{bmatrix} g_1 \\ & g_2 \end{bmatrix}, \det g_1\right){\sf G}_1', \quad \text{and} \\& \jmath: \dot{\sf G} \to {\sf Z}, \quad (g_1, g_2)\dot{\sf G}_1'\mapsto \frac{\det g_1}{\det g_2}, \quad \text{where } g_1, g_2\in {\sf G}_n.\end{aligned}\]
   Then $\Pi:=\pi\boxtimes \eta^{-1}$ is an automorphic representation of ${\sf G}(\A)$.  
The complex L-function for this example is the standard L-function $\CL_\Pi(\chi):=\oL(\frac{1}{2}, \pi\otimes\chi)$ for $\chi$  a Hecke character of $\rk^\times\bs\A_\rk^\times$, which has been studied through Shalika models and Friedberg-Jacquet integrals in \cite{FJ93}.

Let $\mathsf P$ be a parabolic subgroup of $\sf G$ of type $(n,n)$ as in \eqref{elet:gamma}, so  that 
\[
\dot{\mathsf P} := \mathsf P\cap \dot{\mathsf G} = \{ (g, g)\dot{\sf G}_1' \,:\, g\in  \mathsf G_n\}.
\]
It is clear that $\jmath(\dot {\sf P})$ is trivial. Suppose that $P=\mathsf P(\Q_p)$ and $Z_0$ is trivial. Then ${\sf Z}^{\flat, \{0\}} = \CC_\rk(p^\infty)$.

Suppose that $\Q(\Pi)\subset{\sf E}$ and $\Pi_p$ has a nearly ordinary refinement $\Pi_p'\subset \CB_P(\Pi_p)$ defined over ${\sf E}$. Then the coefficient system of $\Pi$ is defined over ${\sf E}$.
%and we suppose that ${\sf V}^\vee$ is its ${\sf E}$-form . %, so that ${\sf V}^\vee$ has the same infinitesimal character as that of $\Pi_\infty$. 
The period relations of $\oL(\frac{1}{2},\pi\otimes\chi)$ for  arbitrary ${\sf V}$-balanced Hecke characters $\chi$ have been proved in full generality in \cite{JLST24}, which extends the earlier work \cite{JST19}.

%$Z_0:=\jmath(\dot P)$ is trivial and ${\sf Z}^{\flat, \{0\}} = \CC_\rk(p^\infty)$. 

As before suppose that the ramification type $\varepsilon$ is valued in $\sf E^\times$. Applying the main theory of this paper, we will construct a continuous linear functional $\mathscr L_\Pi:=\mathcal L_{\varepsilon \otimes \mathscr H}$ on 
$\con({\sf Z}, E)(\varepsilon)$,  %which in particular is an $E$-valued $p$-adic measure on $\CC_\rk(p^\infty)$ if $\varepsilon$ is trivial, 
without assuming the existence of ${\sf V}$-balanced characters.
As the third application of Theorem \ref{thmcint}, we have the following main result  for this example, which shows that %if there is a ${\sf V}$-balanced Hecke character then  
$\mathscr L_\Pi(\chi^\flat)$ interpolates  $\oL(\frac{1}{2}, \pi\otimes \chi)$ 
for all ${\sf V}$-balanced Hecke characters $\chi\in \CX(\varepsilon)$. The  explicit modifying  factors at $p$ are again consistent with \cite{CPR89, Co89}. 

\begin{thml}[Theorem \ref{padicLst}] \label{padicLst0}
Let the notations and assumptions be as above. % If there is a $\sf V$-balanced Hecke character in $\CX(\varepsilon)$, 
Then
\[
\mathscr L_\Pi(\chi^\flat)  = \frac{\Upsilon_{\Pi_\infty'}(\chi_\infty) \cdot \Upsilon_{\Pi_p'}(\chi_p)\cdot \oL(\frac{1}{2}, \pi\otimes \chi)}{\mathscr G_\psi(\chi^{(p)})^n \cdot \Omega_{\Pi_p'}\cdot\Omega_\Pi({\sf w}_\infty\chi_\infty)}
\]
  for all ${\sf V}$-balanced Hecke characters $\chi=\otimes_\ell \chi_\ell  \in \CX(\varepsilon)$, where 
  \begin{itemize}
\item  ${\sf w}$ is the inverse of the weight of $\chi$; 
\item $\Upsilon_{\Pi_\infty'}(\chi_\infty)$ is the modifying factor at $\infty$ in \eqref{MFinfst}; 
\item $\Upsilon_{\Pi_p'}(\chi_p)$ is the modifying factor at $p$ in \eqref{MFpst};
\item $\Omega_{\Pi_p'}$ is in \eqref{omegapi'st};
\item $\{\Omega_\Pi(\varepsilon_\infty)\in \BC^\times\}_{\varepsilon_\infty\in \widehat{\rk_{\infty}^{\times,\natural}}}$  are the Shalika periods of $\Pi$
in \eqref{sha-per}. % (\cf \cite{JST19, JLST24})
  \end{itemize} 
\end{thml}

\begin{remarkl}
The ${\sf V}$-balanced criterion is in \eqref{bal-Sha}, and all ${\sf V}$-balanced Hecke characters are critical for $\pi$ (Definition \ref{critst}). If there is a ${\sf V}$-balanced Hecke character in $\CX(\varepsilon)$, then $\mathscr L_\Pi$ is uniquely determined by the interpolation property
in Theorem \ref{padicLst0}. If  $\rk$ contains no CM subfield, then critical Hecke characters for $\pi$ exist and they are the same as the ${\sf V}$-balanced  Hecke characters.
\end{remarkl}

As mentioned in Theorem \ref{padicLrs0},  $\mathscr G_\psi(\chi^{(p)})$ denotes the Gauss sum of $\chi$ outside $p$. Let us explain the modifying factors  and $\Omega_{\Pi_p'}$ in Theorem \ref{padicLst0}. 
Denote the highest weight of $\sf V^\vee$ by $\mu =\{\mu^\iota\}_{\iota\in\CE_\rk}\in (\BZ^{2n})^{\CE_\rk}$, where 
every $\mu^\iota$ is a sequence of integers $\mu^\iota_1\geq \mu^\iota_2\geq\cdots \geq \mu^\iota_{2n}$.  Then 
\be \label{MFinfst}
\Upsilon_{\Pi_\infty'}(\chi_\infty):={\rm i}^{-\sum_{\iota\in\CE_\rk}\sum^n_{i=1}(\mu^\iota_i - {\sf w}_\iota)}.
\ee

The refinement $\Pi_p'$ is a representation of $L=P/N$. By Theorem \ref{thm:sha},  conjugation of $P$ to the lower triangular parabolic subgroup of $G$ of type $(n,n)$ takes $\Pi_p'$ to a representation of the form 
\[
\kappa = \kappa_1\otimes (\kappa_1^\vee \otimes \eta_p) \otimes \eta_p^{-1}\qquad(\eta=\otimes_\ell \eta_\ell), \]
where $\kappa_1 = \otimes_{\wp\mid p}\kappa_{1,\wp}$ and $\kappa_{1,\wp}$ is an irreducible admissible smooth representation of $\GL_n(\rk_\wp)$.
Then
\be \label{MFpst}
\Upsilon_{\Pi_p'}(\chi_p) :=  \prod_{\wp\mid p}\frac{\gamma(1, \kappa_{1,\wp}^\vee\otimes\chi_\wp^{-1},\psi_\wp^{-1})}{\oL(\frac{1}{2}, \pi_\wp\otimes\chi_\wp)},
\ee
where the local $\gamma$-factors are as in \cite{GJ72}.
Note that $\Upsilon_{\Pi_p'}$ is an algebraic function on $\mathcal X_p$ by Proposition \ref{prp:stSha}.

Recall that $c_p= \prod_{\wp\mid p}c_\wp$ as in Section \ref{ssec:RSGL}. 
Define
\be\label{omegapi'st}
\Omega_{\Pi_p'}:= c_p^{n^2}\cdot \prod_{\wp\mid p}\mathscr{G}_{\psi_\wp}(\chi_{\kappa_{1,\wp}})^{-1}.
\ee

 Using the explicit formula of  $\Upsilon_{\Pi_p'}(\chi_p)$, we also determine the exceptional zeros of 
$\mathscr L_\Pi$ (Proposition \ref{excep0st}).  

\begin{exl}\label{example3}
Let $\mathscr{E}$ be an elliptic curve over $\Q$ without complex multiplication. Then $\Sym^{2n-1}\mathscr{E}$ ($n\geq 1$) corresponds to 
an irreducible regular algebraic cuspidal automorphic representation $\pi$ of $\GL_{2n}(\A)$ such that 
\[
\oL(s, \pi) = \oL\left(s+n-\frac{1}{2}, \Sym^{2n-1}\mathscr{E}\right).
\]
Moreover, $\pi$ has a nonzero $(\eta,\psi)$-Shalika integral with $\eta=1$. 

Assume that $\mathscr{E}$ is $p$-ordinary. 
Given a ramification type 
$\varepsilon$ in \eqref{ram}, we have a $p$-adic L-function $\mathscr L_\Pi$ as in Theorem \ref{padicLst0}. By Proposition \ref{excep0st}, if $\chi^\flat$ is  an exceptional zero of $\mathscr L_\Pi$, where $\chi=\otimes_\ell \chi_\ell\in \CX(\varepsilon)$ is a finite order character of $\Q^\times \bs \A^\times$, then $\chi_p$ is unramified and $\chi_p(p)$ is a root of unity.  There are two cases as in what follows.

If $\mathscr{E}$ has good ordinary reduction at $p$ and $\alpha_p$ is the corresponding $p$-adic unit as in Example \ref{symrs}, then 
$\chi_p(p) = \alpha_p^{-1}$, which is impossible since $\alpha_p$ cannot be a root of unity. 
If $\mathscr{E}$ has multiplicative reduction at $p$, then $\chi_p(p)=1$ or $-1$ according to the reduction is split or nonsplit, in which cases
$\chi_p$ is a simple zero of $\Upsilon_{\Pi_p'}$. 

These results are also consistent with the expectations in terms of Galois representations as in \cite[Page 170]{Gr94}  
(however the construction of $\mathscr L_\Pi$ was not given there), which in particular assert that the trivial character is an exceptional zero of $\mathscr L_\Pi$ when $\varepsilon$ is trivial and $\mathscr{E}$ has split multiplicative reduction at $p$. 
%These results are also consistent with the discussions in \cite[p.170]{Gr94}. To be more precise, consider the compatible system of $\ell$-adic Galois representations $V=\{V_\ell\}$, where $V_\ell := \Sym^{2n-1}\left(V_\ell(\CE)\right)$ and  $V_\ell(\CE)$ denotes the $\ell$-adic Tate module of $\CE$. By {\it loc. cit.} if $\CE$ has split multiplicative reduction $p$, then the $p$-adic L-function of $V$ has an exceptional zero at $s=n$.
\hfill\qed
\end{exl}

It is worth emphasizing that in contrast to earlier works \cite{AG94, G18, DJR20}, our evaluation of $\Upsilon_{\Pi_p'}(\chi_p)$ is again by the technique that compares the Friedberg-Jacquet integral with a new integral over the open orbit of the spherical subgroup $\dot G$ acting on 
the flag variety of $G$. This is given by Theorem \ref{thm:FJ-int}, whose proof is based on an %elegant 
 application of the Godement-Jacquet integrals \cite{GJ72}. 

\subsection{Historical remarks} \label{ssec:HR}

Assuming the existence of ${\sf V}$-balanced characters and rather restrictive ramification conditions, the $p$-adic L-functions for $\GL_n\times \GL_{n-1}$ and $\GL_{2n}$ have been constructed in the previous works such as \cite{KMS00, Jan11, Jan15, Jan16, Jan24, Sch93, Sch01} and \cite{AG94,DJR20} among many others, mainly using a different approach called the local Birch lemma. 
In these works the $p$-adic L-functions are evaluated for Hecke characters of the form $\chi_0\abs{\cdot}_{\A_\rk}^j$ where $\chi_0$ is of finite order and 
$j$ is a critical place. In the resulting modifying factors at $p$ one does not see the local L-factors as in our modifying  factor $\Upsilon_{\Pi_p'}(\chi_p)$, but only the local $\varepsilon$-factors which essentially account for the conductor and Gauss sum of $\chi_p$. Thus the connection  with 
the conjecture by Coates and Perrin-Riou is not clear in general.
In particular the phenomenon of exceptional zeros does not occur in these works.

For the case of $\GL_{2n}$, in \cite{G18} and more recent works \cite{BDW21, BDGJW22, Wi23}, under the condition that
%s  that the number field $\rk$ is totally real %, $\sf V$-balanced characters exist, 
  $\Pi_p$ is  unramified or parahoric spherical, the $p$-adic L-functions are constructed and the modifying factors are evaluated which are compatible with the general result in Theorem \ref{padicLst0}.

Our construction of $p$-adic L-functions $\mathscr L_\Pi$ in Theorems \ref{padicLrs0}, \ref{padicLggp0} and \ref{padicLst0} 
are different from the previous works in the following aspects. 
\begin{itemize}
    %\item %The construction of $\mathscr L_\Pi$ does not require the existence of ${\sf V}$-balanced characters. 
    %For all the three families of examples, $\mathscr{L}_\Pi(\chi^\flat)$ interpolates the classical L-function $\CL_\Pi(\chi)$ for all ${\sf V}$-balanced automorphic characters $\chi$ that are not necessarily of the form $\chi_0\abs{\cdot}_{\A_\rk}^j$ as above in the cases of $\GL_n\times \GL_{n-1}$ and $\GL_{2n}$.

    \item Bases on the recent works of the authors and their collaborators, the modifying factors $\Upsilon_{\Pi_\infty'}(\chi_\infty)$ at $\infty$ are explicitly calculated and are consistent with the conjectures of 
    Deligne and Blasius, which was not known in previous constructions of $p$-adic L-functions for higher rank groups.

    \item The modifying factors $\Upsilon_{\Pi_p'}(\chi_p)$ at $p$ are  defined in a general setting and are explicitly calculated, which are consistent with the conjecture of Coates and Perrin-Riou.
The explicit calculation  is based on the preparatory work in \cite{LLSS23} of the authors and their collaborators, and the idea of integrations over open orbits in {\it loc. cit.} %\cite{LLSS23} 
 is applicable to other cases (for example, the third family of examples in this paper).

    \item To determine  the exceptional zeros is an important open problem in the theory of 
$p$-adic L-functions which has been long-standing. Since the modifying factors $\Upsilon_{\Pi_p'}(\chi_p)$ at $p$ are explicit, we are able to solve this problem 
for these cases  in Propositions \ref{excep0rs}, \ref{excep0ggp} and \ref{excep0st}.

\item No ramification conditions on $\Pi_p$ are required.
%Neither ramification conditions on $\Pi_p$ nor assumptions on the number field $\rk$ are required.

\item The method of constructing $p$-adic L-functions is completely general and is potentially applicable to many other cases in the framework of relative Langlands program (see \cite{BZSV24}). % 
\end{itemize}    

The idea of studying modular symbols of formal vectors as in \eqref{stablems} and their $p$-adic interpolations (Theorem \ref{thmcint}) was first presented by the second named author in the conference on Gan-Gross-Prasad conjecture held in CNRS, Paris, in 2014. This idea is also our guiding principle in the calculations of the modifying factors using principal series representations. However, it took the authors and their collaborators quite a long time to explicitly calculate the modifying factors (both at $\infty$ and at $p$) in the three families of examples considered in this paper. This is the main reason for the delay of the appearance of the current paper.

\section{Relative  cohomologies and  relative completed cohomologies}\label{sec:RCC} %  and Hecke maps

Recall the notation from Section \ref{secnotation}.
In particular, $\mathsf G$ is an arbitrary linear algebraic group defined over $\Q$. Fix  a closed 
 subgroup  of $\mathsf G(\R)$ of the form 
 \[ % \label{kinf}
 K_\infty := \mathsf A(\R)^\circ\cdot K_\infty',
 \]
 where  $\mathsf A$ is a split torus in $\mathsf G$ defined over $\Q$ that is central in the identity connected component of $\sf G$ modulo its  unipotent radical, and $K_\infty'$ is a compact subgroup that normalizes  $\mathsf A(\R)^\circ$. %The group $K_\infty$ is automatically closed in $\mathrm G(\R)$. 
As in the Introduction, write  \[
 \mathsf G^\natural := K_\infty^\natural \times \mathsf G(\A^\infty)= \mathsf G^{\natural,p}\times G,
 \]
 where $\mathsf G^{\natural,p}: =K_\infty^\natural \times \mathsf G(\A^{\infty p})$ and $G :=\mathsf G(\Q_p)$. Set  
 \[
  \mathscr X :=\mathsf G(\A)/K_\infty^\circ =(\mathsf G(\R)/K_\infty^\circ )\times \mathsf G(\A^\infty), 
  \]
  which carries a left action of $\mathsf G(\Q)$ and a right action of $\mathsf G^\natural$ that commute with each other. 
  For each open compact subgroup $K$ of $\mathsf G^\natural$, we define a topological space
  \[
S_K^{\mathsf G}:=\mathsf G(\Q)\backslash \mathscr X/K,
\] 
which is known to be Hausdorff. 

Let $\mathbf G$ be an open subgroup of $\mathsf G^\natural$,  $R$ a commutative ring with identity, 
 and  $M$  an $R[\mathsf G(\Q)\times \BG]$-module.   Recall that  $\p$ is a  Lie subalgebra of the Lie algebra 
$\g$ of $G$. Let $\s \subset \p$ be a  subalgebra. 
In the rest of this paper,  unless otherwise specified, 
\be\label{drs}
\begin{cases}
    \textrm{$D$ denotes an open compact subgroup of $\mathsf G^{\natural, p}\cap \BG$};&\\
    \textrm{$\fG$ denotes an open compact subgroup of $G\cap \BG$};& \\
    \textrm{$\fP$ denotes a compact subgroup of $G\cap \BG$ with Lie algebra $\p$};& \\
    \textrm{$\fS$ denotes a compact subgroup of $G\cap \BG$ with Lie algebra $\s$}. & \\
\end{cases}
\ee
%The letter $D$ denotes an open compact subgroup of $\mathsf G^{\natural,p}\cap \BG$. 
 %A large part o 

  % the theory developed here works for general linear algebraic groups which are not necessarily reductive. 

 \subsection{Neat elements}
Here we review some basic notions  concerning neat elements, following \cite{Pi89}.  For every field $\rk$ of characteristic $0$, and every element $x\in \mathsf G(\rk)$,  we define a finite subgroup $\mathrm{Unit}(x)$ of $\overline \Q^\times$ as in what follows. Fix a faithful finite-dimensional algebraic representation $\mathsf V_\rk$ of $\mathsf G_\rk$ over $\rk$, and a field embedding $\iota_\rk: \overline \Q\rightarrow \overline{\rk}$, where $\overline{\rk}$ is an algebraic closure of $\rk$. All the eigenvalues of the linear operator $x: \mathsf V_\rk\otimes \overline{\rk} \rightarrow \mathsf V_\rk\otimes \overline{\rk}$ generate a subgroup of $\overline{\rk}^\times$. We define  $\mathrm{Unit}(x)$ to be the inverse image under $\iota_\rk$ of the set of all torsion elements in this subgroup. This is a finite subgroup of $\overline \Q^\times$ which is independent of the representation $\mathsf V_\rk$ and the embedding $\iota_\rk$.  

For every 
\[
x^\infty=(x_\ell)_{\ell\neq 
\infty}\in \mathsf G(\A^\infty)
\]
  we define
\[
  \mathrm{Unit}(x^\infty):=\bigcap_{\ell\neq \infty} \mathrm{Unit}(x_\ell). 
\]
We say that $x^\infty$ is neat if the group $ \mathrm{Unit}(x^\infty)$ is trivial. %A subgroup of $\mathsf G(\A^\infty) $ is said to be neat if so are  all its elements. 

 \begin{dfnl}
   (a)  A compact subgroup of $\mathsf G^\natural$ is said to be neat if it is contained in a group of the form $K_\infty^\natural \times K^\infty$, where $K^\infty$ is an open compact subgroup of $\mathsf G(\A^\infty)$ consisting of neat elements. It is said to be completely neat if it is neat and its projection to $K_\infty^\natural$ preserves a ${\mathsf G}(\R)$-invariant orientation  on ${\mathsf G}(\R)/ K_\infty^\circ$.

   \noindent (b) A compact subgroup of $C$ of $G$ is said to be $D$-neat if $D C$ is neat. 
 \end{dfnl} 
 
\subsection{The transfer map of cohomology groups}\label{phicon}

 We will consider cohomologies  with support conditions. For this purpose, let $\Phi$ be a family of closed subsets of $\mathsf G(\Q)\backslash \mathscr X$ with the following properties:
\begin{itemize}
  \item every closed subset of a set in $\Phi$ belongs to $\Phi$;
  \item the union of two sets in $\Phi$ belongs to $\Phi$;
  \item for every set $Z\in \Phi$ and every compact subset $X$ of $ \mathsf G^\natural$, the set $Z.X$ belongs to $\Phi$.
\end{itemize}

For each  open compact subgroup $K$ of $\BG$, we have a sheaf  $M_{[K]}$ of $R$-modules over the topological space
$
S_K^{\mathsf G}
$ %:=\mathsf G(\Q)\backslash \mathscr X/K\] 
such that for every open subset $U$ of  $\mathscr X$ that is left $\mathsf G(\Q)$-invariant and right $K$-invariant, $M_{[K]}(\mathsf G(\Q)\backslash U/K)$ equals 
\begin{equation}\label{sheafm}
\{ \textrm{$\mathsf G(\Q)\times K$-equivariant locally constant maps from $U$ to $M$}\}.
\end{equation}
When $K$ is neat, $S_K^{\mathsf G}$ is a topological manifold and $M_{[K]}$ is a locally constant sheaf. 

For simplicity, write 
\be\label{hkk}
  \oH_\Phi^i(K,M):= \oH_{\Phi/K}^i(S_K^{\mathsf G},M_{[K]}),\qquad (i\in \Z)
\ee
for the $i$-th cohomology of $S_K^{\mathsf G}$ with coefficient $M_{[K]}$ and support $\Phi/K$, 
where $\Phi/K$ denotes the family of closed subsets $Y$ of $S_K^{\mathsf G}$ such that the preimage of $Y$ under the quotient map $\mathsf G(\Q)\backslash \mathscr X\rightarrow S_K^{\mathsf G}$ belongs to $\Phi$. As usual, we drop the subscript $\Phi$ to indicate the cohomology group with arbitrary support, and replace $\Phi$ by ``$\mathrm c$" to indicate the cohomology group with compact support. For example, $\oH^i(K,M)= \oH^i(S_K^{\mathsf G},M_{[K]})$ and $\oH_{\mathrm c}^i(K,M)= \oH_{\mathrm c}^i(S_K^{\mathsf G},M_{[K]})$.  

Let $K'$ be another open compact subgroup of $\BG$. First assume that $K'\subset K$, so that we have the natural map of topological spaces
\[ % \label{quotient}
f_{K', K}: S_{K'}^{\mathsf G}\rightarrow S_K^{\mathsf G}.
\]
We will use a subscript $*$ to indicate the push-forward of sheaves. For instance, $f_{K', K,*}$ denotes the push-forward of sheaves by the map $f_{K',K}$. 
Put 
\[
M':=\Ind^{{\sf G}(\Q)\times K}_{{\sf G}(\Q)\times K'}(M\vert_{{\sf G}(\Q)\times K'})=\Ind^{K}_{K'}(M\vert_{ K'}),
\]
regarded as an $R[\mathsf G(\Q)\times K]$-module.

\begin{lem} \label{lem:ind}
There is a canonical  isomorphism of sheaves $M'_{[K]}\cong f_{K',K, *}M_{[K']}$ on $S^{\mathsf G}_K$.
\end{lem}

\begin{proof}
We realize $M'$ as the space of maps $f: K\to M$ such that 
\[
f(k' k_0) = k'.f(k_0)\quad \textrm{for all } k'\in K', \ k_0\in K,
\] 
on which $K$ acts by right translations. 
Let $U$ be an open subset of $S^{\mathsf G}_K$, with preimage $\tilde U$ in $\mathscr X$. The set of sections $M'_{[K]}(U)$ consists of the $\mathsf{G}(\mathbb{Q})$-equivariant maps 
$s: \tilde{U}\to M'$ such that 
\[
s( x k^{-1})= k.s(x)\quad \textrm{for all }x\in \tilde U, \ k\in K,
\]
i.e., $s(xk^{-1})(k_0) = s(x)(k_0 k)$ for any $k_0\in K$. On the other hand, $(f_{K',K,*}M_{[K']})(U)$ consists of the $\mathsf{G}(\Q)$-equivariant maps $s': \tilde{U}\to M$ such that 
\[
s'(x k'^{-1}) = k'.s'(x)\quad \textrm{for all } x\in \tilde U, \ k'\in K'.
\]
It is easy to verify that we have mutually inverse maps 
\[
\begin{array}{rcl}
M'_{[K]}(U) &\to &(f_{K',K,*}M_{[K']})(U),\\
s&\mapsto& (s':\tilde U\rightarrow M,\ x\mapsto s(x)(1)),
%(s'(x) :=s(x)(1),\ x\in \tilde{U}),
\end{array}
\]
and
\[
\begin{array}{rcl}
 (f_{K',K,*}M_{[K']})(U)&\to& M'_{[K]}(U), \\
 s'&\mapsto &\left(s: \tilde U\rightarrow M',\ \left(x\mapsto (K\rightarrow M,\ k_0\mapsto s'(xk_0^{-1}))\right)\right).
 \end{array}
\]
Thus the sheaves $M'_{[K]}$ and $f_{K',K,*}M_{[K']}$ are canonically isomorphic. 
\end{proof}

Since the fibers of the map $f_{K',K}$ are finite, by Lemma \ref{lem:ind} and \cite[II. Theorem 11.1]{Br97}, we have a canonical isomorphism 
\be \label{eq:ind}
\oH^i_\Phi(K, M') \cong \oH^i_\Phi(K', M).
\ee
The $K$-homomorphism given by the ``orbital map", 
\be \label{eq:orbmap}
\phi: M\to M',\quad v\mapsto (K\to M, \ k\mapsto k.v), 
\ee
induces a homomorphism
\be \label{eq:orb}
\oH^i_\Phi(K, M) \to \oH^i_\Phi(K, M').
\ee
The composition of  \eqref{eq:orb} and \eqref{eq:ind}  defines the pull-back homomorphism
\[ %\label{pb}
  \rho_{K, K'}: \oH_\Phi^i(K,M)\rightarrow \oH_\Phi^i(K',M).
\]

On the other hand, the $K$-homomorphism given by the map 
\be \label{eq:avmap}
\phi': M'\to M, \quad f\mapsto \sum_{kK'\in K/K'} k.f(k^{-1})
\ee
induces a homomorphism 
\be \label{eq:av}
\oH^i_\Phi(K, M')\to \oH^i_\Phi(K, M).
\ee
The composition of \eqref{eq:av} and the inverse of \eqref{eq:ind} defines the push-forward homomorphism 
\[%\label{pf}
  \rho_{K', K}: \oH_\Phi^i(K', M)\rightarrow \oH_\Phi^i(K, M).
\]

\begin{lem}\label{lem:pbpf}
Assume that $K'\subset K$ are open compact subgroups of $\mathbf G$. Then 
\[
\rho_{K',K}\circ \rho_{K,K'}: \oH^i_\Phi(K, M)\to \oH^i_\Phi(K, M)
\]
equals the multiplication by $[K:K']$.
\end{lem}

\begin{proof}
This follows immediately from the fact that the following composition of \eqref{eq:orbmap} and \eqref{eq:avmap},
\[
\phi'\circ\phi: M\to M
\]
is the multiplication by $[K:K']$.
\end{proof}

Now we drop the assumption that $K'\subset K$ and define the homomorphism 
\be\label{rhogg}
  \rho_{K, K'}: \oH_\Phi^i(K,M)\rightarrow \oH_\Phi^i(K',M)
\ee
to be the composition of 
\[
  \oH_\Phi^i(K,M)\xrightarrow{\rho_{K, K\cap K'}} \oH_\Phi^i(K\cap K',M) \xrightarrow{\rho_{ K\cap K', K'}}\oH_\Phi^i(K',M).
\]
We call the above map the transfer map. When the groups $K$ and $K'$ are understood, we also write $\rho$ for $\rho_{K,K'}$.

\begin{lem}\label{neatng}
Let $K_1, K_2, K_3$ be open compact subgroups of $\BG$ such that 
\[
K_2=(K_1\cap K_2)(K_2\cap K_3).
\]
Then the diagram 
\[
 \begin{CD}
        \oH_\Phi^i(K_1, M) @> >>   \oH_\Phi^i(K_2, M) \\
            @VVV          @VVV\\
          \oH_\Phi^i(K_3, M) @>[K_1\cap K_3: K_1\cap K_2\cap K_3 ] >>    \oH_\Phi^i(K_3 , M)\\
  \end{CD}
\]
commutes, where the bottom horizontal arrow is the multiplication map by $[K_1\cap K_3: K_1\cap K_2\cap K_3 ]$, and all the other arrows are the  transfer maps. 

%If, in addition, $K_1\cap K_3 \subset K_2$, then $\rho_{K_2, K_3}\circ \rho_{K_1, K_2} = \rho_{K_1, K_3}$.

%\[
 %\rho_{K_2, K_3} \circ  \rho_{K_1, K_2}=[K_1\cap K_3: K_1\cap K_2\cap K_3 ] \, \rho_{K_1, K_3} 
 %\]
% as homomorphisms from $ \oH_\Phi^i(K_1, M)$ to $ \oH_\Phi^i(K_3, M)$.

\end{lem}
\begin{proof}
This is an exercise in algebraic topology, by using Lemma \ref{lem:pbpf}. We omit the details. 
\end{proof}

% Write $\mathscr G$ for the set of all open compact subgroups of $\mathsf G^\natural$. 
\subsection{Formal   cohomologies}

\begin{dfnl}
    For each compact subgroup $C$ of $\BG$, the formal  cohomology group is defined to be 
%Then the following set is nonempty:\be\label{gcg}   \mathscr G_{C, \BG}:=\{K\in \mathscr G\mid C\subset K\subset \BG\}.\eeGiven $K, K'\in \mathscr G_{C, \BG}$ with $K'\subset  K$, we have the push-forward homomorphism $\rho_{K', K}$ as in \eqref{pf}. Using these homomorphisms and Lemma 
 \be\label{relao1111}
\oH_\Phi^i(C,M):= \varprojlim_{K}\oH_\Phi^i(K, M),
\ee
where $K$ runs over all open compact subgroups of $\BG$ containing $C$, and the implicit homormorphisms are the push-forward maps. %The formal topology on $\oH_\Phi^i(C,M)$ is defined to be the inverse limit topology, with $\oH_\Phi^i(K, M)$ equipped with the discrete topology. 
\end{dfnl}

When $C=K$ is open in $\BG$, the formal cohomology group coincides with the previously defined module $\oH_\Phi^i(K,M)$. When $C$ is neat, in the definition of the formal cohomology \eqref{relao1111}, we may also let $K$ run over all neat open compact subgroups of $\BG$ containing $C$. %Write $R_{\mathrm{dis}}:=R$, to be viewed as a topological ring with the discrete topology. Then under the formal topology, $\oH_\Phi^i(C,M)$ is a topological $R_{\mathrm{dis}}$-module. 

%Two compact subgroups of $\mathsf G^\natural$ are said to be  commensurable with each other if their intersection is open in both of these two groups. This defines an equivalence relation on the set of  compact subgroups of $\mathsf G^\natural$. %Let $\mathscr C$  be an equivalence classes of this commensurable  relation. 

Let $C$ and $ C'$ be two compact subgroups of $\BG$  that are commeasurable with each other, namely  their intersection is open in both of these two groups. %that are commensurable with each other.  
Generalizing the transfer map \eqref{rhogg}, we define the transfer  map
\be \label{rhogg222}
  \rho_{C, C'}: \oH_\Phi^i(C,M)\rightarrow \oH_\Phi^i(C',M)
\ee
 as in the following proposition. 

\begin{prpl}\label{actrt2}
There exists a unique homomorphism 
\[ % \label{actp00}
   \rho_{C,C'}:  \oH_\Phi^i(C, M)\rightarrow \oH_\Phi^i(C' , M)
\]
with the following property: the diagram
\[
 \begin{CD}
          \oH_\Phi^i(C, M) @>\rho_{C, C'} >>    \oH_\Phi^i(C' , M) \\
            @V\textrm{projection}VV          @VV\textrm{projection} V\\
          \oH_\Phi^i(K, M) @>\rho_{K, K'} >>    \oH_\Phi^i(K' , M)\\
  \end{CD}
\]
commutes for all open compact subgroups $K$ and $K'$ of $\BG$ such that $C\subset K$, $C'\subset K'$  and that the natural map 
\[
C/ (C\cap C')\rightarrow  K/(K\cap K')
\]
is bijective. 
\end{prpl}
\begin{proof}
Set 
\[
\mathscr G_{C, \BG}:=\{\textrm{open compact subgroups of $\BG$ containing $C$}\}
\]
and define $\mathscr G_{C', \BG}$ similarly. 
The uniqueness  follows from the following assertion: for every $K''\in \mathscr G_{C', \BG}$, there exist 
$K\in \mathscr G_{C,\BG}$ and $K'\in \mathscr G_{C',\BG}$ such that  $K'\subset K''$ and the natural map 
\be\label{biject1}
C/ (C\cap C')\rightarrow  K/(K\cap K')
\ee
 is bijective. For the proof of this assertion, take  $K'\in \mathscr G_{C',\BG}$ small enough so that 
 \[
 K'\subset K''\quad\textrm{and}\quad K'\cap C=C'\cap C.
 \]
 Note that such $K'$ exists because $C$ and $C'$ are commensurable. 
  Let $K_0$ be an open compact subgroup of $\BG$ that is normalized by $C$ and is small enough so that 
  \[
 K_0\subset K'\quad\textrm{and}\quad ((C\setminus C')K_0)\cap K'=\emptyset. 
 \]
 Take $K:=CK_0$. Then 
 \[
   K\cap K'=(C\cap C')K_0
   \]
   and the map \eqref{biject1} is bijective. This proves the assertion. 
   
    To prove the existence, it suffices to show that the diagram 
 \[
 \begin{CD}
      \oH_\Phi^i(K_1, M) @>\rho >>    \oH_\Phi^i(K'_1 , M)\\
                      @VV\rho V          @VV\rho V\\
                \oH_\Phi^i(K, M) @>\rho >>    \oH_\Phi^i(K' , M)\\
  \end{CD}
\]
commutes for all $K, K_1\in \mathscr G_{C,\BG}$ and all $K', K'_1\in {\mathscr G}_{C',\BG}$ such that 
 \[
   K_1\subset K,\quad K_1'\subset K'
 \]
and that the natural maps 
\[
C/ (C\cap C')\rightarrow  K/(K\cap K')
\quad \textrm{and}\quad 
C/ (C\cap C')\rightarrow  K_1/(K_1\cap K_1')
\]
 are bijective.
 Note that the above conditions imply that the natural map
 \[
   K_1/(K_1\cap K_1')\rightarrow K/(K\cap K')
 \]
 is bijective. Thus 
$ 
   K_1\cap K\cap K'=K_1\cap K_1',
 $
 or equivalently
 \[
   K_1\cap K'=K_1\cap K_1'\cap K'.
 \]
 The proposition then follows from Lemma \ref{neatng}. 
\end{proof}

%Note the transfer maps are continuous under the formal topologies. 
We call the map \eqref{rhogg222} a  pull-back map or a push-forward map when $C\supset C'$ or $C\subset C'$ respectively. Note that the $R$-module $\oH_\Phi^i(C, M)$ is functorial in $M$, and and transfer map \eqref{rhogg222} is natural in  $M$, namely the diagram 
\[
 \begin{CD}
          \oH_\Phi^i(C, M) @>\rho_{C, C'} >>    \oH_\Phi^i(C' , M) \\
            @VV\phi V          @VV\phi V\\
          \oH_\Phi^i(C, M') @>\rho_{C, C'} >>    \oH_\Phi^i(C' , M')\\
  \end{CD}
\]
commutes for all $R[\mathsf G(\Q)\times \BG]$-modules $M'$ and all $R[\mathsf G(\Q)\times \BG]$-module homomorphisms $\phi:M\rightarrow M'$.

The following lemma is a generalization as well as a consequence of Lemma \ref{neatng}, whose proof is omitted. 
\begin{leml}\label{neatngc}
Let $C_1, C_2, C_3$ be  compact subgroups of $\BG$ that are pairwise commeasurable with each other. Assume that 
\[
C_2=(C_1\cap C_2)(C_2\cap C_3).
\]
Then the diagram 
\[
 \begin{CD}
        \oH_\Phi^i(C_1, M) @> >>   \oH_\Phi^i(C_2, M) \\
            @VVV          @VVV\\
          \oH_\Phi^i(C_3, M) @>[C_1\cap C_3: C_1\cap C_2\cap C_3 ] >>    \oH_\Phi^i(C_3 , M)\\
  \end{CD}
\]
commutes, where the bottom horizontal arrow is the multiplication map by $[C_1\cap C_3: C_1\cap C_2\cap C_3 ]$, and all the other arrows are the  transfer maps. 

%If, in addition, $K_1\cap K_3 \subset K_2$, then $\rho_{K_2, K_3}\circ \rho_{K_1, K_2} = \rho_{K_1, K_3}$.

%\[
 %\rho_{K_2, K_3} \circ  \rho_{K_1, K_2}=[K_1\cap K_3: K_1\cap K_2\cap K_3 ] \, \rho_{K_1, K_3} 
 %\]
% as homomorphisms from $ \oH_\Phi^i(K_1, M)$ to $ \oH_\Phi^i(K_3, M)$.

\end{leml}

Recall the convention from \eqref{drs}. 

\begin{dfnl}
    The relative cohomology groups are defined to be 
\[ %\label{hr0}
 \oH_\Phi^i(D\fS, M)^{\la \p\ra}:=\varinjlim_{\fP\supset \fS}  \oH_\Phi^i(D\fP, M)
\]
and
\be\label{hr0001}
 \oH_\Phi^i(\mathsf G, M)^{\la \p\supset \s\ra }:=\varinjlim_{D\fS}  \oH_\Phi^i(D\fS, M)^{\la \p\ra}.
\ee
%where $\fP$ runs over all the groups containing $\fS$.
\end{dfnl} 

Here the transition maps in \eqref{hr0001}  are the pull-back maps, namely the maps that make all the following diagrams %commutative:  
\[
 \begin{CD}
        \oH_\Phi^i(D\fP, M) @> \textrm{pull-back}>>   \oH_\Phi^i(D'\fP', M) \\
            @VVV          @VVV\\
          \oH_\Phi^i(D\fS, M)^{\la \p\ra} @>\textrm{transition map} >>    \oH_\Phi^i(D'\fS', M)^{\la \p\ra}\\
  \end{CD}
\]
commutative. Here $D'$ is an open subgroup of $D$, $\fS'$ is an open subgroup of $\fS$, $\fP\supset \fS$, and $\fP'$ is an open subgroup of $\fP$ containing $\fS'$. In this paper, some obvious maps such as the two vertical arrows in the above diagram, will not be named. 

Note that the group \eqref{hr0001} is independent of $\s$ in the sense that the natural homomorphism
\be\label{isors}
  \oH_\Phi^i(\mathsf G, M)^{\la \p\supset \s\ra }\rightarrow \oH_\Phi^i(\mathsf G, M)^{\la \p\ra }:=\varinjlim_{D\fP}  \oH_\Phi^i(D\fP, M)
\ee
is an isomorphism. %, where $D$ runs over all open compact subgroups of $\mathsf G^{\natural, p}\cap \BG$ and $\fP$ runs over all  compact subgroups of $G\cap \BG$ with Lie algebra $\p$. 
Specializing to the case that $\p=\g$, we define
\[
\oH_\Phi^i(\mathsf G, M):=\oH_\Phi^i(\mathsf G, M)^{\la \g\ra }=\varinjlim_{K}  \oH_\Phi^i(K, M),
\]
where $K$ runs over all open compact subgroups of $\BG$. 
%such that is a compact subgroups of $G$ with Lie algebra $\p$.Put \be\label{hr0} \oH_\Phi^i(\mathsf G, M)^{\la \p\ra}:=\varinjlim_{C}  \oH_\Phi^i(C, M),\eewhere $C$ runs over all compact subgroups of $\BG$ that are commeasurable with  some (and all) groups of the form $D\fP$ where $D$ is an  open compact subgroups of $\mathsf G^{\natural, p}$ and $\fP$ is a compact subgroups of $G$ with Lie algebra $\p$.
%The formal topologies on $\oH_\Phi^i(\mathsf G, M)^{\la \p\ra}$ and $\oH_\Phi^i(\mathsf G, M)^{\la \p\ra}$ are obviously defined in the category  of topological $R_{\mathrm{dis}}$-modules, and the homomorphism \eqref{hr2} is continuous under the formal topologies. 

%\subsection{Relative complete  cohomologies}

\begin{dfnl}
    The relative completed cohomology groups are defined to be 
\be\label{cfc}
\widetilde 
{\oH}_{\Phi}^i(D\fS, M)^{\la\p\ra}:=\varprojlim_{k\in \BN} {\oH}_{\Phi}^i(D\fS, M/p^k)^{\la\p\ra},
\ee
and 
\be\label{cfc2}
\widetilde 
{\oH}_{\Phi}^i(\mathsf G, M)^{\la\p\supset \s\ra}:=\varinjlim_{D\fS} \widetilde {\oH}_{\Phi}^i(D\fS, M)^{\la\p\ra}. 
\ee
\end{dfnl}
Here the transition maps in \eqref{cfc2} are the pull-back maps. 
We have an obvious commutative diagram
\be\label{interprc}
 \begin{CD}
    {\oH}_{\Phi}^i(D\fS, M)^{\la\p\ra}@> >>   \widetilde 
{\oH}_{\Phi}^i(D\fS, M)^{\la\p\ra}\\
            @VVV          @VV V\\
 {\oH}_{\Phi}^i(\mathsf G, M)^{\la\p\ra}@> >> \widetilde 
{\oH}_{\Phi}^i(\mathsf G, M)^{\la\p\supset \s\ra}. 
            \end{CD}
\ee

%natural homomorphisms\be\label{cfc222}{\oH}_{\Phi}^i(D\fS, M)^{\la\p\ra}\rightarrow \widetilde {\oH}_{\Phi}^i(D\fS, M)^{\la\p\ra}\eeand \be\label{cfc2333}{\oH}_{\Phi}^i(\mathsf G, M)^{\la\p\ra}\rightarrow \widetilde {\oH}_{\Phi}^i(\mathsf G, M)^{\la\p\supset \s\ra}.  \ee

%\begin{dfnl}   Let $\CG$ be a topological group. A $\Z[\CG]$-module $M_0$ is said to be $p$-smooth if for every $k\in \BN$, an open subgroup of $\CG$ acts trivially on $M_0/p^k$. \end{dfnl}

Relative completed cohomology groups are most interesting when $M$ is $p$-smooth as an $R[G\cap \BG]$-module (Definition \ref{df:psmooth}). When this is the case,  the $R$-module \eqref{cfc}  (and hence \eqref{cfc2}) 
only depends on the $R[\mathsf G(\Q)\times D\fS]$-module structure on $M$. 

\begin{remarkl}\label{remarkp}
    By using the trivial actions of small open compact subgroups of $G$,  the assignments $\widetilde 
{\oH}_{\Phi}^i(D\fS, \,\cdot\,)^{\la\p\ra}$ and $\widetilde 
{\oH}_{\Phi}^i(\mathsf G, \,\cdot\,)^{\la\p\supset \s\ra}$ are both functors  from the category of $R[\mathsf G(\Q)\times D\fS]$-modules that are $p$-smooth as $R[\fS]$-modules to the category of $R$-modules. 
\end{remarkl}

%Finally, for every continuous  representation of $\fS$ on a finite-dimensional $E$-vector space $V$, we define  the relative complete  cohomogy groups \be\widetilde{\oH}_{\Phi}^i(D\fS, V)^{\la\p\ra}:=\widetilde{\oH}_{\Phi}^i(D\fS, \fV)^{\la\p\ra}\otimes E\eeand\be\label{cf}\widetilde {\oH}_{\Phi}^i({\sf G}, V)^{\la\p\supset \s\ra}:=\widetilde {\oH}_{\Phi}^i({\sf G}, \fV)^{\la\p\supset \s\ra}\otimes E,\eewhere $\fV$ is an $\fS$-stable lattice of $V$, to be viewed as an $\CO[\mathsf G(\Q)\times D\fS]$-module via the trivial action of $\mathsf G(\Q)\times D$. Both of the above spaces are independent of the choice of $\fV$. 

\begin{exl}
When $\p=\g$, $\s=\{0\}$, and $M=\CO$ is the trivial module, the group \eqref{cfc2} agrees with the completed cohomology group introduced by Emerton in \cite{Em06a}. 
\end{exl}

For every compact subgroup $C$ of $\BG$, we define 
\[
\widetilde 
{\oH}_{\Phi}^i(C, M):=\varprojlim_{k\in \BN} {\oH}_{\Phi}^i(C, M/p^k).
\]
\begin{exl}We have that
   \[
   \widetilde 
{\oH}_{\Phi}^i(D\fP, M)^{\la\p\supset \p\ra}=\widetilde  {\oH}_{\Phi}^i(D\fP, M) 
\]
and so that 
\[
   \widetilde 
{\oH}_{\Phi}^i(\mathsf G, M)^{\la\p\supset \p\ra}=\varinjlim_{D\fP} \widetilde  {\oH}_{\Phi}^i(D\fP, M). 
\]
\end{exl}

\begin{exl}
The $\CO$-module $\widetilde 
{\oH}^0({\sf G}, \CO)^{\la\g\supset \s\ra}$ is identified with the $\CO$-module of all continuous functions $f:\mathsf G(\Q)\backslash \mathscr{X}\rightarrow \CO$ that are invariant under some groups of the form $D\fS$. % of $\mathsf G^{\natural, p}$ and compact subgroups $\fS$ of $G$ with Lie algebra $\s$.
%\[ \widetilde {\oH}_{\mathrm{cl}}^0({\sf G}, \CO)^{\la\g\supset \s\ra}=\{f:\mathsf G(\Q)\backslash \mathscr{X}\rightarrow \CO \mid \textrm{$f$ is continuous and $D\fS$-invariant for some open compact subgroups $D$ of $\mathsf G^{\natural, p}$ and compact subgroups $\fS$ of $G$ with Lie algebra $\s$}\}\] 
%Here and henceforth, we drop the subscript  $\Phi$ to indicate the cohomology group with arbitrary support. 
%Here $\mathrm{cl}$ indicates the family of all the closed subsets. 
\end{exl}
%\subsection{pull-back of relative complete  cohomologies and integrations}

\subsection{Hecke maps}
Suppose that $H^+$ is a submonoid of $\mathsf G^\natural$. It is also viewed as a submonoid of $\mathsf G(\Q)\times \mathsf G^\natural$.

% Here and as usual,  ``$\Lie$" indicates the Lie algebra of a real or $p$-adic Lie group, or an algebraic group.  Write $T_{\bar P}^{\dag\dag}$ for all $t\in T^\dag_{\bar P}$ with all the inequalities in \eqref{deftp00} are strict. It equals the set of dominating elements of $T_{\bar P}^+$. 

\begin{dfnl} 
A compatible action of $H^+$ on $M$ is an $R$-linear action 
 \be\label{comact00}
  H^+\times M\rightarrow M, \quad (t, x)\mapsto t \ast x
\ee
such that 
  the diagram 
   \be\label{hecom}
 \begin{CD}
           M@>g>> M \\
            @V{t \ast}VV           @VV t \ast V\\
           M@>tgt^{-1} >> M\\
  \end{CD}
\ee
commutes for all  $t\in H^+$ and $g\in \mathsf G(\Q)\times \BG$ with $tgt^{-1}\in \mathsf G(\Q)\times \BG$.
\end{dfnl}

%\begin{exl}Suppose that $H^+$ is a submonoid of $\BG$ and $M'$ is an $R[H^+]$-module. Then $M\otimes_R M'$ is an $R[\mathsf G(\Q)\times \BG]$-module with a compatible action of $H^+$ on it. Here $\mathsf G(\Q)\times \BG$ acts on $M\otimes_R M'$ through its action on $M$, and $H^+$ acts on $M\otimes_R M'$ through its actions on both $M$ and $M'$. \end{exl}

Assume that we are given a compatible action \eqref{comact00}. 
Let $t\in H^+$. Let $K, K'$ be open compact subgroups  of $\BG$. Generalizing the definition of $\rho_{K, K'}$ in \eqref{rhogg},  we shall define in what follows a  homomorphism 
\be\label{rhogg02}
\rho^*_t:=  \rho^*_{t, K, K'}: \oH_\Phi^i(K,M)\rightarrow \oH_\Phi^i(K',M).
\ee
 We call this map the Hecke map of $t$.

It follows from the commutative diagram \eqref{hecom} that the map
\[
  \mathscr X \times M\rightarrow \mathscr X \times M,\quad (x,u)\mapsto (xt^{-1}, t\ast u)
\]
is equivariant with respect to the isomorphism 
\[
  \mathsf G(\Q)\times (\BG\cap {t^{-1}} \BG t)\rightarrow   \mathsf G(\Q)\times (t \BG t^{-1}\cap \BG),\quad g\mapsto t g t^{-1}.
\]
Thus it descends to a homomorphism
\be\label{homl2}
  M_{[K \cap t^{-1} K' t]}\rightarrow  M_{[t K t^{-1}\cap K']}
\ee
of sheaves over the homeomorphism  
\[
\begin{array}{rcl}
S_{K \cap t^{-1} K' t}^{\mathsf G} &\rightarrow &S_{t K t^{-1}\cap K'}^{\mathsf G},\\
  \textrm{the class of $x\in \mathscr X$} &\mapsto & \textrm{the class of $xt^{-1}$}.
  \end{array}
\]

Write
\[ %\be\label{rtg22}
  \rho^*_t: \oH_\Phi^i(K \cap t^{-1} K' t, M)\rightarrow  \oH_\Phi^i(t K t^{-1}\cap K', M),
\]
for the homomorphism induced by \eqref{homl2}.
We define the homomorphism $\rho^*_t:=\rho^*_{t, K, K'}$ in \eqref{rhogg02} to be the composition of 
\[
   \oH_\Phi^i(K, M) \rightarrow\oH_\Phi^i(K \cap t^{-1} K' t, M)
   \xrightarrow{\rho^*_t}  \oH_\Phi^i(t K t^{-1}\cap K', M) \rightarrow  \oH_\Phi^i(K', M),
\]
where the first arrow is the pull-back map and the third arrow is the push-forward map. 

Hecke maps are natural in the coefficient modules, as in the following lemma. 
\begin{lem}\label{hecken}

Let $M'$ be another $R[\mathsf G(\Q)\times \BG]$-module with a compatible action of $H^+$. Let $\phi: M\rightarrow M'$ be an $H^+$-equivariant homomorphism of $R[\mathsf G(\Q)\times \BG]$-modules. Then for all open compact subgroups $K, K'$ of $\BG$, and all $t\in H^+$, the diagram
\[
 \begin{CD}
       \oH_{\Phi}^{i}(K, M)@> \rho^*_t>>   \oH_{\Phi}^{i}(K', M)\\
            @VV\phi V          @VV \phi V\\
 \oH_\Phi^i(K, M')@> \rho^*_t >>     \oH_\Phi^i(K', M')\\
            \end{CD}
\]
commutes.
\end{lem}
\begin{proof}
This is elementary and we omit its proof. 
\end{proof}

\begin{prpl}\label{hecke001} 
Let $K_1, K_2, K_3$ be open compact subgroups of $\BG$ and let $s,t\in H^+$. If 
\[
  K_2=(s K_1 s^{-1}\cap K_2) (K_2\cap t^{-1}K_3 t),
\]
then  
\[
 \rho^*_{t, K_2, K_3} \circ  \rho^*_{s, K_1, K_2}=[s K_1 s^{-1}\cap t^{-1}K_3 t : s K_1 s^{-1}\cap K_2\cap  t^{-1}K_3 t] \, \rho^*_{ts, K_1, K_3} 
 \]
 as homomorphisms from $ \oH_\Phi^i(K_1, M)$ to $ \oH_\Phi^i(K_3, M)$.

\end{prpl}
\begin{proof}
For simplicity write 
\[
r:=ts \quad\textrm{and}\quad K_0:=s K_1 s^{-1}\cap K_2\cap t^{-1} K_3 t.
\]
 Then the proposition follows from considering the following  commutative diagram:
\[
  \xymatrix{
 \oH(K_1)\ar[r] \ar[rd]\ar[d]& \oH(K_1 \cap s^{-1} K_2 s)  \ar[r]^{\rho^*_s} \ar[d]&   \oH(s K_1 s^{-1}\cap K_2)  \ar[d] \ar[r] & \oH(K_2) \ar[d]  \\
\oH(K_1\cap r^{-1} K_3 r)\ar[r]\ar[dr]^{\rho^*_r} & \oH(s^{-1} K_0 s  )  \ar[r]^{\rho^*_s} \ar[rd]^{\rho^*_{r}} &   \oH(K_0) \ar[d]^{\rho^*_{t}}  \ar[r] & \oH(K_2 \cap t^{-1} K_3 t) \ar[d]^{\rho^*_{t}}  \\   
 & \oH( r K_1 r^{-1}\cap  K_3 )\ar[r]\ar[dr] &  \oH(tK_0 t^{-1} )   \ar[rd] \ar[r] &  \oH(t K_2 t^{-1}\cap K_3) \ar[d]  \\ 
  & & \oH(K_3)\ar[r]^{(*) } & \oH(K_3)\, .  \\   
   }
\]
Here we write   $\oH(K_1)$ for 
 $\oH_\Phi^i(K_1, M)$ and  similarly for other cohomology groups, $(*)$ is the multiplication map by $[s K_1 s^{-1}\cap t^{-1}K_3 t : K_0]$, and all the  unnamed arrows are the transfer maps as defined in Section \ref{phicon}.  
 \end{proof}

\subsection{The representation $\oH_\Phi^i(\mathsf G, M)$ and its formal completion}
%In this subsection, we introduce some  examples of monoid representations given by Hecke maps. % As in the Introduction, set $G:=\mathsf G(\Q_p)$. Let $\mathsf V$ and $V$ be as in the Introduction, to be viewed as representations of $\mathsf G(\Q)$ and $G$ respectively. 

%\begin{exl}\label{exsfv0}
Take $H^+:=\BG$. Then the $R[\mathsf G(\Q) \times \BG]$-module structure on $M$ restricts to an action of $H^+$ on $M$, which is in fact a compatible action. We drop the superscipt $*$ to indicate the Hecke maps attached to this compatible action:
\[
  \rho_t:=  \rho_{t, K, K'}: \oH_\Phi^i(K,M)\rightarrow \oH_\Phi^i(K',M).
\]
Recall that 
\[
  \oH_\Phi^i(\mathsf G, M)=\varinjlim_K  \oH_\Phi^i(K, M), 
\]
where $K$ runs over all open compact subgroups of $\BG$, and the implicit homomorphisms in the direct limit are the pull-back maps. 
By using Proposition \ref{hecke001}, it is easy to see that for every $t\in \BG$, there is a unique homomorphism 
\be\label{actt0}
  \rho_t:  \oH_\Phi^i(\mathsf G, M)\rightarrow  \oH_\Phi^i(\mathsf G, M), \quad \phi\mapsto t.\phi:=\rho_t(\phi)
\ee
such that the diagram 
\[
 \begin{CD}
       \oH_{\Phi}^{i}(K, M)@> \rho_t>>   \oH_{\Phi}^{i}(tK t^{-1}, M)\\
            @VVV          @VV V\\
 \oH_\Phi^i(\mathsf G, M)@> \rho_t >>     \oH_\Phi^i(\mathsf G, M)\\
            \end{CD}
\]
commutes for all open compact subgroups $K$ of $\BG$. 
Moreover, the maps in \eqref{actt0} for various $t$ yield a  representation of $ \BG$ on $ \oH_\Phi^i(\mathsf G, M)$. This representation is smooth in the sense that every element of $ \oH_\Phi^i(\mathsf G, M)$ is fixed by some open subgroup of $\BG$.  When $R$ is a $\Q$-algebra and $K$ is a neat open compact subgroup of $\sf G$,  \cite[II. Section 19.1, (47)]{Br97} implies that 
\be\label{kfix}
  \oH_\Phi^i(\mathsf G, M)^K=\oH_\Phi^i(K, M). 
\ee
%Here we introduce the following definition. 

%\end{exl}

In the rest of this section we assume that $R$ is  a $\Q$-algebra so that the formal completion 
\[
\widehat{\oH_\Phi^i({\mathsf G}, M)}=\varprojlim_{K} \oH_\Phi^i({\mathsf G}, M)^K
\]
is defined, where $K$ runs over all open compact subgroups of $\BG$, and the transition maps are the averaging projections. 
For every 
$\phi\in \widehat{\oH_\Phi^i({\mathsf G}, M)}$ and every open compact subgroups of $K$ of $\BG$, write $\phi_K$ for the projection of $\phi$ to $\widehat{\oH_\Phi^i({\mathsf G}, M)}^K$.

\begin{leml}\label{lemfcc}
    For every compact subgroup $C$ of $\BG$, the natural map 
    \be\label{isofc}
      \widehat{\oH_\Phi^i({\mathsf G}, M)}^C\rightarrow \varprojlim_K {\oH_\Phi^i({\mathsf G}, M)}^K
    \ee
    is an isomorphism, 
    where $K$ runs over all open compact subgroups of $\BG$ containing $C$,  and the transition maps are the averaging projections. 
\end{leml}
\begin{proof}
Note that 
   \be\label{equalfc} \widehat{\oH_\Phi^i({\mathsf G}, M)}=\varprojlim_{K'}{\oH_\Phi^i(K', M)},
   \ee
   where $K'$ runs over all open compact subgroups of $\BG$ normalized by $C$. 
    Let $\{\phi_K\}_K$ be an element in the codomain of the map \eqref{isofc}. For every open compact subgroup $K'$ of $\BG$ normalized by $C$, we define
    \[
    \varphi_{K'}:= \phi_{K'C}\in {\oH_\Phi^i({\mathsf G}, M)},
    \]
    Then for every open subgroup $K''$ of $K'$ normalized by $C$, we have that 
    \begin{eqnarray*}
       &&\frac{1}{[K':K'']} \sum_{g\in K'/K''}  g.\varphi_{K''}\\
      % &=&\frac{1}{[K':K'']} \sum_{g\in K'/K''}  g.\phi_{K''C}\\
       &=&\frac{[K'\cap K''C:K'']}{[K':K'']} \sum_{g\in K'C/K''C}  g.\phi_{K''C}\\
       &=&\frac{1}{[K':K'\cap K''C]} \sum_{g\in K'C/K''C}  g.\phi_{K''C}\\
       &=&\frac{1}{[K'C:K''C]} \sum_{g\in K'C/K''C}  g.\phi_{K''C}\\
       &=&\phi_{K'C}=\varphi_{K'}. 
    \end{eqnarray*}
    In view of the identification \eqref{equalfc}, now it is easily checked that the family $\{\varphi_{K'}\}_{K'}$ defines an element of $\widehat{\oH_\Phi^i({\mathsf G}, M)}^C$, and the map
    \[
   \{\phi_{K}\}_{K}\mapsto  \{\varphi_{K'}\}_{K'}
    \]
    is an inverse of the map \eqref{isofc}. 
   \end{proof}
   
Formal cohomologies are  related to  formal completions as in the following lemma. 
\begin{leml}\label{lemformco}
    Suppose that $C$ is a neat compact subgroup of $\BG$. Then
     \be\label{formccomp}
      \begin{array}{rcl}
         \widehat{\oH_\Phi^i({\mathsf G}, M)}^{C}\otimes \mathrm D(\mathsf G^\natural)  & \rightarrow& \oH_\Phi^i(C, M), \\
          \phi\otimes \mu &\mapsto & \{\mu(K)\cdot  \phi_K\}_{K }
      \end{array}
      \ee
      is a well-defined  $R$-module isomorphism, 
      where $K$ runs over all neat open compact subgroups of $\BG$ containing $C$, and $\phi_K$ denotes the image of $\phi$ under the projection map $\widehat{\oH_\Phi^i({\mathsf G}, M)}\rightarrow \oH_\Phi^i({\mathsf G}, M)^K=\oH_\Phi^i(K, M)$.
\end{leml}
\begin{proof}
    In view of \eqref{kfix}, this is a direct consequence of Lemma \ref{lemfcc}. 
\end{proof}

\begin{leml}\label{lemformco22}
    Suppose that $C$ is a neat compact subgroup of $\BG$ and $C'$ is an open subgroup of $C$. Then the diagram 
     \be\label{ff}
      \begin{CD}
          \widehat{\oH_\Phi^i({\mathsf G}, M)}^{C}\otimes \mathrm D(\mathsf G^\natural) \otimes \RD(C)^\vee @>\eqref{formccomp}\otimes\textrm{(evaluation at $\mu_C$)}>> \oH_\Phi^i(C, M)\\
          @V\textrm{inclusion}VV @VV\textrm{pull-back}V\\
          \widehat{\oH_\Phi^i({\mathsf G}, M)}^{C'}\otimes \mathrm D(\mathsf G^\natural)  \otimes \RD(C)^\vee @>\eqref{formccomp}\otimes\textrm{(evaluation at $\mu_{C,C'}$)}>> \oH_\Phi^i(C', M)\\
      \end{CD}
      \ee
      commutes. Here $\mu_C\in \RD(C)$  denotes the normalized Haar measure, and $\mu_{C,C'}\in \RD(C)$ denotes the Haar measure on $C$ whose restriction to $C'$  is the normalized Haar measure. 
      \end{leml}
      \begin{proof}
          This is routine to verify, by reducing to the case when $C'$ is a normal subgroup of $C$.
      \end{proof}
      
Recall from \eqref{psmooth} that 
\[
\widehat{\oH_\Phi^i({\mathsf G}, M)}_{{\p}-\mathrm{sm}}\subset \widehat{\oH_\Phi^i({\mathsf G}, M)} %:=\varinjlim_{D\fP}\widehat{\oH_\Phi^i({\mathsf G}, M)}^{D\fP}
\]
is the submodule consisting of the elements that are fixed by some groups of the form $D\fP$. %, where $D$ is an open compact subgroup of $\mathsf{G}^{\natural, p}$ and $\fP$ is a  compact subgroup of $G$ with Lie algebra $\p$ such that $D\fP\subset \BG$. 
When $D\fP$ is neat, by Lemma \ref{lemformco}, we have an isomorphism 
\be\label{isof3}
  \begin{array}{rcl}
    \widehat{\oH_\Phi^i({\mathsf G}, M)}^{D\fP}\otimes \mathrm D(\g/\p)   & \rightarrow &\oH_{\Phi}^{i}(D\fP, M), \\
      \phi\otimes \mu\otimes \nu&\mapsto & \{ \mu(\fG) \cdot \la \nu, \mu_\fP\ra\cdot \phi_{D\fG}\}_{\fG}.  
  \end{array}
\ee
Here $\fG$ runs over all $D$-neat open compact subgroups of $G\cap \BG$ containing $\fP$, $\mu_\fP\in \RD(\p)=\mathrm D(\fP)$ is the measure with $\mu_\fP(\fP)=1$, $\phi_{D\fG}$ denotes the projection of $\phi$ to $\oH_\Phi^i({\mathsf G}, M)^{D\fG}=\oH_\Phi^i(D\fG, M)$, and as before we use the identification
\[
 \mathrm D(\g/\p)=\mathrm D(\g)\otimes \mathrm D(\p)^\vee=\mathrm D(G)\otimes \mathrm D(\fP)^\vee.
\]
%The formal topology on $\widehat{\oH_\Phi^i({\mathsf G}, M)}_{{\p}-\mathrm{sm}}$ is defined to be the direct limit topology in the category of the topological $R_{\mathrm{dis}}$-modules, with the $\widehat{\oH_\Phi^i({\mathsf G}, M)}^{D\fP}\subset \widehat{\oH_\Phi^i({\mathsf G}, M)}$ equipped with the subspace topology of the formal topology. 

The following lemma is a direct consequence of Lemma \ref{lemformco22}. 
\begin{leml}\label{isosm00}
    There is a unique isomorphism \be\label{identrsm}
\widehat{\oH_\Phi^i({\mathsf G}, M)}_{{\p}-\mathrm{sm}}\otimes \mathrm D(\g/\p)\rightarrow \oH_\Phi^i({\mathsf G}, M)^{\la \p\ra}
\ee
such that  
the diagram 
\[
 \begin{CD}
    \widehat{\oH_\Phi^i({\mathsf G}, M)}^{D\fP}\otimes \mathrm D(\g/\p)@>\eqref{isof3} >>   \oH_{\Phi}^{i}(D\fP, M)\\
            @VVV          @VV V\\
 \widehat{\oH_\Phi^i({\mathsf G}, M)}_{{\p}-\mathrm{sm}}\otimes \mathrm D(\g/\p)@> \eqref{identrsm} >>     \oH_\Phi^i({\mathsf G}, M)^{\la \p\ra}\\
            \end{CD}
\]
commutes for all $D$ and $\fP$ such that $D\fP$ is neat. 
\end{leml}

\subsection{Cup product}\label{seccupp}
Let $M'$ be another $R[\mathsf G(\Q)\times \BG]$-module so that $M'\otimes M:=M'\otimes_R M$ is also an $R[\mathsf G(\Q)\times \BG]$-module. For the main purpose of this paper, we only consider cup product by degree zero cohomologies. 

\begin{leml}\label{cupcum}
    Let $K$ be an open compact subgroup of $\BG$, and let $K_1, K_2$ be open subgroups of $K$. 
    Then for each $\eta\in \oH^0(K, M')$, the diagram 
    \[
 \begin{CD}
          \oH_\Phi^i(K_1, M) @>\eta_1  \smallsmile (\, \cdot\, )  >>  \oH_\Phi^{i}(K_1, M'\otimes M)  \\
            @V\rho VV          @VV \rho  V\\
          \oH_\Phi^i(K_2, M)  @>\eta_2  \smallsmile  (\, \cdot\, )>>    \oH_\Phi^{i}(K_2, M'\otimes_R M) \\
            \end{CD}
\]
commutes, where $\eta_i:=\rho_{K, K_i}(\eta)\in \oH^i(K_i, M')$ ($i=1,2$) and ``$\,  \smallsmile$" stands for the cup product. 
\end{leml}
\begin{proof}
This is an exercise in algebraic topology. We omit the details. 
\end{proof}

By lemma \ref{cupcum}, we have a unique map
\be\label{cup0}
  \smallsmile: \oH^0(\mathsf G, M')\times \oH_{\Phi}^i(\mathsf G, M)\rightarrow \oH_{\Phi}^{i}( \mathsf G, M'\otimes M)
  \ee
such that the diagram 
    \[
 \begin{CD}
          \oH^0(K, M')\times \oH_{\Phi}^i(K, M)  @>\textrm{cup product}>>  \oH_{\Phi}^{i}( K, M'\otimes M)\\
            @V VV          @VV  V\\
          \oH^0(\mathsf G, M')\times \oH_{\Phi}^i(\mathsf G, M)  @>  \smallsmile  >>  \oH_{\Phi}^{i}( \mathsf G, M'\otimes M)\\
            \end{CD}
\]
commutes for all open compact subgroups $K$ of 
$\BG$.

Lemma \ref{cupcum} also implies that for every compact subgroup $C$ of $\BG$, there is  a unique map
\be\label{cup00}
  \smallsmile: \left(\varinjlim_K\oH^0(K, M')\right)\times \oH_{\Phi}^i(C, M)\rightarrow \oH_{\Phi}^{i}( C, M'\otimes M),
  \ee
where $K$ runs over all open compact subgroups of $\BG$ containing $C$, such that for all such $K$ the diagram 
    \[
 \begin{CD}
          \oH^0(K, M')\times \oH_{\Phi}^i(C, M)  @>  \smallsmile  >>  \oH_{\Phi}^i(C, M'\otimes M)\\
            @V VV          @VV  V\\
         \oH^0(K, M')\times \oH_{\Phi}^i(K, M)  @>  \textrm{cup product}  >>  \oH_{\Phi}^i(K, M'\otimes M)  \\
            \end{CD}
\]
commutes. Similarly, for all $D$ and $\fS$, 
there is  a unique map
\be\label{cup00injl}
  \smallsmile: \left(\varinjlim_{\fG\supset \fS}\oH^0(D\fG, M')\right)\times \left(\varinjlim_{\fP\supset \fS} \oH_{\Phi}^i(D\fP, M)\right) \rightarrow \varinjlim_{\fP\supset \fS} \oH_{\Phi}^{i}( D\fP, M'\otimes M)
  \ee
such that the diagram 
    \[
 \begin{CD}
          \oH^0(D\fG, M')\times \oH_{\Phi}^i(D\fP, M) @>  \eqref{cup00}  >>   \oH_{\Phi}^{i}( D\fP, M'\otimes M)\\
            @V VV          @VV  V\\
        \left(\varinjlim_{\fG\supset \fS}\oH^0(D\fG, M')\right)\times \left(\varinjlim_{\fP\supset \fS} \oH_{\Phi}^i(D\fP, M)\right) @>\smallsmile>>\varinjlim_{\fP\supset \fS} \oH_{\Phi}^{i}( D\fP, M'\otimes M) \\
            \end{CD}
\]
commutes whenever $\fG\supset \fP\supset \fS$.

By using Lemma \ref{cupcum}, it is routine to check that there are three cup product maps as defined in what follows. The cup product map 
\[ %\label{cup005}
  \smallsmile: \oH^0(\mathsf G, M')\times \oH_{\Phi}^i(\mathsf G, M)^{\la \p\ra}\rightarrow \oH_{\Phi}^{i}( \mathsf G, M'\otimes M)^{\la \p\ra}
  \]
  for relative cohomologies is defined by requiring 
that the diagram 
    \[
 \begin{CD}
          \oH^0(K, M')\times \oH_{\Phi}^i(D\fP, M)  @>\eqref{cup00} >>  \oH_{\Phi}^{i}( D\fP, M'\otimes M)\\
            @V VV          @VV   V\\
          \oH^0(\mathsf G, M')\times \oH_{\Phi}^i(\mathsf G, M)^{\la \p\ra}  @>  \smallsmile  >>  \oH_{\Phi}^{i}( \mathsf G, M'\otimes M)^{\la \p\ra}\\
            \end{CD}
\]
commutes for all $D$, $\fP$, and all open compact subgroups $K$ of 
$\BG$ containing $D\fP$.

The cup product map 
\[ %\label{cup006}
  \smallsmile: \widetilde \oH^0(\mathsf G, M')^{\la \g\supset \s\ra}\times \widetilde \oH_{\Phi}^i(\mathsf G, M)^{\la \p\supset \s\ra}\rightarrow \widetilde \oH_{\Phi}^{i}( \mathsf G, M'\otimes M)^{\la \p\supset \s\ra}
  \]
  for relative completed cohomologies is defined by requiring 
that the diagram 
    \[
 \begin{CD}
         \widetilde \oH^0(D\fS, M')^{\la \g\ra}\times \widetilde \oH_{\Phi}^i(D\fS, M)^{\la \p\ra} @>    >> \widetilde \oH_{\Phi}^{i}( D\fS, M'\otimes M)^{\la \p\ra}\\
            @V VV          @VV   V\\
          \widetilde \oH^0(\mathsf G, M')^{\la \g\supset \s\ra}\times \widetilde \oH_{\Phi}^i(\mathsf G, M)^{\la \p\supset \s\ra} @>  \smallsmile  >> \widetilde \oH_{\Phi}^{i}( \mathsf G, M'\otimes M)^{\la \p\supset \s\ra}\\
            \end{CD}
\]
commutes for all $D$ and  $\fS$, where the top horizontal arrow is the inverse limit over $k\in \BN$ of the maps
\[
  \smallsmile: \left(\varinjlim_{\fG\supset \fS}\oH^0(D\fG, M'/p^k)\right)\times \left(\varinjlim_{\fP\supset \fS} \oH_{\Phi}^i(D\fP, M/p^k)\right) \rightarrow \varinjlim_{\fP\supset \fS} \oH_{\Phi}^{i}( D\fP, (M'\otimes M)/p^k)
  \]
  as in \eqref{cup00injl}.

When $R$ is a $\Q$-algebra, the cup product map 
\be\label{cup007}
  \smallsmile:  \oH^0(\mathsf G, M')\times \widehat{\oH_{\Phi}^i(\mathsf G, M)} \rightarrow \widehat{\oH_{\Phi}^i(\mathsf G, M'\otimes M)}
  \ee
  for the formal completion  is defined by requiring 
that the diagram 
    \[
 \begin{CD}
         \oH^0( K, M')\times \widehat{\oH_{\Phi}^i(\mathsf G, M)}@>  \smallsmile  >> \widehat{\oH_{\Phi}^i(\mathsf G, M'\otimes M)}\\
            @V VV          @VV   V\\
          \oH^0( K, M') \times {\oH_{\Phi}^i(\mathsf G,  M)}^K@>  \eqref{cup0} >> {\oH_{\Phi}^i(\mathsf G, M'\otimes M)}^K\\
            \end{CD}
\]
commutes for all open compact subgroups $K$ of $\BG$.
Note that the map \eqref{cup007} restricts to a map 
\[
  \smallsmile:  \oH^0(\mathsf G, M')\times \widehat{\oH_{\Phi}^i(\mathsf G, M)}_{\p-\mathrm{sm}} \rightarrow \widehat{\oH_{\Phi}^i(\mathsf G, M'\otimes M)}_{\p-\mathrm{sm}}.
 \]

It is  routine to check the following lemma by definitions. 
 
\begin{leml}\label{lemcup}
    The diagram
    \[
 \begin{CD}
      \oH^0(\mathsf G, M')\times \oH_{\Phi}^i(\mathsf G, M)^{\la \p\ra} @>  \smallsmile  >> \oH_{\Phi}^{i}( \mathsf G, M'\otimes M)^{\la \p\ra}\\
            @V VV          @VV   V\\
          \widetilde \oH^0(\mathsf G, M')^{\la \g\supset \s\ra}\times \widetilde \oH_{\Phi}^i(\mathsf G, M)^{\la \p\supset \s\ra} @>  \smallsmile  >> \widetilde \oH_{\Phi}^{i}( \mathsf G, M'\otimes M)^{\la \p\supset \s\ra}\\
            \end{CD}
\]
commutes, and when $R$ is a $\Q$-algebra, the diagrams 
\[
 \begin{CD}
      \oH^0(\mathsf G, M')\times \oH_{\Phi}^i(\mathsf G, M) @>  \smallsmile  >> \oH_{\Phi}^{i}( \mathsf G, M'\otimes M)\\
            @V VV          @VV   V\\
         \oH^0(\mathsf G, M')\times \widehat{\oH_{\Phi}^i(\mathsf G, M)} @>  \smallsmile  >> \widehat{ \oH_{\Phi}^{i}( \mathsf G, M'\otimes M)}\\
            \end{CD}
\]
and \[
 \begin{CD}
      \oH^0(\mathsf G, M')\times \widehat{\oH_{\Phi}^i(\mathsf G, M)}_{\p-\mathrm{sm}}\otimes \RD(\g/\p) @>  \smallsmile  >> \widehat{ \oH_{\Phi}^{i}( \mathsf G, M'\otimes M)}_{\p-\mathrm{sm}}\otimes \RD(\g/\p)\\
            @V \eqref{identrsm} VV          @VV  \eqref{identrsm} V\\
         \oH^0(\mathsf G, M')\times \oH_{\Phi}^i(\mathsf G, M)^{\la \p\ra} @>  \smallsmile  >>  \oH_{\Phi}^{i}( \mathsf G, M'\otimes M)^{\la \p\ra}\\
            \end{CD}
\]
commute. 
\end{leml}

\section{Pull-backs and integrations} \label{sec:PBI}

All unexplained notations in this section will be as in the last section. 

Recall the homomorphism $\imath: \dot{\mathsf G}\rightarrow \mathsf G$ from Section \ref{secnotation}.  Fix  a closed 
 subgroup of $\dot{\mathsf G}(\R)$ of the form $\dot K_\infty = \dot{\mathsf A}(\R)^\circ\cdot \dot K_\infty'$, where $\dot{\mathsf A}$ is a split torus in $\dot{\mathsf G}$ defined over $\Q$ that is central in the identity connected component of $\dot{\sf G}$ modulo its  unipotent radical, and $\dot K_\infty'$ is a compact subgroup that normalizes  $\dot{\mathsf  A}(\R)^\circ$. Assume that  $\imath(\dot K_\infty)\subset K_\infty$. 
%Let  $\dot K_\infty\subset K_\infty$  be  a fixed closed subgroup of $\dot{\mathsf G}(\R)$ of the form $\dot K_\infty = \dot{\mathsf A}(\R)^\circ \cdot \dot K_\infty'$, where $\dot{\mathsf A}$ is a subtorus of  $\dot{\mathsf G} \cap \mathsf A$ defined over $\Q$, and $\dot K_\infty'$ is a compact subgroup of $\dot{\sf G}(\R)\cap K_\infty'$. 
Then we have a  map
\[
\imath: \dot{\mathscr X}\to \mathscr X, 
\]
where
\[
  \dot{\mathscr X}:=(\dot{\mathsf G}(\R)/\dot K_\infty^\circ)\times \dot{ \mathsf G}(\A^\infty).
\]
 % Assume that $\imath(\dot D\dot \fS)\subset D\fS$.
Put 
\[
\dot \BG:=\imath^{-1}(\BG)\subset \dot{\sf G}^\natural:=\dot K_\infty^\natural\times \dot{\sf G}(\A^\infty).
\]
Recall that $\dot G:=\dot{\mathsf G}(\Q_p)$ whose Lie algebra is denoted by $\dot \g$. Also recall that  $\dot \p:=\imath^{-1}(\p)\subset \dot \g$,  and let $\dot \s$ be a Lie subalgebra of $\imath^{-1}(\s)\subset \dot \g$.

Similar to \eqref{drs}, in the rest of this paper, 
\[ %\label{drs2}
\begin{cases}
    \textrm{$\dot D$ denotes an open compact subgroup of $\dot{\mathsf G}^{\natural, p}\cap \dot \BG$};&\\
     \textrm{$\dot \fG$ denotes an open compact subgroup of $\dot G\cap \dot \BG$};& \\
    \textrm{$\dot \fP$ denotes a compact subgroup of $\dot G\cap \dot \BG$ with Lie algebra $\dot \p$};& \\
    \textrm{$\dot \fS$ denotes a compact subgroup of $\dot G\cap \dot \BG$ with Lie algebra $\dot \s$}. & \\
\end{cases}
\]

\subsection{Pull-back for formal cohomologies}

Let $C$ be a neat compact subgroup of $\BG$ and $\dot C$ a neat compact subgroup of $\dot \BG$ such that 
\begin{itemize}
    \item $\imath(\dot C)\subset C$,  and
    \item the map
$
\imath:  \dot{\sf G}^\natural/\dot C\rightarrow {\sf G}^\natural/C
  $
  is  a local homeomorphism.
  \end{itemize}

\begin{lem}\label{pb1}
For every open compact subgroup  $\dot K'$ of $\dot \BG$ containing $\dot C$, there exist a neat open compact subgroup $\dot K$ of $\dot \BG$ containing $\dot C$ and a neat open compact subgroup $ K$ of $\BG$ containing $C$  such that
\begin{itemize}
    \item $ \dot K\subset \dot K'$,  $\imath(\dot K)\subset K$, and
    \item the map
$
\imath:   \dot K/\dot C\rightarrow  K/C
$
is bijective.
\end{itemize}
  
\end{lem}

\begin{proof}
Without loss of generality, assume that $\dot K'$ is sufficiently small so that
the map
\[
  \imath: \dot K'/\dot C\rightarrow  {\sf G}^\natural/C
\]
is injective. Take a neat open compact subgroup $ K$ of $\BG$ containing $C$ that  is small enough so that
\[
  K/C\subset  \imath( \dot K'/\dot C).
 \]
Now  the lemma follows by taking 
$
  \dot K:=\dot K'\cap \imath^{-1}(K).
$
\end{proof}

 Let $\dot \Phi$ be a family of closed subsets of $\dot{\mathsf G}(\Q)\backslash \dot{\mathscr X}$ that satisfies the analogous conditions for $\Phi$ in Section \ref{phicon}.
Assume that under the map
 $
 \imath: \dot{\mathsf G}(\Q)\backslash \dot{\mathscr X}\rightarrow  {\mathsf G}(\Q) \backslash \mathscr X
 $, 
 \be\label{precompact}
\textrm{the preimage of every set in $\Phi$ 
 belongs to $\dot \Phi$.}
  \ee
We define the pull-back map
\[
   \imath^*: \oH_\Phi^i(C, M)\rightarrow \oH^{i}_{\dot \Phi}(\dot C, M)
   \]
as in the following lemma.

\begin{prpl}\label{pb2}
There  is a unique map
\[
   \imath^*: \oH_\Phi^i(C, M)\rightarrow \oH^{i}_{\dot \Phi}(\dot C, M)
   \]
   such that the diagram
\be\label{comcdc}
 \begin{CD}
          \oH_\Phi^i(C, M)@> \imath^* >>    \oH^{i}_{\dot \Phi}(\dot C, M)\\
            @V\textrm{the projection map} VV          @VV \textrm{the projection map} V\\
  \oH_\Phi^i(K, M)@> \imath^* >>    \oH^{i}_{\dot \Phi}(\dot K, M)\\
            \end{CD}
\ee
commutes for all  neat open compact subgroups $\dot K$ of $\dot \BG$ containing $\dot C$ and  all neat open compact subgroups $ K$ of $\BG$ containing $C$ such that 
\be\label{bijec1}
\textrm{ $\imath(\dot K)\subset K
\ \ $ and the map 
$\  \imath:  \dot K/\dot C\rightarrow  K/C  $ is bijective.}
\ee
 Here the bottom horizontal arrow $\imath^*$ of \eqref{comcdc} is the usual pull-back map for cohomology groups (similar notation will be used later on).

\end{prpl}

\begin{proof}
The uniqueness follows from Lemma \ref{pb1}. To prove the existence, it suffices to show that
the diagram
\be\label{cdiota}
 \begin{CD}
          \oH_\Phi^i(K', M)@> \imath^* >>    \oH^{i}_{\dot \Phi}(\dot K', M)\\
            @V\rho VV          @VV \rho V\\
  \oH_\Phi^i(K, M)@> \imath^* >>    \oH^{i}_{\dot \Phi}(\dot K, M)\\
            \end{CD}
\ee
commutes  all neat open compact subgroups $\dot K'$ of $\dot \BG$ containing $\dot C$ and  all neat open compact subgroups $ K'$ of $\BG$ containing $C$ such that 
\[
  K'\subset K, \quad \dot K'\subset \dot K, \quad\imath(\dot K')\subset K'
\]
and that the map
\be\label{bijec2}
  \imath:  \dot K'/\dot C\rightarrow K'/ C \ \textrm{ is bijective. }
\ee
It follows from \eqref{bijec1} and \eqref{bijec2} that  the map
\be\label{bijec22}
  \imath:  \dot K/\dot K'\rightarrow K/K' \ \textrm{ is bijective. }
\ee
Consider the commutative  diagram 
\[
 \begin{CD}
            S^{\mathsf G}_{K'}@<\imath<<  S^{\dot{ \mathsf G}}_{\dot K'}\\
            @VVV          @VVV\\
            S^{\mathsf G}_{K}@<\imath<<  S^{\dot{ \mathsf G}}_{\dot K}.\\
  \end{CD}
\]
The neatness condition implies that the left vertical arrow and the right  vertical arrow are finite folds covering maps of topological manifolds, with fibers $K/K'$ and $\dot K/\dot K'$ respectively. Then \eqref{bijec22} implies that the above commutative diagram   
is Cartesian. This implies that  the diagram \eqref{cdiota} commutes. 
\end{proof}

\begin{leml}\label{limitpc}
    Let $C'\subset C$  and $\dot C'\subset\dot C$ be open subgroups such that 
$\imath(\dot C')\subset C'$. Then  the diagram
\[
 \begin{CD}
          \oH_{\Phi}^i(C, M) @> \imath^* >>     \oH_{\dot \Phi}^i(\dot C, M) \\
            @VVV          @VVV\\
 \oH_{\Phi}^i(C', M) @> \imath^* >>     \oH_{\dot \Phi}^i(\dot C', M) \\
            \end{CD}
\]
commutes, where the vertical arrows are the pull-back maps.
\end{leml}
\begin{proof}
  % We assume without loss of generality that $\dot C=\imath^{-1}(C)$ and $\dot C'=\imath^{-1}(C')$.
\begin{comment}
   Pick a neat  open compact subgroup  $\dot K'_0$ of $\dot \BG$ containing $\dot C'$ such that
   \begin{itemize}
       \item it is normalized by $\dot C$, 
       \item $\dot K_0'\cap \dot C=\dot C'$, and 
       \item the map
    $
    \imath :  \dot K_0'/\dot C'\rightarrow 
    {\sf G}^\natural / C'
$
is injective. 
   \end{itemize}
   Note that such groups form a neighorhood basis of $\dot C'$ in $\dot \BG$.
       Pick a neat  open compact subgroup $K'$ of $\BG$ containing  $C'$ such that
        \begin{itemize}
       \item it is normalized by $C$, 
       \item $K'\cap C=C'$, and 
       \item $K'\subset \imath(\dot K_0')C'$.
   \end{itemize}

       Now we set 
       \[
       K:=K' C, \quad \dot K':=\dot K_0'\cap \imath^{-1}(K'), \quad \textrm{and} 
       \quad \dot K:=\dot K' \dot C.
       \]
\end{comment}
Pick  neat open compact subgroups $\dot K'$ and $\dot K$ of $\dot \BG$ containing $\dot C'$ and $\dot C$ respectively, such that 
$\dot K' \subset \dot K$ and that $\dot K'/\dot C' \to \dot K /\dot C$ is bijective. It follows from the proof of Proposition \ref{actrt2} that such groups $\dot K'$ form a neighborhood basis of $\dot C'$ in $\dot \BG$. Using Lemma \ref{pb1}, by shrinking $\dot K'$ and $\dot K$ if necessary we may assume that there exist neat open compact subgroups $K'$ and $K$ of $\BG$ such that $\imath(\dot K')\subset K'$, $\imath(\dot K)\subset K$ and that the maps $\imath: \dot K' /\dot C' \to K'/C'$ and $\imath: \dot K/\dot C \to K/C$ are bijective. It follows that 
\[
K' = \imath(\dot K')C' \subset \imath(\dot K)C = K.
\]
Then we have a commutative diagram
    \[
 \begin{CD}
          K/C @< \imath<<    \dot K/\dot C \\
            @AAA          @AAA\\
  K'/C' @< \imath<<    \dot K'/\dot C' .\\
            \end{CD}
\]
Since all arrows except the left vertical one are bijective,  the left vertical arrow is  bijective as well. Now the lemma follows by considering the diagram 
    \[
\xymatrix{
\oH_{\Phi}^i(K, M) \ar[rrr]\ar[ddd] &    &   & \oH_{\dot \Phi}^i(\dot K, M)\ar[ddd]\\
& \oH_{\Phi}^i(C, M)\ar[r]\ar[d]\ar[lu]&\oH_{\dot \Phi}^i(\dot C, M)\ar[d]\ar[ru]&\\
&\oH_{ \Phi}^i(C', M)\ar[r]\ar[ld] &\oH_{\dot \Phi}^i(\dot C', M)\ar[rd] & \\
\oH_{ \Phi}^i(K', M) \ar[rrr]&& & \oH_{\dot \Phi}^i(\dot K', M) . 
}
\]
\end{proof}

%, which is a Lie subalgebra of the Lie algebra  $\dot \g$ of $\dot G:=\dot{\sf G}(\Q_p)$.  

Recall that $\p$ is transversal to $\imath(\dot \g)$. 
The following two lemmas are direct consequences of  Lemma \ref{limitpc}. 
\begin{leml}\label{pullr}
    Assume that $D\fS$ and  $\dot D\dot \fS$ are neat, and $\imath(\dot D\dot \fS)\subset D\fS$. Then there is a unique homomorphism 
    \be\label{pull-backr00}
\imath^*: \oH_{\Phi}^i(D\fS, M)^{\la  \p\ra}\rightarrow \oH_{\dot \Phi}^i(\dot{D}\dot{\fS},M)^{\la \dot \p\ra}
\ee
such that the diagram
\[
 \begin{CD}
        \oH_\Phi^i(D\fP, M) @> \textrm{pull-back}>>   \oH_\Phi^i(\dot D\dot \fP, M) \\
            @VVV          @VVV\\
          \oH_\Phi^i(D\fS, M)^{\la \p\ra} @>\imath^* >>    \oH_\Phi^i(\dot D\dot \fS, M)^{\la \dot \p\ra}\\
  \end{CD}
\]
commutes for all $\fP$ and $\dot \fP$ such that 
\[
\fP\supset \fS,\quad 
\dot \fP\supset \dot \fS,\quad D\fP\textrm{ and } \dot D \dot \fP\textrm{ are neat, $\quad$ and }\imath(\dot \fP)\subset \fP.
\]

\end{leml}

\begin{leml}\label{pullr2}
    There is a unique homomorphism 
    \be\label{pull-backr00000}
\imath^*: \oH_{\Phi}^i(\mathsf G, M)^{\la  \p\supset \s\ra}\rightarrow \oH_{\dot \Phi}^i(\dot{\mathsf G},M)^{\la \dot \p\supset \dot \s\ra}
\ee
such that the diagram
\[
 \begin{CD}
        \oH_\Phi^i(D\fS, M)^{\la \p\ra} @>\eqref{pull-backr00} >>    \oH_\Phi^i(\dot D\dot \fS, M)^{\la \dot \p\ra}\\
            @VVV          @VVV\\
         \oH_{\Phi}^i(\mathsf G, M)^{\la  \p\supset \s\ra}@>\imath^* >>    \oH_{\dot \Phi}^i(\dot{\mathsf G},M)^{\la \dot \p\supset \dot \s\ra}\\
  \end{CD}
\]
commutes for all $D$, $\fS$, $\dot D$ and $\dot \fS$ as in Lemma \ref{pullr}.
\end{leml}

It is clear that the homomorphism \eqref{pull-backr00000} is independent of $\s$ and $\dot \s$ (see \eqref{isors}). We also call it the pull-back map and write it as   
\[%\label{pull-backr0000000}
\imath^*: \oH_{\Phi}^i(\mathsf G, M)^{\la  \p\ra}\rightarrow \oH_{\dot \Phi}^i(\dot{\mathsf G},M)^{\la \dot \p\ra}.
\]
Specializing the above map to the case when $\p=\g$ and $\dot \p=\dot \g$, we obtain the pull-back map 
\[ %\label{pull-backr2}
\imath^*: \oH_{\Phi}^i({\sf G}, M)\rightarrow \oH_{\dot \Phi}^i(\dot{\sf G}, M). 
\]

In the setting of Lemma \ref{pullr}, the homomorphism 
    \eqref{pull-backr00} is natural in the coefficient module $M$. Thus it yields a homomorphism 
\[
\imath^*: \widetilde \oH_{\Phi}^i(D\fS, M)^{\la  \p\ra}\rightarrow \widetilde \oH_{\dot \Phi}^i(\dot{D}\dot{\fS},M)^{\la \dot \p\ra}.
\]
By taking the direct limits, this further induces a 
homomorphism 
\[
\imath^*: \widetilde \oH_{\Phi}^i(\mathsf G, M)^{\la  \p\supset \s\ra}\rightarrow \widetilde \oH_{\dot \Phi}^i(\dot{\mathsf G}, M)^{\la \dot \p\supset \dot \s\ra}.
\]
We have obvious commutative diagrams
\[
 \begin{CD}
    \oH_{\Phi}^i(D\fS, M)^{\la  \p\ra}@>\imath^*>>  \oH_{\dot \Phi}^i(\dot{D}\dot{\fS},M)^{\la \dot \p\ra}\\
            @VVV          @VV V\\
 \widetilde \oH_{\Phi}^i(D\fS, M)^{\la  \p\ra}@>\imath^*>> \widetilde \oH_{\dot \Phi}^i(\dot{D}\dot{\fS},M)^{\la \dot \p\ra} 
            \end{CD}
\]
and
\[
 \begin{CD}
    \oH_{\Phi}^i(\mathsf G, M)^{\la  \p\ra}@>\imath^*>>  \oH_{\dot \Phi}^i(\dot{\mathsf G},M)^{\la \dot \p\ra}\\
            @VVV          @VV V\\
 \widetilde \oH_{\Phi}^i(\mathsf G, M)^{\la  \p\supset \s\ra}@>\imath^*>> \widetilde \oH_{\dot \Phi}^i(\dot{\mathsf G},M)^{\la \dot \p\supset\dot \s\ra}.
            \end{CD}
\]

\subsection{Pull-back for formal completions}
In this subsection we assume that $R$ is a $\Q$-algebra. 
The proof of Proposition \ref{pb2} also proves the following proposition.

\begin{prpl}\label{pbf3} 
There  is a unique map
\be\label{pullfceq}
   \imath^*: \widehat{\oH_\Phi^i(\mathsf G, M)}^C\rightarrow \widehat{\oH^{i}_{\dot \Phi}(\dot{\mathsf G}, M)}^{\dot C}
   \ee
   such that the diagram
\[ %\label{comcdc}
 \begin{CD}
          \widehat{\oH_\Phi^i(\mathsf G, M)}^C @> \imath^* >>    \widehat{\oH^{i}_{\dot \Phi}(\dot{\mathsf G}, M)}^{\dot C}\\
            @V\textrm{the projection map} VV          @VV \textrm{the projection map} V\\
  \oH_\Phi^i(K, M)@> \imath^* >>    \oH^{i}_{\dot \Phi}(\dot K, M)\\
            \end{CD}
\]
commutes for all  neat open compact subgroups $\dot K$ of $\dot \BG$ containing $\dot C$ and  all neat open compact subgroups $ K$ of $\BG$ containing $C$ such that 
\[
\textrm{ $\imath(\dot K)\subset K
\ \ $ and the map 
$\  \imath:  \dot K/\dot C\rightarrow  K/C  $ is bijective.}
\]
 Here the identifications $ \oH_\Phi^i(K, M)=\widehat{\oH_\Phi^i(\mathsf G, M)}^K$ and $ \oH_{\dot \Phi}^i(\dot K, M)=\widehat{\oH_{\dot \Phi}^i(\dot{\mathsf G}, M)}^{\dot K}$ are used.

\end{prpl}

The proof of Lemma \ref{limitpc} also shows the following result. 
\begin{leml}\label{limitpcfc}
    Let $C'\subset C$  and $\dot C'\subset\dot C$ be open subgroups such that 
$\imath(\dot C')\subset C'$. Then  the diagram
\[
 \begin{CD}
          \widehat{\oH^{i}_{ \Phi}({\mathsf G}, M)}^{ C} @> \eqref{pullfceq} >>     \widehat{\oH^{i}_{\dot \Phi}(\dot{\mathsf G}, M)}^{\dot C} \\
            @V\subset VV          @VV\subset V\\
 \widehat{\oH^{i}_{\Phi}({\mathsf G}, M)}^{ C'} @> \eqref{pullfceq} >>     \widehat{\oH^{i}_{\dot \Phi}(\dot{\mathsf G}, M)}^{\dot C'} \\
            \end{CD}
\]
commutes.
\end{leml}

By Lemma \ref{limitpcfc}, there is a unique homomorphism
\[
   \imath^*: \widehat{\oH_\Phi^i(\mathsf G, M)}_{\p-\mathrm{sm}}\rightarrow \widehat{\oH^{i}_{\dot \Phi}(\dot{\mathsf G}, M)}_{\dot \p-\mathrm{sm}},
   \]
   to be called the pull-back map for formal completions, such that the diagram 
   \[
 \begin{CD}
        \widehat{\oH_\Phi^i(\mathsf G, M)}^{D\fP}@>\eqref{pullfceq} >>    \widehat{\oH_\Phi^i(\dot{\mathsf G}, M)}^{\dot D\dot \fP}\\
            @VVV          @VVV\\
         \widehat{\oH_\Phi^i(\mathsf G, M)}_{\p-\mathrm{sm}} @>\imath^* >>    \widehat{\oH^{i}_{\dot \Phi}(\dot{\mathsf G}, M)}_{\dot \p-\mathrm{sm}}\\
  \end{CD}
\]
commutes for $D$, $\dot D$, $\fP$ and  $\dot \fP$ such that 
\[
\textrm{$D\fP$ and $\dot D\dot \fP$ are neat, and  $\imath(\dot D\dot \fP)\subset D\fP$. 
}
\]

The transversality condition implies that $\g/\p=\dot \g/\dot \p$. 

\begin{leml}\label{pullbackeq}
    The diagram 
     \[
 \begin{CD}
        \widehat{\oH_\Phi^i(\mathsf G, M)}_{\p-\mathrm{sm}}\otimes \mathrm D(\g/\p) @>\imath^* >>    \widehat{\oH^{i}_{\dot \Phi}(\dot{\mathsf G}, M)}_{\dot \p-\mathrm{sm}}\otimes \mathrm D(\dot \g/\dot \p)\\
            @V \eqref{identrsm} VV          @VV \eqref{identrsm} V\\
         {\oH_\Phi^i(\mathsf G, M)}^{\la \p\ra} @>\imath^* >>    {\oH^{i}_{\dot \Phi}(\dot{\mathsf G}, M)}^{\la \dot \p\ra}\\
  \end{CD}
\]
commutes. 
\end{leml}
\begin{proof}
    It suffices to show that the diagram
     \[
 \begin{CD}
        \widehat{\oH_\Phi^i(\mathsf G, M)}^{D\fP}\otimes \mathrm D(\g/\p) @>\imath^* >>    \widehat{\oH^{i}_{\dot \Phi}(\dot{\mathsf G}, M)}^{\dot D\dot \fP}\otimes \mathrm D(\dot \g/\dot \p)\\
            @VVV          @VVV\\
         {\oH_\Phi^i({D\fP}, M)} @>\imath^* >>    {\oH^{i}_{\dot \Phi}(\dot D \dot \fP}, M)\\
  \end{CD}
\]
commutes when $D\fP$ and $\dot D \dot \fP$ are neat, and $\imath(\dot D \dot \fP)\subset D\fP$. This is routine to check and we omit the details. 
\end{proof}

 %Finally, for every continuous representation of $\fS_0$ on a finite-dimensional $E$-vector space $V$,Applying  the homomorphisms \eqref{pull-backr00c} and \eqref{pull-backr00c2} to the case of $\fV=\fV^t$ (and tensoring with $E$), we obtain linear maps\be\label{pull-backr00ce} \imath^*: \widetilde \oH_{\Phi}^i(D\fS, V)^{\la  \p\ra}\rightarrow \widetilde \oH_{\dot \Phi}^i(\dot{D}\dot{\fS},V)^{\la \dot \p\ra}.\ee and \eqref{pull-backr00c}\be\label{pull-backr00c2e}\imath^*: \widetilde \oH_{\Phi}^i(\mathsf G, V)^{\la  \p\supset \s\ra}\rightarrow \widetilde \oH_{\dot \Phi}^i(\dot{\mathsf G},V)^{\la \dot \p\supset \dot \s\ra}.\ee

\subsection{Integrations for formal cohomologies}

Fix a $\dot {\mathsf G}(\R)$-invariant orientation $\omega_{\dot{\mathsf G}}$ on $\dot{\mathsf G}(\R)/ \dot K_\infty^\circ$, which always exists and is unique up to sign. It yields a generator of $\mathrm O(\dot{\mathsf G}(\R)/\dot K_\infty^\circ)$, which is still denoted by $\omega_{\dot {\mathsf G}}$. 
It also induces  an orientation on the manifold
\[
  S^{\dot {\mathsf G}}_{ \dot K}:=\dot{\mathsf G}(\Q)\backslash \dot{ \mathscr X}/\dot K,
\]
for every completely neat open compact subgroup $\dot K$ of ${\dot{\mathsf G}^\natural}$. Using this orientation, the pairing against the fundamental class gives a map
  \be\label{intm0}
  \int_{\omega_{\dot {\mathsf G}}}:   \oH^{i}_{\mathrm c}( \dot K, R)\rightarrow R.
\ee
Here $R$ carries the trivial action of $\dot{\mathsf G}(\Q)\times  \dot{\mathsf G}^\natural$. %, and the subscript $\mathrm c$ indicates cohomology with compact support.  
By convention, the map \eqref{intm0} is identically zero unless $i=\dim ( \dot{\mathsf G}(\R)/\dot K^\circ_\infty)$. 
Note that for each open compact subgroup $\dot K'$ of $\dot K$, the diagrams
\be\label{lemintc2} %\be\label{comcdc0pf}
 \begin{CD}
           \oH^{i}_{\mathrm c}(\dot{K'}, R)@>\int_{\dot\omega_{{\mathsf G}}}>>    R \\
            @V\textrm{the push-forward  map}VV          @|\\ 
            \oH^{i}_{\mathrm c}(\dot{K}, R)@> \int_{\omega_{\dot{\mathsf G}}}>>  R\\
            \end{CD}
\ee
and
\be\label{lemintc23} %\be\label{comcdc0pf}
 \begin{CD}
           \oH^{i}_{\mathrm c}(\dot{K}, R)@>\int_{\omega_{\dot{\mathsf G}}}>>    R \\
            @V\textrm{the pull-back map}VV          @VV\textrm{multiplication by $[\dot K: \dot K']$}V \\ 
            \oH^{i}_{\mathrm c}(\dot{ K'}, R)@> \int_{\omega_{\dot{\mathsf G}}}>>  R\\
            \end{CD}
\ee
are commutative. 

For every completely neat compact subgroup $\dot C$ of $\dot{\mathsf G}^\natural$, we define the integration map \be\label{intm1}
  \int_{\omega_{\dot{\mathsf G}}}:   \oH^{i}_{\mathrm c}( \dot C, R)\rightarrow R
\ee 
to be the composition of 
\[ %\label{intm0'}
    \oH^{i}_{\mathrm c}( \dot C, R) \xrightarrow{\textrm{the projection}}\oH^{i}_{\mathrm c}(\dot K, R)\xrightarrow{\int_{\omega_{\dot{\mathsf G}}}} R,
\]
where $\dot K$ is a completely neat open compact subgroup of $\dot{\mathsf G}^\natural$ containing $\dot C$. 
By \eqref{lemintc2}, this map is independent of the choice of $\dot K$. 

\begin{leml}\label{lemintc233}
    Let $\dot C$ and $\dot C'$ be completely neat compact subgroups of $\dot{\mathsf G}^\natural$ that are commeasurable with each other. %Assume that $C\cap C'$ is open  in both $C$ and $C'$. 
    Then the diagram
    \[
    \begin{CD}
           \oH^{i}_{\mathrm c}(\dot C, R)@>\int_{\omega_{\dot{\mathsf G}}}>>    R \\
            @V\textrm{the transfer map}VV          @VV \textrm{multiplication by $[\dot C:\dot C\cap \dot C']$} V \\ 
            \oH^{i}_{\mathrm c}({ \dot C'}, R)@> \int_{\omega_{\dot{\mathsf G}}}>>  R\\
            \end{CD}
                \]
                commutes. 
\end{leml}
\begin{proof}
    In view of Proposition \ref{actrt2}, this easily follows from the commutative diagrams \eqref{lemintc2} and \eqref{lemintc23}. 
\end{proof}

As in \eqref{dgp}, set 
\[ 
 \RD(\dot{\mathsf G}, \dot\p):= \RD(\dot{\mathsf G}^{\natural,p})\otimes \RO(\dot{\mathsf G}(\R)/\dot K_\infty^\circ)\otimes \RD( \dot \p).
\]
 When $\dot D\dot \fP$ is completely neat,
we define the integration map
\be\label{intdrr}
\begin{array}{rcl}
    \int: \oH_\mathrm c^i(\dot D\dot\fP, R)\otimes \RD(\dot{\mathsf G},\dot \p)&\rightarrow & \Q\otimes R, 
 \\
    \phi\otimes \mu\otimes \omega_{\dot{\sf G}}\otimes \mu' &\mapsto & \mu(\dot D)\cdot \mu'(\dot \fP)\otimes \int_{\omega_{\dot{\sf G}}} \phi.
\end{array}
\ee
%where the top horizontal arrow is the map \[\phi\otimes \mu\otimes \omega_{\sf G}\otimes \nu\mapsto \mu(D)\cdot \nu(\fP)\cdot \int_{\omega_{\sf G}} \phi.\]
Here we use the obvious identification $\RD(\dot \fP)=\RD(\dot \p)$ via the logarithmic map, and  similar identifications will be used without further explanation.

By Lemma \ref{lemintc233}, 
we have a unique homomorphism 
\be\label{intr}
\int: \oH_\mathrm c^i(\dot{\mathsf G}, R)^{\la \dot \p\ra}\otimes \RD(\dot{\mathsf G},\dot \p)\rightarrow \Q\otimes R 
\ee
such that the diagram
    \[
    \begin{CD}
           \oH^{i}_{\mathrm c}(\dot D\dot \fP, R)\otimes \RD(\dot{\mathsf G}, \dot \p)@>\int >>   \Q\otimes  R \\
            @VVV          @VV= V \\ \oH_\mathrm c^i(\dot{\mathsf G}, R)^{\la \dot \p\ra}\otimes \RD(\dot{\mathsf G}, \dot \p)@> \int>>  \Q\otimes R\\
            \end{CD}
                \]
                commutes for all $\dot D$ and $\dot \fP$ such that $\dot D\dot \fP$ is completely neat.  Specifying the homomorphism \eqref{intr} to the case when  $\dot \p=\dot \g$, we get a homomorphism
\be\label{intr2}
\int: \oH_\mathrm c^i(\dot{\mathsf G}, R)\otimes \RD(\dot{\mathsf G})\rightarrow R, 
\ee
where \[
  \mathrm{D}(\dot{\mathsf G}):=\mathrm D(\dot{\mathsf G}^{\natural})\otimes \mathrm{O}(\dot{\mathsf G}(\R)/ \dot K_\infty^\circ),
  \] 
  as in  \eqref{rmd}. % we have a one-dimensional rational vector space 

  Set $\widehat R:=\varprojlim_{k\in \BN} R/p^k$. Since the integration map \eqref{intm1} is natural in  $R$, it yields a 
map 
\[
\int_{\omega_{\dot{\sf G}}}: \widetilde \oH_\mathrm c^i(\dot C, R)\rightarrow\widehat R. 
\]
  Now we define a map 
  \[ %\label{intr2rc}
\int: \widetilde \oH_\mathrm c^i(\dot{\mathsf G}, R)^{\la\dot \p\supset \dot \p\ra}\otimes \RD(\dot{\mathsf G}, \dot \p)\rightarrow \Q\otimes \widehat R,
\]
to be called the integration map  for relative completed cohomologies, 
by the following lemma. 

\begin{leml}
 There is a unique homomorphism  
 \[
 \label{intr2rc22}
\int: \widetilde \oH_\mathrm c^i(\dot{\mathsf G}, R)^{\la\dot \p\supset \dot \p\ra}\otimes \RD(\dot{\mathsf G}, \dot \p)\rightarrow \Q\otimes \widehat R
\]
such that the diagram 
\[
 \begin{CD}
\widetilde \oH_{\mathrm c}^{i}(\dot D\dot \fP, R) \otimes  \RD(\dot{\mathsf G}, \dot \p)@>>>\Q\otimes \widehat R\\
            @V VV          @VV = V\\
\widetilde \oH_\mathrm c^i(\dot{\mathsf G}, R)^{\la\dot \p\supset \dot \p\ra}\otimes \RD(\dot{\mathsf G}, \dot \p)@>\int >> \Q\otimes \widehat R
            \end{CD}
\]
commutes for all $\dot D$ and $\dot \fP$ such that $\dot D\dot \fP$ is completely neat, where the top horizontal arrow is the map
\[
\phi\otimes \mu\otimes \omega_{\dot{\sf G}}\otimes \mu'\mapsto \mu(\dot D)\cdot \mu'(\dot \fP)\otimes \int_{\omega_{\dot{\sf G}}}\phi.
\]

\end{leml}
 \begin{proof}
     This is an easy consequence of Lemma \ref{lemintc233}. 
 \end{proof}
 
The following lemma is easily verified. 
\begin{leml}
    The diagram 
\[
 \begin{CD}
\oH_{\mathrm c}^{i}(\dot{\mathsf G}, R)^{\la \dot \p\ra} \otimes \mathrm D(\dot{\mathsf G},\dot \p)@>>>   \widetilde \oH_{\mathrm c}^{i}(\dot{\mathsf G}, R)^{\la \dot \p\supset \dot \p\ra} \otimes \mathrm D(\dot{\mathsf G},\dot \p)\\
            @V\int VV          @VV \int V\\
\Q\otimes R@>  >>     \Q\otimes \widehat R\\
            \end{CD}
\]
commutes.
\end{leml}

%Then by using Lemma \ref{lemformco}, we get an identificationwhere the top horizontal arrow is the map\[\widehat \phi\otimes \mu\otimes \nu\mapsto  \{ \mu(\fG) \cdot \la \nu, \mu_\fP\ra\cdot \phi_{D\fG}\}_{\fG}.\]

\subsection{Integrations for formal completions}
In this subsection we assume that $R$ is a $\Q$-algebra. We define the  integration map 
\[
\int: \widehat{\oH_\mathrm c^i(\dot{\mathsf G}, R)}\otimes \RD(\dot{\mathsf G})\rightarrow R
\]
as in the following lemma, whose easy proof is omitted.  %is easily checked. %The formal topology on $\widehat{\oH_\Phi^i({\mathsf G}, M)}$ is defined to be the inverse limit topology, with $\oH_\Phi^i({\mathsf G}, M)^K$ equipped with the discrete topology. 

\begin{leml}%Suppose that $R$ is a $\Q$-algebra.
  There is a unique $R$-homomorphism    
  \be\label{intr222}
\int: \widehat{\oH_\mathrm c^i(\dot{\mathsf G}, R)}\otimes \RD(\dot{\mathsf G})\rightarrow R
\ee
such that the diagram 
\[
 \begin{CD}
    \widehat{\oH_\mathrm c^i(\dot{\mathsf G}, R)}\otimes \RD(\dot{\mathsf G})@> \int >>   R\\
            @V\textrm{projection}VV          @VV =V\\
 {\oH_{\mathrm{c}}^i(\dot{\mathsf G}, R)^{\dot K}}\otimes \mathrm D(\dot{\mathsf G})@> \int >>  R\\
            \end{CD}
\]
commutes for all completely neat open compact subgroups $\dot K$ of $\dot{\mathsf G}^{\natural}$, where the bottom horizontal arrow is the restriction of the map \eqref{intr2}. Moreover, the map \eqref{intr222} extends the map \eqref{intr2}.
  %The map   \[\int: \oH_\mathrm c^i(\mathsf G, R)\otimes \RD(\mathsf G)\rightarrow R \]uniquely extends to a continuous map\be\label{intr222}\int: \widehat{\oH_\mathrm c^i(\mathsf G, R)}\otimes \RD(\mathsf G)\rightarrow R.\eeHere $\widehat{\oH_\mathrm c^i(\mathsf G, R)}$ is equipped with the formal topology, and $R$ is equipped with the discrete topology. 
\end{leml}

We also have 
the obvious identification 
\be\label{diso}
 \mathrm{D}(\dot{\mathsf G})=\mathrm D(\dot \g/\dot \p)\otimes  \mathrm{D}(\dot{\mathsf G}, \dot \p).
\ee
\begin{leml}\label{intint}
%Suppose that $R$ is a $\Q$-algebra. Then t
The diagram 
\[
 \begin{CD}
\widehat{\oH_{\mathrm c}^i(\dot{\mathsf G}, R)}_{{\dot \p}-\mathrm{sm}}\otimes \mathrm D(\dot{\mathsf G})@> \eqref{identrsm}\textrm{ and } \eqref{diso} >>   \oH_{\mathrm c}^{i}(\dot{\mathsf G}, R)^{\la \dot \p\ra} \otimes \mathrm D(\dot{\mathsf G},\dot \p)\\
            @V\int VV          @VV \int V\\
R@> = >>     R\\
            \end{CD}
\]
commutes.
\end{leml}
\begin{proof}
    Suppose that $\dot D\dot \fP$ is completely neat,  and let $\phi\in \widehat{\oH_\Phi^i(\dot{\mathsf G}, M)}^{\dot D\dot \fP}$. Write
    \[
    \eta:=\mu\otimes \nu_{\dot \fP}\otimes \mu_{\dot D}\otimes \mu_{\dot \fP}\otimes \omega_{\dot{\mathsf G}}\in \mathrm D(\dot{\mathsf G}), 
    \]
where $\mu\in \mathrm D(\dot \g)=\mathrm D(\dot G)$, $\nu_{\dot \fP}\in \mathrm D(\dot \p)^\vee=\mathrm D(\dot \fP)^\vee$ is the functional $\lambda\mapsto \lambda(\dot \fP)$, $\mu_{\dot D}\in \mathrm D(\dot{\mathsf G}^{\natural, p})$ is the distribution such that $\mu_{\dot D}(\dot D)=1$, and $\mu_{\dot \fP}\in \mathrm D(\dot \p)=\mathrm D(\dot \fP)$ is the distribution such that $\mu_{\dot \fP}(\dot \fP)=1$.  
%As in the Introduction we write $G:=\mathsf G(\Q_p)$.
The image of $\phi\otimes \eta$ under the top horizontal arrow is represented by 
\[
\{ \mu(\dot \fG) \cdot \phi_{\dot D\dot \fG}\}_{\dot \fG}\otimes \mu_{\dot D}\otimes \mu_{\dot \fP}\otimes \omega_{\dot{\mathsf G}}\in \oH_{\mathrm c}^{i}(\dot{\mathsf G}, M)^{\la \dot \p\ra} \otimes \mathrm D(\dot{\mathsf G},\dot \p) ,  
\]
where $\dot \fG$ runs over all  $\dot D$-neat 
open compact subgroups of $\dot G$ containing $\dot \fP$, and $\phi_{\dot D\dot \fG}$ denotes the projection of $\phi$ to $\widehat{\oH_\Phi^i(\dot{\mathsf G}, M)}^{\dot D\dot \fG}$. % such that $\dot D \dot \fG$ is neat. 
We have that
\begin{eqnarray*}
    &&\int \{ \mu(\dot \fG) \cdot \phi_{\dot D\dot \fG}\}_{\dot \fG}\otimes \mu_{\dot D}\otimes \mu_{\dot \fP}\otimes \omega_{\dot{\mathsf G}}\\
    &=& \mu(\dot \fG) \cdot \int_{\omega_{\dot{\mathsf G}} }\phi_{\dot D\dot \fG}
    \qquad (\textrm{see }  \eqref{intdrr})\\
    &=& \int \phi\otimes\eta.
    \end{eqnarray*}
    This proves the lemma. 
\end{proof}

\section{Modular symbols} \label{secpint}

%We continue with the notation of the last section. 
Recall from  the Introduction that $\mathsf E$ is a subfield of $\overline \Q$ that is contained in a closed subfield $E$ of $\C_p$.  %Write \[\CO:=\{x\in E\,:\, \abs{x}_p\leq 1\}\]for the ring of integers in $E$. 
In this section, let $\sf V$ be a representation of ${\sf G}(\Q)$ over $\sf E$, and $V$ a continuous finite-dimensional representation of $G$ over $E$.  Suppose that the representation $\sf V$ is identified with a $\mathsf G(\Q)$-stable $\sf E$-form of $V$. In particular, $V=E\otimes \mathsf V$ as  vector spaces. %via $\iota_\mathsf V$.  we are given an $\sf E$-linear map \[\iota_{\mathsf V}: \mathsf V\rightarrow V\]that is $\mathsf G(\Q)$-equivariant with respect to the inclusion map $\mathsf G(\Q)\subset G$. 

Throughout this section we assume that  under the map
 $
 \imath: \dot{\mathsf G}(\Q)\backslash \dot{\mathscr X}\rightarrow  {\mathsf G}(\Q) \backslash \mathscr X
 $, 
 \be\label{precompact2}
\textrm{the preimage of every set in $\Phi$ 
is compact.}
  \ee

\subsection{Some cohomology groups}

Recall that $\p\supset \s$ are Lie subalgebras of $\g$. Note that every finite-dimensional representation of $\s$ over $E$ integrates to continuous  representations of some compact subgroups of $G$ with Lie algebra $\s$.  This establishes an equivalence between the category of all finite-dimensional representations of $\s$ over $E$ with the category of all pairs $(\fS_1, V_1)$ where $\fS_1$ is a compact subgroup of $G$ with Lie algebra $\s$, and $V_1$ is a continuous finite-dimensional representation of $\fS_1$ over $E$. A morphism $ (\fS_1, V_1)\rightarrow (\fS_2, V_2)$ in the latter category is defined to be a linear map $V_1\rightarrow V_2$ that is equivariant under some open subgroups  of $\fS_1\cap \fS_2$. 

\begin{dfnl}\label{defintegral}
     Let $V_1$ be a finite-dimensional representation of $\s$ over $E$. Let $\fV_1$ be an $\CO$-lattice of $V_1$. Integrate the representation $V_1$ of $\s$ to a compact subgroup $\fS$ of $G$ with Lie algebra $\s$ such that $\fS$ stabilizes $\fV_1$.  View $\fV_1$ as  an $\CO[\mathsf G(\Q)\times {\mathsf G}^{\natural,p}\times \fS] $-module with the given action of $\fS$ and the trivial action of $\mathsf G(\Q)\times {\mathsf G}^{\natural,p}$. Define 
\[
\widetilde{\oH}_\Phi^i({\mathsf G},  V_1)^{\la \p\supset \s\ra, \circ}:=E\otimes \widetilde{\oH}_\Phi^i({\mathsf G},  \fV_1)^{\la \p\supset \s \ra}\qquad (\textrm{see Remark \ref{remarkp}})
\]
to be called the integral relative completed cohomology space. 
\end{dfnl}
Note that Definition  \ref{defintegral} is independent of the choices of $\fV_1$, $\fS$, and the representation of $\fS$ on $V_1$ that integrates the representation of $\s$. Moreover, $\widetilde{\oH}_\Phi^i({\mathsf G},  \,\cdot\, )^{\la \p\supset \s\ra, \circ}$ is a functor from the category of finite-dimensional representation of $\s$ over $E$ to the category of locally convex topological vector spaces over $E$. Here we write
\[
\widetilde{\oH}_\Phi^i({\mathsf G},  V_1)^{\la \p\supset \s\ra, \circ} =
\varinjlim_{D\fS_1} E\otimes \widetilde{\oH}_\Phi^i(D\fS_1,  \fV_1)^{\la \p \ra},
\]
($D$ runs over open compact subgroups of $\mathsf G^{\natural, p}$ and $\fS_1$ runs over open compact subgroups of $\fS$) and view it as a locally convex topological space under the direct limit topology, where the topology on $E\otimes \widetilde{\oH}_\Phi^i(D\fS_1,  \fV_1)^{\la \p \ra}$ is given by the seminorm associated to the image of the natural map 
\[
\widetilde{\oH}_\Phi^i(D\fS_1,  \fV_1)^{\la \p \ra}\rightarrow E\otimes \widetilde{\oH}_\Phi^i(D\fS,  \fV_1)^{\la \p \ra}.
\]

%and the integration 

Set $\sf V_\C:=\C \otimes_{\sf E} \sf V $. View $\sf V$ and $\sf V_\C$ as  $\sf E[\mathsf G(\Q)\times {\mathsf G}^\natural] $-modules with the given action of $\mathsf G(\Q)$ and the trivial action of ${\mathsf G}^\natural$. View $V$ as  an $E[\mathsf G(\Q)\times {\mathsf G}^\natural] $-module with the given action of $G$ and the trivial action of $\mathsf G(\Q)\times {\mathsf G}^{\natural,p}$. 
%Recall that $\p$ is a Lie subalgebra of $\g:=\Lie(G)$. 
%Let $\s$ be a Lie subalgebra of $\p$ and let $V_0$ be a subspace of $V$ that is stabilized by some  compact subgroups of $G$ with Lie algebra $\s$.

%We make the following as in the Introduction.
\begin{dfnl}\label{defintegral2}
     Let $\fV$ be an $\CO$-lattice of $V$. Define 
\[
{\oH}_\Phi^i({\mathsf G},  V)^{\la \p\ra, \circ}:=E\otimes {\oH}_\Phi^i({\mathsf G},  \fV)^{\la \p \ra},
\]
to be called the integral relative  cohomology space. 

\end{dfnl}
Definition \ref{defintegral2}
is  independent of $\fV$. 

As an example of Definition  \ref{defintegral}, we have the integral relative completed cohomology space
\[\widetilde{\oH}_\Phi^i({\mathsf G},  V)^{\la \p\supset \s\ra, \circ}:=E\otimes \widetilde{\oH}_\Phi^i({\mathsf G},  \fV)^{\la \p\supset \s \ra},
\]
where $\fV$ is an $\CO$-lattice of $V$. The natural map
\[
 {\oH}_\Phi^i({\mathsf G},  \fV)^{\la \p \ra}\rightarrow  \widetilde{\oH}_\Phi^i({\mathsf G},  \fV)^{\la \p\supset \s \ra} \qquad (\textrm{see \eqref{interprc}})
\]
yields a linear map 
\be\label{interp5}
{\oH}_\Phi^i({\mathsf G},  V)^{\la \p \ra,\circ}\rightarrow \widetilde{\oH}_\Phi^i({\mathsf G},  V)^{\la \p\supset \s\ra, \circ}.
\ee
The inclusion map $\fV\rightarrow V$ yields a homomorphism 
\[
{\oH}_\Phi^i({\mathsf G},  \fV)^{\la \p \ra}\rightarrow  {\oH}_\Phi^i({\mathsf G},  V)^{\la \p \ra},
\]
which further induces a linear map
\be\label{interp6}
{\oH}_\Phi^i({\mathsf G},  V)^{\la \p \ra,\circ}\rightarrow  {\oH}_\Phi^i({\mathsf G},  V)^{\la \p \ra}.
\ee
Note that both linear maps  \eqref{interp5} and \eqref{interp6} are independent of $\fV$.

The map 
 \[
   \mathscr X\times \mathsf V  \rightarrow \mathscr X\times V, \quad (x,u)\mapsto (x, x_p^{-1}. u)
 \]
 is $\mathsf G(\Q)\times \mathsf G^\natural$-equivariant, where $x_p$ denotes the image of $x$ under the projection map $\mathscr X\rightarrow G$. Thus it induces a homomorphism 
 \[
    \mathsf V_{[K]}\rightarrow V_{[K]} 
 \]
 of sheaves over $S^{\mathsf G}_{K}$, for every open compact subgroup $K$ of $\mathsf G^\natural$. This further induces  homomorphisms of various cohomology spaces, to be denoted by 
 \be\label{homsfvv}
 \begin{cases}
     \iota_{\mathsf V}: \oH_\Phi^i(K, \mathsf V)\rightarrow \oH_\Phi^i(K, V);& \smallskip\\
     \iota_{\mathsf V}: {\oH_\Phi^i({\mathsf G}, \mathsf V)}\rightarrow {\oH_\Phi^i({\mathsf G},  V)};&\smallskip\\
     \iota_{\mathsf V}: \widehat{\oH_\Phi^i({\mathsf G}, \mathsf V)}_{{\p}-\mathrm{sm}}\rightarrow \widehat{\oH_\Phi^i({\mathsf G},  V)}_{{\p}-\mathrm{sm}}. %;&\smallskip\\ \iota_{\mathsf V}: \widehat{\oH_\Phi^i({\mathsf G}, \mathsf V)}_{{\p}-\mathrm{sm}} \otimes \RD(\g/\p)\rightarrow  \widehat{\oH_\Phi^i({\mathsf G}, V)}_{{\p}-\mathrm{sm}}\otimes \RD(\g/\p). 
 \end{cases}
 \ee
% \[    \iota_{\mathsf V}: \oH_\Phi^i(K, \mathsf V)\rightarrow \oH_\Phi^i(K, V),\quad   \iota_{\mathsf V}: \widehat{\oH_\Phi^i({\mathsf G}, \mathsf V)}_{{\p}-\mathrm{sm}}\rightarrow  \widehat{\oH_\Phi^i({\mathsf G}, V)}_{{\p}-\mathrm{sm}},  \] and  \[  \iota_{\mathsf V}: \widehat{\oH_\Phi^i({\mathsf G}, \mathsf V)}_{{\p}-\mathrm{sm}} \otimes \RD(\g/\p)\rightarrow  \widehat{\oH_\Phi^i({\mathsf G}, V)}_{{\p}-\mathrm{sm}}\otimes \RD(\g/\p). \]
 
Let $V_0$ be an $E\otimes \s$-submodule of $V$. In this section, we will consider modular symbols on  the spaces appearing in  the following diagram: 

\be\label{thediag00220000}
\xymatrix{
&\widehat{\oH_\Phi^i({\mathsf G}, \sf V_\C)}_{{\p}-\mathrm{sm}}\otimes \RD(\g/\p)&\widehat{\oH_\Phi^i({\mathsf G}, \sf V)}_{{\p}-\mathrm{sm}}\otimes \RD(\g/\p)\ar[l]_{\C\otimes(\,\cdot\,)}\ar[d]^{\iota_{\sf V}}\\
\widetilde{\oH}_\Phi^i({\mathsf G},  V_0)^{\la \p\supset \s\ra,\circ}\ar[d]^{\textrm{the map induced by $V_0\subset V$}} &%\mathscr H\ar[l]\ar[d]\ar[ru] \ar[u]\ar[r]
&\widehat{\oH_\Phi^i({\mathsf G},  V)}_{{\p}-\mathrm{sm}}\otimes \RD(\g/\p)\ar[d]^{\eqref{identrsm}}_{\cong}\\
\widetilde{\oH}_\Phi^i({\mathsf G},  V)^{\la \p\supset \s\ra,\circ} &{\oH}_\Phi^i({\mathsf G},  V)^{\la \p\ra, \circ}\ar[l]_{\eqref{interp5}}\ar[r]^{\eqref{interp6}} &   {\oH_\Phi^i({\mathsf G},  V)}^{\la \p\ra}.\\
            }
\ee

\subsection{Modular symbols}

Recall that $\sf Z$ is an algebraic torus over $\Q$, and  $\jmath:  \dot{\sf G} \to {\sf Z}$ is an algebraic  homomorphism over $\Q$.  With the fixed closed subgroup of 
$\sf Z(\R)$ as in \eqref{maximalkz}, we define various cohomology groups attached to $\sf Z$ as in the previous sections.

%The fixed closed subgroup of  $\sf Z(\R)$ (as in \eqref{kinf}) is taken to be the largest possible one, namely,  the group\[ K_{\mathsf Z,\infty}:= \mathsf A_{\mathsf Z}(\R)^\circ \cdot \textrm{(the maximal compact subgroup of $\sf Z(\R)$)},\]where $\mathsf A_{\mathsf Z}$ is the largest split torus in $\mathsf Z$. Here and as before, we still use $\jmath$ to denote various maps induced by the algebraic homomorphism $\jmath$. 
Recall from Section \ref{secnotation} that $Z_0$ is an algebraic subtorus  of $Z:=\mathsf Z(\Q_p)$ whose Lie algebra is denoted by $\z_0$. The space 
$\widetilde{\oH}^0({\mathsf Z},  E)^{\la \z \supset \z_0\ra, \circ}$
 is identified with the space of all continuous $E$-valued functions on $\sf Z(\Q)\backslash  \sf Z^\natural$ that are invariant under some open compact subgroups of $D_{Z_0}^\mathrm{max}$. 
 %of the form $D_{\mathsf Z}\cdot\fZ_{\jmath(\dot\p)}$, where $D_{\mathsf Z}$ is an open compact subgroups of $\mathsf Z(\A^{\infty p})$ and $\fZ_{\jmath(\dot\p)}$ is a compact subgroups of $\mathsf Z(\Q_p)$ with Lie algebra $\jmath(\dot\p)$. Here $\z$ denotes the Lie algebra of $Z:=\mathsf Z(\Q_p)$. 
%We define the $p$-adic period integral map\[  \widetilde{\oH}^0({\mathsf Z},  E)^{\la \z \supset \{0\}\ra, \circ}\]

%\subsection{$p$-adic interpolations}

Recall from the Introduction that  
\[
\s\supset \imath(\dot \p)\quad\textrm{and}\quad \z_0\supset \jmath(\dot \p).
\]
  %  contains $\jmath(\dot \p)$. .  
Fix a functional $\lambda_0\in  \Hom_{E\otimes \dot \p}(V_0, E)$.

\begin{dfnl}\label{df:RCMS}
    The relative completed  modular symbol map
\be\label{RCMS}     %RCMS indicates Relative Completed Modular Symbol; other labels are similar
\widetilde{\mathscr M}_{\lambda_0}: \widetilde{\oH}^0({\mathsf Z},  E)^{\la \z \supset \z_0\ra, \circ}\times \left(\widetilde{\oH}_\Phi^i(\mathsf{G},  V_0)^{\la \p\supset \s \ra, \circ} \otimes \RD(\dot{\mathsf G},\dot \p)\right)\rightarrow E
\ee
is the bilinear map whose pull-back to 
\[
\widetilde{\oH}^0({\mathsf Z},  \CO)^{\la \z \supset \z_0\ra}\times \left(\widetilde{\oH}_\Phi^i(\mathsf{G},  \fV_0)^{\la \p\supset \s \ra}\otimes \RD(\dot{\mathsf G},\dot \p)\right)
\]
equals the composition of 
\begin{eqnarray*}
    &&\widetilde{\oH}^0({\mathsf Z},  \CO)^{\la \z \supset \z_0\ra}\times \left(\widetilde{\oH}_\Phi^i(\mathsf{G},  \fV_0)^{\la \p\supset \s \ra}\otimes \RD(\dot{\mathsf G},\dot \p)\right)\\
&\xrightarrow{\textrm{pull-back}}&\widetilde{\oH}^0(\dot{\mathsf G},  \CO)^{\la \dot \g \supset \dot \p \ra}\times \left(\widetilde{\oH}_{\mathrm c}^i(\dot{\mathsf{G}},  \fV_0)^{\la \dot \p\supset \dot \p \ra}\otimes  \RD(\dot{\mathsf G},\dot \p)\right)\\
 &\xrightarrow{\textrm{cup product}}& \widetilde{\oH}_{\mathrm c}^i(\dot{\mathsf{G}},  \fV_0)^{\la \dot \p\supset \dot \p \ra} \otimes \RD(\mathsf G,\dot \p)\\ %&\xrightarrow{\textrm{pull-back}}& \widetilde{\oH}_\Phi^i(\dot{\mathsf{G}},  \fV_0)^{\la \dot \p\supset \dot \p \ra}\otimes \RD(\dot{\mathsf G},\dot \p)\\
&\xrightarrow{\lambda_0}& \widetilde{\oH}_{\mathrm c}^i(\dot{\mathsf{G}},  \CO)^{\la \dot \p\supset \dot \p \ra}\otimes \RD(\dot{\mathsf G},\dot \p)\\
      &\xrightarrow{\int}& E,
\end{eqnarray*}
where $\fV_0$ is an $\CO$-lattice of $V_0$ such that $\lambda_0(\fV_0)\subset  \CO$.
\end{dfnl}
It is clear that the bilinear map \eqref{RCMS} is well-defined, separately continuous under the natural topologies, and independent of $\fV_0$.

%Similarly, we have a $\C_p$-bilinear  map \[\widetilde{\mathscr M} : \con(\mathsf{Z}, \C_p)\times \left(\widetilde{\oH}_\Phi^i(\mathsf{G},  \C_0\otimes V_0)^{\la \p\supset \s \ra, \circ}\otimes \RD(\dot{\mathsf G},\dot \p)\right)\rightarrow \C_p\]which extends the map \eqref{padiint}. Here $\con(\mathsf{Z}, \C_p)$ denotes the space of all $\C_p$-valued continuous functions on 
%that is $\C_p$-linear on the first variable and $E$-linear on the second variable, 
%where \[\con(\mathsf{Z}, \C_p):=\varinjlim_{E_1} \con(\mathsf{Z},  E_1),\]and $E_1$ runs over all finite extensions of $E$ in $\C_p$. 

Let $\sf w$ be an algebraic character of $\sf Z_{\sf E}$ defined over $\sf E$. As in \eqref{defwp} we have a character 
\[
  \mathsf w_p: Z\subset \mathsf Z(E)\xrightarrow{\sf w} E^\times. 
\]
Write $E_{\sf w}:=E$, to be viewed as an
$E[\sf Z(\Q)\times \sf Z^\natural]$-module via the  action of $Z$ by $\mathsf w_p$ and the trivial action of ${\sf Z}(\Q)\times {\sf Z}^{\natural,p}$.  
%Then $\oH^0(\sf Z, \overline \Q_{\sf w})$ equals the space of all locally constant maps form 
Fix  a functional $\lambda_{V, \sf w}\in  \Hom_{E\otimes \dot \g}(E_{\mathsf w} \otimes V, E)$. Similar to Definition \ref{df:RCMS} (and Definition \ref{df:IRMS} below), we have the relative completed  modular symbol map
\[    %RCMS indicates Relative Completed Modular Symbol; other labels are similar
\widetilde{\mathscr M}_{\lambda_{V, \sf w}}: \widetilde{\oH}^0({\mathsf Z},  E_\mathsf w)^{\la \z \supset \z_0\ra, \circ}\times \left(\widetilde{\oH}_\Phi^i(\mathsf{G},  V)^{\la \p\supset \s \ra, \circ} \otimes \RD(\dot{\mathsf G},\dot \p)\right)\rightarrow E.
\]

Write $\CO_\mathsf w:=\CO$, to be viewed as an $\CO[\sf Z(\Q)\times \sf Z^\natural]$-submodule of $E_\mathsf w$.  
Similar to Definitions \ref{df:RCMS}, %and \ref{df:SMS}, 
we make the following three definitions. 

\begin{dfnl}\label{df:IRMS}
    The integral relative modular  symbol map
\be\label{IRMS}
{\mathscr M}_{\lambda_{V, \sf w}}^\circ: {\oH}^0({\mathsf Z},  E_{\sf w})\times \left({\oH}_\Phi^i(\mathsf{G},  V)^{\la \p\ra, \circ} \otimes \RD(\dot{\mathsf G},\dot \p)\right)\rightarrow E
\ee
is the bilinear map whose pull-back to 
\[
{\oH}^0({\mathsf Z},  \CO_{\sf w})\times \left({\oH}_\Phi^i(\mathsf{G},  \fV)^{\la \p \ra}\otimes \RD(\dot{\mathsf G},\dot \p)\right)
\]
equals the composition of 
\begin{eqnarray*}
    &&{\oH}^0({\mathsf Z},  \CO_{\sf w})\times \left({\oH}_\Phi^i(\mathsf{G},  \fV)^{\la \p \ra}\otimes \RD(\dot{\mathsf G},\dot \p)\right)\\
     &\xrightarrow{\textrm{pull-back}}&{\oH}^0(\dot{\mathsf G},  \CO_{\mathsf w})\times \left({\oH}_{\mathrm c}^i(\dot{\mathsf{G}},  \fV)^{\la \dot \p \ra}\otimes  \RD(\dot{\mathsf G},\dot \p)\right)\\
     &\xrightarrow{\textrm{cup product}}& {\oH}_{\mathrm c}^i(\dot{\mathsf{G}},  \CO_{\mathsf w}\otimes \fV)^{\la \dot \p \ra} \otimes \RD(\mathsf G,\dot \p)\\
           &\xrightarrow{\lambda_{V, \sf w}}& {\oH}_{\mathrm c}^i(\dot{\mathsf{G}},  \CO)^{\la  \dot \p \ra}\otimes \RD(\dot{\mathsf G},\dot \p)\\
      &\xrightarrow{\int}& E,
\end{eqnarray*}
where $\fV$ is an $\CO$-lattice of $V$ such that $\lambda_{V, \sf w}(\CO_\mathsf w\otimes \fV)\subset \CO$. 
\end{dfnl}

The above definition is independent of the choice of $\fV$. 

\begin{dfnl}\label{df:RMS}
    The  relative modular  symbol map
\be\label{RMS}
{\mathscr M}_{\lambda_{V, \sf w}}: {\oH}^0({\mathsf Z},  E_{\sf w})\times \left({\oH}_\Phi^i(\mathsf{G},  V)^{\la \p\ra} \otimes \RD(\dot{\mathsf G},\dot \p)\right)\rightarrow E
\ee
is the composition of 
\begin{eqnarray*}
    &&{\oH}^0({\mathsf Z},  E_{\sf w})\times \left({\oH}_\Phi^i(\mathsf{G},  V)^{\la \p \ra}\otimes \RD(\dot{\mathsf G},\dot \p)\right)\\
     &\xrightarrow{\textrm{pull-back}}&{\oH}^0(\dot{\mathsf G},  E_{\sf w})\times \left({\oH}_{\mathrm c}^i(\dot{\mathsf{G}},  V)^{\la \dot \p \ra}\otimes  \RD(\dot{\mathsf G},\dot \p)\right)\\
     &\xrightarrow{\textrm{cup product}}& {\oH}_{\mathrm c}^i(\dot{\mathsf{G}},  E_{\sf w}\otimes V)^{\la \dot \p \ra} \otimes \RD(\mathsf G,\dot \p)\\    &\xrightarrow{\lambda_{V, \sf w}}& {\oH}_{\mathrm c}^i(\dot{\mathsf{G}},  E)^{\la \dot \p\supset \dot \p \ra}\otimes \RD(\dot{\mathsf G},\dot \p)\\
      &\xrightarrow{\int}& E.
\end{eqnarray*}
\end{dfnl}

\begin{dfnl}
    The stable modular symbol map
\[
\widehat{\mathscr M}_{\lambda_{V, \sf w}} : \oH^0({\sf Z}, { E}_{\sf w})\times \left(\widehat{\oH_\Phi^i(\mathsf{G},   V)}_{\p-\mathrm{sm}}\otimes \RD(\dot{\mathsf G})\right)\rightarrow  E
\]
is the composition of 
\begin{eqnarray*}
    &&\oH^0({\sf Z}, { E}_{\sf w})\times \left(\widehat{\oH_\Phi^i(\mathsf{G},   V)}_{\p-\mathrm{sm}}\otimes\RD(\dot{\mathsf G})\right)\\
     &\xrightarrow{\textrm{pull-back}}&{\oH}^0(\dot{\mathsf G},  { E}_{\sf w})\times \left(\widehat{\oH_{\mathrm c}^i(\dot{\mathsf{G}},   V)}_{\dot \p-\mathrm{sm}}\otimes \RD(\dot{\mathsf G})\right)\\
     &\xrightarrow{\textrm{cup product}}& \widehat{\oH_{\mathrm c}^i(\dot{\mathsf{G}},   E_{\sf w}\otimes V)}_{\dot \p-\mathrm{sm}}\otimes  \RD(\dot{\mathsf G})\\
            &\xrightarrow{\lambda_{V, \sf w}}& \widehat{\oH_{\mathrm c}^i(\dot{\mathsf{G}},  E)}_{\dot{\p}-\mathrm{sm}}\otimes \RD(\dot{\mathsf G})\\
      &\xrightarrow{\int}& E.
\end{eqnarray*}
\end{dfnl}

Write $\mathsf E_{\sf w}:=\mathsf E$, to be viewed as an
$\sf E[\sf Z(\Q)\times \sf Z^\natural]$-module via the  action of $\mathsf Z(\Q)\subset \mathsf Z_{\mathsf E}(\sf E)$ by $\sf w$ and the trivial action of $\sf Z^\natural$. 
Fix a functional $\lambda_{\sf V, w}\in \Hom_{\dot{\sf G}(\Q)}(\sf E_{\sf w}\otimes\sf V, \sf E)$.
\begin{dfnl}\label{df:SMS}
    The stable modular symbol map
\[
\widehat{\mathscr M}_{\lambda_{\sf V, w}} : \oH^0({\sf Z}, {\sf E}_{\sf w})\times \left(\widehat{\oH_\Phi^i(\mathsf{G},  \mathsf V)}_{\p-\mathrm{sm}}\otimes \RD(\dot{\mathsf G})\right)\rightarrow \sf E
\]
is the composition of 
\begin{eqnarray*}
    &&\oH^0({\sf Z}, {\sf E}_{\sf w})\times \left(\widehat{\oH_\Phi^i(\mathsf{G},  \mathsf V)}_{\p-\mathrm{sm}}\otimes\RD(\dot{\mathsf G})\right)\\
     &\xrightarrow{\textrm{pull-back}}&{\oH}^0(\dot{\mathsf G},  {\sf E}_{\sf w})\times \left(\widehat{\oH_{\mathrm c}^i(\dot{\mathsf{G}},  \mathsf V)}_{\dot \p-\mathrm{sm}}\otimes \RD(\dot{\mathsf G})\right)\\
     &\xrightarrow{\textrm{cup product}}& \widehat{\oH_{\mathrm c}^i(\dot{\mathsf{G}},  \sf E_{\sf w}\otimes\mathsf V)}_{\dot \p-\mathrm{sm}}\otimes  \RD(\dot{\mathsf G})\\
            &\xrightarrow{\lambda_{\sf V, w}}& \widehat{\oH_{\mathrm c}^i(\dot{\mathsf{G}},  \sf E)}_{\dot{\p}-\mathrm{sm}}\otimes \RD(\dot{\mathsf G})\\
      &\xrightarrow{\int}& \sf E.
\end{eqnarray*}

\end{dfnl}

Write $\C_\mathsf w:=\C\otimes \mathsf E_\mathsf w$, which is a 
$\C[\sf Z(\Q)\times \sf Z^\natural]$-module. Similar to Definition \ref{df:SMS}, we have the stable modular symbol map
\[
\widehat{\mathscr M}_{\lambda_{\sf V, w}} : \oH^0({\sf Z}, {\C}_{\sf w})\times \left(\widehat{\oH_\Phi^i(\mathsf{G},  \mathsf V_\C)}_{\p-\mathrm{sm}}\otimes \RD(\dot{\mathsf G})\right)\rightarrow \sf \C. 
\]

\subsection{Compatibility of the modular symbols}

In this subsection, we show that the seven versions of the modular symbols defined in the last subsection are compatible with respect to the six arrows in the diagram  \eqref{thediag00220000}.

\begin{prpl}\label{lem56}
    The diagram
\[
 \begin{CD}
 \oH^0({\sf Z}, {\sf E}_{\sf w})\times \left(\widehat{\oH_\Phi^i(\mathsf{G},  \mathsf V)}_{\p-\mathrm{sm}}\otimes \RD(\dot{\mathsf G})\right)@> \widehat{\mathscr M}_{\lambda_{\sf V, w}} >>   \mathsf E\\
            @V \C\otimes(\,\cdot\,) VV          @VV \subset V\\
  \oH^0({\sf Z}, {\C}_{\sf w})\times \left(\widehat{\oH_\Phi^i(\mathsf{G},  \mathsf V_\C)}_{\p-\mathrm{sm}}\otimes \RD(\dot{\mathsf G})\right)@>\widehat{\mathscr M}_{\lambda_{\sf V, w}} >>   \C\\
            \end{CD}
\]
commutes.
\end{prpl}

\begin{proof}
    This follows from the fact that all the relevant operations are natural in   the coefficient modules. 
\end{proof}

Note that 
\[
\oH^0(\mathsf Z, E_{\mathsf w})=\oH^0(\mathsf Z,  E_{\mathsf w})^{\la \z \ra, \circ}, 
\]
which is identified with the space of all functions $f:\mathsf Z(\Q)\backslash \mathsf Z^\natural\rightarrow E$ such that for some open compact subgroups $D_{\sf Z}$ of $\mathsf Z(\A^{\infty p})$ and $\fZ$ of $Z:={\sf Z}(\Q_p)$, 
\[
 f(xg_0 g_1^{-1})=\mathsf w_p(g_1) \cdot f(x)\ \ \textrm{ for all $x\in \sf Z(\Q)\backslash \sf Z^\natural$, $g_0\in D_\mathsf Z$, and $g_1\in \fZ$}. 
\]
As an example of the second map in \eqref{homsfvv}, we have  the map  
\be\label{embpheck23}
  \oH^0({\mathsf Z}, {\sf E}_{\mathsf w})\rightarrow \oH^0({\mathsf Z}, { E}_{\mathsf w}), \qquad f\mapsto f\cdot \mathsf w_p^{-1}. 
\ee

%Suppose that the diagram\[ \begin{CD}\mathsf E_\mathsf w\otimes \mathsf V @> \lambda_{\sf V, w} >>     \sf E\\            @V \textrm{(inclusion)}    \otimes \iota_{\mathsf V} VV          @VV \subset V\\ E_\mathsf w\otimes  V @> \lambda_{V, \sf w} >>      E\\           \end{CD}\]commutes.

\begin{prpl} %\label{lem510}
Suppose that $\lambda_{V, \sf w}$ extends $\lambda_{\sf V, w}$. 
   Then the diagram
\[
 \begin{CD}
\oH^0(\mathsf Z, \mathsf E_{\mathsf w})\times \left(\widehat{\oH_\Phi^i({\mathsf G}, \sf V)}_{{\p}-\mathrm{sm}}\otimes \RD(\dot{\mathsf G})\right) @> \widehat{\mathscr M}_{\lambda_{\sf V, w}}  >>     \sf E\\
            @V \eqref{embpheck23}  \times \iota_{\mathsf V} VV          @VV \subset V\\
 {\oH}^0({\mathsf Z},  E_{\sf w})\times \left(\widehat{\oH_\Phi^i({\mathsf G},  V)}_{{\p}-\mathrm{sm}}\otimes \RD(\dot{\mathsf G})\right) @> \widehat{\mathscr M}_{\lambda_{V, \sf w}}  >>     E\\
            \end{CD}
\]
commutes.
\end{prpl}

\begin{proof}
It is routine to check that the following diagram is commutative: 
\[
 \begin{CD}
\oH^0(\mathsf Z, \mathsf E_{\mathsf w})\times \left(\widehat{\oH_\Phi^i({\mathsf G}, \sf V)}_{{\p}-\mathrm{sm}}\otimes \RD(\dot{\mathsf G})\right) @> \eqref{embpheck23}  \times \iota_{\mathsf V} >>  {\oH}^0({\mathsf Z},  E_{\sf w})\times \left(\widehat{\oH_\Phi^i({\mathsf G},  V)}_{{\p}-\mathrm{sm}}\otimes \RD(\dot{\mathsf G})\right)  \\
            @V \textrm{pull-back} VV          @VV \textrm{pull-back} V\\
 {\oH}^0(\dot{\mathsf G},  {\mathsf E}_{\sf w})\times \left(\widehat{\oH_{\mathrm c}^i(\dot{\mathsf{G}},   \mathsf V)}_{\dot \p-\mathrm{sm}}\otimes \RD(\dot{\mathsf G})\right) @>\eqref{embpheck23}  \times \iota_{\mathsf V} >>    {\oH}^0(\dot{\mathsf G},  {E}_{\sf w})\times \left(\widehat{\oH_{\mathrm c}^i(\dot{\mathsf{G}},   V)}_{\dot \p-\mathrm{sm}}\otimes \RD(\dot{\mathsf G})\right)\\
  @V \textrm{cup product} VV          @VV \textrm{cup product} V\\
 \widehat{\oH_{\mathrm c}^i(\dot{\mathsf{G}},   \mathsf E_\mathsf w\otimes \mathsf V)}_{\dot \p-\mathrm{sm}}\otimes \RD(\dot{\mathsf G}) @> \iota_{\mathsf V} >>    \widehat{\oH_{\mathrm c}^i(\dot{\mathsf{G}},   E_\mathsf w\otimes V)}_{\dot \p-\mathrm{sm}}\otimes \RD(\dot{\mathsf G})\\
 @V \lambda_{\sf V, w}  VV          @VV {\lambda_{V, \sf w}} V\\
 \widehat{\oH_{\mathrm c}^i(\dot{\mathsf{G}},   \mathsf E)}_{\dot \p-\mathrm{sm}}\otimes \RD(\dot{\mathsf G}) @> \mathsf E\subset E >>    \widehat{\oH_{\mathrm c}^i(\dot{\mathsf{G}},   E)}_{\dot \p-\mathrm{sm}}\otimes \RD(\dot{\mathsf G})\\
 @V \int  VV          @VV \int V\\
 \mathsf E @> \subset  >>    E.\\
            \end{CD}
\]
This proves the proposition. %We omit the details. 
\end{proof}

\begin{prpl}%\label{lem510}
    The diagram
\[
 \begin{CD}
\oH^0(\mathsf Z,  E_{\mathsf w})\times \left(\widehat{\oH_\Phi^i({\mathsf G},  V)}_{{\p}-\mathrm{sm}}\otimes \RD(\dot{\mathsf G})\right) @> \widehat{\mathscr M}_{\lambda_{V, \sf w}}  >>     E\\
            @V \eqref{identrsm}\textrm{ and }\eqref{diso} VV          @VV = V\\
 {\oH}^0({\mathsf Z},  E_{\sf w})\times \left({\oH}_\Phi^i(\mathsf{G},  V)^{\la \p \ra}\otimes \RD(\dot{\mathsf G},\dot \p)\right) @> {\mathscr M}_{\lambda_{V, \sf w}}  >>     E\\
            \end{CD}
\]
commutes.
\end{prpl}

\begin{proof}
This follows from Lemmas \ref{pullbackeq}, \ref{lemcup}, and \ref{intint}. 
\end{proof}

\begin{prpl}%\label{lem510}
    The diagram
\[
 \begin{CD}
 \widetilde{\oH}^0({\mathsf Z},  E_\mathsf w)^{\la \z \supset \z_0\ra, \circ}\times \left(\widetilde{\oH}_\Phi^i(\mathsf{G},  V)^{\la \p\supset \s \ra, \circ} \otimes \RD(\dot{\mathsf G},\dot \p)\right) @> {\widetilde{\mathscr M}}_{\lambda_{V, \sf w}} >> E\\
            @A  \eqref{interp5} AA          @AA = A\\
 {\oH}^0({\mathsf Z},  E_{\sf w})\times \left({\oH}_\Phi^i(\mathsf{G},  V)^{\la \p \ra,\circ}\otimes \RD(\dot{\mathsf G},\dot \p)\right) @> {\mathscr M}_{\lambda_{V, \sf w}}^\circ  >>     E\\
@V \eqref{interp6} VV @VV = V \\
 {\oH}^0({\mathsf Z},  E_{\sf w})\times \left({\oH}_\Phi^i(\mathsf{G},  V)^{\la \p \ra}\otimes \RD(\dot{\mathsf G},\dot \p)\right) @> {\mathscr M}_{\lambda_{V, \sf w}}>> E
            \end{CD}
\]
commutes.
\end{prpl}

\begin{proof}
    Similar to Proposition \ref{lem56}, this follows from the fact that all the relevant operations are natural in   the coefficient modules. 
\end{proof}

%\begin{prpl}\label{lem510}   The diagram\[ \begin{CD} {\oH}^0({\mathsf Z},  E_{\sf w})\times \left({\oH}_\Phi^i(\mathsf{G},  V)^{\la \p \ra,\circ}\otimes \RD(\dot{\mathsf G},\dot \p)\right) @> {\mathscr M}_{\lambda_{V, \sf w}}^\circ  >>     E\\@V\eqref{interp6}VV @VV = V \\ {\oH}^0({\mathsf Z},  E_{\sf w})\times \left({\oH}_\Phi^i(\mathsf{G},  V)^{\la \p \ra}\otimes \RD(\dot{\mathsf G},\dot \p)\right) @> {\mathscr M}_{\lambda_{V, \sf w}}>> E         \end{CD}\]commutes.\end{prpl}

Obviously identify $E_\mathsf w\otimes V$ with $V$ as vector spaces. Assume that  $\lambda_0=(\lambda_{V, \sf w})|_{V_0}$.  The following proposition also follows  from the fact that all the relevant operations are natural in  the coefficient modules. 
\begin{prpl}\label{lem510}
  Suppose that  ${\sf w}_p$ is trivial on $Z_0$ so that  
  \[
 \widetilde{\oH}^0({\mathsf Z},  E_\mathsf w)^{\la \z \supset \z_0\ra, \circ}=\widetilde{\oH}^0({\mathsf Z},  E)^{\la \z \supset \z_0\ra, \circ}.
  \]
  Then the diagram
\[
 \begin{CD}
 \widetilde{\oH}^0({\mathsf Z},  E)^{\la \z \supset \z_0\ra, \circ}\times \left(\widetilde{\oH}_\Phi^i(\mathsf{G},  V_0)^{\la \p\supset \s \ra, \circ} \otimes \RD(\dot{\mathsf G},\dot \p)\right) @> {\widetilde{\mathscr M}}_{\lambda_0} >> E\\
@V^{\textrm{the map induced by $V_0\subset V$}}VV @VV = V \\
 \widetilde{\oH}^0({\mathsf Z},  E_\mathsf w)^{\la \z \supset \z_0\ra, \circ}\times \left(\widetilde{\oH}_\Phi^i(\mathsf{G},  V)^{\la \p\supset \s \ra, \circ} \otimes \RD(\dot{\mathsf G},\dot \p)\right) @> {\widetilde{\mathscr M}}_{\lambda_{V, \sf w}} >> E\\
            \end{CD}
\]
commutes.
\end{prpl}

%The ordinary part $\mathscr H$ of $\CB_P(\oH_{\Phi}^{i}({\mathsf G}, \mathsf V)) \otimes \RD(\g/\p)$ fits into the commutative diagram
%\be\label{thediag} \begin{CD}\widetilde{\oH}_\Phi^i({\mathsf G},  V_0)^{\la \p\supset \s\ra,\circ} @<\widetilde \xi<<\mathscr H@> \subset >>  \widehat{\oH_\Phi^i({\mathsf G}, \sf V)}_{{\p}-\mathrm{sm}}\otimes \RD(\g/\p)\\ @V\eqref{mapv0}VV      @V V \xi V          @VV \eqref{sfvv} V\\\widetilde{\oH}_\Phi^i({\mathsf G},  V)^{\la \p\supset \s\ra,\circ} @<\eqref{embrcirc}<<{\oH}_\Phi^i({\mathsf G},  V)^{\la \p\ra, \circ}@> \eqref{embrcirc} >>     {\oH_\Phi^i({\mathsf G},  V)}^{\la \p\ra}.\\            \end{CD}\ee

Suppose that we are given  an $\sf E$-vector space $\mathscr H$ that fits into a commutative diagram 
\be\label{thediag0022000}
\xymatrix{
&\widehat{\oH_\Phi^i({\mathsf G}, \sf V_\C)}_{{\p}-\mathrm{sm}}\otimes \RD(\g/\p)&\widehat{\oH_\Phi^i({\mathsf G}, \sf V)}_{{\p}-\mathrm{sm}}\otimes \RD(\g/\p)\ar[l]_{\C\otimes(\,\cdot\,)}\ar[d]^{\iota_{\sf V}}\\
\widetilde{\oH}_\Phi^i({\mathsf G},  V_0)^{\la \p\supset \s\ra,\circ}\ar[d]^{\textrm{the map induced by $V_0\subset V$}} &\mathscr H\ar[l]_{\widetilde \xi}\ar[d]^{\xi^\circ}\ar[ru]^{\widehat \xi}\ar[u]^{\widehat \xi}\ar[r]^{\widehat \xi}&\widehat{\oH_\Phi^i({\mathsf G},  V)}_{{\p}-\mathrm{sm}}\otimes \RD(\g/\p)\ar[d]^{\eqref{identrsm}}_{\cong}\\
\widetilde{\oH}_\Phi^i({\mathsf G},  V)^{\la \p\supset \s\ra,\circ} &{\oH}_\Phi^i({\mathsf G},  V)^{\la \p\ra, \circ}\ar[l]_{\eqref{interp5}}\ar[r]^{\eqref{interp6}} &   {\oH_\Phi^i({\mathsf G},  V)}^{\la \p\ra}.\\
            }
\ee
We will construct such a space $\mathscr H$ in the next two  sections.

Note that $\dot{\mathsf G}(\Q)$ is dense in an open subgroup of $\dot G$. Hence as in \eqref{emblambda} we have natural inclusions
\[
\Hom_{\dot{\mathsf G}(\Q)}(\mathsf E_{\mathsf w}\otimes\mathsf V, \mathsf E)\subset \Hom_{E\otimes \dot \g}(E_{\mathsf w}\otimes V,  E)\subset \Hom_{E\otimes \dot \p}( V,  E),
\]
where $\sf w$ is an algebraic character of $\sf Z_{\sf E}$ defined over $\sf E$ such that  ${\sf w}_p$ is trivial on $Z_0$.  
In particular, $\lambda_{\sf V, w}$ is also viewed as an element of $\Hom_{E\otimes \dot \p}( V,  E)$. 

The following theorem asserts that relative completed modular symbols interpolate the stable modular symbols on $\mathscr H$, for various weights $\sf w$.  
\begin{thml}\label{thmc}
%Assume that $\lambda_0=\lambda|_{V_0}$.
 Let  $\lambda_0\in  \Hom_{E\otimes \dot \p}(V_0, E)$. Then for all algebraic characters $\sf w$ of $\sf Z_{\sf E}$ defined over $\sf E$ such that  ${\sf w}_p$ is trivial on $Z_0$, and all $\lambda_{\sf V, w}\in \Hom_{\dot{\mathsf G}(\Q)}(\mathsf E_\mathsf w\otimes \mathsf V, \mathsf E)$ such that  $(\lambda_{\sf V, w})|_{V_0}=\lambda_0$,
 the diagram
\[
 \begin{CD}
 {\oH}^0({\mathsf Z},  \C_{\sf w})\times \left(\widehat{\oH_\Phi^i({\mathsf G}, \sf V_\C)}_{{\p}-\mathrm{sm}}\otimes \RD(\dot{\mathsf G})\right)@> \widehat{\mathscr M}_{\lambda_{\sf V, w}} >>   \C\\
 @A\widehat \xi AA @AA\subset A \\
 \oH^0(\mathsf Z, \mathsf E_{\mathsf w})\times \left(\mathscr H\otimes \RD(\dot{\mathsf G},\dot \p)\right)@> \widehat{\mathscr M}_{\lambda_{\sf V, w}} \circ \widehat \xi >>  \mathsf E\\
            @V \eqref{embpheck00}\textrm{ and }\widetilde{\xi} VV          @VV \subset V\\
\widetilde{\oH}^0({\mathsf Z},  E)^{\la \z \supset \z_0\ra, \circ}\times \left(\widetilde{\oH}_\Phi^i(\mathsf{G},  V_0)^{\la \p\supset \s \ra, \circ} \otimes \RD(\dot{\mathsf G},\dot \p)\right) @> \widetilde{\mathscr M}_{\lambda_0}  >>     E\\
            \end{CD}
\]
commutes. Here the middle horizontal arrow is the composition of 
\[
\oH^0(\mathsf Z, \mathsf E_{\mathsf w})\times \left(\mathscr H\otimes \RD(\dot{\mathsf G},\dot \p)\right)\xrightarrow{\widehat{\xi}}\oH^0(\mathsf Z, \mathsf E_{\mathsf w})\times \left(\widehat{\oH_\Phi^i({\mathsf G}, \sf V)}_{{\p}-\mathrm{sm}}\otimes \RD(\dot{\mathsf G})\right)\xrightarrow{\widehat{\mathscr M}_{\lambda_{\sf V, w}} } E.
\]

\end{thml}
\begin{proof}
    This follows by combining Propositions \ref{lem56}--\ref{lem510}. 
\end{proof}

\section{Parabolic cohomologies and the $t$-stable part} \label{sec:PC}

In this section, we define parabolic cohomology groups which include the space $\CB_P(\oH_{\Phi}^{i}({\mathsf G}, \mathsf V)) \otimes \RD(\g/\p)$ in the Introduction as an example. These can be viewed as variants of the Jacquet modules studied in \cite{Em06b}. We will also define their ``$t$-stable parts", which are related to the nearly ordinary part.  %Here and henceforth, 

\subsection{Parabolic pairs}

%(b) An open compact subgroup $\fL$ of $L$ is said to be $t$-good if there is an
%where $t\fL t^{-1}=\fL$
%\end{dfnl}
In view of \cite[Section 5.2]{Be87} we make the following definition.
\begin{dfnl}
    A parabolic pair in $\g$ is  a pair $\p\supset \n$  of  its Lie subalgebras such that some $t\in G$ defines $\p\supset \n$ in the following sense:      \[
    \begin{cases}
    \textrm{ $t$ normalizes both $\p$ and $\n$;}&\\ 
         \textrm{all eigenvalues (in $\C_p^\times $) of $\Ad_t: \n\rightarrow \n$ have $p$-adic norms $>1$;}&\\
        \textrm{all eigenvalues of $\Ad_t: \p/\n\rightarrow\p/ \n$ have $p$-adic norm $1$;} & \\
         \textrm{all eigenvalues  of $\Ad_t: \g/\p\rightarrow\g/ \p$ have $p$-adic norms $< 1$.}  & 
    \end{cases} 
    \]
    Such an element $t$  is called a defining element of the parabolic pair. 
\end{dfnl}
Here and henceforth $\Ad$ indicates various maps induced by the conjugations. It is clear that every element of $G$  defines a unique parabolic pair in $\g$.

In what follows we suppose that $\p\supset \n$ is a parabolic pair in $\g$.  %Then there are connected algebraic subgroups $P\supset N$ of $G$ whose Lie algebras equal $\p$ and $\n$ respectively. Note that $N$ is a unipotent normal subgroup of $P$.  

\begin{dfnl}
    An element of $G$ is said to be split if it is the image of $p$ under some (unique) algebraic homomorphism $\Q_p^\times \rightarrow G$. 
\end{dfnl}

Put   
\[
\mathrm{Df}(\p,\n):=\{\textrm{split defining elements $t\in G$ of the parabolic pair $\p\supset \n$}\}.
\]
%This set is invariant under the conjugations by $P$. 

\begin{leml}
    The set $\mathrm{Df}(\p,\n)$ is nonempty. 
\end{leml}
\begin{proof}
    The lemma follows by the following observation: If we replace $t\in G$ by the semi-simple part of its Jordan decomposition,  or its positive power, or its multiplication by an element in a compact subgroup of $G$ that commutes with $t$, then the defined parabolic pair remains unchanged. 
\end{proof}

Let $P$ denote the normalizer in $G$ of the pair $\p\supset \n$. Note that the Lie algebra of $P$ is equal to $\p$. 

 Let $t\in \mathrm{Df}(\p,\n)$. Note that all eigenvalues of $\Ad_t:\n\rightarrow \n$ are negative powers of $p$. We have a unique $\Ad_t$-stable decomposition 
 \[
   \g=\bar{\n}_t\oplus \l_t\oplus \n
 \]
 such that all eigenvalues of $\Ad_t: \bar{\n}_t\rightarrow \bar{\n}_t$ are positive powers of $p$, and  $\l_t$ equals the %centralizer
 invariant space of $t$ in $\g$. 
 Denote by $L_t$ the normalizer of $\l_t$ in $P$. Note that the Lie algebra of $L_t$ is equal to $\l_t$. 

 \begin{leml}\label{algsubgp}
   The subgroups $P$ and $L_t$ of $G$ are both algebraic, and there are unipotent algebraic subgroups $N$ and $\bar N_t$ of $G$ whose Lie algebras are respectively equal to $\n$ and $\bar \n_t$. % P=L_{t}\ltimes N$.

\end{leml}
\begin{proof}
By realizing $G$ as an algebraic subgroup of a general linear group, it is easy to see that there are connected algebraic subgroups $P_0$, $L_{t,0}$, $N$ and $\bar N_t$ of $G$ whose Lie algebras are respectively equal to $\p$, $\l_t$, $\n$ and $\bar \n_t$. Moreover, $P\supset P_0$, $L_t\supset L_{t,0}$, and both $N$ and $\bar N_t$ are unipotent.  This implies the lemma.  %. We_{ sketch a proof for the convenience of the reader.  Realize $G$ as an algebraic subgroup of a general linear group. Then it is easy to see that $t$ is a central element of $L$, and the lemma then easily follows.     %such that $t$ is realized as a dia
\end{proof}

Let $N$ and $\bar N_t$ be as in Lemma \ref{algsubgp}. 
It is clear that the multiplication map \[
\bar N_t\times P\rightarrow G
\]
is an open embedding.

Set 
$L:=P/N$. Denote by  $P_0\subset P$ and $L_0\subset L$ the identity connected components under the Zariski topologies. 

\begin{leml}\label{lemc}
    For all $t_1, t_2\in \mathrm{Df}(\p,\n)$, there exists $g\in N$ such that $g t_1 g^{-1}$ commutes with $t_2$.
\end{leml}
\begin{proof}
In view of Lemma \ref{algsubgp} we assume without loss of generality that $G=P$ and $P$ is connected as an algebraic group. Then $L_{t_i}$ is connected as an algebraic group, $P=L_{t_i}\ltimes N$, and $t_i$ is a central element of $L_{t_i}$ ($i=1,2$).
    Write 
    \[
    L_{t_i}=L_i\ltimes N_i\quad 
    \]
    for a Levi decomposition of $L_{t_i}$ where $N_i$ is the unipotent radical. Then $L_i$ is also a Levi factor of $P$,  and $N_i\ltimes N$ equals the unipotent radical of $P$. By uniqueness of Levi factors in characteristic zero, we have  elements $g\in N$ and $g'\in  N_1$ such that 
    \[
    (gg') L_1 (gg')^{-1}=L_2.
    \]
    Then we have that
    \[
      g t_1 g^{-1}=(gg') t_1 (gg')^{-1}\in L_2\subset L_{t_2}.
    \]
    This proves the lemma. 
\end{proof}

\begin{leml}\label{lemc2}
    For all $t_1, t_2\in \mathrm{Df}(\p,\n)$, 
    \[
    t_1 t_2 =t_2 t_1\Longleftrightarrow \l_{t_1}=\l_{t_2}\Longleftrightarrow L_{t_1}=L_{t_2}.
    \]
\end{leml}
\begin{proof}
  First assume that $t_1 t_2 =t_2 t_1$.  Note that ${t_2}$ stabilizes $\l_{t_1}$. Hence it centralizes $\l_{t_1}$, by considering the eigenvalues of the operator $\Ad_{t_2}: \p\rightarrow \p$. Thus $\l_{t_1}\subset \l_{t_2}$. Similarly $\l_{t_2}\subset \l_{t_1}$ and hence $\l_{t_1}=\l_{t_2}$. All other implications are obvious.  
\end{proof}
%Following \cite[Section 5.2]{Be87} we set
%\[P:=P_t:= \left\{g\in G\mid \textrm{ $\{t^{-k} g t^k\,:\,{k\in \BN}\}$ is relative compact in $G$}\right\}\]and\[N:=N_t:= \left\{g\in G\mid \lim_{k\rightarrow \infty} t^{-k} g t^k=1\right\}.\]
%\be\label{parabolicl}\left\{\begin{array}{l}P:=P_t:= \left\{g\in G\mid \textrm{ $\{t^{-k} g t^k\,:\,{k\in \BN}\}$ is relative compact in $G$}\right\};\\N:=N_t:= \left\{g\in G\mid \lim_{k\rightarrow \infty} t^{-k} g t^k=1\right\};\\\bar P:=P_{t^{-1}};\\\bar N:=N_{t^{-1}};\\L:=P\cap \bar P.\end{array}\right.\ee
%We call $P$ the parabolic subgroup of $G$ attached to $t$, and call $t$ a defining element of $P$. %, and the Lie algebra of $P$ equals $\p$.
% When $G$ is connected and reductive as an  algebraic group, this notion of parabolic subgroups defined above agrees with the usual one, and $N$ is the unipotent radical of $P$.  

Note that by Lemmas \ref{lemc} and \ref{lemc2}, all the groups 
\be\label{conj}
\textrm{$\{L_t \,:\, t\in \mathrm{Df}(\p,\n)\}$ are conjugate to each other by $N$.}
\ee

\begin{leml}
   The equality $P=L_{t}\ltimes N$ holds.
         \end{leml}
         
\begin{proof}
Let $g\in P$. By Lemma \ref{lemc}, there is an element $g'\in N$ such that $g'gt(g'g)^{-1}$  commutes with $t$. Then by Lemma \ref{lemc2}, we have that 
\[
\l_t=\l_{g'gt(g'g)^{-1}}=\Ad_{g'g} (\l_t)
\]
and hence $g'g\in L_t$. This implies the lemma. 
%This is known to experts and we sketch a proof for completeness. By realizing $G$ as an algebraic subgroup of a general linear group, it is easy to see that there are two connected algebraic subgroups of $G$ whose Lie algebras respectively equal $\p$ and $\l_t$. This implies the first assertion.  %. We sketch a proof for the convenience of the reader.  Realize $G$ as an algebraic subgroup of a general linear group. Then it is easy to see that $t$ is a central element of $L$, and the lemma then easily follows.     %such that $t$ is realized as a dia
\end{proof}

%\subsection{A partial  order}\label{secorder}
 %When no  confusion is possible, we will not distinguish an algebraic group over $\Q_p$ with its group of $\Q_p$-points. Let $G$ be a connected reductive linear algebraic group over $\Q_p$, and let $P$ be a parabolic of $G$.

\subsection{A partial order}
\begin{dfnl} For any two open compact subgroups $\fG$ and $\fG'$ of $G$, define
\be\label{partial00}
  \fG\preceq_P \fG'\quad \textrm{if and only if }\quad  \fP\subset \fP'\textrm{ and } \fG'\subset \fG\fP',
\ee
where $\fP:= \fG\cap P$ and $\fP':= \fG'\cap P$.
\end{dfnl}

It is easily checked that $\preceq_P$ is a  partial order on the set of all open compact subgroups of $G$.
 Note that the conditions in \eqref{partial00} imply that
\be\label{fgprime}
  \fG'=(\fG\cap \fG')\fP' \quad\textrm{and}\quad P\cap(\fG\fG')=\fP'.
\ee
The following lemma will be useful.
\begin{lem}\label{ordp1}
Let $\fG_1\preceq_P \fG_2\preceq_P \fG_3$ be three open compact subgroups of $G$. Then
\[
  \fG_2=(\fG_1\cap \fG_2)(\fG_2\cap \fG_3)\quad\textrm{and}\quad \fG_1\cap \fG_3\subset \fG_2.
\]
\end{lem}
\begin{proof}
The first assertion is obvious in view of the first equality in \eqref{fgprime}. For the second assertion, we have that
\begin{eqnarray*}
  &&\fG_1\cap \fG_3\\
  &=& \fG_1\cap ((\fG_2\cap \fG_3)\fP_3)\qquad (\fP_3:=\fG_3\cap P)\\
  &\subset & \fG_1\cap (\fG_2 P)\\
 &\subset &  \fG_2.
\end{eqnarray*}
Here the last inclusion follows from the second  equality in \eqref{fgprime}.
\end{proof}

The following lemma is easily checked. 
\begin{leml}\label{bijggpp}
   For all open compact subgroups $\fG, \fG'$ of $G$ with $\fG\preceq_P \fG'$, the natural map
    \[
        (\fG'\cap P)/(\fG\cap \fG'\cap P)\rightarrow \fG'/(\fG\cap\fG')
        \]
        is bijective.
\end{leml}

  %Set 
  % \[
%\bar P:=P_{t^{-1}}:= \left\{g\in G\mid \textrm{ $\{t^{k} g t^{-k}\,:\,{k\in \BN}\}$ is relative compact in $G$}\right\},\]\[\bar N:=N_{t^{-1}}:= \left\{g\in G\mid \lim_{k\rightarrow \infty} t^{k} g t^{-k}=1\right\},\]and\[ L_{\bar P}:=P\cap \bar P.\]

 Let $\fL$ be an open compact subgroup of $L$. 
 Let $\fL_t$ denote the open compact subgroup of $L_t$ that corresponds to $\fL$ under the  obviously isomorphism  $L\cong L_t$. Recall that  $\BG$ is an open subgroup of $\mathsf G^{\natural}$. Set $\mathbf P:=P\cap \BG$, 
 \[
{\mathscr P}_{\fL,\mathbf P}:=\{\textrm{open compact subgroups $\fP$ of $\mathbf P$ such that $\fP/(\fP\cap N)=\fL$}\}, 
\]
and \[
 {\mathscr G}_{\fL, \BG}:=\{\textrm{open compact subgroups $\fG$ of $G\cap \BG$ such that $\fG\cap P\in {\mathscr P}_{\fL,\mathbf P}$}\}.
 \] 
Then
\[ %\label{nonemptypg}
{\mathscr P}_{\fL,\mathbf P} \neq \emptyset \ \Longleftrightarrow \  {\mathscr G}_{\fL,\BG}\neq \emptyset\  \Longleftrightarrow\  \fL\subset \mathbf P/(N\cap \mathbf P),
\]
 in which case ${\mathscr P}_{\fL,\mathbf P}$ is a directed set under the inclusion relation, and ${\mathscr G}_{\fL, \BG}$ is a directed set under the partial order $\preceq_P$.

\begin{leml}\label{lemp}
    Let $\fP$ be an open compact subgroup of $P_0$. Then 
    \begin{eqnarray*}
        &&\fP\subset t^k\fP t^{-k}\textrm{ for some positive integer $k$}\\
       &\Longleftrightarrow& \fP\subset t^k\fP t^{-k}\textrm{ for all sufficiently large  positive integer $k$}\\
        &\Longleftrightarrow& 
        \fP=(\fP\cap L_t)(\fP\cap N).     
            \end{eqnarray*} %such that $\fP\subset t\fP t^{-1}$, then 
  % \[  \fP=\fL\ltimes \fN\quad \textrm{and}\quad  \fN\subset  t \fN t^{-1},\] where   $\fL:=\fP\cap L$ and $\fN:=\fP\cap N$.  
    \end{leml}
    
   \begin{proof}
     Suppose that $\fP\subset t^k\fP t^{-k}$, where $k$ is a positive integer.   Let  $g=g_1 g_2\in \fP$ with $g_1\in L_t$ and $g_2\in N$. Then
       \[
     g_1=\lim_{r\rightarrow \infty} g_1 t^{-rk} g_2 t^{rk}=\lim_{r\rightarrow \infty} t^{-rk} g t^{rk}\in \fP.
       \]
       This proves the equality $\fP=(\fP\cap L_t)(\fP\cap N)$. The rest of the lemma is obvious.  
   \end{proof} 

Recall that $D$ denotes an open compact subgroup of $\mathsf G^{\natural,p}\cap \BG$. 
 We say that $\fL$ is $D$-neat if so is $\fL_t$. This is independent of $t\in \mathrm{Df}(\p,\n)$.  
%An compact subgroup $\fL$ of $L$ is said to be $t$-stable if $\Ad_t(\fL)=\fL$, it is said to be $D$-neat if there is exists a $D$-neat group in ${\mathscr G}_\fL$.
Write
 \[
 {\mathscr G}_{\fL, \BG}^{D,t}:=\{\textrm{$D$-neat open compact subgroups $\fG$ of $G\cap \BG$ such that $\fL_t\subset \fG\subset \bar N_t \fL_t N$} \}.   
 \]
 %which is a cofinal subset of ${\mathscr G}_{\fL, \BG}$.
  %\[{\mathscr G}_{\fL, \BG}^{D,t}:=\{\fG\subset G\cap \BG\in {\mathscr G}_{\fL,\BG}\mid \textrm{$\fG$ is $D$-neat, and   $\fL\subset \fG\subset \bar N P$} \}.   \]
The  following  lemma is obvious. 

\begin{lem}\label{upbd}
   Suppose that $\fL\subset L_0$.  The set ${\mathscr G}_{\fL, \BG}^{D,t}$ is nonempty if and only if $\fL_t$ is  $D$-neat and contained in $\mathbf P$. When this is the case ${\mathscr G}_{\fL, \BG}^{D,t}$ is cofinal in the directed set ${\mathscr G}_{\fL, \BG}$. 
\end{lem}
%\begin{proof} This easily follows from Lemma \ref{lemp}. % In view of Lemma \ref{tstablel}, the ``only if" part of the first assertion is obvious. Now we suppose that $\fL$ is $t$-stable, $D$-neat, and contained in $\mathbf P/(N\cap \mathbf P)$. Take a $D$-neat open compact subgroup $\fG_0\in \mathcal \CG_\fL$.\end{proof}

When $\fL\subset L_0$, for each  $\fG\in {\mathscr G}_{\fL, \BG}^{D,t}$ we set
    \[
   \BN_\fG:= \BN_{\fG,t}:=\{k\in \BN \,:\, \fG\preceq_P t^k \fG t^{-k}\}.
    \]
  This is  a  submonoid of the additive  monoid $\BN$ which contains all but finitely many elements of $\BN$. %\begin{dfnl}An open compact subgroup $\fG$ of $G$ is said to be $t$-good if \[ \fG\subset \bar N P\quad\textrm{and} \quad  \fG\preceq_P t\fG t^{-1}.   \]\end{dfnl}

   \subsection{Parobolic   cohomology}\label{sec3.2}
 %$M$ is an $R[\mathsf G(\Q)\times \mathsf G^\natural]$-module.

%Recall the partial order $\preceq_P$ defined in Section \ref{secorder}.
% Recall the finite ring $\CO_k$ from the Introduction. View it as a trivial representation of
%$\mathsf G(\Q)\times \mathsf G(\A)^\natural$.
%\begin{dfnl}
%A compact subgroup $\fH$ of $L$ is said to be $D$-neat if there is an open compact subgroup $\fG$ of $G$ such that $D\fG$ is neat and  $\fH\subset (\fG\cap P)/(\fG\cap N)$.
%\end{dfnl}

%\end{itemize}

%\begin{dfnl}An open compact subgroup $\fG$ of $G$ is said to be $t$-good if it is of the form $\fG=\fN^-\fL\fN$ such that\[  \fN^-\subset N^-,\quad \fL\subset L,\quad \fN\subset N \]  and \[   t\fN^- t^{-1}\subset \fN^-, \quad t\fL t^{-1}=\fL, \quad t\fN t^{-1} \supset \fN.\]\end{dfnl}
\begin{lem}\label{lemfg0123}
Let  $\fG_1, \fG_2, \fG_3$ be three open compact subgroups of $G\cap \BG$ such that $\fG_1\preceq_P \fG_2\preceq_P \fG_3$. Then the diagram
 \[
  \xymatrix{
 \oH^i_\Phi(D \fG_3, M) \ar[rr]^{ \rho_{D \fG_3, D \fG_2}}  \ar[rrd]_(0.4){ \rho_{D \fG_3, D \fG_1}}&& \ \oH^i_\Phi(D \fG_2, M)  \ar[d]^{ \rho_{D \fG_2, D \fG_1}} \\
& &  \oH^i_\Phi(D \fG_1, M)  \\
    }
\]
commutes. 
\end{lem}
\begin{proof}
This follows from Lemmas \ref{neatng} and  \ref{ordp1}. 
\end{proof}

%Suppose that $\fL$ is contained in $\BG$.  In the rest of this subsection, we suppose that $\BG\supset N$. Then ${\mathscr G}_{\fL,\BG}$ is cofinal in ${\mathscr G}_{\fL}$.

Suppose that $\fL$ is contained in $\mathbf P/(N\cap \mathbf P)$ so that both ${\mathscr P}_{\fL,\mathbf P}$ and ${\mathscr G}_{\fL,\BG}$ are nonempty directed sets. 

\begin{dfnl}\label{defparaboliccoh}
    In view of %Lemma \ref{upbd} and 
    Lemma \ref{lemfg0123}, we  define the parabolic   cohomology  group 
\[
  \oH^{i}_{\Phi,\mathbf P}(D\fL, M):= \varprojlim_{\fG\in {\mathscr G}_{\fL, \BG}} \oH^i_\Phi(D \fG, M). 
\]
\end{dfnl}

%As before, let $M'$ be another $R[\mathsf G(\Q)\times \mathsf G^\natural]$-module and let $\phi: M\rightarrow M'$ be a homomorphism of $R[\mathsf G(\Q)\times \mathsf G^\natural]$-modules. We define the change of coefficients map \[ \phi:    \oH^i_{\Phi,P}(D\fL, M)\rightarrow    \oH^i_{\Phi,P}(D\fL, M')\]as in the following lemma. \begin{lem}\label{ntransfer00}There is a unique homomorphism \be\label{phidlm}\phi:    \oH^i_{\Phi,P}(D\fL, M)\rightarrow    \oH^i_{\Phi,P}(D\fL, M')\eesuch  that the diagram\[ \begin{CD}        \oH^i_{\Phi,P}(D\fL, M)@> \phi >>     \oH^i_{\Phi,P}(D\fL, M')\\                 @VVV          @VVV\\\oH^i_{\Phi}(D\fG, M)@> \phi >>    \oH_{\Phi}^i(D\fG, M')\\        \end{CD}\] commutes for all $\fG\in {\mathscr G}_{\fL}$. \end{lem}\begin{proof}The lemma follows by noting that the diagram\[ \begin{CD}      \oH^i_{\Phi}(D\fG', M)@>\phi>>    \oH_{\Phi}^i(D\fG', M')\\               @VV\rho V          @VV\rho V\\\oH^i_{\Phi}(D\fG, M)@>\phi>>    \oH_{\Phi}^i(D\fG, M')\\         \end{CD}\]commutes for all  $\fG'\in {\mathscr G}_{\fL}$ with $\fG\preceq_P \fG'$.  \end{proof}

%The following lemma easily follows  from Lemma \ref{ntransfer00}.

%\begin{lem}\label{funct2}
%The group \eqref{cohL} essentially depends on $P\cap \BG$. 
It is clear that the assignment $ \oH^{i}_{\Phi,\mathbf P}(D\fL, \,\cdot\,)$ is a functor from the category of $R[\mathsf G(\Q)\times \BG]$-modules to the category of $R$-modules. We remark that the parabolic cohomology group only depends on $\mathbf P$ is the following sense:  if $\BG'$ is an open subgroup of $\BG$ containing $D$ such that $\BG'\cap P=\BG\cap P$, then  ${\mathscr G}_{\fL, \BG'}$ is a cofinal subset of ${\mathscr G}_{\fL, \BG}$ and hence 
\[
 \varprojlim_{\fG\in {\mathscr G}_{\fL, \BG}} \oH^i_\Phi(D \fG, M)=\varprojlim_{\fG\in {\mathscr G}_{\fL, \BG'}} \oH^i_\Phi(D \fG, M). 
\]
%\end{lem}
%Denote by $\CP_\fL$ the set of all open compact subgroups $\fP$ of $P$ such that $\fP/(\fP\cap N)=\fL$. 

We have an identification
\be\label{isopara}
  \oH^{i}_{\Phi,\mathbf P}(D\fL, M)=\varprojlim_{\fP\in {\mathscr P}_{\fL,\mathbf P}} \oH^{i}_{\Phi}(D\fP, M)
\ee
as in the following proposition, where the transition maps are the pull-back maps. 
\begin{prpl}\label{isopara2}
The projections yield an isomorphism
\be\label{isopc}
  \oH^{i}_{\Phi,\mathbf P}(D\fL, M)\rightarrow \varprojlim_{\fP\in {\mathscr P}_{\fL,\mathbf P}} \oH^{i}_{\Phi}(D\fP, M)
\ee
 whose inverse is the unique homomorphism  
   \be\label{isojacquet} \varprojlim_{\fP\in {\mathscr P}_{\fL, \mathbf P}}\oH^{i}_{\Phi}(D\fP, M)\rightarrow \oH^{i}_{\Phi,\mathbf P}(D\fL, M)
   \ee
   such that  the diagram
\[
 \begin{CD}
      \varprojlim_{\fP\in {\mathscr P}_{\fL,\mathbf P}} \oH^{i}_{\Phi}(D\fP, M)@> >>   \oH^{i}_{\Phi,\mathbf P}(D\fL, M)\\
            @VVV          @VV V\\
\oH_{\Phi}^{i}(D (\fG\cap P), M)@> >>    \oH_{\Phi}^{i}(D\fG, M)\\
            \end{CD}
\]
commutes for all $\fG\in {\mathscr G}_{\fL, \BG}$.  %Moreover, this homomorphism is an isomorphism.  

\end{prpl}
\begin{proof}
    The uniqueness of \eqref{isojacquet} obvious. Lemma \ref{bijggpp} implies that  the diagram  
\[
    \begin{CD}
      \oH^{i}_{\Phi}(D(\fG'\cap P), M)@> >>  \oH^{i}_{\Phi}(D\fG', M) )\\
            @VVV          @VV V\\
 \oH^{i}_{\Phi}(D(\fG\cap P), M)@> >>  \oH^{i}_{\Phi}(D\fG, M)\\
    \end{CD}
\]
commutes, for all $\fG, \fG'\in {\mathscr G}_{\fL,\BG}$ with $\fG\preceq_P \fG'$. This proves  the existence of the homomorphism \eqref{isojacquet}. 

Similarly the diagram 
\[
  \xymatrix{
\oH^{i}_{\Phi,\mathbf P}(D\fL, M) \ar[rr]^{\textrm{projection}}  \ar[rrd]_(0.4){\textrm{projection}}&& \oH^i_\Phi(D \fP', M)\ar[d]^{\textrm{pull-back}} \\
& &  \oH^i_\Phi(D \fP, M)  \\
    }
\]
commutes for all $\fP, \fP'\in {\mathscr P}_{\fL,\mathbf P}$ with $\fP\subset \fP'$. Hence we have  a natural map
\[
   \oH^{i}_{\Phi,\mathbf P}(D\fL, M)\rightarrow \varprojlim_{\fP\in {\mathscr P}_{\fL,\mathbf P}}\oH^{i}_{\Phi}(D\fP, M). 
\]
It is routine to check that this is the inverse of the map \eqref{isojacquet}.
\end{proof}

%Suppose that $\p$ equals the Lie algebra of $P$.

 %  \subsection{Pull-back maps}
   Let $D'$ be an open subgroup of $D$ and let $\fL'$ be an open subgroup of $\fL$. We define the pull-back map 
   \[
\rho_{D\fL, D' \fL'}: \oH_{\Phi, \mathbf P}^{i}(D \fL, M)\rightarrow \oH_{\Phi, \mathbf P}^{i}(D' \fL', M)
\]
as in the following lemma. %Here all the arrows except the top horizontal one are the projection maps. 

\begin{lem}\label{pull-backord}
There is a unique homomorphism 
\[
\rho_{D\fL, D' \fL'}: \oH_{\Phi, \mathbf P}^{i}(D \fL, M)\rightarrow \oH_{\Phi, \mathbf P}^{i}(D' \fL', M)
\]
 such that the diagram
\[
 \begin{CD}
       \oH_{\Phi, \mathbf P}^{i}(D \fL, M)@> \rho_{D \fL, D' \fL'} >>   \oH_{\Phi, \mathbf P}^{i}(D' \fL', M)\\
            @VVV          @VV V\\
\oH_{\Phi}^{i}(D \fG, M)@>\rho_{D\fG, D' \fG'} >>    \oH_{\Phi}^{i}(D'\fG', M)\\
            \end{CD}
\]
commutes for all $\fG\in {\mathscr G}_{\fL, \BG}$ and  $\fG'\in {\mathscr G}_{\fL',\BG}$  with $\fG'\preceq_P \fG$.  
\end{lem}
\begin{proof}
It is easy to see that for every $\fG'\in {\mathscr G}_{\fL',\BG}$, there is a group  $\fG\in {\mathscr G}_{\fL,\BG}$ with $\fG'\preceq_P \fG$. This implies the uniqueness assertion of the lemma.
Let $\fG_1\in {\mathscr G}_{\fL, \BG}$ and $\fG_1'\in {\mathscr G}_{\fL',\BG}$ with
\[
  \fG \preceq_P \fG_1,\quad \fG' \preceq_P \fG_1', \quad \fG_1'\preceq_P \fG_1.
\]
 Lemma \ref{neatng} and  Lemma \ref{ordp1} imply that the diagram 
\[
 \begin{CD}
     \oH_{\Phi}^{i}(D \fG_1, M)@>\rho>>    \oH_{\Phi}^{i}(D'\fG_1', M)\\
            @V\rho VV          @VV \rho V\\
\oH_{\Phi}^{i}(D \fG, M)@>\rho >>    \oH_{\Phi}^{i}(D'\fG', M)\\
            \end{CD}
\]
commutes. This implies the existence assertion of the lemma. 
\end{proof}

\begin{lem}\label{pull-backord22}
Let $D_1\supset D_2\supset D_3$ be open compact subgroups of $\mathsf G^{\natural, p}\cap \BG$, and let $\fL_1\supset \fL_2\supset \fL_3$ be open compact subgroups of $L\cap \BG$. Then  the diagram 
 \[
  \xymatrix{
 \oH_{\Phi, \mathbf P}^{i}(D_1 \fL_1, M) \ar[rr]^{}  \ar[rrd]^{}&& \  \oH_{\Phi, \mathbf P}^{i}(D_2 \fL_2, M) \ar[d]^{} \\
& &   \oH_{\Phi, \mathbf P}^{i}(D_3 \fL_3, M)  \\
    }
\]
commutes, where the three arrows are the pull-back maps. 
\end{lem}
\begin{proof}
This also follows from Lemma \ref{neatng} and  Lemma \ref{ordp1}. 
\end{proof}

It is clear that  the pull-back  maps are natural in the coefficient modules. The following lemma is routine to check.

\begin{lem}\label{pull-backordcom}
The diagram
\[
 \begin{CD}
       \oH_{\Phi, \mathbf P}^{i}(D \fL, M)@> \textrm{pull-back} >>   \oH_{\Phi, \mathbf P}^{i}(D' \fL', M)\\
            @V\textrm{projection}VV          @VV \textrm{projection} V\\
\oH_{\Phi}^{i}(D \fP, M)@>\textrm{pull-back} >>    \oH_{\Phi}^{i}(D'\fP', M)\\
            \end{CD}
\]
commutes for all $\fP\in {\mathscr P}_{\fL, \mathbf P}$ and  $\fP'\in {\mathscr P}_{\fL',\mathbf P}$  with $\fP'\subset  \fP$.  
\end{lem}
\begin{proof}
    This is implied by  Lemma \ref{bijggpp}. 
\end{proof}
In view of Lemma \ref{pull-backord22}, we define 
\be
\oH_{\Phi, \mathbf P}^{i}(\mathsf G,M):=\varinjlim_{D\fL}\oH_{\Phi, \mathbf P}^{i}(D \fL, M), 
\ee
where $D$ runs over open compact subgroups of $\mathsf G^{\natural,p}\cap \BG$ and $\fL$ runs over open compact subgroups of $\mathbf P/(N\cap \mathbf P)$.  
Define a homomorphism
\be\label{paratorela}
  \oH_{\Phi, \mathbf P}^{i}(D \fL, M)\rightarrow \oH_{\Phi}^{i}(\mathsf G, M)^{\la \p\ra}
\ee
as the composition of 
\[
 \oH_{\Phi, \mathbf P}^{i}(D \fL, M)\xrightarrow{\textrm{projection}} \oH_{\Phi}^{i}(D \fP, M)\rightarrow\oH_{\Phi}^{i}(\mathsf G, M)^{\la \p\ra},
\]
where $\fP\in {\mathscr P}_{\fL, \BG}$. This homomorphism is obviously independent of $\fP$. 
It follows from Lemma \ref{pull-backordcom} that the diagram 
\[
  \xymatrix{
 \oH_{\Phi, \mathbf P}^{i}(D \fL, M)\ar[rr]^{\textrm{pull-back}}  \ar[rrd]_(0.4){\eqref{paratorela}} && \   \oH_{\Phi, \mathbf P}^{i}(D' \fL', M)\ar[d]^{\eqref{paratorela}} \\
& &  \oH_{\Phi}^{i}(\mathsf G, M)^{\la \p\ra}  \\
    }
\]
commutes. Thus the homomorphisms for various $D$ and $\fL$ yield a homomorphism 
\be\label{paratorela2}
  \oH_{\Phi, \mathbf P}^{i}(\mathsf G, M)\rightarrow \oH_{\Phi}^{i}(\mathsf G, M)^{\la \p\ra}.
\ee

Note that if $\fL$ is $D$-neat and $\fL\subset L_0$, then  every compact subgroup of $\fL_t N$ is $D$-neat. In the rest of this subsection, assume that $R$ is a $\Q$-algebra, $\fL$ is $D$-neat, and $\fL\subset L_0$. 

\begin{prpl}\label{prpparaf}
The identification \eqref{isopara} and the horizontal arrows in \eqref{ff} yield an isomorphism 
\[ %\label{isoff}
\widehat { \oH^{i}_{\Phi}(\mathsf G, M)}^{D((\fL_t N)\cap \mathbf P)}\otimes \RD(\g/\p)\xrightarrow{\cong}  \oH^{i}_{\Phi,\mathbf P}(D\fL, M).
\]
Moreover, the diagram 
     \[
      \begin{CD}
          \widehat { \oH^{i}_{\Phi}(\mathsf G, M)}^{D((\fL_t N)\cap \mathbf P)}\otimes \RD(\g/\p) @>\cong >> \oH^{i}_{\Phi,\mathbf P}(D\fL, M)\\
          @V\textrm{inclusion}VV @VV\textrm{pull-back}V\\
          \widehat { \oH^{i}_{\Phi}(\mathsf G, M)}^{D'((\fL'_t N)\cap \mathbf P)}\otimes \RD(\g/\p)  @>\cong >> \oH^{i}_{\Phi,\mathbf P}(D'\fL', M)\\
      \end{CD}
      \]
      commutes for all open subgroups $D'$ and $\fL'$ of $D$ and $\fL$ respectively.
\end{prpl}
\begin{proof}
   This easily follows from Lemmas \ref{lemformco22} and \ref{pull-backordcom}. 
\end{proof}

Proposition \ref{prpparaf} implies that 
\be\label{paracomp}
\oH_{\Phi, \mathbf P}^{i}(\mathsf G,M)=\left(\widehat { \oH^{i}_{\Phi}(\mathsf G, M)}_{\p\mathrm{-sm}}\right)^{N\cap \mathbf P}\otimes \RD(\g/\p),
\ee
and the natural map $\oH_{\Phi, \mathbf P}^{i}(D\fL,M)\rightarrow \oH_{\Phi, \mathbf P}^{i}(\mathsf G,M)$ is injective so that $\oH_{\Phi, \mathbf P}^{i}(D\fL,M)$ is considered as a subspace of $\oH_{\Phi, \mathbf P}^{i}(\mathsf G,M)$. It is clear that both sides of \eqref{paracomp} carry natural representations of the group
\[
%(\textrm{the normalizer of $\mathbf P$ in $\mathbf G$}) 
\mathbf P/(N\cap \mathbf P),
\]
and the identification \eqref{paracomp} respects these representations. Moreover,
\be\label{identifydl}
\oH_{\Phi, \mathbf P}^{i}(D\fL,M)=\oH_{\Phi, \mathbf P}^{i}(\mathsf G,M)^{D\fL}.
\ee
\subsection{The $t$-stable part}\label{sectstable}

In the rest of  this section we assume that $\fL$ is $D$-neat, $\fL\subset L_0$,  and   $\fL_t\subset  \mathbf P$. Then   ${\mathscr G}_{\fL, \BG}^{D,t}$ is a nonempty cofinal subset of ${\mathscr G}_{\fL, \BG}$.  In particular, 
\[
 \oH^{i}_{\Phi,\mathbf P}(D\fL, M) = \varprojlim_{\fG\in {\mathscr G}_{\fL, \BG}^{D,t}} \oH^i_\Phi(D \fG, M). 
\]
Let $\la t\ra$ denote the submonoid of $G$ generated by $t$. Suppose that we are given a  compatible action of $\la t\ra$ on $M$:
\[
\la t\ra\times M\rightarrow M,\quad  (t^k,u)\mapsto t^k*u. 
\]
This action yields Hecke maps  as in \eqref{rhogg02}, which are 
homomorphisms to be denoted by 
\[ %\label{rhogg0244}
\rho_{t^k}^*:=  \rho^*_{t^k, K, K'}: \oH_\Phi^i(K,M)\rightarrow \oH_\Phi^i(K',M),
\]
 where $k\in \BN$, and $K$ and $K'$ are open compact subgroups of $\BG$. %We call this map the Hecke map of $t$. 

 Let $\fG\in {\mathscr G}_{\fL, \BG}^{D,t}$. 
\begin{leml}\label{hecket}
The map
\[
  k\mapsto  \rho^*_{t^k, D\fG, D\fG}
\]
defines a representation of the monoid $\BN_\fG$ on $\oH_\Phi^i(D\fG, M)$. 

\end{leml}
  \begin{proof}
This follows from  Proposition \ref{hecke001} and Lemma \ref{ordp1}.
  \end{proof}

We call the representation in Lemma \ref{hecket} the Hecke representation. 
Set
\[
\oH^{i,t}_\Phi(D \fG, M):=\bigcap_{k\in \BN_\fG} \rho^*_{t^k}(\oH^i_\Phi(D \fG, M)).
\]

\begin{leml}\label{lemprest}
    Suppose that $\fG, \fG'\in {\mathscr G}_{\fL, \BG}^{D,t}$ with $\fG\preceq_P \fG'$. Then the transfer map
    \be\label{thansfer12}
    \rho: \oH^i_\Phi(D \fG', M)\rightarrow \oH^i_\Phi(D \fG, M)
    \ee
    sends $\oH^{i,t}_\Phi(D \fG', M)$ into $\oH^{i,t}_\Phi(D \fG, M)$.
\end{leml}
\begin{proof}
     It follows from Proposition
 \ref{hecke001} that the transfer map $\rho$ in \eqref{thansfer12} is equivariant under the Hecke representations of the monoid
 $\{t^k \,:\, k\in \BN_{\fG}\cap \BN_{\fG'}\}$. 
 Thus 
 \begin{eqnarray*}
  \rho\left(\oH^{i,t}_\Phi(D \fG', M)\right)&=&\rho\left(\bigcap_{k\in \BN_\fG\cap \BN_{\fG'}} \rho^*_{t^k}(\oH^i_\Phi(D \fG', M))\right)\\
 &\subset &\bigcap_{k\in \BN_\fG\cap \BN_{\fG'}} \rho^*_{t^k}(\oH^i_\Phi(D \fG, M))\\
 &=&\oH^{i,t}_\Phi(D \fG, M).
    \end{eqnarray*}
\end{proof}

\begin{dfnl}
    In view of %Lemma \ref{upbd} and 
    Lemma \ref{lemprest}, we  define the $t$-stable part of the parabolic   cohomology  group $\oH^{i}_{\Phi,\mathbf P}(D\fL, M)$ to be 
\[
  \oH^{i,t}_{\Phi,\mathbf P}(D\fL, M):= \varprojlim_{\fG\in {\mathscr G}_{\fL, \BG}^{D,t}} \oH^{i,t}_\Phi(D \fG, M)\subset \oH^{i}_{\Phi,\mathbf P}(D\fL, M). 
\]
\end{dfnl}
Similar to the case of   parabolic   cohomology  groups, the $t$-stable parts also only depend on $\mathbf P$.  
\begin{leml}\label{lemprest2}
     Let $D'$ be an open subgroup of $D$, and $\fL'$ an open subgroup of $\fL$. Then  the pull-back map
    \be\label{thansfer1234}
    \rho: \oH^i_{\Phi,\mathbf P}(D \fL, M)\rightarrow \oH^i_{\Phi,\mathbf P}(D' \fL', M)
    \ee
    sends $\oH^{i,t}_{\Phi,\mathbf P}(D \fL, M)$ into $\oH^{i,t}_{\Phi,\mathbf P}(D' \fL', M)$.
\end{leml}
\begin{proof}
This is similar to the proof of Lemma \ref{lemprest}. 
\end{proof}

\begin{dfnl}
 In view of Lemma \ref{lemprest2}, we define the $t$-stable part 
\be
\oH_{\Phi, \mathbf P}^{i,t}(\mathsf G,M):=\varinjlim_{D\fL}\oH_{\Phi, \mathbf P}^{i,t}(D \fL, M)\subset \oH_{\Phi, \mathbf P}^{i}(\mathsf G,M), 
\ee
where $D$ runs over open compact subgroups of $\mathsf G^{\natural,p}\cap \BG$ and $\fL$ runs over $D$-neat open compact subgroups of $(P_0\cap \mathbf G)/(N\cap \mathbf P)$. 
\end{dfnl}

\subsection{A condition on $\Phi$}

Recall the following well-known result, which is more or less a consequence of the Borel-Serre compactification (\cf \cite{BS73}, \cite{Pi89}, \cite[Theorem 17.10]{Bo19} and \cite[Page 32]{Em06a}). %\cite[Lemma 15.2.3]{GeH24}).

\begin{lem}\label{homotopy}
Suppose that $\sf G$ is connected and reductive. Then for every neat open compact subgroup $K$  of $\mathsf G^\natural$,
 $\mathsf G(\Q)\backslash {\mathscr X}/K$ is a topological manifold that is homotopic to a finite simplicial complex.
\end{lem}

 The above lemma has the following consequence. 
\begin{prpl}\label{prop:cohfin}
Suppose that $\sf G$ is connected and reductive, and  
 $\Phi$ is the family of all closed subsets or the family of all compact subsets of $\mathsf G(\Q)\backslash \mathscr X$. Then  for all Noetherian commutative rings $R_0$ with identity, all neat open compact subgroups $K$ of $\sf G^\natural$, all local systems  $\mathcal L_0$ over $S^{\mathsf G}_{K}$ of
 finitely generated $R_0$-modules, and all $i\in \Z$, 
\be\label{phi1}
 \textrm{$\oH_\Phi^i(S^{\mathsf G}_{K}, \mathcal L_0)$ is finitely generated as an $R_0$-module.}
\ee
 Moreover, the  
 natural map
 \be\label{phi2}
 \oH_\Phi^i(S^{\mathsf G}_{K}, \mathcal L_0)\otimes_{R_0} M_0 \rightarrow \oH_\Phi^i(S^{\mathsf G}_{K}, \CL_0\otimes_{R_0} M_0)\ \  \textrm{is an isomorphism}
 \ee
  for every flat $R_0$-module $M_0$.
\end{prpl}

\begin{proof}
If $\Phi$ is the family of all compact subsets of $\mathsf G(\Q)\backslash\mathscr X$, the second assertion of the proposition  follows from the universal coefficient theorem for compactly supported cohomology (see \cite[II. Theorem 15.3]{Br97} for example). If $\Phi$ is the family of all closed subsets of $\mathsf G(\Q)\backslash\mathscr X$, it follows from the same theorem in view of Lemma \ref{homotopy}. 

If $\Phi$ is the family of all closed subsets of $\mathsf G(\Q)\backslash\mathscr X$, then the first assertion of  the proposition  is also implied by Lemma \ref{homotopy}. If $\Phi$ is the family of all compact subsets of $\mathsf G(\Q)\backslash\mathscr X$, then it is still implied by  Lemma \ref{homotopy}, by using Poincar\'e duality. 
\end{proof}

From now on, we assume that 
\be \label{conphi}
\textrm{the support set $\Phi$ satisfies the conditions \eqref{phi1} and \eqref{phi2}.} 
\ee

%\begin{prpl}Suppose that $R$ is Noetherian and $M$ is finitely generated as an $R$-module. Then for every open compact subgroup $K$ of $\BG$, $\oH_\Phi^i(K,M)$ is finitely generated as an $R$-module.\end{prpl}
%\begin{proof}   Pick a neat open compact subgroup $K_0$ of $\mathsf G(\A^\infty)$ that is contained in $K$ as a normal subgroup. 
%\[  \oH^j(K/K_0,  \oH_\Phi^i(K_0,M))\Rightarrow \oH_\Phi^{i+j}(K,M).  \]\end{proof}

 % \subsection{Fi}\subsection{The module $\oH_{\Phi, \mathbf P}^{i}(D \fL, M)$}  
 \subsection{Some properties of the $t$-stable part}\label{secmoduledl}
 In this subsection we further assume that $R$ is Noetherian and $M$ is finitely generated as an $R$-module. 
  We need the following elementary lemma. 
\begin{lem}\label{fsuriso}
   % Suppose that $R$ is Noetherian.
   Let $J$ be a finitely generated $R$-module and let $\phi: J\rightarrow J$ be an $R$-module endomorphism. Then $\phi$ induces an automorphism on 
    \[
      \bigcap_{k\in \BN} \phi^k(J). 
    \]
\end{lem}
\begin{proof}
   Since $J$ is a Noetherian $R$-module, the sequence 
   \[
   \ker \phi^0\subset \ker \phi^1\subset \ker \phi^2 \subset \dots 
   \]
   is stable. In particular, there is a non-negative  integer $r$ such that 
   \be\label{equalker}
   \ker \phi^r=\ker \phi^{r+1}. 
\ee
   Let $x\in \bigcap_{k\in \BN} \phi^k(J)$. Then 
   \[
   x=\phi^{r+1}(y)
   \]
   for some $y\in J$.
   
   For every $k\in \BN$, there is an element $z_k\in J$ such that
   \[
   x=\phi^{r+k+1}(z_k).
   \]
Then 
\[
\phi^{r+1}(y)=\phi^{r+k+1}(z_k),
\]
  and \eqref{equalker} implies that 
 \[
\phi^{r}(y)=\phi^{r+k}(z_k).
\] 
This proves that
\[
  \phi^{r}(y)\in \bigcap_{k\in \BN} \phi^{r+k}(J)=\bigcap_{k\in \BN} \phi^k(J).
\]
In conclusion, we have proved that   $\phi$ restricts to a surjective endomorphism of $\bigcap_{k\in \BN} \phi^k(J)$. It is thus an automorphism since all surjective endomorphisms of Noetherian modules are injective.  
\end{proof}

%We say that a compact subgroup $C$ of $G$ is $D$-neat if the subgroup $D C$ of $ \mathsf G^\natural$ is neat.  The following lemma is a consequence of Lemma \ref{fsuriso}. 

The following result is a consequence of Lemma \ref{fsuriso}. 
\begin{leml}\label{fsuriso2}
For all $\fG\in {\mathscr G}_{\fL,\BG}^{D,t}$ and $k\in \BN_\fG$,  the Hecke map 
\[
  \rho^*_{t^k}: \oH_{\Phi}^{i,t}(D\fG, M)\rightarrow \oH_\Phi^{i,t}(D\fG, M)
\]
is a well-defined  automorphism. 
\end{leml}
%In view of Lemma \ref{exneat}, now we suppose  that $\CG_\fL^{D,\bar P}\neq \emptyset$ and $\fG\in \CG_\fL^{D,\bar P}$. For each $R[T_{\bar P}^\dag]$-module $J$, recall from Section \ref{secfsur} the $R$-submodule $J_{t\mathrm{-sur}}\subset J$.

  \begin{prpl}\label{lemfg0123450}
For all $\fG, \fG'\in {\mathscr G}_{\fL,\BG}^{t,D}$ such that $\fG\preceq_P\fG'$, the transfer map
 \[
  \rho: \oH_{\Phi}^i(D \fG', M)\rightarrow \oH_{\Phi}^i(D \fG, M)
 \]
 induces an isomorphism
  \[
   \rho: \oH_{\Phi}^{i,t}(D \fG', M)\rightarrow \oH_{\Phi}^{i,t}(D \fG, M).
    \]
\end{prpl}

  \begin{proof}
  
Take a sufficiently large $k\in \BN$ so that $k\in \BN_\fG\cap \BN_{\fG'}$ and $\fG' \preceq_P t^k \fG t^{-k}$.  It follows from Proposition
 \ref{hecke001} that the  diagrams 
  \[
  \xymatrix{
 \oH_{\Phi}^i(D\fG', M)\ar[rr]^{\rho}  \ar[rrd]_(0.4){\rho^*_{t^k}}&& \ \oH_{\Phi}^i(D\fG,M)  \ar[d]^{\rho^*_{t^k}} \\
& &  \oH_{\Phi}^i(D\fG',M)  
    }
\]
and
  \[
  \xymatrix{
 \oH_{\Phi}^i(D\fG,M)\ar[rr]^{\rho^*_{t^k}}  \ar[rrd]_(0.4){\rho^*_{t^k}}&& \ \oH_{\Phi}^i(D\fG',M)  \ar[d]^{\rho} \\
& &  \oH_{\Phi}^i(D\fG,M)  \
    }
\]
are commutative, and all the six arrows are $\la t^k\ra$-equivariant. Thus these two commutative diagrams induce commutative diagrams 
 \[
  \xymatrix{
 \oH_{\Phi}^{i,t}(D\fG',M)\ar[rr]^{\rho}  \ar[rrd]_(0.4){\rho^*_{t^k}}&& \ \oH_{\Phi}^{i,t}(D\fG,M)  \ar[d]^{\rho^*_{t^k}} \\
& &  \oH_{\Phi}^{i,t}(D\fG',M)
    }
\]
and
  \[
  \xymatrix{
 \oH_{\Phi}^{i,t}(D\fG,M)
 \ar[rr]^{\rho^*_{t^k}}  \ar[rrd]_(0.4){\rho^*_{t^k}}&& \ \oH_{\Phi}^{i,t}(D\fG',M)  \ar[d]^{\rho} \\
& &  \oH_{\Phi}^{i,t}(D\fG,M).
    }
\]
Therefore the proposition follows, in view of Lemma \ref{fsuriso2}. 
  \end{proof}

By Proposition \ref{lemfg0123450}, we have that 
\[
\oH_{\Phi,\mathbf P}^{i,t}(D\fL,M) =\oH_{\Phi}^{i,t}(D\fG,M) 
\]
for all $\fG\in {\mathscr G}_{\fL,\BG}^{D,t}$. In particular, $\oH_{\Phi,\mathbf P}^{i,t}(D\fL,M)$ is independent of $\mathbf G$ in the following sense: if $\mathbf G'$ is an open subgroup of $\mathbf G$ containing $D\fL_t$, then 
\be\label{indp}
\oH_{\Phi,\mathbf P}^{i,t}(D\fL,M)=\oH_{\Phi,\mathbf P'}^{i,t}(D\fL,M), \quad\textrm{where }\, \mathbf P':=P\cap \mathbf G'. 
\ee

%Recall that $\fL$ is an open compact subgroup of $L$ that is contained in $\mathbf P/(N\cap \mathbf P)$. We say that $\fL$ is $t$-stable if is

%In the case when there is a $t$-good, $D$-neat group $\fG\in {\mathscr G}_{\fL, \BG}$, 

%Now we define the following second version of the  parabolic cohomology group: \[\oH^i_{\Phi,t}(D\fL, M):= (\oH^i_{\Phi}(D\fG, M))_{t\mathrm{-sur}},\]where $\fG\in {\mathscr G}_{\fL, \BG}^{D, t}$. By Proposition \ref{lemfg0123450} this is independent of $\fG$, and by the following proposition  it is consistent with the definition in \eqref{cohL}. 

%Denote by $\fL_{\bar P}$ its corresponding open compact subgroup of $L_{\bar P}$. 

\begin{prpl}\label{ordm}
%Keep the assumptions of Lemma \ref{fsuriso2}. 
 %Let $M$ be an $R[\mathsf G(\Q)\times \mathsf G^\natural]$-module with a compatible action of $T_{\bar P}^\dag$.  Assume that %$R$ is Artinian or $R=\CO$, and that $M$ is finitely generated as an $R$-module. 
 Assume that $\mathbf P\supset N$ and $\la t\ra$ acts on $M$ by automorphisms. Then for all $\fG\in {\mathscr G}_{\fL, \BG}^{D, t}$, the projection map
 \[
   \oH^i_{\Phi,\mathbf P}(D \fL, M)\rightarrow \oH^i_{\Phi}(D \fG, M)
 \]
 induces an isomorphism 
 \[
 \oH^i_{\Phi,\mathbf P}(D \fL, M)\rightarrow \oH^{i,t}_{\Phi}(D \fG, M).
 \]
 Consequently, 
 \[
 \oH^i_{\Phi,\mathbf P}(D \fL, M)=\oH^{i,t}_{\Phi,\mathbf P}(D \fL, M). 
 \]
  
\end{prpl}

\begin{proof}
%It follows from Lemma \ref{domt} that \[\varprojlim_{\fG'\in {\mathscr G}_{\fL,\BG}} \oH^i_\Phi(D \fG', M)=\varprojlim_{ \fG'\in {\mathscr G}_{\fL, \BG}^{D, t}} \oH^i_\Phi(D \fG', M).\]
%Note that the assumption \eqref{conphi} and the assumption on $\fG$ implies that $\oH^i_{\Phi}(D\fG, M)$ is finitely generated as an $R$-module.  
%Let $\fG\in {\mathscr G}_{\fL, \BG}^{D, t}$. % Note that the family of all such sets is cofinal in ${\mathscr G}_{\fL, \BG}^{D, t}$. 
For all sufficiently large positive integers $k$ we have that $t^k\fG t^{-k}\in {\mathscr G}_{\fL, \BG}^{D, t}$ and $\fG\preceq_P t^k\fG t^{-k}$.  It follows from Proposition
 \ref{hecke001} that the  diagram
 \[
  \xymatrix{
 \oH_{\Phi}^i(D\fG,M)\ar[rr]^{\rho^*_{t^k}}  \ar[rrd]_(0.4){\rho^*_{t^k}}&& \ \oH_{\Phi}^i(D t^k\fG t^{-k},M)  \ar[d]^\rho \\
& &  \oH_{\Phi}^i(D\fG,M)  \\
    }
\]
commutes. % for every $k\in \BN$. 
The top horizontal arrow is an isomorphism since $\la t\ra$ acts on $M$ by automorphisms. 
Thus the image of the projection map
\be\label{pr0}
  \oH^i_{\Phi,\mathbf P}(D \fL, M)\rightarrow \oH^i_{\Phi}(D \fG, M)
\ee
is contained in the image of 
\[
\rho^*_{t^k}: \oH_{\Phi}^i(D\fG,M)\rightarrow \oH_{\Phi}^i(D\fG,M).
\]
Therefore the image of \eqref{pr0} is contained in $\oH^{i,t}_{\Phi}(D\fG, M)$. The proposition then  follows from Proposition \ref{lemfg0123450}.
\end{proof}

%%\begin{lem}\label{pull-backordcom6}For all open subgroups $D'$ of $D$ and $\fL'$ of $\fL$, the pull back map\[  \oH_{\Phi, \mathbf P}^{i}(D \fL, M)\rightarrow   \oH_{\Phi, \mathbf P}^{i}(D' \fL', M)\]sends $\oH_{\Phi}^{i,t}(D \fL, M)$ into $\oH_{\Phi}^{i,t}(D' \fL', M)$. \end{lem}\begin{proof}  Suppose that  $\fG\in {\mathscr G}_{\fL, \BG}^{D,t}$ and  $\fG'\in {\mathscr G}_{\fL',\BG}^{D',t}$  with $\fG'\preceq_P \fG$. It suffices to show that the transfer map \[ \rho: \oH_{\Phi}^{i}(D \fG, M)\rightarrow   \oH_{\Phi, \mathbf P}^{i}(D' \fG', M)\]sends $\oH_{\Phi}^{i,t}(D \fG, M)$ into $\oH_{\Phi}^{i,t}(D' \fG', M)$. Pick a positive integer   $k\in \BN_{\fG}\cap \BN_{\fG'}$. It follows from Proposition \ref{hecke001} that the diagram \[ \begin{CD}     \oH_{\Phi}^{i}(D \fG, M)@>\rho>>    \oH_{\Phi}^{i}(D'\fG', M)\\           @VV\rho^*_{t^k} V          @VV \rho^*_{t^k} V\\\oH_{\Phi}^{i}(D \fG, M)@>\rho >>    \oH_{\Phi}^{i}(D'\fG', M)\\           \end{CD}\]commutes. The lemma then easily follows. \end{proof}

  %By Example \ref{exsfv} and \eqref{conphi}, we have an admissible smooth representation  $\oH_\Phi^i(\mathsf G, V)$ of $\mathsf G^\natural$
  
\section{The nearly ordinary part} \label{sec:NOP}

In this section we study the nearly ordinary part of the automorphic cohomology. The nearly ordinary part in the setting of representation theory of $p$-adic groups has been studied in \cite{Em10}. % the $p$-adic setting has been also studied in \cite{Em10}.

Let $V$ be a nonzero continuous finite-dimensional representation of $G$ over $E$. Assume that $V$ is definable over a closed subfield  of $E$ that is a finite extension of $\Q_p$.  In other words, there is a continuous representation of $G$ on a finite-dimensional vector space $V'$ over $E'$ such that $V=E\otimes V'$ as representations of $G$, where $E'$ is a closed subfield of $E$ that is a finite extension of $\Q_p$. 
View $V$  as an $E[\mathsf G(\Q)\times {\mathsf G}^\natural] $-module with the given action of $G$ and the trivial action of $\mathsf G(\Q)\times {\mathsf G}^{\natural,p}$.

Recall that $t\in \mathrm{Df}(\p,\n)$. Denote by $V^t$ the unique $t$-stable nonzero  subspace of $V$ 
with the following properties: 
\begin{itemize}
\item there is a constant  $r_V\in \Q$ such that  $\abs{a_0}_p=p^{r_V}$ for all eigenvalues $a_0\in \C_p^\times $ of the operator $t: V^t\rightarrow V^t$,
\item  $\abs{a_1}_p < p^{r_V}$ for all  eigenvalues $a_1\in \C_p^\times$ of the operator $t: V/V^t\rightarrow V/V^t$.
\end{itemize}
Here and henceforth, when no confusion is possible, we still use $t$ to denote various maps attached to $t$. Note that $V^t$ is contained in $V^N$, is $P_0$-stable, and is defined over $E'$ in the following sense:   
%\be\label{decv200}
\[
V^t=E\otimes {V'}^t, \qquad \textrm{where }\ {V'}^t:=V^t\cap V'.
\]
%Also note that $V^t\subset V^N$.

\subsection{The nearly ordinary part and Hida's inequality}\label{sechida}

As a special case of  \eqref{paracomp}, we have an identification (for $\BG=\mathsf G^\natural$ and $M=V$)
\be\label{paracomp27}
\oH_{\Phi,P}^{i}(\mathsf G,V)=\left(\widehat { \oH^{i}_{\Phi}(\mathsf G, V)}_{\p\mathrm{-sm}}\right)^{N}\otimes \RD(\g/\p).
\ee
This identification respects the natural actions of $\mathsf G^{\natural, p}\times P$ on the both sides. These actions are still denoted by $g\mapsto \rho_g$.

Unless otherwise specified, throughout this section 
$D$ denotes an open compact subgroup of $\mathsf G^{\natural,p}$ and   $\fL$ denotes a $D$-neat  open compact subgroup of $L_0$. 
Note that 
\[ %\label{subpcoh}
  \oH^{i}_{\Phi,P}(D\fL, V)=\oH^{i}_{\Phi,P}(\mathsf G, V)^{D\fL}\subset \oH^{i}_{\Phi,P}(\mathsf G, V)\qquad  (\textrm{see \eqref{identifydl}})
\]
is a $\rho_t$-stable subspace and $\oH^{i}_{\Phi,P}(\mathsf G, V)$ is the union of all such subspaces. Proposition \ref{ordm} implies that $\oH^{i}_{\Phi,P}(D\fL, V)$ is finite-dimensional. In conclusion $\oH_{\Phi,P}^i({\mathsf G},  V)$ is an admissible smooth representation of $\mathsf G^{\natural,p}\times L$.

Consequently,  we have a decomposition
\[
\C_p\otimes \oH_{\Phi,P}^i({\mathsf G},  V)=\bigoplus_{\nu\in \C_p^\times} \left(\C_p\otimes \oH_{\Phi,P}^i({\mathsf G},V)\right)_\nu,
\]
where $\left(\C_p\otimes \oH_{\Phi,P}^i({\mathsf G},V)\right)_\nu$ is the generalized eigenspace of $\rho_t$ with eigenvalue $\nu$. 
%Note that $V_0$ is $P$-stable and $V_1$ is $\bar P$-stable.
%Suppose that $\BG=\mathsf G^\natural$ for the moment so that $\mathbf P=P$. 
In view of Hida's inequality  (Proposition \ref{prpcontrol}), we define the nearly ordinary part 
\[
 \oH_{\Phi,P}^{i, \mathrm{ord}}({\mathsf G},  V):=\oH_{\Phi,P}^{i, t\mathrm{-ord}}({\mathsf G},  V)\subset  \oH_{\Phi,P}^i({\mathsf G},  V)
\]
to be the subspace such that
\[
\C_p\otimes \oH_{\Phi,P}^{i, \mathrm{ord}}({\mathsf G},  V):=\bigoplus_{\nu\in \C_p^\times,\abs{\nu}_p=p^{r_V}} \left(\C_p\otimes\oH_{\Phi,P}^i({\mathsf G},V)\right)_\nu. 
\]
\begin{leml}\label{confg000}
     For all $\fG\in {\mathscr G}_{\fL, \mathsf G^\natural}$ and all $k\in \BN$ such that $\fG\preceq t^{k}\fG t^{-k}$, 
the diagram
 \[ %\label{comflg}
 \begin{CD}
     \oH^{i}_{\Phi,P}(D\fL, V)@>\textrm{projection}>>\oH^{i}_{\Phi}(D\fG, V)\\
     @V \rho_{t^k}VV @VV\rho_{t^k} V \\
     \oH^{i}_{\Phi,P}(D\fL, V)@>\textrm{projection}>>\oH^{i}_{\Phi}(D\fG, V)
 \end{CD}
 \]
 commutes. 
\end{leml}

\begin{proof}
This follows by considering the commutative diagram  
   % \be\label{comflg}\begin{CD}    \oH^{i}_{\Phi,P}(D\fL, V)@>\textrm{projection}>>\oH^{i}_{\Phi}(D\fG, V)\\   @V \rho_{t^k}VV @VV\rho_{t^k} V \\  \oH^{i}_{\Phi}(D\fL, V)@>\textrm{projection}>>\oH^{i}_{\Phi}(Dt^{k}\fG t^{-k}, V).\end{CD}\ee
 \[
  \xymatrix{
 \oH^{i}_{\Phi,P}(D\fL, V)
 \ar[rr]^{\textrm{projection}}  \ar[d]^{\rho_{t^k}}&& \  \oH^{i}_{\Phi}(D\fG, V)  \ar[d]_{\rho_{t^k}} \ar[rdd]^{\rho_{t^k}}&\\
  \oH^{i}_{\Phi,P}(D\fL, V)
 \ar[rr]^{\textrm{projection}}  \ar[rrrd]_{\textrm{projection}}&& \  \oH^{i}_{\Phi}(D t^{k}\fG t^{-k}, V) \ar[rd]^{\rho} & \\
 & & &\ \oH^{i}_{\Phi}(D\fG, V).  \\
    }
\]
Here the commutativity of the upper right triangle follows from Proposition \ref{hecke001}, and the commutativity of the square and the lower left triangle follows by the definitions.
\end{proof}

Write $d_V>0$ for the denominator of $r_V$ as an irreducible fraction.  Set $t':=t^{d_V}$. We define the star action of the monoid $\la t'\ra $ on $V$ by 
\be\label{starac} 
  \la t'\ra\times V\rightarrow V, \quad (t'^{k}, u)\mapsto t'^k*u:= p^{k\cdot d_V \cdot r_V}\cdot (t'^{k}.u).
\ee
As before, the associated  Hecke maps are written as 
\[
  \rho^*_{t'^k}: \oH^{i}_{\Phi}(K, V)\rightarrow \oH^{i}_{\Phi}(K, V),
\]
where $K$ is an open compact subgroup of $\mathsf G^\natural$. 

The following lemma is obvious. 
\begin{leml}\label{stareq}
For all open compact subgroups $\fG$ of $G$, the equality
\[
  \rho^*_{t'}=p^{ d_V \cdot r_V}\cdot \rho_{t'} : \oH^{i}_{\Phi}(D\fG, V)\rightarrow \oH^{i}_{\Phi}(D\fG, V)
\]
holds. 
\end{leml}

Write $V_t:=V/V^t$. Recall that $\CO$ denotes the ring of integers in $E$. Similarly let $\CO'$ denote the ring of integers in $E'$. 
Take a pair $(\fV, \fG_\fV)$ where $\fV$ is a lattice of $V$ and $\fG_\fV$ is an open compact subgroup of $G$ such that 
\be\label{fvfg}
\begin{cases}
     \fV=\CO\otimes \fV', \qquad \textrm{where }\ \fV':=\fV\cap V';
    &\\
        t'*\fV\subset\fV \ 
   \textrm{and}\  t'*\fV^t=\fV^t,\ 
\textrm{ where $\fV^t:=\fV\cap V^t$}; &\\ 
 \textrm{$\fG_\fV$  stabilizes  $\fV$.} % such that $(\fG_\fV\cap P)/(\fG_\fV\cap N)\supset \fL$. 
   \end{cases}
\ee
%Take a pair $(\fV, \fG_\fV)$ with the following properties:  \begin{itemize}   \item $\fV$ is a  lattice of $V$ that is defined over $\CO_0$ in the following sense:  \[ \fV=\CO\otimes \fV_0, \qquad \textrm{where }\ \fV_0:=\fV\cap V_0; \]   \item \[   t'*\fV\subset\fV \quad \textrm{and}\quad t'*\fV^t=\fV^t,\]where $\fV^t:=\fV\cap V^t$; %and $\fV_t\subset V_t$ denotes the image of $\fV$ under the quotient map $V\rightarrow  V_t$; \item $\fG_\fV$ is an open compact subgroup of $G$ stabilizing $\fV$. % such that $(\fG_\fV\cap P)/(\fG_\fV\cap N)\supset \fL$. \end{itemize}  
It is easy to see that such a pair  exists. 

In the rest of this section suppose that   
\[
\BG=\mathsf G^{\natural,p}\times \fG_\fV.
\]
 %This holds at least when the linear map $t: V\rightarrow V$ is definable over $E\cap \overline{\Q}_p\subset \C_p$.
 The star action \eqref{starac} induces  a monoid  action  
\be\label{starfv}
  \la t'\ra\times \fV\rightarrow \fV,
\ee
which is still called a star action. 
The associated Hecke maps are still written as 
\be\label{hechestar}
  \rho^*_{t'^k}: \oH^{i}_{\Phi}(D\fG, \fV)\rightarrow \oH^{i}_{\Phi}(D\fG, \fV),
\ee
where $\fG$ is an open compact subgroup of $\fG_\fV$. %ever the open compact subgroup $\fG$ of $G$ stabilizes $\fV$.

As mentioned in the Introduction, the following result  is essentially  due to Hida.  We will sketch a proof in what follows. 
 \begin{prpl}\label{prpcontrol}
    For every eigenvalue $a\in \C_p^\times$ of the operator $\rho_t$ on the space \eqref{paracomp27}, the inequality
    \[
      \abs{a}_p\leq p^{r_V} 
    \]
    holds. 
 \end{prpl}
 \begin{proof}
   Suppose that $\fL_t\subset \fG_\fV$. It suffices to show that  every eigenvalue of the operator 
   \[
   \rho_t: \oH^{i}_{\Phi,P}(D\fL, V) \rightarrow 
   \oH^{i}_{\Phi,P}(D\fL, V) 
   \]
   has $p$-adic norm $\leq 
   p^{r_V}$.

Pick a group $\fG\in {\mathscr G}_{\fL, \BG}^{D,t}$ such that $\fG\preceq_P t\fG t^{-1}$. Proposition \ref{ordm} implies  that the projection map 
\[
\oH^{i}_{\Phi,P}(D\fL, V)\rightarrow \oH^{i}_{\Phi}(D\fG, V)
\]
is injective. In view of Lemma \ref{confg000}, it remains to show that all eigenvalues of the operator 
   \[
   \rho_t: \oH^{i}_{\Phi}(D\fG, V) \rightarrow 
   \oH^{i}_{\Phi}(D\fG, V) 
   \]
   have $p$-adic norms $\leq 
   p^{r_V}$. By Lemmas \ref{hecket} and \ref{stareq}, this is equivalent to saying that all eigenvalues of the operator 
   \[
   \rho^*_{t'}: \oH^{i}_{\Phi}(D\fG, V) \rightarrow 
   \oH^{i}_{\Phi}(D\fG, V) 
   \]
   have $p$-adic norms $\leq 1$. This is clear by considering the commutative diagram 
   \[
 \begin{CD}
     \oH^{i}_{\Phi}(D\fG, \fV)@>>>\oH^{i}_{\Phi}(D\fG, V)\\
     @V \rho^*_{t'}VV @VV\rho^*_{t'} V \\
     \oH^{i}_{\Phi}(D\fG, \fV)@>>>\oH^{i}_{\Phi}(D\fG, V).
 \end{CD}
 \]
 \end{proof}

\subsection{The $t'$-stable part and the nearly ordinary part}

Every finitely generated $\CO$-module is viewed as a topological module under the $p$-adic topology. 

 \begin{leml}\label{prpord22}
Let $J$ be an $\CO$-module and let $\phi:J\rightarrow J$ be an $\CO$-module homomorphism. Assume that there is a finitely generated $\CO'$-submodule $J'$ of $J$ such that $J=\CO\otimes J'$ and $\phi(J')\subset J'$. Then 
\be\label{decomj}
  J=J_{\mathrm{sta}}\oplus J_{\mathrm{nil}}
\ee
and $\phi$ induces an isomorphism on $J_{\mathrm{sta}}$, where 
\[
J_{\mathrm{sta}}:=\bigcap_{k\in \BN} \phi^k(J)
\quad \textrm{
and}\quad 
J_{\mathrm{nil}}:=\{x\in J \,:\, \lim_{k\rightarrow \infty}\phi^k(x)=0\}.
\]
Moreover, the decomposition \eqref{decomj} is defined over $\CO'$ in the sense that 
\[
J_{\mathrm{sta}}=\CO\otimes (J_{\mathrm{sta}}\cap J')\quad \textrm{
and}\quad J_{\mathrm{nil}}=\CO\otimes (J_{\mathrm{nil}}\cap J'). 
\]
\end{leml}

\begin{proof}
%At least some variants of this proposition are known. We sketch a proof for completeness. 
The assumption of the lemma implies that
\[
  J=\varprojlim_{k\in \BN} J/p^k.
\]
In view of Lemma \ref{fsuriso}, the lemma is easily reduced to the case when $\CO=\CO'$ and $J$ has finite cardinality. Suppose this is the case. Then it is clear that $J_{\mathrm{sta}}\cap J_{\mathrm{nil}}=\{0\}$.

%Now assume without loss of generality that $J$ is $R$-finite. A similar argument as in Lemma \ref{defnil} shows that\be\label{intfin0}   J_{t\mathrm{-sur}}= \bigcap_{t\in H^+} t.J=\bigcap_{k\geq 0} t_+^k. J,\eewhere $t_+$ is a dominating element of $H^+$ as before. 

Since $J$ has finite cardinality, there is a positive integer $k$ such that 
\[
\phi^k(J)=\phi^{k+1}(J)=\phi^{k+2}(J)=\cdots.
\]
Let $x\in J$. Then
$\phi^k(x)= \phi^{2k}(y)$ for some $y\in J$. Write
\[
x = \phi^k(y) + (x - \phi^k(y)).
\]
The lemma follows by noting  that 
\[
\phi^k(y)\in \phi^k(J)=J_{\mathrm{sta}},
\]
and 
\[
x-\phi^k(y) \in J_{\rm nil}
\]
since $\phi^k(x-\phi^k(y))=0$. 
\end{proof}

%\begin{prpl}\label{prpord2}Let $H_1^+$ be a submonoid of $H^+$ that  has a common dominating element with $H^+$. Then for each $R[H^+]$-module $J$, \[  (J |_{H_1^+})_{\mathrm{nil}}=J_{\mathrm{nil}}.\] If  $R$ is Artinian and $J$ is locally $R$-finite, then \[  (J |_{H_1^+})_{t\mathrm{-sur}}=J_{t\mathrm{-sur}}.\]Here $J |_{H_1^+}:=J$, to be viewed as an $R[H_1^+]$-module.\end{prpl}\begin{proof}The first assertion follows from \eqref{equnil}, and the second one follows from \eqref{intfin}.\end{proof}

%We are particularly interested in the cases when $R=\CO$ or $\CO_k=\CO/p^k$ ($k\in \BN$). 

In the rest of this section we further assume that $\fL_t\subset  \fG_\fV$. As before write 
\[
 \mathbf P:=P\cap \BG=P\cap \fG_\fV. 
 \] %Then there is an obvious identification \[\oH^{i}_{\Phi,P}(D\fL, V)=  \left(\oH^{i}_{\Phi,P}(\mathsf G, V)\right)^{\fL N}.\]
Define 
\[
  \oH^{i,\mathrm{ord}}_{\Phi,P}(D\fL, V):=\oH^{i}_{\Phi,P}(D\fL, V)\cap \oH^{i,\mathrm{ord}}_{\Phi,P}(\mathsf G, V). 
\]
By using the Hecke maps associated to the star action, we define the $t'$-stable part  $\oH^{i,t'}_{\Phi,\mathbf P}(D\fG, \fV)$ and so on as in Section \ref{sectstable}. 

\begin{prpl}
Let $\fG\in {\mathscr G}_{\fL,\BG}^{D,t}$. 
    Then the image of $\oH^{i,\mathrm{ord}}_{\Phi,P}(D\fL, V)$ under the projection map
    \[
\oH^{i}_{\Phi,P}(D\fL, V)\rightarrow 
\oH^{i}_{\Phi}(D\fG, V)
\]
equals  $E\otimes \oH^{i,t'}_{\Phi}(D\fG, \fV)$. Consequently, there are  identifications
\be\label{pcoh}
\oH^{i,\mathrm{ord}}_{\Phi,P}(D\fL, V)=E\otimes \oH^{i,t'}_{\Phi,\mathbf P}(D\fL, \fV)
\ee
and 
\[
\oH^{i,\mathrm{ord}}_{\Phi,P}(\mathsf G, V)=E\otimes \oH^{i,t'}_{\Phi,\mathbf P}(\mathsf G, \fV).
\]
\end{prpl}

\begin{proof}
    Pick a positive integer $k$ such that 
    $\fG\preceq_P t'^k \fG t'^{-k}$. Write $J:=\oH^{i}_{\Phi}(D\fG, \fV)$ and 
    \[
    \phi:=\rho_{t'^k}^*: J\rightarrow J.
    \]
   Then we have a decomposition $  J=J_{\mathrm{sta}}\oplus J_{\mathrm{nil}}$ as in \eqref{decomj}, with $J_{\mathrm{sta}}=\oH^{i,t'}_{\Phi}(D\fG, \fV)$. Note that all eigenvalues of  
   \[ \phi: E\otimes J_{\mathrm{sta}}\rightarrow E\otimes  J_{\mathrm{sta}}\]
   have $p$-adic norm $1$, and all 
   eigenvalues of  
   \[ \phi: E\otimes J_{\mathrm{nil}}\rightarrow E\otimes  J_{\mathrm{nil}}
   \]
   have $p$-adic norms $<1$. This implies the first assertion of proposition. The second assertion then follows by Proposition \ref{ordm}. 
\end{proof}

By the identification \eqref{pcoh}, we have an inclusion 
\[
\oH^{i,\mathrm{ord}}_{\Phi,P}(D\fL, V)\subset E\otimes \oH^{i}_{\Phi,\mathbf P}(D\fL, \fV).
\]

%is injective, and hence we identify  $\oH^{i}_{\Phi,P}(D\fL, V)$ with  subspce of 

\subsection{The $t$-stable part and relative completed cohomologies}
Put 
\[
{\mathscr G}_{\fL,\BG,k}^{D,t}:=\{\fG_k\in {\mathscr G}_{\fL,\BG}^{D,t} \,:\, \fG_k\textrm{ stabilizes $\fV^t/p^k\subset \fV/p^k$}\}\qquad (k\in \BN). 
\]
Since $N$ acts trivially on $\fV^t/p^k$, this set is cofinal in ${\mathscr G}_{\fL,\BG}^{D,t}$.

The star action \eqref{starfv} induces compatible actions 
 \[
  \la t'\ra\times \fV/p^k\rightarrow \fV/p^k
\]
and 
 \[
  \la t'\ra\times \fV^t/p^k\rightarrow \fV^t/p^k.
\]
As before, associated to these compatible actions we define the Hecke maps and the $t'$-stable parts 
\[
\begin{cases}
    \oH^{i,t'}_{\Phi}(D\fG, \fV/p^k)\subset \oH^{i}_{\Phi}(D\fG, \fV/p^k);&\smallskip\\
    \oH^{i,t'}_{\Phi}(D\fG_k, \fV^t/p^k)\subset \oH^{i}_{\Phi}(D\fG_k, \fV^t/p^k);&\smallskip\\
    \oH^{i,t'}_{\Phi, \mathbf P}(D\fL, \fV/p^k):= \varprojlim_{\fG\in {\mathscr G}_{\fL,\BG}^{D,t}}\oH^{i,t'}_{\Phi}(D\fG, \fV/p^k)\subset \oH^{i}_{\Phi, \mathbf P}(D\fL, \fV/p^k);&\smallskip\\
        \oH^{i,t'}_{\Phi, \mathbf P}(D\fL, \fV^t/p^k):= \varprojlim_{\fG_k\in {\mathscr G}_{\fL,\BG,k}^{D,t}}\oH^{i,t'}_{\Phi}(D\fG_k, \fV^t/p^k)\subset \oH^{i}_{\Phi, \mathbf P}(D\fL, \fV^t/p^k),&\\
    \end{cases}
\]
where $\fG\in {\mathscr G}_{\fL,\BG}^{D,t}$ and $\fG_k\in {\mathscr G}_{\fL,\BG,k}^{D,t}$. 

\begin{prpl}\label{eqv0}
   Let $\fG_k\in {\mathscr G}_{\fL,\BG,k}^{D,t}$.  The inclusion map $\fV^t/p^k\to \fV/p^k$ induces  an isomorphism
   \[
   \oH^{i,t'}_{\Phi}(D\fG_k, \fV^t/p^k)\cong \oH^{i,t'}_{\Phi}(D\fG_k, \fV/p^k).
   \]
\end{prpl}

\begin{proof}
Let $\fV_t\subset V_t$ denote the image of $\fV$ under the quotient map $V\rightarrow  V_t$.     It is clear that the $t'$-stable part 
    \[
    \oH^{j,t'}_{\Phi}(D\fG_k, \fV_t/p^k)=\{0\}\qquad (j\in \Z).
    \]
     In view of Lemma \ref{prpord22}, the proposition then follows by considering the long exact sequence 
  \begin{eqnarray*}
&\cdots&\rightarrow   \oH_\Phi^{i-1}(D\fG_k, \fV_t/p^k)\rightarrow \oH_\Phi^i(D\fG_k, \fV^t/p^k)\\
  && \rightarrow \oH_\Phi^i(D\fG_k, \fV/p^k )\rightarrow \oH_\Phi^i(D\fG_k, \fV_t/p^k)\rightarrow \cdots . 
\end{eqnarray*}
%One checks that all arrows in the above sequence are $\rho^*{t'^r}$-equivariant. In view of Proposition \ref{prpord3}, it suffices to show that  \[    \left( \oH_\Phi^{i}(D\fG_k, ((\fV/\fV^\n)\otimes \fV')/p^k)\right)_{t\mathrm{-sur}}=0.\]This holds because a dominating element of  $T_{\bar P}^\dag$ acts as a nilpotent operator on $\oH_\Phi^{i}(D\fG_k, ((\fV/\fV^\n)\otimes  \fV')/p^k)$. 
\end{proof}

By Proposition \ref{eqv0}, we have an identification
   \[
   \oH^{i,t'}_{\Phi, \mathbf P}(D\fL, \fV^t/p^k)\cong \oH^{i,t'}_{\Phi, \mathbf P}(D\fL, \fV/p^k).
   \]
  % As in the Introduction, we define   \be\label{defthi}\widetilde{\oH}^{i}_{\Phi}(\mathsf G, V)^{\la \p\supset \p\ra,\circ}:=E\otimes \widetilde{\oH}^{i}_{\Phi}(\mathsf G, \fV)^{\la \p\supset \p\ra},\eeand similarly \be\label{defthi2}\widetilde{\oH}^{i}_{\Phi}(\mathsf G, V^t)^{\la \p\supset \p\ra,\circ}:=E\otimes \widetilde{\oH}^{i}_{\Phi}(\mathsf G, \fV^t)^{\la \p\supset \p\ra}.\ee Here  \[\widetilde{\oH}^{i}_{\Phi}(\mathsf G, \fV)^{\la \p\supset \p\ra,\circ}:= \varinjlim_{D\fP}  \varprojlim_{k\in \BN}  \oH^{i}_{\Phi}(D\fP, \fV/p^k)  \]  and   \[\widetilde{\oH}^{i}_{\Phi}(\mathsf G, \fV^t)^{\la \p\supset \p\ra,\circ}:=\varinjlim_{D\fP}  \varprojlim_{k\in \BN}  \oH^{i}_{\Phi}(D\fP, \fV^t/p^k),  \] where $D$ runs all open compact subgroups of $\mathsf G^{\natural,p}$ and $\fP$ runs over all open compact subgroups of $\mathbf P=P\cap \fG_\fV$. It is easy to see that the spaces \eqref{defthi} and \eqref{defthi2} are independent of the pair $(\fV, \fG_\fV)$. 
   For each $\fP\in {\mathscr P}_{\fL,\mathbf P}$, we have a commutative diagram 
 \be\label{deftxi}
  \xymatrix{
 &\oH^{i,t'}_{\Phi, \mathbf P}(D\fL, \fV) \ar[rd]\ar[ld]& \\
 \varprojlim_{k\in \BN} \oH^{i,t'}_{\Phi, \mathbf P}(D\fL, \fV^t/p^k)
\ar[d]\ar[rr]^\cong  && \  \varprojlim_{k\in \BN}\oH^{i,t'}_{\Phi, \mathbf P}(D\fL, \fV/p^k) \ar[d]\\
\varprojlim_{k\in \BN}  \oH^{i}_{\Phi}(D\fP, \fV^t/p^k)
\ar[rr]&& \  \varprojlim_{k\in \BN}\oH^{i}_{\Phi}(D\fP, \fV/p^k) \\
      }
\ee
where the two vertical arrows are the inverse limits of the projections with respect to the isomorphism \eqref{isopc},  the down left arrow is the map that makes the diagram commute, and all the other arrows are the coefficient change maps. 
By compositions in the above diagram we get a commutative diagram
\[
  \xymatrix{
 &\oH^{i,t'}_{\Phi, \mathbf P}(D\fL, \fV) \ar[rd]\ar[ld]& \\
 \varprojlim_{k\in \BN}  \oH^{i}_{\Phi}(D\fP, \fV^t/p^k)
\ar[rr]&& \  \varprojlim_{k\in \BN}\oH^{i}_{\Phi}(D\fP, \fV/p^k). 
      }
\]
By using Lemma \ref{pull-backordcom} and taking the direct limits, we further get a commutative diagram (see \eqref{cfc2})
\[
  \xymatrix{
 &\oH^{i,t'}_{\Phi,  \mathbf P}(\mathsf G, \fV) \ar[rd]\ar[ld]& \\
   \widetilde{\oH}^{i}_{\Phi}(\mathsf G, \fV^t)^{\la \p\supset \p\ra}
\ar[rr]&& \  \widetilde{\oH}^{i}_{\Phi}(\mathsf G, \fV)^{\la \p\supset \p\ra}. \\ 
      }
\]
Finally, taking tensor product with $E$ we get a commutative diagram 
\be\label{defxi}
  \xymatrix{
 &\oH^{i,\mathrm{ord}}_{\Phi}(\mathsf G, V) \ar[rd]\ar[ld]& \\
   \widetilde{\oH}^{i}_{\Phi}(\mathsf G, V^t)^{\la \p\supset \p\ra,\circ}
\ar[rr]&& \  \widetilde{\oH}^{i}_{\Phi}(\mathsf G, V)^{\la \p\supset \p\ra, \circ}. \\
      }
\ee

   \subsection{A commutative diagram} 
%As in the Introduction we define \[  {\oH}_\Phi^i({\mathsf G},  V)^{\la \p\ra, \circ}:=E\otimes {\oH}_\Phi^i({\mathsf G},  \fV)^{\la \p\ra},\] which is independent of the pair $(\fV, \fG_\fV)$. 

We  now explain the following  diagram of linear maps (see also  \eqref{thediag00}): 
\be\label{thediag0011}
 \begin{CD}
\widetilde{\oH}_\Phi^i({\mathsf G},  V^t)^{\la \p\supset \p\ra,\circ} @<\widetilde \xi<<\oH_{\Phi,P}^{i,\mathrm{ord}}({\mathsf G},  V)@> \widehat{\xi} >>   \widehat{\oH_\Phi^i({\mathsf G}, V)}_{{\p}\mathrm{-sm}}\otimes \RD(\g/\p)\\
      @VVV      @V V\xi    V          @VV V\\
\widetilde{\oH}_\Phi^i({\mathsf G},  V)^{\la \p\supset \p\ra,\circ} @<<<{\oH}_\Phi^i({\mathsf G},  V)^{\la \p\ra, \circ}@>  >>     {\oH_\Phi^i({\mathsf G},  V)}^{\la \p\ra}.\\
            \end{CD}
\ee
\begin{itemize}
    \item The map $\widehat \xi$ is the inclusion map with respect to the identification \eqref{paracomp27}.
    \item The right vertical arrow is the isomorphism given in Lemma \ref{isosm00}.
    \item The map $\widetilde \xi$ and the left vertical arrow are given in \eqref{defxi}. 
    \item Recall that ${\oH}_\Phi^i({\mathsf G},  V)^{\la \p\ra, \circ}:=E\otimes {\oH}_\Phi^i({\mathsf G},  \fV)^{\la \p\ra}$. The right bottom horizontal arrow is the linear map induced by the inclusion map $\fV\rightarrow V$, and the left bottom horizontal arrow is the linear map induced by the maps $\fV\rightarrow \fV/p^k$ ($k\in \BN$). 
    
    \item The map $\xi$ is defined by requiring that the diagram 
     \be\label{defxi9}
 \begin{CD}
     \oH^{i,t'}_{\Phi,\mathbf P}(D\fL, \fV)@>\eqref{pcoh}  >> {\oH}_{\Phi,P}^{i,\mathrm{ord}}({\mathsf G},  V) \\
     @V \textrm{projection} VV @VV \xi V \\
     \oH^{i}_{\Phi}(D\fP, \fV) @> >>{\oH}_\Phi^i({\mathsf G},  V)^{\la \p\ra,\circ}
 \end{CD}
 \ee
  commutes for every  open compact subgroup $D$ of $\mathsf G^{\natural,p}$, every $D$-neat open compact subgroup $\fL$ of $L_0$ such that $\fL_t\subset \fG_\fV$, and every $\fP\in {\mathscr P}_{\fL, \mathbf P}$. The existence of $\xi$ is guaranteed  by Lemma \ref{pull-backordcom}.  %, where  the bottom horizontal arrow is the composition of  \[\oH^{i}_{\Phi,\mathbf P}(D\fL, \fV)\xrightarrow{\eqref{paratorela}}\oH^{i}_{\Phi}(\mathsf G, \fV)^{\la \p\ra}\rightarrow\oH^{i}_{\Phi}(\mathsf G, V)^{\la \p\ra}.\]
\end{itemize} 

%Note that all the six spaces in \eqref{thediag0011} carry natural representations of $\mathsf G^{\natural,p}\times P$. 

\begin{thml}
The two squares in \eqref{thediag0011} are commutative. %, and all the seven arrows in  \eqref{thediag0011}  are   $\mathsf G^{\natural,p}\times P$-equivariant. 
\end{thml}
\begin{proof}
As before, suppose that $D$ is an open compact subgroup of $\mathsf G^{\natural,p}$, $\fL$ is a $D$-neat open compact subgroup of $L_0$ such that $\fL_t\subset \fG_\fV$, and $\fP\in {\mathscr P}_{\fL, \mathbf P}$. The commutativity of the left square easily follows by considering the commutative diagram 
\[
  \xymatrix{
  \varprojlim_{k\in \BN}  \oH^{i}_{\Phi}(D\fP, \fV^t/p^k)  \ar[d]& &\varprojlim_{k\in \BN}  \oH^{i,t'}_{\Phi,\mathbf P}(D\fL, \fV^t/p^k)\ar[ll]_{\textrm{projection}}%^{\eqref{isoff}\qquad\quad} 
  \ar[d]^{\cong}& \oH^{i,t'}_{\Phi,\mathbf P}(D\fL, \fV) \ar[ld] \ar[l] \\
  \varprojlim_{k\in \BN} \oH^{i}_{\Phi}(D\fP,\fV/p^k) & &\varprojlim_{k\in \BN}  \oH^{i,t'}_{\Phi,\mathbf P}(D\fL, \fV/p^k)\ar[ll]_{\textrm{projection}}.  & 
        }
        \]
        
The   commutativity of the right square follows by considering the commutative diagram 
\[
  \xymatrix{
  \oH^{i,t'}_{\Phi,\mathbf P}(D\fL, \fV)\ar[r]^{\eqref{pcoh}}  \ar[d]^{\textrm{projection}}& \oH^{i,\mathrm{ord}}_{\Phi,P}(D\fL, V)  \ar[r]%^{\eqref{isoff}\qquad\quad} 
  \ar[d]^{\textrm{projection}}& \widehat{\oH_\Phi^i({\mathsf G}, V)}^{D(\fL_t N)}\otimes \RD(\g/\p)\ar[ld] \\
  \oH^{i}_{\Phi}(D\fP,\fV)\ar[r] & \oH^{i}_{\Phi}(D\fP, V).  & 
        }
        \]
\end{proof}

Consider the diagram %\eqref{thediag0011}:
\[ %\label{thediag0022}
\xymatrix{&\oH_{\Phi,P}^{i}(\mathsf G,V)\ar[r]&\widehat{\oH_\Phi^i({\mathsf G}, V)}_{{\p}\mathrm{-sm}}\otimes \RD(\g/\p)\ar[d]\\
\widetilde{\oH}_\Phi^i({\mathsf G},  V)^{\la \p\supset \p\ra,\circ} &{\oH}_\Phi^i({\mathsf G},  V)^{\la \p\ra, \circ}\ar[l]\ar[r] &   {\oH_\Phi^i({\mathsf G},  V)}^{\la \p\ra},\\
            }
\]
where the top horizontal arrow is the inclusion map with respect to the identification \eqref{paracomp27}, and the other three arrows are as in the diagram  \eqref{thediag0011}. 
All the five spaces in the above diagram carry naturally representations of $\mathsf G^{\natural, p}\times P$. We still use $g\mapsto \rho_g$ to denote these representations. It is routine to check that the two bottom horizontal arrows are independent of the choice of the lattice $\fV$, and all the four arrows above are $\mathsf G^{\natural, p}\times P$-equivariant. 

\subsection{Independence of $t$}

\begin{leml}
    If  $V'$ is irreducible as a representation of $\g$, then the space $V^t$ is independent of $t\in \mathrm{Df}(\p,\n)$. 
\end{leml}
\begin{proof}
    Assume without loss of generality that $E=E'$ so that $V$ is irreducible as an $E\otimes \g$-module. %Note that $V^t\subset V^\n$ is  $\l_t$-stable, and for every subspace $V_1$ of $V^t$, \[  \mathrm U(\g)\cap V^t=V_1.  \] Thus that $V^t$ an irreducible as a  $\l_t$-submodule of $V$.

Note that $V^t$ is contained in $V^\n$ and is $\l_t$-stable.   Suppose that $t_1\in \mathrm{Df}(\p,\n)$. First we assume that $t_1$ commutes with $t$.  Then $\l_t=\l_{t_1}$ by Lemma \ref{lemc2}, and hence 
    \[
    V= \mathrm U(\g). V^t=\mathrm U(\bar \n_{t_1}).V^t. 
    \]
    Here $\mathrm U$ indicates the universal enveloping algebra. % and $\bar \n_{t_1}$ denotes the Lie algebra of $\bar N_{t_1}$. 
    This implies that $V^{t_1}\subset V^t$. Similarly $V^{t}\subset V^{t_1}$, and hence $V^{t}=V^{t_1}$. The lemma in general then follows by Lemma \ref{lemc}. 
\end{proof}

If $V^t$  is independent of $t\in \mathrm{Df}(\p,\n)$, then the space $\widetilde{\oH}_\Phi^i({\mathsf G},  V^t)^{\la \p\supset \p\ra,\circ}$ also carries a representation of $\mathsf G^{\natural, p}\times P$, which is still denoted by $g\mapsto \rho_g$. It is easy to see that the left vertical arrow in \eqref{thediag0011} is independent of the choice of the lattice $\fV$, and is $\mathsf G^{\natural, p}\times P$-equivariant. 

The main purpose of this subsection is to prove the following result.
\begin{thml}\label{thmindpt}
    Assume that the subspace $V^t$ of $V$ is independent of $t\in \mathrm{Df}(\p,\n)$. Then the followings hold ture.

      \noindent  (a) The subspace $\oH_{\Phi,P}^{i,\mathrm{ord}}(\mathsf G,V)$ of $\oH_{\Phi,P}^{i}(\mathsf G,V)$ is  $\mathsf G^{\natural,p}\times L$-stable and is independent of $t\in \mathrm{Df}(\p,\n)$.

   \noindent  (b) The maps $\xi$ and $ \widetilde{\xi}$ in \eqref{thediag0011} are both  independent of $t\in \mathrm{Df}(\p,\n)$ and the pair $(\fV,\fG_\fV)$  in \eqref{fvfg}.

       \noindent (c) All the arrows in \eqref{thediag0011} are $\mathsf G^{\natural,p}\times P$-equivariant. 
\end{thml}
In the rest of this subsection we assume that the subspace $V^t$ of $V$ is independent of $t\in \mathrm{Df}(\p,\n)$. Then  $V^t$ is $P$-stable, and the representation of $P$ on it descends to a representation of $L$.

We say that an element $g_1\in P$ is a  canonical lift of an element $g\in L$ if it belongs to $L_{t_1}$ for some $t_1\in \mathrm{Df}(\p,\n)$ and the quotient map $P\rightarrow L$ sends it to $g$. By \eqref{conj}, all canonical lifts of an element $g\in L$ form an $N$-conjugacy class. 

Write $T_L$ for the largest central torus in $L_0$. Denote by $T_L^+$ the subset of all $t_1\in T_L$ such that for some (and hence all) canonical lifts $t_2\in P$ of $t_1$,  all eigenvalues of $\Ad_{{t}_2}: \n\rightarrow \n $
have $p$-adic norms $\geq 1$, and all eigenvalues of $\Ad_{{t}_2}: \g/\p\rightarrow \g/\p $
have $p$-adic norms $\leq 1$.

\begin{leml}\label{betav}
There is a unique locally constant homomorphism $\beta_V: T_L\rightarrow \Q$ such that for every $t_1\in T_L$, all eigenvalues of $t_1:V^t\rightarrow V^t$ have $p$-adic norm $p^{\beta_V(t_1)}$. Moreover, if $t_1\in T_L^+$, then all eigenvalues of $t_1:V/V^t\rightarrow V/V^t$ have $p$-adic norms $\leq p^{\beta_V(t_1)}$. 
\end{leml}

\begin{proof}
Let $t_1\in T_L$. For the first assertion, it suffices to show that  all eigenvalues of $t_1:V^t\rightarrow V^t$ have the same $p$-adic norm. Since $t_1$ is the product of a split element with an element in the maximal compact subgroup of $T_L$, without loss of generality we assume that $t_1$ is split.   
   % Write $\Lambda_L$ for the set of all split elements in $T_L$. Then $T_L=\Lambda_L\times T_L^c$, where $T_L^c$ is the maximal compact subgroup of $T_L$. Let $t_1\in \Lambda_L$. 
   Write $t_1'\in L_t$  for the  element corresponding to $t_1$ under the isomorphism $L\cong L_t$. Pick a sufficiently large positive integer $k$ so that $t_1' t^k\in \mathrm{Df}(\p,\n)$. By assumption, all eigenvalues of $t_1' t^k: V^t\rightarrow V^t$ have the same $p$-adic norm, say $p^c$ $(c\in \Q)$. Then all  eigenvalues of $t_1' :  V^t\rightarrow V^t$ have the same $p$-adic  norm $p^{c-k\cdot  r_V}$. This proves the first assertion of lemma. 

   Now suppose that $t_1\in T_L^+$. As before we assume without loss of generality that $t_1$ is split.  Then for all $k\in \BN$, all eigenvalues of $t_1^k t : V/V^t\rightarrow V/V^t$ have $p$-adic norms   $<p^{\beta_V(t_1^k t)}$. Hence all eigenvalues of $t_1: V/V^t\rightarrow V/V^t$ have $p$-adic norms 
   \[
    <p^{(k\cdot \beta_V(t_1) -c')/k}, 
   \]
   where $c'\in \Q$  is independent of $k$. This proves the second assertion. 
\end{proof}

Let $\beta_V: T_L\rightarrow \Q$ be as in Lemma \ref{betav}. 
Recall that $\oH_{\Phi,P}^{i}(\mathsf G,V)$ is an admissible smooth representation of $\mathsf G^{\natural,p}\times L$. Hence it is a union of finite-dimensional $T_L$-subrepresentations. 

\begin{prpl}\label{hida3}
    For every character $\chi: T_L\rightarrow \C_p^\times $ that occurs in $\C_p\otimes \oH_{\Phi,P}^{i}(\mathsf G,V)$, 
    \be\label{hida222}
      \abs{\chi(t_+)}_p\leq p^{\beta_V(t_+)}\quad \textrm{for all } t_+\in T_L^+.
    \ee
\end{prpl}
\begin{proof}
 Without loss of generality assume that $ t_+\in T_L^+$ is split. Write $t_+'\in L_t$ for the element corresponding to $t_+$ under the isomorphism $L\cong L_t$.   Define a character 
   \be\label{definebeta}
    \beta_\chi: T_L\rightarrow \R^\times, \quad t_1 \mapsto \abs{\chi(t_1)}_p \cdot p^{-\beta_V(t_1)}.
   \ee
Let $k\in \BN$. Since  $t (t_+')^k $ belong to $\mathrm{Df}(\p,\n)$, by  Proposition \ref{prpcontrol}  we have that 
 \[
   \beta_\chi([t] t_+^k )\leq 1,
   \]
   where $[t]$ denotes the image of $t$ under the quotient map $P\rightarrow L$. 
   Since $k$ is arbitrary, this implies that $\beta_\chi(t_+)\leq 1$, and hence the proposition follows. 
\end{proof}

\begin{leml}\label{lemordchi}
    Let $\chi: T_L\rightarrow \C_p^\times$ be a character such that \eqref{hida222} is satisfied. Then 
    \[
    \abs{\chi(t_1)}_p= p^{\beta_V(t_1)} \ \textrm{for all} \ t_1\in T_L \ \Longleftrightarrow \ \abs{\chi([t])}_p= p^{r_V},
    \]
    where $[t]$ denotes the image of $t$ under the quotient map $P\rightarrow L$. 
\end{leml}

\begin{proof}
   Since $[t]\in T_L$, we only need to prove the implication ``$\Longleftarrow$". Suppose that $\abs{\chi([t])}_p= p^{r_V}$.

    Let $\beta_\chi$ be as in \eqref{definebeta}. Then $\beta_\chi([t])=1$. For every $t_1\in T_L$, we have that 
    \[
   \beta_\chi(t_1)= \beta_\chi(t_1 [t]^k)\leq 1, 
    \]
    where $k$ is a sufficiently large integer so that $t_1 [t]^k\in T_L^+$. This proves  the lemma. 
\end{proof}

We prove part (a) of Theorem \ref{thmindpt} in the following lemma. 
\begin{leml}
    The subspace $\oH_{\Phi,P}^{i,\mathrm{ord}}({\mathsf G},  V)$ of $\oH_{\Phi,P}^{i}({\mathsf G},  V)$ is  independent of $t\in \mathrm{Df}(\p,\n)$ and is $\mathsf G^{\natural, p}\times L$-stable.
\end{leml}

\begin{proof}
 It follows from  Proposition \ref{hida3} and Lemma \ref{lemordchi} that the subspace $\oH_{\Phi,P}^{i,\mathrm{ord}}({\mathsf G},  V)$ is  independent of $t\in \mathrm{Df}(\p,\n)$. It is $\mathsf G^{\natural, p}\times L$-stable since $\beta_V$ is $L$-invariant.   
\end{proof}

%Note that the map $\widehat{\xi}$ in \eqref{thediag0011} is independent of $t\in \mathrm{Df}(\p,\n)$ and the pair $(\fV,\fG_\fV)$.

%All the six spaces in \eqref{thediag0011} carry naturally representations of $\mathsf G^{\natural, p}\times P$. We still use $g\mapsto \rho_g$ to denotes these representation. 

\begin{leml}
    The maps $\xi$ and $ \widetilde{\xi}$ in \eqref{thediag0011} are both  independent of the pair $(\fV,\fG_\fV)$ as in \eqref{fvfg}.
\end{leml}
\begin{proof}
    The independence of $\fG$ follows from the observation \eqref{indp}. The independence  of $\fV$ follows from the fact that all the  arrows in  \eqref{deftxi} and the left vertical arrow in \eqref{defxi9} are natural in  the coefficient system $\fV$. 
\end{proof}

\begin{leml}\label{imagexi}
    The images of $\xi$ and $\widetilde \xi$ are pointwise fixed by  $N$. 
\end{leml}
\begin{proof}
    Let $\fN$ be an open compact subgroup of $N$. We choose the pair $(\fV, \fG_\fV)$ appropriately so that $\fG_\fV\supset \fN$. Then it is easy to see that the images of $\xi$ and $\widetilde \xi$ are pointwise fixed by  $\fN$. This implies the lemma. 
\end{proof}

Write $\xi_{t}:=\xi$ and $\widetilde \xi_{t}:=\widetilde \xi$ to indicate the dependence on $t\in \mathrm{Df}(\p,\n)$. 

\begin{leml}\label{lemint0}
  For every $g\in N$, $\xi_{t}=\xi_{gtg^{-1}}$ and  $\widetilde \xi_{t}=\widetilde \xi_{gtg^{-1}}$. % and the pair $(\fV,\fG_\fV)$ as in \eqref{fvfg}.
\end{leml}
\begin{proof}

For every $g\in P$, it is routine to check that the diagram
\[
 \begin{CD}
\widetilde{\oH}_\Phi^i({\mathsf G},  V^t)^{\la \p\supset \p\ra,\circ} @<\widetilde \xi_{t}<<\oH_{\Phi,P}^{i,\mathrm{ord}}({\mathsf G},  V)\\
      @V\rho_g VV      @V V\rho_g    V         \\
\widetilde{\oH}_\Phi^i({\mathsf G},  V^t)^{\la \p\supset \p\ra,\circ} @<\widetilde \xi_{gtg^{-1}}<<\oH_{\Phi,P}^{i,\mathrm{ord}}({\sf G}, V)
            \end{CD}
\]
commutes. If $g\in N$, then the right vertical arrow is the identity map, and it follows from Lemma \ref{imagexi} that $\widetilde \xi_{t}=\widetilde \xi_{gtg^{-1}}$. The equality $\xi_{t}= \xi_{gtg^{-1}}$ is similarly proved. 
\end{proof}

\begin{leml}\label{lemint}
  For every $t_1\in \mathrm{Df}(\p,\n)$ that commutes with $t$, $\xi_{t_1}=\xi_{t}$ and  $\widetilde \xi_{t_1}=\widetilde \xi_{t}$. % and the pair $(\fV,\fG_\fV)$ as in \eqref{fvfg}.
\end{leml}
\begin{proof}
     Let $T_{L}^{+,\Z}$ denote the monoid of all elements of $ T_{L}^{+}$ whose image under $\beta_V$ is an integer. %such that $\beta_V(t_+)\in \Z$. 
     Let $T_{L_t}^{+,\Z}$ denote the submonoid of $L_t$ that corresponds to $T_{L}^{+,\Z}$ under the isomorphism 
     $L\cong L_t$. Define the star action of $T_{L_t}^{+,\Z}$ on $V$ by
     \[
     t_+* u:= p^{\beta_V([t_+])}\cdot t_+.u,
     \]
     where $[t_+]$ denotes the image of $t_+$ under the isomorphism $L_t\to L$. 
     
By using Lemma \ref{betav}, we choose the pair $(\fV, \fG_\fV)$ appropriately so that besides \eqref{fvfg} the following condition is also satisfied: 
     \[
      t_+*\fV\subset\fV \ 
   \textrm{ and }\  t_+*\fV^t=\fV^t\qquad \textrm{for all } t_+\in T_{L_t}^{+,\Z}. 
     \]
  Similar to \eqref{hechestar} we have the Hecke maps \be\label{heckestar2}
  \rho^*_{t_+}: \oH^{i}_{\Phi}(D\fG, \fV)\rightarrow \oH^{i}_{\Phi}(D\fG, \fV),\qquad t_+\in T_{L_t}^{+,\Z},
\ee
where $D$ is an open compact subgroup of $\mathsf G^{\natural, p}$ and $\fG$ is an open compact subgroup of $\fG_\fV$.  Similar to Lemma \ref{hecket}, the map \eqref{heckestar2} yields an action of the monoid 
\[
  T_{\fG}:=\{t_+\in T_{L_t}^{+,\Z} \,:\,  \fG \preceq_P  t_+ \fG t_+^{-1} \}
\]
on $\oH^{i}_{\Phi}(D\fG, \fV)$.

Note that $\bar N_{t_1}=\bar N_t$ for every $t_1\in \mathrm{Df}(\p,\n)$ that commutes with $t$. Let $\fL$ be an open compact subgroup of $L_0$, and recall the corresponding group $\fL_t\subset L_t$. Suppose  $\fL_t\subset \fG\subset \bar N_t \fL_t N$ so that $t'^k\in T_{\fG}$ for all sufficiently large $k\in \BN$. Note that $\la t'\ra \cap T_{\fG}$ is cofinal in $T_\fG$ in the following sense: 
for every $t_1\in T_{\fG}$, there is an element $t_2\in T_\fG$ such that $t_1 t_2\in \la t'\ra \cap T_{\fG}$. Therefore  
\be\label{tstaint}
  \bigcap_{t_+\in T_\fG} \rho^*_{t_+}\left(\oH^{i}_{\Phi}(D\fG, \fV)\right)=\bigcap_{k\in \BN,\, \fG\preceq_P {t'}^k \fG {t'}^{-k}} \rho^*_{{t'}^k}\left(\oH^{i}_{\Phi}(D\fG, \fV)\right).
\ee
This easily implies that $\xi_t$ only depends on $L_t$ and hence $\xi_{t_1}=\xi_t$. The equality $\widetilde \xi_{t_1}=\widetilde \xi_t$ is similarly proved. 
\end{proof}

%\begin{lem}\label{lem:ntransfer}The diagram\[ \begin{CD}    \oH_{\Phi, \mathbf P}^{i}(D \fL, M) @> \phi>>   \oH_{\Phi, \mathbf P}^{i}(D \fL, M') \\@VV\rho_{D\fL, D'\fL'} V          @VV\rho_{D\fL, D'\fL'}  V\\ \oH_{\Phi, \mathbf P}^{i}(D' \fL', M) @> \phi>>    \oH_{\Phi, \mathbf P}^{i}(D' \fL', M) \          \end{CD}\]commutes, where the horizontal arrows are defined as in Lemma \ref{ntransfer00}.\end{lem}\begin{proof}This is routine to check. \end{proof}

In view of Lemma \ref{lemc}, part (b) of Theorem \ref{thmindpt} now follows by Lemmas \ref{lemint0} and \ref{lemint}. To finish the proof of Theorem  \ref{thmindpt}, it remains to prove the following result. 
\begin{leml}
    The maps $\xi$ and $\widetilde \xi$ are $\mathsf G^{\natural,p}\times P$-equivariant. 
\end{leml}
\begin{proof}
    It is clear that the maps $\xi$ and $\widetilde \xi$ are $\mathsf G^{\natural,p}$-equivariant. By Lemma \ref{imagexi} they are  $N$-equivariant. It remains to show that they are also $L_t$-equivariant. For every $g\in L_t$, in the notation of the proof of Lemma \ref{lemint}, we have a commutative diagram 
    \[
 \begin{CD}
    \varprojlim_{\fG} \bigcap_{t_+\in T_\fG} \rho^*_{t_+}\left(\oH^{i}_{\Phi}(D\fG, \fV)\right)@>\rho_g  >> \varprojlim_{\fG} \bigcap_{t_+\in T_{\fG} }\rho^*_{gt_+ g^{-1}}\left(\oH^{i}_{\Phi}(D(g\fG g^{-1}), g.\fV)\right)  \\
     @V \textrm{projection} VV @VV \textrm{projection} V \\
     \oH^{i}_{\Phi}(D\fP, \fV) @> \rho_g >>\oH^{i}_{\Phi}(D(g\fP g^{-1}), g.\fV),
 \end{CD}
 \]
 where $\fG$ runs over ${\mathscr G}^{D,t}_{\fL, \BG}$, $\fL$ is a $D$-neat open compact subgroup of $L_0$, and $\fP\in {\mathscr P}_{\fL, \mathbf P}$. In view of the equality \eqref{tstaint}, this  implies that $\xi$ is $L_t$-equivariant. Similar argument shows that $\widetilde \xi$ is also $L_t$-equivariant. This finishes the proof of the lemma. \end{proof}

\subsection{The commutative diagram \eqref{thediag0022000}}

%We continue to assume that $V^t$ is independent of $t\in \mathrm{Df}(\p,\n)$. 
As in Section \ref{secpint}, suppose that $V$ has a $\mathsf G(\Q)$-stable $\sf E$-form  $\sf V$. Let $\mathscr H$ be an $\sf E$-vector space that fits into a commutative diagram  
%Denote by $\oH_{\Phi,P}^{i,\mathrm{ord}}({\mathsf G},  \sf V)$ the representation of $\mathsf G^{\natural,p}\times L$ that fits to a Cartetian diagram
\[
 \begin{CD}
\mathscr H @>>>\oH_{\Phi,P}^{i,\mathrm{ord}}({\mathsf G},  V)\\
      @VVV      @V V \subset    V         \\
\widehat{{\oH}_\Phi^i({\mathsf G},  \mathsf V)}_{ \p\textrm{-sm}} \otimes \RD(\g/\p)@>\iota_{\mathsf V}  >>     \widehat{{\oH}_\Phi^i({\mathsf G}, V)}_{ \p\textrm{-sm}}\otimes \RD(\g/\p).\\
            \end{CD}
\]
Then by using the commutative diagrams \eqref{thediag0011} and 
\[
 \begin{CD}
\widetilde{\oH}_\Phi^i({\mathsf G},  V^t)^{\la \p\supset \s\ra,\circ} @<\textrm{pull-back}<<\widetilde{\oH}_\Phi^i({\mathsf G},  V^t)^{\la \p\supset \p\ra,\circ} \\
      @VVV         @VV V\\
\widetilde{\oH}_\Phi^i({\mathsf G},  V)^{\la \p\supset \s\ra,\circ} @<\textrm{pull-back}<<\widetilde{\oH}_\Phi^i({\mathsf G},  V)^{\la \p\supset \p\ra,\circ},\\
            \end{CD}
\] 
$\mathscr H$ obviously fits into the commutative diagram \eqref{thediag0022000}.

\section{Application I: Rankin-Selberg %$p$-adic 
L-functions for $\GL_n\times \GL_{n-1}$} \label{sec:rs}

In this section we retain the setup in Section \ref{ssec:RSGL} and construct the $p$-adic L-function $\mathscr L_\Pi$ in Theorem \ref{padicLrs0} following the general formalism in Sections \ref{ssec:CLR}--\ref{ssec:PL}. Then we determine the exceptional zeros of $\mathscr L_\Pi$.

\subsection{Rankin-Selberg integrals} \label{ssec:RSI}
Let $\Pi=\Pi_n\boxtimes \Pi_{n-1}$ be an irreducible representation of ${\sf G}(\A) = \GL_n(\A_\rk)\times \GL_{n-1}(\A_\rk)$ ($n\geq 2$) that is realized as a space of smooth automorphic forms on $\mathsf G(\Q)\backslash \mathsf G(\A)$. % occurs as a subrepresentation of the space of smooth automorphic forms  globally generic automorphic reresentation of  ${\sf G}(\A) = \GL_n(\A_\rk)\times \GL_{n-1}(\A_\rk)$, where 
Assume that $\Pi_n$ is cuspidal. %isobaric, cohomological and globally generic. 
Let $\chi: \rk^\times\bs\A_\rk^\times\to \C^\times$ be a Hecke character. We have the global Rankin-Selberg integral 
\[
\CP_\chi: \Pi\otimes \RM(\dot{\sf G}(\Q)\bs \dot{\sf G}(\A))\to\chi^{-1}, \quad f\otimes \tau\mapsto \int_{\dot{\sf G}(\Q)\bs \dot{\sf G}(\A)}\chi(\det g)f(g)\od\!\tau(g),
\]
which converges absolutely. In this example we take ${\sf H}:={\sf U}$ to be the upper triangular maximal unipotent subgroup of ${\sf G}$, and define 
\[
\psi_{\sf U}: {\sf U}(\Q)\bs {\sf U}(\A) \to\C^\times, \quad
\left([u_{i,j}], [u'_{k,l}]\right) \mapsto \psi\left(\sum^{n-1}_{i=1} u_{i, i+1}- \sum^{n-2}_{k=1}u'_{k, k+1}\right),
\]
where $\psi$ is as in \eqref{psi}. Then \eqref{assh} holds. 
Assume that $\Pi$ is globally generic, namely the integrals in \eqref{whit} yield a nonzero functional 
\[
\lambda_{\sf U}\in \Hom_{{\sf U}(\A)}(\Pi,\psi_{\sf U}).
\]
The latter space is known to be at most one-dimensional, and the functional $\lambda_{\sf U}$ is called the Whittaker period. 

At this point, it will be more familiar to switch the notation and work over the number field $\rk$. Let $\RU$ be the upper triangular maximal unipotent subgroup of  
\[
\RG:=\GL_n\times \GL_{n-1} \quad (\text{so that ${\sf G}=\Res_{\rk/\Q}\RG$ and ${\sf U}=\Res_{\rk/\Q}\RU$}). 
\]
We also have $\dot\RG:=\GL_{n-1}$ diagonally embedded into $\RG$. For every place $v$ of $\rk$, write $\RU_v:=\RU(\rk_v)$, $\RG_v:=\RG(\rk_v)$,  $\dot \RG_v:=\dot \RG(\rk_v)$, etc. Accordingly we have the decompositions 
\[
\Pi = \widehat\otimes'_v \Pi_v:=\left(\widehat\otimes_{v\mid\infty} \Pi_v\right)\otimes \left(\otimes'_{v\nmid \infty} \Pi_v\right), \quad \psi_{\sf U} = \otimes_v \psi_{\RU_v}, \quad \textrm{and}\quad \lambda_{\sf U}= \otimes_v \lambda_{\RU_v}, 
\]
where $\widehat\otimes$ stands for the completed projective tensor product. Similar notations will be used without explanation. 

Likewise  let $\CX_v$ be the group of complex continuous characters of $\rk_v^\times$ for every place $v$ of $\rk$, and  for a locally constant character $\varepsilon_v: \CO_v^\times\to \overline{\Q}^\times$ when $v$ is finite, let $\CX_v(\varepsilon_v)\subset \CX_v$ be the subset of characters whose restriction to $\CO_v^\times$ equals $\varepsilon_v$.

The normalized Rankin-Selberg integral 
\[
\begin{aligned}
\CP^\circ_v: \CX_v\times \left(\Pi_v\otimes \RM(\dot\RU_v\bs \dot\RG_v)\right) & \to \C, \\
(\chi_v', f\otimes \tau)& \mapsto \frac{1}{\oL(\frac{1}{2}, \Pi_v\times\chi_v')}\int_{\dot\RU_v\bs \dot\RG_v}
\chi_v'(\det g)\langle \lambda_{\RU_v}, g.f\rangle\od\!\tau(g)
\end{aligned}
\]
is defined by holomorphic continuation. Write $\chi =\otimes_v\chi_v$.  By \cite{JPSS83, J09} and \eqref{IDmeas} we have that 
\[
\CP_\chi = \oL(\frac{1}{2}, \Pi\times \chi) \cdot \otimes_v \CP^\circ_v(\chi_v,\cdot).
\]
 
%The integrals $\CP_v^\circ$ satisfy the same properties as those of \eqref{nz}. In particular there is a family \be \label{familyrs} \{\phi_v^\circ \in \Pi_v\otimes \RD(\dot\RU_v\bs \dot\RG_v)\}_{v\nmid \infty}\eesuch that for all $v\nmid \infty$, $\CP^\circ_v(\cdot, \phi^\circ_v)$ takes a nonzero constant value $(\Omega_{\Pi_v}(\varepsilon_v))^{-1}$ on $\CX(\varepsilon_v)$, and  for all but finitely many $v\nmid\infty$, $\Omega_{\Pi_v}(\varepsilon_v)=1$ and $\phi_v^\circ$ is the fixed spherical vector used in the restricted tensor product $\Pi\otimes \RD(\dot\RU(\A_\rk)\bs \dot\RG(\A_\rk)) = \otimes'_v(\Pi_v\otimes\RD(\dot\RU_v\bs \dot\RG_v)) $. Let us explain this family and the constant  $\Omega_{\Pi_v}(\varepsilon_v)$, as well as the rational structure of $\Pi_v$  following \cite[Section 5.1]{LLS24}. 

In the rest of this section, assume that $\Pi$ is regular algebraic,  $\Pi_{n-1}$ is tamely isobaric as in \cite{LLS24}, and $\chi$ is algebraic.  Let $v$ be a finite place of $\rk$. Let $\ell$ be the residue characteristic of $\rk_v$, and let $\mu_{\ell^\infty}\subset \C^\times$ be the subgroup of $\ell$-power roots of unity. The cyclotomic character at $\ell$ is given by 
\[
\Aut(\C)\xrightarrow{{\rm restriction}} \Aut(\Q(\mu_{\ell^\infty})/\Q) \to \Z_\ell^\times,\quad \sigma\mapsto t_{\sigma, \ell},
\]
such that $\sigma(\zeta)= \zeta^{t_{\sigma, \ell}}$ for all $\zeta\in \mu_{\ell^\infty}$. 
Put 
\be  \label{tl}
{\bf t}_{\sigma, \ell} :=(t_{n, \sigma, \ell}, t_{n-1, \sigma, \ell})\in \RG_v,
\ee
where 
\[
t_{m, \sigma, \ell}: = \diag(t_{\sigma, \ell}^{-(m-1)}, \ldots, t_{\sigma, \ell}^{-1}, 1)\in \GL_m(\rk_v)\qquad (m=n,n-1).
\]
Define an action of $\Aut(\C)$ on  $\Ind^{\RG_v}_{\RU_v} \psi_{\RU_v}$ (the smooth induction) by
\be \label{autc}
{}^\sigma \varphi(g) := \sigma(\varphi({\bf t}_{\sigma,\ell} \cdot g)),\quad \varphi\in \Ind^{\RG_v}_{\RU_v} \psi_{\RU_v}, \  g\in \RG_v.
\ee
This action gives  the $\Q(\Pi_v)$-form  of $\Pi_v \hookrightarrow \Ind^{\RG_v}_{\RU_v} \psi_{\RU_v}$ (see \cite[Page 594]{Mah05}). 
%Let $\chi_{\Pi_{n-1,v}}$ denote the central character of $\Pi_{n-1, v}$. 
The following result is a consequence of \cite[Proposition 5.1]{LLS24}.

\begin{prpl} \label{napr}
For every finite place $v$ of $\rk$, the linear functional
\[
{\mathscr G}_{\psi_v}(\chi_v)^{\frac{n(n-1)}{2}}\cdot {\mathscr G}_{\psi_v}(\chi_{\Pi_{n-1,v}})\cdot \CP_v^\circ(\chi_v,\cdot): \Pi_v\otimes \RM(\dot\RU_v\bs \dot\RG_v)\rightarrow \C
\]
is defined over $\Q(\Pi_v, \chi_v)$, where  ${\mathscr G}_{\psi_v}$ denotes the Gauss sum of a character as in \eqref{df:gauss}.
\end{prpl}

Here $\Q(\Pi_v,\chi_v)$ denotes the composition of the rationality fields $\Q(\Pi_v)$ and $\Q(\chi_v)$ which are number fields (see \cite[3.1]{Cl90}), and similar notation will be used without explanation. Given a character $\omega_v: \rk_v^\times\rightarrow \C^\times$, the usual definition of Gauss sum (see \cite{LLS24}) implicitly depends on a generator $y_v$ of the fractional ideal $\frak{d}_v^{-1} \cdot \c(\omega_v)^{-1}$ of $\CO_v$, where $\c(\omega_v)$ is the conductor of $\omega_v$. We make a slight modification and put
\be \label{df:gauss}
{\mathscr G}_{\psi_v}(\omega_v) := \int_{\CO_v^\times} \omega_v^{-1}(y_v x_v)\cdot \psi_v(y_v x_v) \od\! x_v,
\ee
where $\od\!x_v$ is the normalized Haar measure such that $\CO_v^\times$ has total volume 1. Then it is independent of the choice of $y_v$ and satisfies the property that
\be \label{gauss}
\sigma({\mathscr G}_{\psi_v}(\omega_v)) = {}^\sigma\omega_v(t_{\sigma,\ell})\cdot {\mathscr G}_{\psi_v}({}^\sigma\omega_v),\quad \sigma \in \Aut(\C).
\ee
\begin{comment}
As is well-known (see \cite[(3.2.6)]{T79}),  ${\mathscr G}_{\psi_v}(\omega_v)$ coincides with $\varepsilon_1(0, \omega_v, \psi_v)$ up to a scalar in $\Q(\omega_v)^\times$, where 
$\varepsilon_1(s, \omega_v, \psi_v)$ is the local $\varepsilon$-factor defined in  \cite[(3.6.10)]{T79} using the  measure on 
$\rk_v$ for which $\CO_v$ has volume 1.   Thus Proposition \ref{napr} also holds with the Gauss sums replaced by local $\varepsilon$-factors. The local $\varepsilon$-factors will also appear in the modifying   factors at $p$ given in Section \ref{ssec:OOI}.
\end{comment}

As before suppose that $\Q(\Pi)\subset {\sf E}$. The $\Q(\Pi_v)$-form  of $\Pi_v$ induces an $\sf E$-form  of $\Pi_v$, to be denoted by $\Pi_v(\sf E)$.  Put
\[
\Omega_{\Pi_v}(\varepsilon_v):= {\mathscr G}_{\psi_v}(\chi_v')^{\frac{n(n-1)}{2}}\cdot {\mathscr G}_{\psi_v}(\chi_{\Pi_{n-1,v}})
\]
for an arbitrary $\chi_v'\in \CX(\varepsilon_v)$, which is clearly well-defined. 
%Then the desired family \eqref{familyrs} exists by Proposition \ref{napr} and the theory of Rankin-Selberg integrals \cite{JPSS83}. %, there exists  $\phi^\circ_v\in \Pi_v({\sf E})\otimes \RD(\dot\RU_v \bs \dot\RG_v)$ as in \eqref{familyrs} such that $\CP^\circ_v(\cdot, \phi^\circ_v)$ equals the constant $(\Omega_{\Pi_v}(\varepsilon_v))^{-1}$ on $\CX(\varepsilon_v)$.
By Proposition \ref{napr} and the theory of Rankin-Selberg integrals \cite{JPSS83}, there is a family
\be \label{familyrs}
\{\phi_v^\circ \in \Pi_v\otimes \RD(\dot\RU_v\bs \dot\RG_v)\}_{v\nmid \infty}
\ee
such that 
\begin{itemize}
    \item for all $v \nmid \infty$, $\phi^\circ_v\in \Pi_v({\sf E})\otimes \RD(\dot\RU_v \bs \dot\RG_v)$ and $\CP^\circ_v(\cdot, \phi^\circ_v)$ takes the nonzero constant value $(\Omega_{\Pi_v}(\varepsilon_v))^{-1}$ on $\CX(\varepsilon_v)$;
    \item for all but finitely many $v\nmid\infty$, $\Omega_{\Pi_v}(\varepsilon_v)=1$ and $\phi_v^\circ$ is the fixed spherical vector used in the restricted tensor product $\Pi\otimes \RM(\dot\RU(\A_\rk)\bs \dot\RG(\A_\rk)) = \widehat\otimes'_v(\Pi_v\otimes\RM(\dot\RU_v\bs \dot\RG_v))$.
\end{itemize}

For a Hecke character $\omega=\otimes_v\omega_v: \rk^\times \bs \A_\rk^\times \to \C^\times$, define its Gauss sum outside $p$ by
\be \label{gaussp}
\mathscr G_\psi(\omega^{(p)}):  = \prod_{v\nmid \infty p} \mathscr G_{\psi_v}(\omega_v).
\ee

\subsection{Archimedean modular symbols} \label{ssec:AMSM}

Take $K_\infty=\mathsf A(\R)\cdot K_\infty'\subset \mathsf G(\R)$, where $\mathsf A=\GL(1)/_{\Q}\times \GL(1)/_{\Q}$ is the largest central split  torus in $\mathsf G$ and $K_\infty'$ is the standard maximal compact subgroup (which is a product of orthogonal groups and unitary groups). Take $\dot K_\infty':=K^\infty\cap \dot{\mathsf G}(\R)$, which is the standard maximal compact subgroup of  $\dot{\sf G}(\R)$. Recall that $\Q(\Pi)\subset {\sf E}$. Then by \cite{Cl90} there is a geometrically irreducible algebraic representation ${\sf F}_\mu\boxtimes {\sf F}_\nu$ of ${\sf G}_{\sf E}$ %, whose complexification is called  the coefficient system of $\Pi$,  
such that the total relative Lie algebra cohomology 
\[
\oH^\bullet(\g_\C, K_\infty^\circ; ({\sf F}_\mu^\vee\boxtimes {\sf F}_\nu^\vee)\otimes \Pi_\infty)\neq \{0\},
\]
where $(\mu, \nu)\in (\Z^n)^{\CE_\rk}\times (\Z^{n-1})^{\CE_\rk}$ is as in Section \ref{ssec:RSGL}.  See \cite{LLS24} for more details. 
Suppose that  ${\sf V}={\sf F}_\mu^\vee\boxtimes {\sf F}_\nu^\vee$ (which is also viewed as  a representation of ${\sf G}(\Q)$), and that $V=E\otimes({\sf F}_\mu^\vee\boxtimes {\sf F}_\nu^\vee)$ as a representation of $G={\sf G}(\Q_p)\subset \mathsf G_\mathsf E(E)$.

Recall the ${\sf G}^\natural$-homomorphism \eqref{embcoh}, where $\Phi$ is now the family of closed subsets of ${\sf G}(\Q)\bs \mathscr X_{{\sf G}, K_\infty}$ whose obvious projections to ${\sf G}_n(\Q) \bs \mathscr X_{{\sf G}_n, K_{n,\infty}}$ are relatively compact ($K_{n,\infty}$ is the projection of $K_\infty$ to ${\sf G}_n(\R)$). Then it satisfies \eqref{precompact2}
and \eqref{conphi} (see Proposition \ref{prop:cohfin} and  the K\"unneth formula \cite[IV. Theorem 7.6]{Br97}).
 The bottom degree component  $\Pi_\infty'$ of $\oH^\bullet(\g_\C, K_\infty^\circ; {\sf V}\otimes \Pi_\infty)$ and its ${\sf E}$-form are as in \eqref{Piinf'}, with $i_0 =  \dim (\dot{\sf G}(\R)/\dot K_\infty^\circ)$.
As a representation of the component group $\dot K_\infty^\natural =  \rk_\infty^{\times, \natural} $, there is a multiplicity-free decomposition 
\[
\Pi_\infty' \cong  \bigoplus_{\varepsilon_\infty \in \widehat{\rk_\infty^{\times, \natural}}}\varepsilon_\infty.
\]

 \begin{dfnl} \label{critrs}
(a)  A character $\chi_\infty$ of $\rk_\infty^\times$ is said to be critical for $\Pi$ if it is algebraic and $s=\frac{1}{2}$ is a pole of neither $\oL(s, \Pi_\infty\times\chi_\infty)$ nor $\oL(1-s,  \Pi_\infty^\vee \times \chi_\infty^{-1})$.

\noindent (b) A  Hecke character $\chi$ of $\rk^\times\bs\BA_\rk^\times$ is said to be critical for $\Pi$ if so is its archimedean component $\chi_\infty$.
\end{dfnl}

It is known that all  $\sf V$-balanced characters of $\rk_\infty^\times$ are critical for $\Pi$. 
Assume that $\chi_\infty$ is algebraic of weight ${\sf w}^{-1}$.
 By the  well-known branching rule for algebraic representations of general linear groups, 
 ${\sf w}=\prod_{\iota\in\CE_\rk}\iota^{{\sf w}_\iota}$ (hence $\chi_\infty$) is $\sf V$-balanced if and only if   (\cf \cite{KS13, Rag16})
\be \label{eq:bp}
 \mu^\iota_{i+1}+\nu^\iota_{n-i} \leq  {\sf w}_\iota \leq \mu^\iota_{i} +\nu^\iota_{n-i} \quad \text{for all $i=1,2,\dots, n-1$ and $\iota\in \CE_\rk$}.
\ee

%In rest of this subsection, assume that ${\sf w}$ is ${\sf V}$-balanced. We first specify the parabolic subgroup $P$ of $G$, and the elements $\lambda_{\sf V, w}\in \Hom_{\dot{\sf G}_{\sf E}}({\sf E_w}\otimes {\sf V}, {\sf E})$ and $\lambda_0\in \Hom_E(V^\n, E)$ such that $\lambda_0 = \lambda_{\sf V, w}|_{V^\n}$.

Following \cite[Section 1.3]{LLSS23}, define a family  $\{z_m\in \GL_m(\BZ)\}_{m\in \BN}$ of matrices
	inductively by
	\[
	z_0:=\emptyset\ \  (\textrm{the unique element of $\GL_0(\BZ)$}), \quad  z_1:=[1]\quad \text{and}
	\]
	\[
		z_m :=
		\begin{bmatrix}
			w_{m-1}& 0 \\
			0 & 1 
		\end{bmatrix}
		\begin{bmatrix}
			z_{m-2}^{-1}& 0 \\
			0 & 1_2 
		\end{bmatrix}
		\begin{bmatrix}
			{}^tz_{m-1}  w_{m-1} z_{m-1}& {}^t e_{m-1}\\
			0 & 1 \\
		\end{bmatrix}, \quad m\geq 2,
	\]
	where $w_m : = \left[\begin{smallmatrix}  &  & 1 \\  & \begin{sideways} $\ddots$ \end{sideways} & \\  1 & & \end{smallmatrix}\right]$ denotes  the 
    $m\times m$ anti-diagonal permutation matrix, $e_{m-1}:=[0 \ \cdots \  0 \ 1]\in \BZ^{1\times(m-1)}$ is a row vector, and ${}^t g$ denotes the transpose of a matrix $g$. Put
\be \label{eletz}
z:=(z_n, z_{n-1})\in \GL_n(\Z)\times \GL_{n-1}(\Z) \hookrightarrow {\sf G}(\Q).
\ee
Let $\overline{\sf B}$ be the Borel subgroup of lower triangular matrices in ${\sf G}$, with unipotent radical $\overline{\sf U}$. By \cite[Lemma 1.1]{LLSS23}, $\overline{\sf B} z \dot{\sf G}$ is Zariski open in $\dot{\sf G}$, and in fact 
\be \label{transP}
{\sf P}\cap \dot{\sf G}=\{1\},\quad \textrm{where}\quad {\sf P}:=z^{-1} \overline{\sf B} z.
\ee
Let ${\sf B}$ be the Borel subgroup of upper  triangular matrices in ${\sf G}$. By using algebraic induction from ${\sf B}$ as in \cite{LLS24}, we realize ${\sf V}$ as a space of algebraic functions on ${\sf G}_{\sf E}$. Let 
${\sf v} \in {\sf V^{\overline{\sf U}}}$
be the unique  algebraic function in ${\sf V}$ which equals 1 on $\overline{\sf U}$.

Suppose that $P ={\sf P}(\Q_p)$ so that $N$ is its unipotent radical, and that $\lambda_0 \in  \Hom_E(V^\n, E)$ is the unique generator defined over ${\sf E}$ such that $\langle \lambda_0, z^{-1}.{\sf v} \rangle =1$. Recall that  ${\sf Z}={\sf G}_1$. 
For every   $\sf V$-balanced algebraic character $\sf w$ of $\mathsf Z_{\sf E}$, let $\lambda_{\sf V, w}\in \Hom_{\dot{\sf G}(\Q)}({\sf E_w}\otimes{\sf V}, {\sf E})$ be as in Lemma \ref{lambda01}.  
Then  for all algebraic characters $\chi_\infty$ of $\rk_\infty^\times$ of weight ${\sf w}^{-1}$, 
we have the archimedean modular symbol map
\[
\widehat{\mathcal P}_\infty^{\lambda_{\sf V, w}} : \oH^0(\z_\C,K_{\mathsf Z,\infty}^\circ; \C_{\mathsf w_\infty}\otimes \chi_\infty) \times\left(\Pi_\infty' \otimes \mathrm M 
 (\dot{\mathsf U}(\R))^\vee\otimes \mathrm{O}(\dot{\mathsf G}(\R)/\dot{K}_\infty^\circ)\right) \to \C
\]
defined as in \eqref{AMS}.

Let $\Pi_{0,\infty}$ be the irreducible %regular algebraic 
tempered Casselman-Wallach  representation 
of ${\sf G}(\R)$ whose infinitesimal character equals that of the trivial representation and %corresponding to the case that $\mu$ and $\nu$ are zero, 
whose central character equals that of 
$({\sf F}_\mu^\vee\boxtimes {\sf F}_\nu^\vee)\otimes \Pi_\infty$. Define the cohomology group $\Pi_{0,\infty}'$ as in \eqref{Piinf'}. Let $\lambda_0'\in \Hom_{\sf U}({\sf V}, {\sf E})$ be the generator such that $\langle\lambda_0', {\sf v}\rangle =1$. Following 
\cite{LLS24}, with the fixed Whittaker functionals and $\lambda_0'$ we have the translation map 
\[
 \jmath_{\mu}\otimes \jmath_{\nu}: \Pi_{0,\infty}' \to \Pi_{\infty}'
\]
which is a $\rk_\infty^{\times,\natural}$-equivariant isomorphism. 
As a specialization of \eqref{AMS}, we have a map 
\[
 \widehat\CP_\infty: \oH^0(\z_\C,K_{\mathsf Z,\infty}^\circ; \varepsilon_\infty) \times\left(\Pi_{0,\infty}' \otimes \mathrm M 
 (\dot{\mathsf U}(\R))^\vee\otimes \mathrm{O}(\dot{\mathsf G}(\R)/\dot{K}_\infty^\circ)\right) \to \C
\]
for every $\varepsilon_\infty \in \widehat{\rk_\infty^{\times,\natural}}$. Define  $\widehat{\CP}^\circ_\infty$ in \eqref{NAMS} to be the map such that for every $\varepsilon_\infty \in \widehat{\rk_\infty^{\times,\natural}}$, $\widehat{\CP}^\circ_\infty(\varepsilon_\infty, \,\cdot\,)$ equals  the composition of 
\begin{eqnarray*}
&&  \Pi_{\infty}' \otimes\mathrm M 
 (\dot{\mathsf U}(\R))^\vee\otimes \mathrm{O}(\dot{\mathsf G}(\R)/\dot{K}_\infty^\circ) \\
& \xrightarrow{(1, \, (\jmath_\mu\otimes \jmath_\nu)^{-1}\otimes {\rm id}\otimes {\rm id})} & %\bigoplus_{\varepsilon_\infty\in \widehat{\rk_\infty^{\times,\natural}}}
\oH^0(\z_\C,K_{\mathsf Z,\infty}^\circ; \varepsilon_\infty) \times \left(\Pi_{0, \infty}' \otimes\mathrm M (\dot{\mathsf U}(\R))^\vee\otimes \mathrm{O}(\dot{\mathsf G}(\R)/\dot{K}_\infty^\circ)\right)  \\
 & \xrightarrow{\widehat\CP_\infty}& \C.
 \end{eqnarray*}
% where in the first arrow  $1\in \C = \varepsilon_\infty = \oH^0(\z_\C,K_{\mathsf Z,\infty}^\circ; \varepsilon_\infty)$ (similar notation will be used without explanation).
 Here and as usual, $ \rm{id}$ denotes the identity map.

We have the following archimedean nonvanishing hypothesis and period relations, which are proved in \cite{Sun17, LLS24} for the essentially tempered case, and extended to all the 
generic cohomological cases in \cite{JLS24} after a suggestion of Michael Harris.

\begin{thml}\label{thm:sun} 
 (a) There is an element \[ \widehat \phi^\circ_\infty\in \Pi_\infty'({\sf E})\otimes  \RD  
 (\dot{\mathsf U}(\R))^\vee\otimes \mathrm{O}(\dot{\mathsf G}(\R)/\dot{K}_\infty^\circ)
\]
such that $\widehat\CP_\infty^\circ(\varepsilon_\infty, \widehat\phi^\circ_\infty)\neq 0$ for all $\varepsilon_\infty\in \widehat{\rk_\infty^{\times,\natural}}$. 

\noindent (b)  Let $\lambda_{\sf V, w}\in \Hom_{\dot{\sf G}(\Q)}({\sf E_w}\otimes{\sf V}, {\sf E})$ be as in Lemma \ref{lambda01}, where $\sf w$ is a  $\sf V$-balanced algebraic character  of $\mathsf Z_{\sf E}$. Then
\[\widehat\CP^{\lambda_{\sf V, w}}_\infty(1, \,\cdot\,) = \Upsilon_{\Pi_\infty'}(\chi_\infty)\cdot \widehat\CP_\infty^\circ({\sf w}_\infty \chi_\infty, \,\cdot\,)
\]
for all algebraic characters $\chi_\infty$ of $\rk_\infty^\times$ of weight ${\sf w}^{-1}$, 
 where $\Upsilon_{\Pi_\infty'}(\chi_\infty)$ is in \eqref{MFinfrs}.
\end{thml}

Following Definition \ref{df:period}, define the Whittaker periods
\be\label{whit-per}
\Omega_{\Pi}(\varepsilon_\infty):=\left(\widehat{\mathcal P}_\infty^\circ(\varepsilon_\infty,  \widehat \phi_\infty^\circ)\right)^{-1},\quad \varepsilon_\infty\in \widehat{\rk_\infty^{\times,\natural}}.
\ee
%which are also called the Whittaker periods of $\Pi$.

%\subsection{Nearly ordinary refinement} \label{ssec:NOR}

\subsection{Open orbit integrals and normalized refined period} \label{ssec:OOI}

%Recall  the Borel subgroup $\mathsf P\subset \mathsf G$ from \eqref{transP}, and that $P=\mathsf P(\Q_p)= z^{-1}\overline{B} z$, where $\overline{B}=\overline{\sf B}(\Q_p)$. 
Assume that $\Pi_p'\subset\CB_P(\Pi_p)$ is a nearly ordinary refinement of $\Pi_p$ defined over ${\sf E}$, which 
is also viewed as a character of $P$ that descends to a character of the torus $L=P/N$.

For every locally compact Hausdorff topological group $\mathcal G$, write $\delta_\CG:\CG\rightarrow \C^\times$ for its modular  character.  We use $\Ind$ to denote the normalized smooth induction, and still use a superscript $\,^\vee$ to indicate the contragredient  representations  of admissible smooth representations of totally disconnected groups and 
Casselman-Wallach representations of real reductive groups.
%(or some other groups when no confusion is possible). 

\begin{leml} \label{fdim1}
One has that
$
  \dim \Hom_L( \Pi_p', \CB_P(\Pi_p))= 1.
$
\end{leml}

\begin{proof}
We have that 
\begin{eqnarray*}
 \Hom_L( \Pi_p', \CB_P(\Pi_p))&=& \Hom_P(\Pi_p^\vee, \Pi_p'^\vee)\\
 &=&\Hom_G(\Pi_p^\vee, \mathrm{Ind}^G_P \,(\Pi_p'^\vee \otimes \delta_P^{-1/2}))\\
 &=&\Hom_G( \mathrm{Ind}^G_P \,(\Pi_p' \otimes \delta_P^{1/2}),\Pi_p).
\end{eqnarray*}
This is at most one-dimensional by the uniqueness of the Whittaker functionals on the principal series representations. %since $\Pi_p$ is irreducible generic. 
%By a form of the second adjointness theorem (see \cite[Proposition 6.5]{Be87}), \[ \Hom_T( \Pi_p', \CB_P(\Pi_p))  \cong   \Hom_{G}(\textrm{Ind}^G_P \,(\Pi_p'\cdot\delta_P^{1/2}), \Pi_p),\]which has dimension at most one  because $\Pi_p$ is irreducible generic. 
\end{proof}

%Here and henceforth, 
By the proof of Lemma \ref{fdim1},  $\Pi_p$ is isomorphic to a quotient representation of 
$
\textrm{Ind}^G_P \,(\Pi_p'\cdot\delta_P^{1/2}).
$
Note that $P = z^{-1}\overline{B}z$, where $\overline{B} := 
\overline{\sf B}(\Q_p)$. Define a character
\[
\kappa:= \Pi_p' \circ {\rm Ad}(z^{-1}): \overline{B}\to {\sf E}^\times,
\]
so that $\Pi_p$ is also isomorphic to a quotient representation of 
\[
I(\tilde\kappa):=\Ind^G_{\overline{B}}(\tilde\kappa),\quad \text{where}\quad \tilde\kappa:=\kappa\otimes\delta_{\overline B}^{1/2}.
\]

The most technical input for the evaluation of the modifying factors at $p$ and $\infty$ is an application of the preparatory result in \cite{LLSS23} to be recalled below. Let $v$ be a finite place of $\rk$. The result in {\it loc. cit.}  compares the normalized Rankin-Selberg integral $\CP_v^\circ$ with the integral over  the open $\dot \RG_v$-orbit in the flag variety $ \overline{\RB}_v \bs \RG_v$. The same result for the archimedean places has been used in \cite{LLS24, JLS24} to evaluate $\Upsilon_{\Pi_\infty'}(\chi_\infty)$ and prove the archimedean period relations in Theorem \ref{thm:sun}.% (see \cite[Theorem 3.2]{LLS24}).

Write a continuous character $\varrho: \overline{\RB}_v \to \C^\times$ as
\[ % \label{varrho}
\varrho:=(\varrho_1,\ldots, \varrho_n, \varrho_1',\ldots, \varrho_{n-1}') \in (\widehat{\rk_v^\times})^{2n-1}.
\]
Let $I(\varrho):=\Ind^{\RG_v}_{\overline{\RB}_v}\,\varrho$.
By \cite[Theorem 15.4.1]{Wa92}, there is a unique Whittaker functional 
$
\lambda_{\RU_v}' \in \Hom_{\RU_v}(I(\varrho), \psi_{\RU_v})
$
such that 
\be \label{jacint}
\langle \lambda_{\RU_v}', f\rangle  = \int_{\RU_v} f(u) \psi_{\RU_v}^{-1}(u)\od\! u
\ee
for all $f\in I(\varrho)$ such that $f\vert_{\RU_v}\in \CS(\RU_v)$, where we fix the Haar measure $\od\!u$ on $\RU_v$ to be the product of the self-dual Haar measures on $\rk_v$ with respect to $\psi_v$. Here
and henceforth, $\CS(X)$ denotes 
%the space of Schwartz functions on $X$ when $X$ is a Nash manifold, and 
the space of compactly supported locally constant complex functions on $X$ when $X$ is a totally disconnected topological space.

We fix the similar Haar measure on $\dot\RU_v$. For $\tau\in \RM(\dot\RG_v)$, denote by $\bar\tau\in \RM(\dot\RU_v\bs \dot\RG_v)$  the quotient of $\tau$ by the fixed measure on $\dot\RU_v$. We have the unnormalized Rankin-Selberg integral map 
\be \label{Lambdaint1}
\begin{array}{rcl}
\CP_{v}: \CX_v\times \left(I(\varrho)\otimes \RM(\dot\RU_v\bs \dot\RG_v)\right) & \rightarrow &\C\cup\{\infty\},\\
  (\chi_v', f\otimes  \bar\tau) &\mapsto & \int_{ \dot\RU_v\bs \dot\RG_v}\chi_v'(\det g) \langle \lambda_{\RU_v}', g.f\rangle \od\!\bar\tau(g)
\end{array}
\ee
defined by meromorphic continuation of absolutely convergent integrals.

Following \cite{LLSS23}, we 
%Let $I(\varrho)_\sharp$ be the subspace of $f\in I(\varrho)$ such that $ f\vert_{z \dot{\RG}_v}\in \CS(z \dot\RG_v),$ where $z$ is in \eqref{eletz}. 
also have the open orbit  integral map 
\be \label{Lambdaint2}
\begin{array}{rcl}
\Lambda_{v}: \CX_v\times \left(I(\varrho)\otimes \RM(\dot\RG_v)\right)& \rightarrow &\C\cup\{\infty\},\\
   (\chi_v', f  \otimes  \tau) &\mapsto & \int_{ \dot\RG_v} \chi_v'(\det g) f(z g)\od\!\tau(g),
\end{array}
\ee
defined by meromorphic continuation (in variables $\varrho$ and $\chi_v'$) of absolutely convergent integrals. %which converges absolutely and is holomorphic in $\chi_v'$.

%Similar to \eqref{nz2}, the maps \eqref{Lambdaint1} and \eqref{Lambdaint2} naturally extend to maps\[\CP_v: \CX_v\times \left(\widehat{I(\varrho)}_{\p_v\mathrm{-sm}}\otimes \RM(\dot\RU_v\bs \dot\RG_v)\right)\rightarrow \C\cup\{\infty\}\]and\[\Lambda_{v}: \CX_v\times \left(\widehat{I(\varrho)}_{\p_v\mathrm{-sm}}\otimes \RM(\dot\RG_v)\right)\rightarrow \C\cup\{\infty\},\]where $\p_v$ denotes the Lie algebra of $\mathsf P(\rk_v)$. 

Define a meromorphic function on $\CX_v$ by
\[
\Gamma_{\varrho, \psi_v}(\chi_v'): =\prod_{i>j,\, i+j\leq n}(\varrho_i\cdot\varrho'_j\cdot\chi_v')(-1) \cdot \prod_{i+j\leq n}\gamma\left(\frac{1}{2}, \varrho_i\cdot \varrho_j\cdot\chi_v', \psi_v\right).
\]
Here and henceforth
\[
\gamma(s, \omega, \psi_v):= \varepsilon(s, \omega, \psi_v) \cdot \frac{\oL(1-s, \omega^{-1})}{\oL(s, \omega)}
\]
denotes the Tate $\gamma$-factor of a character $\omega\in \widehat{\rk_v^\times}$ defined in \cite{T79, J79, Ku03} using the self-dual Haar measure on 
$\rk_v$ with respect to $\psi_v$. %, under which $\CO_v$ has volume $c_v$ as in \eqref{cond}. 
The following result is implied by \cite[Theorem 1.6 (b)]{LLSS23} and \cite[Corollary 4.3]{LLS24}. %, which holds for all local fields. 

\begin{thml}  \label{thm:llss}
For every $\widehat f\otimes \tau\in \widehat{I(\varrho)}_{\p_v\mathrm{-sm}} \otimes \RM( \dot\RG_v)$ the equality 
\[
\Lambda_{v}(\chi_v', \widehat f\otimes  \tau) = \Gamma_{\varrho, \psi_v}(\chi_v')\cdot \CP_{v}(\chi_v', \widehat f\otimes \bar\tau)
\]
holds as meromorhic functions of $\chi_v'\in \CX_v$.
\end{thml}

Write $\tilde\kappa = \otimes_{\wp\mid p}\tilde\kappa_\wp$. Then we have that $I(\tilde\kappa) = \otimes_{\wp\mid p}I(\tilde\kappa_\wp)$. By tensor product  over $\wp\mid p$, the maps in \eqref{Lambdaint1} yield a map 
 \[
\CP_{p}: \CX_p\times \left(I(\tilde \kappa)\otimes \RM(\dot U\bs \dot G)\right) \rightarrow \C\cup\{\infty\}.
\]
Similar to \eqref{nz2}, it naturally extends to a map 
 \[  %\label{Lambdaint4}
\CP_{p}: \CX_p\times \left(\widehat{I(\tilde \kappa)}_{\p\mathrm{-sm}}\otimes \RM(\dot U\bs \dot G)\right) \rightarrow \C\cup\{\infty\}.
\]
Similarly, the maps in \eqref{Lambdaint2} yield a map 
 \[ % \label{Lambdaint5}
\Lambda_{p}: \CX_p\times \left(\widehat{I(\tilde \kappa)}_{\p\mathrm{-sm}}\otimes \RM( \dot G)\right) \rightarrow \C\cup\{\infty\}. 
\]
%Put $I(\tilde\kappa)_\sharp := \otimes_{\wp\mid p} I(\tilde\kappa_\wp)_\sharp$ and define a function $\Lambda_p$ on $\CX_p \times \left(I(\tilde\kappa)_\sharp\otimes \RM(\dot G)\right)$ such that $\Lambda_p(\chi_p',\cdot) = \otimes_{\wp\mid p}\Lambda_\wp(\chi_\wp',\cdot)$ for all $\chi_p' = \otimes_{\wp\mid p}\chi_\wp'\in\CX_p$. Similarly, define a function $\CP^\circ_p$ on $\CX_p\times \left(\Pi_p\otimes \RM(\dot U\bs \dot G)\right)$ using the normalized Rankin-Selberg integrals $\{\CP_\wp^\circ\}_{\wp\mid p}$. 

Recall the Whittaker functional $\lambda_{\RU_\wp}$ on $\Pi_\wp$.
Write 
\be\label{surjhom0}
\xi_p: I(\tilde\kappa)\rightarrow \Pi_p=\otimes_{\wp\mid p} \Pi_\wp
\ee
for the $G$-homomorphism 
whose composition with $\otimes_{\wp\mid p}\lambda_{\RU_\wp}$ equals $\otimes_{\wp\mid p}\lambda_{\RU_\wp}'$. It is surjective and  naturally extends to a surjective $G$-homomorphism 
\be\label{surjhom}
\xi_p: \widehat{I(\tilde\kappa)}\twoheadrightarrow \widehat{\Pi_p}. 
\ee
Note that $\widehat{I(\tilde\kappa)}$ is naturally identified with  a space of generalized functions on $G$. Denote by $\widehat{I(\tilde\kappa)}_z$ its subspace of the generalized functions supported in  $\overline{ B} z$, where where $z$ is as in \eqref{eletz}. It is easy to see that $\widehat{I(\tilde\kappa)}_z$  is one-dimensional and the map $\xi_p$ in \eqref{surjhom} restricts to a $P$-isomorphism 
\[
  \xi_p: \widehat{I(\tilde\kappa)}_z\xrightarrow{\sim}\Pi_p'. 
\]
Moreover, for every generator $\widehat f\otimes \tau$ of $\widehat{I(\tilde\kappa)}_z\otimes \RM( \dot G)$,  $\Lambda_p(\,\cdot\, ,\widehat f\otimes \tau)$ is a constant function on $\CX_p$ with values in $\C^\times$. 

 Fix the  Haar measure on $\dot U$ to be the product of the measures on $\dot\RU_\wp$ for all $\wp\mid p$, and for 
$\tau\in \RM(\dot G)$ denote by $\bar\tau\in \RM(\dot U \bs \dot G)$ the quotient of $\tau$ by the fixed measure on $\dot U$. 
Now we define the normalized refined period map to be the composition 
\begin{eqnarray*}
   \widehat\CP_p^\circ&:& \CX_p \times \left(\Pi_p'\otimes \RM(\dot U \bs \dot G)\right)\\
   &\xrightarrow{\xi_p^{-1}}& \CX_p \times \left(\widehat{I(\tilde\kappa)}_z\otimes \RM(\dot U \bs \dot G)\right)\\
    &\xrightarrow{\bar \tau\mapsto \tau}& \CX_p \times \left(\widehat{I(\tilde\kappa)}_z\otimes \RM( \dot G)\right)\\
    &\xrightarrow{\Lambda_p}&\C.
\end{eqnarray*}

 \subsection{Rational test vectors and modifying factors at $p$}  \label{ssec:SI} 
%Fix a $\RG_\wp$-equivariant surjective homomorphism  \be\label{surjhom}\xi_\wp: I(\tilde\kappa_\wp)\twoheadrightarrow \Pi_\wp\ee whose composition with $\lambda_{\RU_\wp}$ equals $\lambda_{\RU_\wp}'$.

In what follows we define a rational test vector $\widehat\phi_p^\circ\in \Pi_p'({\sf E})\otimes \RD(\dot U \bs \dot G)$. Recall the $\Q(\Pi_\wp)$-form  of $\Pi_{\wp}$  determined by \eqref{autc}.   Define an  $\Aut(\C/ \Q(\kappa_\wp))$-action on $I(\tilde\kappa_\wp)$ by 
\[
{}^\sigma f(g) := \kappa_\wp({\bf t}_{\sigma, p})\cdot \sigma(f(g)),\quad \sigma\in \Aut(\C/\Q(\kappa_\wp)), \  f\in I(\tilde\kappa_\wp), \ g\in \RG_\wp. 
\]
 By
\eqref{gauss} we have that
\[
\kappa_\wp({\bf t}_{\sigma, p})= \frac{\sigma(\omega_{\psi_\wp}(\kappa_\wp))}{\omega_{\psi_\wp}(\kappa_\wp)},
\]
where $\omega_{\psi_\wp}(\kappa_\wp)$ is in  \eqref{gaussrs}. 
Hence 
\[
I(\tilde\kappa_\wp)^{\Aut(\C/\Q(\kappa_\wp))} =\{
f\in I(\tilde\kappa_\wp) \,:\, \text{$\omega_{\psi_\wp}(\kappa_\wp)\cdot f$ is $\Q(\kappa_\wp)$-valued}\},
\]
which is a $\Q(\kappa_\wp)$-form  of $I(\tilde\kappa_\wp)$.   By tensor products, we have a $\Q(\Pi_p)$-form  of $\Pi_{p}$ as well as a $\Q(\kappa)$-form  on $I(\tilde\kappa)$. 
 
Recall that $c_p = \prod_{\wp\mid p}c_\wp$, where $c_\wp$ as in \eqref{cond} is the volume of $\CO_\wp$ under the self-dual Haar measure on $\rk_\wp$ with respect to $\psi_\wp$. 

%\begin{prpl}   \label{prat}One has that  $\Q(\Pi_\wp)  \subset \Q(\kappa_\wp)$ and the surjective homomorphism $c_\wp^{-(n-1)^2}\xi_\wp: I(\tilde\kappa_\wp)\twoheadrightarrow \Pi_\wp$ is $\Q(\kappa_\wp)$-rational, where $\xi_\wp$ is in \eqref{surjhom}. \end{prpl}

\begin{prpl}   \label{prat}
One has that  $\Q(\Pi_p)  \subset \Q(\kappa)$ and the surjective homomorphism $c_p^{-(n-1)^2}\xi_p: I(\tilde\kappa)\twoheadrightarrow \Pi_p$ is $\Q(\kappa)$-rational, where $\xi_p$ is in \eqref{surjhom0}.
\end{prpl}
\begin{proof} 
From  \eqref{autc}, it suffices to show that for every $\sigma\in \Aut(\C/\Q(\kappa_\wp))$
and $f\in I(\tilde\kappa_\wp)$ it holds that 
\[
 c_\wp^{-(n-1)^2} \cdot \langle\lambda_{\RU_\wp}', {}^\sigma f\rangle = \sigma( c_\wp^{-(n-1)^2} \cdot\langle  \lambda_{\RU_\wp}', {\bf t}_{\sigma, p}.f\rangle).
\]
We may assume that $f|_{\RU_\wp}\in \CS(\RU_\wp)$. Then the above equality can be verified directly using the 
 Jacquet integral \eqref{jacint}.
\end{proof}

Let $\widehat\phi_p$ denote the generator  of $\Pi_p'\otimes \RM(\dot U \bs \dot G)$
such that  $\widehat\CP_p^\circ(\,\cdot\, ,\widehat\phi_p)$ equals the constant function $1$ on $\CX_p$. Write $\widehat\phi_p = \phi_p' \otimes \bar \tau$ such that $\phi_p'\in \Pi_p'$
and $\tau\in \RD(\dot G)$. 
Proposition \ref{prat} implies that
\[
c_p^{-(n-1)^2}\omega_{\psi_p}(\kappa)^{-1}\cdot \phi_p' \in \Pi_p'({\sf E}). 
\]
It is clear that 
$c_p^{(n-1)(n-2)/2}\cdot \bar\tau\in \RD(\dot U\bs \dot G)$. Define
\be \label{ptestrs}
\begin{aligned}
\widehat\phi_p^\circ: \, & = (c_p^{-(n-1)^2}\omega_{\psi_p}(\kappa)^{-1}\cdot\phi_p' )\otimes (c_p^{(n-1)(n-2)/2}\cdot \bar\tau) \\
& = \Omega_{\Pi_p'}^{-1}\cdot \widehat\phi_p \in \Pi_p'({\sf E})\otimes \RD(\dot U\bs \dot G),
\end{aligned}
\ee
where $\Omega_{\Pi_p'}$ is in \eqref{omegapi'}. Then 
\[
\widehat\CP_p^\circ(\chi_p', \widehat\phi_p^\circ) = \Omega_{\Pi_p'}^{-1}\quad \text{for all }\chi_p'\in \CX_p.
\]
%By tensor product over $\wp\mid p$, the normalized Rankin-Selbert integrals $\CP_\wp^\circ$ yield a function $\CP_p^\circ$ on $\CX_p\times \left(\Pi_p\otimes \RM(\dot U\bs \dot G)\right)$ as in \eqref{nz}, which naturally extends to a map $\CP_p^\circ$  in \eqref{nz2}.
From \eqref{MFprs} it is clear that
\[
\Upsilon_{\Pi_p'}(\chi_p') =  \frac{1}{ \prod_{\wp\mid p}\left(\oL(\frac{1}{2}, \Pi_\wp\times\chi_\wp') \cdot\Gamma_{\tilde\kappa_\wp, \psi_\wp}(\chi_\wp')\right)},\quad \chi_p' = \otimes_{\wp\mid p}\chi_\wp'\in\CX_p.
\]
Then the following is a direct consequence of Theorem \ref{thm:llss}.

\begin{comment}
\begin{corl} \label{pfin}
For every $f\otimes \tau\in I(\tilde\kappa)_\sharp \otimes \RM(\dot G)$, the equality
\[
\Upsilon_{\Pi_p'}(\chi'_p)\cdot \Lambda_p(\chi'_p, f\otimes\tau) =\CP_p^\circ(\chi'_p, f\otimes\bar\tau) 
\]
holds as holomorphic functions of $\chi'_p\in \CX_p$. 
\end{corl}

Now  we define the normalized refined period
\[
\widehat\CP_p^\circ: \CX_p \times \left(\Pi_p'\otimes \RM(\dot U \bs \dot G)\right)\to\C.
\]
Fix a generator $\tau_p\in \RD(\dot G)$.  Define an element
\[
\widehat f=\{\widehat f_\fG\}_{\fG} \in \CB_P(I(\tilde\kappa))
\]
such that for all  open compact subgroups $\fG$ of $G$ satisfying that $\fG=(\fG\cap P)\dot \fG$ where $\dot \fG:=\fG\cap \dot G$, 
$\widehat f_{\fG} $ is the unique element of $I(\tilde\kappa)_\sharp$ supported on $\overline B z \dot\fG$ whose restriction to $z \dot \fG$ takes the constant value $\tau_p(\dot\fG)^{-1}$.  It is clear that for every $\chi_p'\in \CX_p$, when $\fG$ is sufficiently small it holds that
\[
\Lambda_p(\chi_p', \widehat f_\fG\otimes \tau_p)=1.
\]
 Let
$
\widehat\phi_p\in \CB_P(\Pi_p)
$
be the image of $\widehat f$ under the natural map $\CB_P(I(\tilde\kappa))\to \CB_P(\Pi_p)$ induced by \eqref{surjhom}, which clearly belongs to $\Pi_p'$.  
Define the normalized refined period to be 
\[
\begin{aligned}
\widehat\CP_p^\circ: \CX_p  \times \left(\Pi_p'\otimes \RM( \dot U \bs \dot G)\right) & \to\C, \\
 (\chi_p', \widehat\phi_p\otimes\bar\tau_p) & \mapsto 1. 
\end{aligned}
\]

The following is immediate from Corollary \ref{pfin}.
\end{comment}

\begin{prpl} \label{prp:st}
It holds that 
\[
\CP^\circ_p(\chi_p', \widehat\phi_p^\circ)= \Upsilon_{\Pi_p'}(\chi_p')\cdot \Omega_{\Pi_p'}^{-1}
%\widehat\CP_p^\circ(\chi_p', \widehat\phi_p^\circ)
\quad \textrm{ for all }\chi_p'\in \CX_p.
\]
Consequently 
$\Upsilon_{\Pi_p'}$ is algebraic on $\mathcal X_p$.
\end{prpl}

By Proposition \ref{prp:st}, $\Upsilon_{\Pi_p'}(\chi_p)$ is the modifying factor at $p$ as in Definition \ref{df:MFp}. It is consistent with the conjecture given by Coates and Perrin-Riou  in \cite{CPR89, Co89}. 

\subsection{$p$-adic L-functions and exceptional zeros} \label{ssec:PLEZ} 
Now we  construct Rankin-Selberg $p$-adic L-function $\mathscr L_\Pi$ by combining the previous results. Suppose that $Z_0$ is trivial and recall $\varepsilon = \otimes_{v\nmid\infty p}\varepsilon_v$ in \eqref{ram}.

\begin{itemize}
    \item For  $v\nmid \infty p$, let  
$\phi^\circ_v\in \Pi_v({\sf E})\otimes \RD(\dot\RU_v \bs \dot\RG_v)$ be as in Section \ref{ssec:RSI}   such that 
\[
\CP^\circ_v(\chi_v', \phi^\circ_v) = \frac{1}{\mathscr G_{\psi_v}(\chi_{\Pi_{n-1,v}})\cdot \mathscr G_{\psi_v}(\chi_v')^{\frac{n(n-1)}{2}}},\quad \chi_v'\in \CX(\varepsilon_v).
\]
\item  Let $\lambda_0 \in  \Hom_E(V^\n, E)$ and $\widehat \phi^\circ_\infty\in \Pi_\infty'({\sf E})\otimes  \RD  
 (\dot{\mathsf U}(\R))^\vee\otimes \mathrm{O}(\dot{\mathsf G}(\R)/\dot{K}_\infty^\circ)$ be as in Section \ref{ssec:AMSM} such that 
 \[
 \widehat\CP_\infty^{\lambda_{\sf V, w}}(1,\widehat\phi^\circ_\infty)  = \Upsilon_{\Pi_\infty'}(\chi_\infty) \cdot \widehat\CP_\infty^\circ({\sf w}_\infty\chi_\infty, \widehat\phi^\circ_\infty)
  = \frac{\Upsilon_{\Pi_\infty'}(\chi_\infty)}{\Omega_\Pi({\sf w}_\infty\chi_\infty)}
 \]
 for all ${\sf V}$-balanced characters ${\sf w}$ and all algebraic characters $\chi_\infty$ of weight ${\sf w}^{-1}$, where  $\lambda_{\sf V, w}\in \Hom_{\dot{\sf G}_{\sf E}}(\sf E_{\sf w}\otimes {\sf V}, {\sf E})$ is the generator such that $\lambda_{\sf V, w}|_{V^\n} = \lambda_0$.
 
 \item Let $\widehat\phi_p^\circ\in \Pi_p'({\sf E})\otimes \RD(\dot U \bs \dot G)$ be as in \eqref{ptestrs}.
\end{itemize}
Using all the local test vectors, let
\[
\widehat\phi^\circ:= \widehat\phi_\infty^\circ \otimes \widehat\phi_p^\circ \otimes (\otimes_{v\nmid\infty p} \phi_v^\circ)
\in \mathscr{H}\otimes \RD(\dot{\sf G}, \dot\p),
\]
and following Definition \ref{df:padicL} let
\[
\mathscr{L}_\Pi:= \mathcal L_{\varepsilon \otimes  \mathscr H } := (\mathcal L_{\lambda_0\otimes \widehat \phi^\circ})|_{\con({\sf Z}, E)(\varepsilon)}.
\]

We restate  Theorem \ref{padicLrs0} below, which is now an immediate consequence of the above results and \eqref{pintp}.

\begin{thml} \label{padicLrs}
Let the notations and assumptions be as above.  %If  there is a $\sf V$-balanced Hecke character in $\CX(\varepsilon)$, 
Then
\[
\mathscr L_\Pi(\chi^\flat)  = \frac{\Upsilon_{\Pi_\infty'}(\chi_\infty)\cdot \Upsilon_{\Pi_p'}(\chi_p)\cdot \oL(\frac{1}{2}, \Pi\times \chi)}{\mathscr G_\psi(\chi^{(p)})^{\frac{n(n-1)}{2}}\cdot  \mathscr G_\psi(\chi_{\Pi_{n-1}}^{(p)}) 
\cdot \Omega_{\Pi_p'}\cdot \Omega_{\Pi}({\sf w}_\infty \chi_\infty)   }
  \]
  for all ${\sf V}$-balanced Hecke characters $\chi=\otimes_\ell \chi_\ell \in \CX(\varepsilon)$, where ${\sf w}$ is the inverse of the weight of $\chi$.
\end{thml}

As mentioned earlier, the existence of exceptional zeros of $p$-adic L-functions is of great importance in arithmetic applications. Let us give the application of Theorem \ref{padicLrs}.

% Recall from Section \ref{ssec:rar} that  $\CF_\mu$ and $\CF_\nu$ are of pure weights $w_\mu$ and $w_\nu$ respectively. As is well-known, the existence of critical places implies that  $ w_\mu\equiv w_\nu \mod 2,$ in which case the  set of critical places is symmetric about the central critical place \[ \frac{1}{2}+j_0,\quad \textrm{where} \ j_0: = \frac{w_\mu+w_\nu}{2}.\]

\begin{prpl} \label{excep0rs}
Under the assumptions in   Theorem \ref{padicLrs}, $\chi^\flat$ is an exceptional zero of $\mathscr L_\Pi$  if and only if there is a place $\wp\mid p$ such that one of the followings holds:
\begin{itemize}
\item $\oL(s, \Pi_\wp\times\chi_\wp)$ has a pole at $s=\frac{1}{2}$;
    \item there exists $1\leq i<n$ such that 
\[ % \label{crit0}
\chi_\wp = \kappa_{i, \wp}^{-1} \cdot \kappa_{n-i, \wp}'^{-1} \cdot \abs{\,\cdot\,}_\wp,
\]
where $\kappa_{i,\wp}$'s and $\kappa_{j,\wp}'$'s are as in \eqref{MFprs}.
\end{itemize}
\end{prpl}

The rest of this section is devoted to the proof of Proposition \ref{excep0rs}. 
Let us  introduce the notion of $p$-exponent  at a place $\wp\mid p$. Denote by $\CE_{\rk_\wp}$ the set of continuous field embeddings 
$\rk_\wp \hookrightarrow \C_p$, so that $\CE_\rk$ can be  identified with the disjoint union of $\{\CE_{\rk_\wp}\}_{\wp |p}$. Let
$d_\wp:=|\CE_{\rk_\wp} |= [\rk_\wp: \Q_p]$.

\begin{dfnl}
The $p$-exponent of a character $\omega: \rk_\wp^\times \to \C_p^\times$ is the number ${\rm ex}_p(\omega)\in \Q$ such that
\[
\abs{\omega(a)}_p = \abs{a}_\wp^{{\rm ex}_p(\omega)} = \abs{a}_p^{d_\wp\cdot {\rm ex}_p(\omega)},\quad a\in \rk_\wp^\times.
\]
\end{dfnl}
Here and henceforth, $\abs{\,\cdot\,}_v$ denotes the normalized absolute value on $\rk_v$ for every place $v$ of $\rk$. 

\begin{exl}
 (a) The normalized absolute value $\abs{\,\cdot\,}_\wp: \rk_\wp^\times \to \Q^\times \hookrightarrow \C_p^\times$ has $p$-exponent $-1$.
 
    \noindent (b) The inclusion $\rk_\wp^\times \hookrightarrow \C_p^\times$ has $p$-exponent $d_\wp^{-1}$. % It is easy to check that \begin{enumerate}    \item[(a)] The normalized absolute value $\abs{\,\cdot\,}_\wp: \rk_\wp^\times \to \Q^\times \hookrightarrow \C_p^\times$ has $p$-exponent $-1$.    \item[(b)] The inclusion $\rk_\wp^\times \hookrightarrow \C_p^\times$ has $p$-exponent $d_\wp^{-1}$.\end{enumerate} 
\end{exl}

Similarly we define the $p$-exponent of $\tilde\kappa_\wp$ as an element of $\Q^{2n-1}$. 
Note that $\RG_\wp \subset G$ acts diagonally on 
$
V_\wp:=  \otimes_{\iota\in \CE_{\rk_\wp}}(F_{\mu^\iota}^\vee\boxtimes F_{\nu^\iota}^\vee),
$
where $F_{\mu^\iota}$ and $F_{\nu^\iota}$  are of highest weights $\mu^\iota$ and  $\nu^\iota$ respectively. % defined over an extension of $\iota(\rk_\wp)$. 
Let
\[
\mu^\wp:=\frac{1}{d_\wp}\sum_{\iota\in \CE_{\rk_\wp}}\mu^\iota \in  \Q^n, \quad \nu^\wp:=\frac{1}{d_\wp}\sum_{\iota\in \CE_{\rk_\wp}}\nu^\iota \in  \Q^{n-1}.
\]

Recall that $L$ acts on $V^\n$ through a character $\alpha_V: L\to E^\times$. Then the component of $\alpha_V\circ {\rm Ad}(z^{-1})$ at $\wp$
has $p$-exponent $(-\mu^\wp, -\nu^\wp)\in \Q^{2n-1}$. The   character $\delta_{\overline{\RB}_\wp}$ has $p$-exponent $2(\rho_n, \rho_{n-1})$, where
\[
\rho_m:= \left(\frac{m-1}{2}, \frac{m-3}{2}, \ldots,  -\frac{m-1}{2}\right)\in \Q^m\quad (m=n,n-1).
\]
Since  $\Pi_p'\otimes \delta_P$ and $\alpha_V$ have the same $p$-adic norm by Definition \ref{df:NVO}, we have the following result.

\begin{lem} \label{lem:exp}
For $\wp\mid p$, the character $\tilde\kappa_\wp$ has $p$-exponent
\[
(-\mu^\wp-\rho_n, -\nu^\wp-\rho_{n-1}) \in \Q^{2n-1}.
\]
\end{lem}

\begin{lem} \label{lem:llss}
If $\chi =\otimes_v\chi_v$ is $\sf V$-balanced, then for each $\wp\mid p$ one has that
\[
\frac{1}{2}+{\rm ex}_p(\tilde\kappa_{i, \wp}\cdot \tilde\kappa'_{j, \wp}\cdot \chi_\wp)\textrm{ is }
\begin{cases}
\leq 0, & \textrm{if }i+j\leq n, \\
\geq 1, & \textrm{if }i+j > n.
\end{cases}
\]
\end{lem}

\begin{proof}
Let ${\sf w}^{-1}$ be the weight of $\chi$. By Lemma \ref{lem:exp}, we find that
\[
\frac{1}{2}+{\rm ex}_p(\tilde\kappa_{i,\wp}\cdot \tilde\kappa'_{j,\wp}\cdot \chi_\wp) = {\sf w}_\wp - \mu^\wp_{i}- \nu^\wp_{j} + i+j-n,
\]
where ${\sf w}_\wp : = d_\wp^{-1}\sum_{\iota\in\CE_{\rk_\wp}}{\sf w}_\iota$. 
The assertion follows easily from \eqref{eq:bp}.
\end{proof}

%By the nontrivial estimate of Langlands parameters towards the generalized Ramanujan conjecture in\cite{LRS99, MS04} and the multiplicativity of Rankin-Selberg L-functions, $\oL(\frac{1}{2}, \Pi_p\times \chi_p)$ is finite. 
By \eqref{MFprs}, $\chi^\flat$ is an exceptional zero of $\mathscr L_\Pi$ if and only if 
\[
\prod_{\wp\mid p}\frac{\prod_{i+j\leq n}\gamma\left(\frac{1}{2}, \tilde\kappa_{i, \wp}^{-1} \cdot \tilde\kappa_{j,\wp}'^{-1} \cdot \chi_\wp^{-1}, \psi_\wp^{-1}\right)}{\oL(\frac{1}{2}, \Pi_\wp\times\chi_\wp)} = 0.
\]
%Applying \cite[Theorem (3.1)]{JPSS83} for $\Pi_p$ and $I(\tilde\kappa)$, 
Proposition \ref{excep0rs} follows easily from Lemma \ref{lem:llss} and its proof.

\section{Application II: Rankin-Selberg %$p$-adic  
L-functions for  $\RU_n\times \RU_{n-1}$}

In this section we retain the setup  in Section \ref{ssec:RSU} and construct the $E$-valued $p$-adic measure $\mathscr L_\Pi$  
 in Theorem \ref{padicLggp0} following Sections \ref{ssec:CLR}--\ref{ssec:PL}, for which we assume that the character $\varepsilon$ in \eqref{ram} is trivial. Then we determine the exceptional zeros of $\mathscr L_\Pi$.

\subsection{Refined Gan-Gross-Prasad conjecture} \label{ssec:ggp} 
We first review the Gan-Gross-Prasad conjecture and its refinement for the Bessel periods of unitary groups. 
Let $\tilde\pi_m$, $m=n, n-1$ ($n\geq 2$), be an irreducible isobaric automorphic representation of $\GL_m(\A_{\rk'})$ that is hermitian in the sense of \cite[Definition 1.5]{BPLZZ21}. Put $\tilde\pi:=\tilde\pi_n \boxtimes\tilde\pi_{n-1}$, 
which is an automorphic representation of $\GL_n(\A_{\rk'})\times \GL_{n-1}(\A_{\rk'})$. 

Denote by $\rk'\rightarrow \rk', \, a \mapsto \bar a$ the Galois involution for $\rk'/\rk$, which induces an $\A_\rk$-algebra involution $\A_{\rk'}\rightarrow \A_{\rk'}, \, a \mapsto \bar a$.

\begin{thml} [Gan-Gross-Prasad conjecture] \label{ggp}  The following  two statements are equivalent.

\noindent (a) $\oL(\frac{1}{2}, \tilde\pi)\neq 0$.

\noindent (b) There exists a quadruple $(V_n, V_{n-1}, \pi_n, \pi_{n-1})$ such that 

\begin{itemize}
\item $V_{n-1}$ is a hermitian spaces over $\rk'$ of rank $n-1$ and  $V_n$ is a Hermitian space over $\rk'$ such that $V_n=V_{n-1} \oplus \rk'$ is an orthogonal direct sum, where $\rk'$ is viewed as a hermitian space under the form $(a,b)\mapsto a \bar b$, $a, b\in \rk'$;  
\item  for $m=n,n-1$, $\pi_m$ is an irreducible cuspidal automorphic representations of $\RU_m(\A_{\rk})$ such that  $\RB\RC(\pi_m)\cong \tilde\pi_m$, where $\RU_m:= \RU(V_m)$ is the unitary group attached to $V_m$, and $\RB\RC(\pi_m)$ denotes the weak base-change of $\pi_m$ to $\GL_m(\A_{\rk'})$;
\item
there are cusp forms $\varphi_m\in \pi_m$ ($m=n, n-1$) such that 
\[
\int_{\RU_{n-1}(\rk)\bs \RU_{n-1}(\A_{\rk})}\varphi_{n}(g) \varphi_{n-1}(g)\od\! g\neq 0, 
\]
where $\od\!g$ is the Tamagawa measure.
\end{itemize}

%\begin{enumerate}\item[(a)] $\oL(\frac{1}{2}, \tilde\pi)\neq 0$.\item[(b)] \end{enumerate}

%\begin{itemize}\item hermitian spaces $V_n$ of rank $n$ and $V_{n-1}$ of rank $n-1$ over $\rk'$ such that $V_n=V_{n-1} \oplus \rk'$ is an orthogonal direct sum, where $\rk'$ is viewed as a hermitian space under the form $(a,b)\mapsto a \bar b$, $a, b\in \rk'$; and \item  irreducible cuspidal automorphic representations $\pi_m$ of $\RU_m(\A_{\rk})$ for $\RU_m:= \RU(V_m)$, satisfying  $\RB\RC(\pi_m)\cong \tilde\pi_m$, and cusp forms $\varphi_m\in \pi_m$, $m=n, n-1$, such that \[\int_{\RU_{n-1}(\rk)\bs \RU_{n-1}(\A_{\rk})}\varphi_{n}(g) \varphi_{n-1}(g)\od\! g\neq 0\]where $\od\!g$ is the Tamagawa measure, and $\RB\RC(\pi_m)$ denotes the weak base-change of $\pi_m$ to $\GL_m(\A_{\rk'})$. \end{itemize}\end{enumerate}
\end{thml}

Suppose that the equivalent statements in Theorem \ref{ggp} hold, and we keep the notation in Theorem \ref{ggp}. Then the pair $(V_n, V_{n-1})$ is unique up to isomorphism,
and the pair $(\pi_n, \pi_{n-1})$  is uniquely determined, by the local Gan-Gross-Prasad conjecture \cite{BP16, BP20, Xue23} and  Arthur's multiplicity formula for unitary groups \cite{KMSW14}.
Suppose that  $\pi:=\pi_n \boxtimes \pi_{n-1}$ is everywhere tempered.  
Recall the L-function $\CL(s, \pi\times \chi)$ in \eqref{cLggp}, where $\chi: \RU_1(\rk)\backslash \RU_1(\A_\rk)\rightarrow \C^\times $ is an automorphic character with base-change
\[
\tilde\chi={\rm BC}(\chi): \rk'^\times\bs \A_{\rk'}^\times \to \BC^\times, \quad a\mapsto \chi(a/\bar a).
\]
Note that 
\[
\oL(s, \pi, \Ad) = \oL(s, \tilde\pi_n, {\rm As}^{(-1)^n}) \oL(s, \tilde\pi_{n-1}, {\rm As}^{(-1)^{n-1}})
\]
is a product of Asai L-functions.% which only depends on $\tilde\pi \cong {\rm BC}(\pi)$.  

Recall the embedding 
\[
\imath: \dot{\sf G}= \Res_{\rk/\Q}(\RU_{n-1} \times \RU_{n-1})\hookrightarrow {\sf G}= \Res_{\rk/\Q}(\RG_0\times\RG_0),
\]
where $\RG_0 = \RU_n\times \RU_{n-1}$, and the homomorphism $\jmath: \dot{\sf G}\to {\sf Z}=\Res_{\rk/\Q}\RU_1$. Let 
\[
{\sf H}:=\Res_{\rk/\Q}\RG_0 \xrightarrow{\rm diag} {\sf G}
\]
diagonally embedded as a subgroup 
of ${\sf G}$, and let $\psi_{\sf H}$ be the trivial character of ${\sf H}(\A)$. Then it is clear that \eqref{assh} holds. 
At this point it is again more familiar to work with $\rk$ and introduce
\[
\dot\RG :=\RU_{n-1}\times \RU_{n-1} \xrightarrow{\imath} \RG:= \RG_0 \times \RG_0 \quad \text{and}  \quad \RH := \RG_0 \xrightarrow{\diag} \RG.
\]
Fix factorizations 
$\pi = \widehat\otimes'_v \pi_v $ and $ \pi^\vee = \widehat\otimes'_v \pi_v^\vee$, where $v$ runs over all places of $\rk$,
such that there is a decomposition of $\RH(\A_\rk)$-invariant map
\[
\la \cdot, \cdot \ra = \otimes_v \la \cdot, \cdot \ra_v: \Pi=\pi\boxtimes \pi^\vee\to\C,
\]
where $\la \cdot, \cdot \ra$ is the pairing between $\pi$ and $\pi^\vee$ given by the Petersson inner product on cusp forms with respect to the Tamagawa measure of 
$\RH(\A_\rk)$, and $\la\cdot, \cdot\ra_v$ is the natural pairing between $\pi_v$ and $\pi_v^\vee$. 
%Fix  local Haar measures on $\dot\RG_v = \dot\RG(\rk_{v})$, rational at all finite places,  whose product is  the Tamagawa measure on  $\dot\RG(\A_\rk)$. 

Following Theorem \ref{ggp} (b), define the global period integral map
\[
\begin{array}{rcl}
   \mathcal P_\chi:  \Pi \otimes \mathrm M(\dot\RG(\rk)\backslash\dot \RG(\A_\rk))&\rightarrow  & \chi^{-1},\\
      \varphi\otimes \varphi^\vee \otimes \tau &\mapsto &\int_{\dot\RG(\rk) \backslash\dot\RG(\A_\rk)}\chi(\jmath(g)) \cdot (\varphi\otimes \varphi^\vee)(g)\od\!\tau(g).
\end{array}
\]

%For every place $v$ of $\rk$, write  \[\rk'_v:=\rk'\otimes_\rk \rk_v = \prod_{v'|v}\rk'_{v'}, \] which is a quadratic \'etale $\rk_v$-algebra. Let $\chi_v$ be a unitary character of \[\RU_1(\rk_v)=\ker(\RN_{\rk_v'/\rk_v}: \rk_v'^\times \to \rk_v^\times).  \]

Let $\CX_v$ be the group of complex continuous characters of $\RU_1(\rk_v)$ for every place $v$ of $\rk$, and let $\CX_v^{\rm un}\subset \CX_v$ be the subgroup of unramified characters. In particular $\CX_v^{\rm un}$ consists of the trivial character if $v$ is nonsplit in $\rk$. 
Define the $\dot\RG_v$-invariant normalized zeta integral
\[
\begin{aligned}
\CP^\circ_v: \CX_v\times \left(\Pi_v\otimes \RM(\dot\RH_v\bs \dot\RG_v)\right) & \to \C, \\
(\chi_v', \varphi_v\otimes \varphi_v^\vee\otimes\tau_v)& \mapsto \frac{\int_{\RU_{n-1}(\rk_v)}\chi_v'(\det g) \la \pi_v(g)\varphi_v, \varphi_v^\vee\ra_v \od\! \tau_v(g)}{\CL(\frac{1}{2}, \pi_v\times\chi_v')} ,
\end{aligned}
\]
where  $\RU_{n-1}$ embeds into $\RG_0$ diagonally as before, $\varphi_v\otimes \varphi_v^\vee\in \Pi_v = \pi_v\boxtimes \pi_v^\vee$ and $\tau_v\in \RM(\dot\RH_v\bs \dot\RG_v) =\RM(\RU_{n-1}(\rk_v))$.  Here and henceforth $\CL(\frac{1}{2}, \pi_v\times\chi_v')$ denotes the local factor of $\CL(\frac{1}{2}, \pi\times \chi')$ at $v$. By the temperedness assumption,  the integral converges absolutely and $\CL(\frac{1}{2},\pi_v\times\chi_v')\in \C^\times$  when $\chi_v'$ is unitary. By Lemma \ref{lem:zhang} below, $\CP_v^\circ$ has a holomorphic continuation in $\chi_v'\in \CX_v$ when $v$ is split in $\rk$. Moreover if $\chi=\otimes\chi_v$, $\varphi \otimes \varphi^\vee = \otimes_v (\varphi_v \otimes \varphi_v^\vee)$ and 
$\tau =\otimes_v\tau_v$ under the identification \eqref{IDmeas}, then  $\CP_v^\circ(\chi_v, \varphi_v\otimes \varphi_v^\vee\otimes\tau_v)=1$ for  all but finitely many  $v$ by \cite[Theorem 2.12]{Ha14}.

The following Ichino-Ikeda conjecture refines the Gan-Gross-Prasad conjecture.
%and gives the proportionality constant between $\CP_\chi$ and $\CP_\chi^\circ=\otimes_v \CP_v(\chi_v,\cdot)$ for $\chi = \otimes_v \chi_v$.

\begin{thml} [Ichino-Ikeda conjecture] \label{iic}
Let the assumptions be as above. Then for all $\varphi \otimes\varphi^\vee = \otimes_v (\varphi_v \otimes \varphi_v^\vee) \in \Pi$, 
$\tau=\otimes_v\tau_v\in \RM(\dot\RG(\rk)\bs \dot\RG(\A_\rk))$, and all automorphic characters $\chi = \otimes_v\chi_v : \RU_1(\rk)\backslash \RU_1(\A_\rk)\rightarrow \C^\times $,  it holds that
\[
\CP_\chi( \varphi\otimes \varphi^\vee\otimes\tau)  = \abs{S_{\tilde\pi}}^{-1} \CL(\frac{1}{2}, \pi\times\chi) \prod_v \CP_v^\circ(\chi_v,\varphi_v \otimes \varphi_v^\vee\otimes \tau_v).
\] 
\end{thml}

In the above, $S_{\tilde\pi}  =S_{\tilde\pi_n} \times S_{\tilde\pi_{n-1}}$ and  $S_{\tilde\pi_m}$ ($m=n,n-1$) is an elementary abelian 2-group whose rank equals the number of cuspidal summands of the isobaric sum $\tilde\pi_m$. Thanks to the work 
of \cite{Mok15, KMSW14, AGI+24},  it is known that $|S_{\tilde\pi_m}|$ is the size of the global Arthur packet of $\pi_m$. 

Theorem \ref{ggp} and Theorem \ref{iic} are proved in \cite{BPLZZ21} for $\tilde\pi$ cuspidal (in which case $|S_{\tilde\pi}|=4$), and in \cite{BPCZ22} for $\tilde\pi$ hermitian isobaric, improving the previous results
in \cite{Zh14a, Zh14b, Xue19, BP21a, BP21b}. They have been further extended to some Eisenstein series very recently in \cite{BPC23}.

\subsection{Rational forms and integrals of matrix coefficients} \label{ssec:IMC} 
In the rest of this section,  assume that $\pi$ is regular algebraic,  $\Q(\pi)\subset {\sf E}$ and the coefficient system  of $\pi$ is defined over ${\sf E}$, and that $\chi$ is algebraic.

\begin{leml} \label{lem:ratggp}
For every finite place $v$ of $\rk$, there is a unique $\Q(\pi_v)$-form  on $\Pi_v$ such that the natural pairing 
\[
\langle\cdot, \cdot\rangle_v: \Pi_v= \pi_v\boxtimes \pi_v^\vee\to\C
\]
is defined over $\Q(\pi_v)$. 
\end{leml}

\begin{proof}
We sketch the proof, which is similar to that of \cite[Proposition 3.6]{JST19}. We have the $\RG_v$-equivariant embedding
\[
\Pi_v \to C^\infty(\RG_{0,v}),\quad \varphi\otimes \varphi^\vee \mapsto ( f_{\varphi, \varphi^\vee}(g):= \langle \pi_v(g)\varphi, \varphi^\vee\rangle_v,\quad g\in \RG_{0,v}),
\]
where $C^\infty(\RG_{0,v})$ denotes the space of locally constant complex functions on $\RG_{0,v}$, on which $\RG_v=\RG_{0,v}\times \RG_{0,v}$ acts by translations. Let $\Aut(\C)$ act on $C^\infty(\RG_{0,v})$ by $(\sigma.f)(g):=\sigma(f(g))$, 
where $\sigma\in\Aut(\C)$, $f\in C^\infty(\RG_{0,v})$. Since $\Hom_{\RG_v}(\Pi_v, C^\infty(\RG_{0,v}))$ is one-dimensional, 
 $\Aut(\C/\Q(\pi_v))$ stabilizes the image of $\Pi_v$ in $C^\infty(\RG_{0,v})$, which induces an action of $\Aut(\C/\Q(\pi_v))$ on $\Pi_v$.

By \cite[Thm. A.2.4]{GS17}, $\Pi_v$ has a rational form  $\Pi_v^\circ$ defined over a finite extension of $\Q(\pi_v)$ such that $\langle\cdot,\cdot\rangle|_{\Pi_v^\circ}$
is $\overline{\Q}$-valued. By choosing $\varphi\otimes\varphi^\vee\in \Pi_v^\circ$ such that $\langle \varphi,\varphi^\vee\rangle_v\neq 0$ and using the above multiplicity one result, it is easy to show that some open subgroups of $\Aut(\C/ \Q(\pi_v))$  fix $\Pi_v^\circ$ pointwise, where $\Aut(\C/\Q(\pi_v))$ is equipped with the coarsest topology such that the natural map $\Aut(\C/\Q(\pi_v))\to \Aut(\overline\Q / \Q(\pi_v))$ is continuous.  It follows that 
$
\Pi_v^{\Aut(\C/\overline{\Q})} = \Pi_v^\circ\otimes \overline{\Q}.
$

Hence by \cite[Proposition 11.1.6]{Sp09},
\[
\Pi_v^{\Aut(\C/\Q(\pi_v))} = \left(\Pi_v^{\Aut(\C/\overline{\Q})}\right)^{\Aut(\overline{\Q}/\Q(\pi_v))}
\]
is a $\Q(\pi_v)$-form of $\Pi_v$, which is the unique $\Q(\pi_v)$-form  satisfying the lemma. 
\end{proof}

For $v\nmid\infty$, define the $\Q(\pi_v)$-form  of $\Pi_v$ as in Lemma \ref{lem:ratggp}, which induces an ${\sf E}$-form  of $\Pi_v$ to be denoted by  $\Pi_v({\sf E})$. 
Recall the normalized zeta integral $\CP_v^\circ$. 

\begin{lem}\label{lem:autggp}
For every finite place $v$ of $\rk$, the linear functional $\CP^\circ_v(\chi_v, \cdot)$ is defined over $\Q(\Pi_v,\chi_v)$.
\end{lem}

\begin{proof}
By the definitions, this follows easily from the $\Aut(\C)$-equivariance of the local L-factors in $\CL(\frac{1}{2}, \pi_v\times\chi_v)$, which is
proved in \cite[Lemma 2.5]{Liu23}.
\end{proof}

We  need the following result for the split case. Assume that  $v$ is a place of $\rk$ that splits in $\rk'$ and fix an isomorphism $\rk_v\otimes_\rk \rk' \cong  \rk_v\times \rk_v$. Via the first factor $\rk_v$ we have an identification $\RG_{0,v} = \GL_n(\rk_v)\times \GL_{n-1}(\rk_v)$  (with respect to certain bases) and a diagonal embedding of $\GL_{n-1}(\rk_v)$ into $\RG_{0,v}$.

Similar to Section \ref{ssec:RSI}, let $\RU_v$ be the upper triangular maximal unipotent subgroup of $\RG_{0,v}$, and let $\psi_{\RU_v}$ be the generic character of $\RU_v$ defined using the character $\psi_v$ of $\rk_v$ which is trivial on $\dot\RU_v:=\RU_v\cap \GL_{n-1}(\rk_v)$. By the temperedness assumption,  $\pi_v$ is generic unitary. Following \cite[(3.2)]{Zh14b}, we fix the Whittaker models 
\[
\pi_v \subset \Ind^{\RG_{0,v}}_{\RU_v}\psi_{\RU_v},\quad \pi_v^\vee \subset \Ind^{\RG_{0, v}}_{\RU_v} \psi_{\RU_v}^{-1}
\]
such that the pairing $\la \cdot, \cdot \ra_v$ on $\pi_v\boxtimes \pi_v^\vee$ is given by
\be \label{zhpairing}
\la \varphi, \varphi^\vee \ra_v = \int_{\RU_v\bs \RQ_v}\varphi(g) \varphi^\vee(g)\od\!g.
\ee
Here $\RQ_v$ is a mirabolic subgroup of $\RG_{0,v}$ containing $\RU_v$,  and $\od\!g$ is a fixed Borel measure on $\RU_v\bs \RQ_v$ that is rational if $v$ is finite. Note that the integral \eqref{zhpairing} converges absolutely. 

Fix the similar Haar measure on $\dot\RU_v$ as in Section \ref{ssec:OOI}, and for $\tau\in \RM(\GL_{n-1}(\rk_v))$ denote by $\bar\tau
\in \RM(\dot\RU_v\bs \GL_{n-1}(\rk_v))$ the quotient of $\tau$ by the fixed measure on $\dot\RU_v$. 
The following is \cite[Proposition 4.10]{Zh14b} (stated as \cite[Lemma 2.10]{Liu23}).

\begin{lem} \label{lem:zhang}
Let  $v$ be a place of $\rk$ that splits in $\rk'$, $\chi_v'\in \CX_v$ a unitary character, and $\tau\in \RM(\GL_{n-1}(\rk_v))$. Then there is a  constant $c_{\tau}$ only depending on $\tau$ such that 
\[
\begin{aligned}
& \int_{\GL_{n-1}(\rk_v)} \chi_v'(\det g) \la\pi_v(g)\varphi, \varphi^\vee\ra_v \od\!\tau(g) \\
=\, & c_{\tau} \int_{\dot\RU_v\bs \GL_{n-1}(\rk_v)}\chi_v'(\det g)\varphi(g)  \od\!\bar\tau(g)   \int_{\dot\RU_v\bs \GL_{n-1}(\rk_v)}\chi_v'^{-1}(\det g)\varphi^\vee(g) \od\!\bar\tau(g)
\end{aligned}
\]
for all $\varphi\otimes \varphi^\vee\in \Pi_v$, 
where both sides converge absolutely. Moreover if $v$ is finite and 
$\tau_v\in \RD(\GL_{n-1}(\rk_v))$, then $c_{\tau}$ is rational. 
\end{lem}

\begin{prpl}  \label{lem:unr}
There is a family
\be \label{familyggp}
\{\phi_v^\circ \in \Pi_v\otimes \RD(\dot\RH_v\bs \dot\RG_v)\}_{v\nmid \infty}
\ee
such that 
\begin{itemize}
    \item for all $v \nmid \infty$, $\phi^\circ_v\in \Pi_v({\sf E})\otimes \RD(\dot\RH_v \bs \dot\RG_v)$ and $\CP^\circ_v(\,\cdot\,, \phi^\circ_v)=1$ on $\CX_v^{\rm un}$;
    
    \item for all but finitely many $v\nmid\infty$, $\phi_v^\circ$ is the fixed spherical vector used in the restricted tensor product $\Pi\otimes \RM(\dot\RH(\A_\rk)\bs \dot\RG(\A_\rk)) = \widehat\otimes'_v(\Pi_v\otimes\RM(\dot\RH_v\bs \dot\RG_v))$.
\end{itemize}
\end{prpl}

\begin{proof}
This follows from \cite[Theorem 2.12]{Ha14}, \cite[Proposition 2.11]{Liu23} and Lemmas \ref{lem:autggp} and \ref{lem:zhang}. 
\end{proof}

\subsection{Archimedean modular symbols}  \label{ssec:AMSMggp}
Take $K_\infty$ to be a maximal compact subgroup of $\mathsf G(\R)$ such that the corresponding Cartan involution preserves $\dot{\mathsf G}(\R)$. Then  $\dot K_\infty:=K_\infty\cap \dot{\mathsf G}(\R)$ is a maximal compact subgroup of  $\dot{\sf G}(\R)$.   There is a geometrically irreducible algebraic representation $ {\sf V}_\pi\boxtimes {\sf V}_\pi^\vee$ of ${\sf G}_{\sf E}$ such that the total relative Lie algebra cohomology 
\[
\oH^\bullet(\g_\C, K_\infty^\circ; ({\sf V}_\pi\boxtimes {\sf V}_\pi^\vee)\otimes \Pi_\infty)\neq \{0\},
\]
where ${\sf V}_\pi={\sf F}_\mu^\vee\boxtimes {\sf F}_\nu^\vee$ and $(\mu, \nu)\in (\Z^n)^{\CE_\rk}\times (\Z^{n-1})^{\CE_\rk}$ is as in Section \ref{ssec:RSGL}. 
Suppose that  ${\sf V}={\sf V}_\pi\boxtimes {\sf V}_\pi^\vee$ (which is also viewed as  a representation of ${\sf G}(\Q)$), so that $V=E\otimes({\sf V}_\pi\boxtimes {\sf V}_\pi^\vee)$ as a representation of $G={\sf G}(\Q_p)
\subset {\sf G_E}(E)$.

\begin{dfnl} \label{critggp}
 (a) A character $\chi_\infty$ of  $\RU_1(\rk_\infty)$ is said to be critical for $\pi$ if 
it is algebraic and $s=\frac{1}{2}$ is a pole of neither $\CL(s, \pi_\infty\times \chi_\infty)$ nor $\CL(1-s, \pi_\infty^\vee\times \chi_\infty^{-1})$.

\noindent (b) An automorphic character $\chi:\RU_1(\rk)\bs\RU_1(\A_\rk)\to \C^\times$ is said to be critical for $\pi$  if so is its archimedean component $\chi_\infty$.
\end{dfnl}

\begin{leml}\label{lem:crit}
(a) All ${\sf V}$-balanced characters  $\chi_\infty$ of $\RU_1(\rk_\infty)$
are critical for $\pi$.

\noindent (b) All algebraic automorphic characters $\chi: \RU_1(\rk)\bs\RU_1(\A_\rk)\to \C^\times$ are critical for $\pi$.
\end{leml}

\begin{proof}
 By \cite{Mok15, KMSW14, Ram18} (\cf \cite[Remark 1.1.4.1]{BPCZ22}), a weak base-change is automatically a strong base-change, hence $\tilde\pi_\infty$ is also regular algebraic. 
It follows easily from the classical branching rule that $\chi_\infty$ is $\sf V$-balanced if and only if $\tilde\chi_\infty:={\rm BC}(\chi_\infty)$ is ${\sf V}_{\tilde\pi}$-balanced in the sense of Section \ref{sec:rs}, where ${\sf V}_{\tilde\pi}^\vee$ is the rational form of the coefficient system of $\tilde\pi$. Part (a)  follows  from  this and the fact that $\chi_\infty$ is critical for $\pi$ if and only if $\tilde\chi_\infty$ is critical for $\tilde\pi$.

Part (b) follows from the fact that $\tilde\pi_\infty$ is tempered and $\chi$ is unitary. 
\end{proof}

Recall the ${\sf G}^\natural$-homomorphism \eqref{embcoh} whose image is defined over ${\sf E}$ (see \cite[1.4.2]{GL21}), where $\Phi$ is the family of all compact subsets of ${\sf G}(\Q)\bs \mathscr X$. The  bottom degree component of $\oH^\bullet(\g_\C, K_\infty^\circ; {\sf V}\otimes \Pi_\infty)$ 
is 
\[
\Pi_\infty' = \oH^{i_0}(\g_\C, K_\infty^\circ; {\sf V}\otimes \Pi_\infty),\quad \textrm{where}\quad i_0= \dim (\dot{\sf G}(\R)/\dot{K}_\infty^\circ).
\]
Similarly let $\pi_\infty'$ and ${\pi_\infty^\vee}'$  be the bottom degree relative Lie algebra cohomologies
 of ${\sf V}_{\pi}\otimes\pi_\infty$ and ${\sf V}_\pi^\vee\otimes\pi_\infty^\vee$ respectively, so that 
 \be \label{piinf'}
 \Pi_\infty' = \pi_\infty' \boxtimes {\pi_\infty^\vee}' \cong \bigoplus_{\varepsilon_\infty, \varepsilon_\infty' \in \widehat{\RU_1(\rk_\infty)^\natural}}
 \varepsilon_\infty\boxtimes \varepsilon_\infty'
 \ee
as representations of $\dot K_\infty^\natural = \RU_1(\rk_\infty)^\natural\times  \RU_1(\rk_\infty)^\natural$.
In view of Lemma \ref{lem:ratggp}, we have the ${\sf E}$-form  $\Pi_\infty'({\sf E})$ as in \eqref{Piinf'}.

Recall that ${\sf E}_0:=\Q_p\cap {\sf E}$ and an isomorphism $\rk_{\sf E_0}\otimes_\rk \rk' \cong \rk_{\sf E_0}\times \rk_{\sf E_0}$ is fixed. By using the first factor of the product, we have an identification 
\[
{\sf G}_{\sf E_0} =\Res_{\rk_{\sf E_0}/{\sf E}_0} (\GL_{n}\times \GL_{n-1}\times \GL_n \times \GL_{n-1})\quad \text{(with respect to certain bases)}.
\]
 Let $\overline{\sf B}_{\sf E_0}$ be the Borel subgroup of lower triangular matrices in ${\sf G}_{\sf E_0}$, and let 
 \be \label{transP0}
{\sf P}_{\sf E_0} := z'^{-1} \overline{\sf B}_{\sf E_0} z', 
 \ee
where $z':= (z, z)$ and $z$ is as in \eqref{eletz}. Let ${\sf B}_{\sf E}$ be the Borel subgroup of upper triangular matrices in ${\sf G}_{\sf E}$. Similar to \cite{LLS24}, by using algebraic induction from ${\sf B}_{\sf E}$ we realize ${\sf V}$ as a space of algebraic functions on ${\sf G}_{\sf E}$. Let 
${\sf v}\in {\sf V}^{\overline{\sf U}_{\sf E}}$ be the unique algebraic function in ${\sf V}$ which equals 1 
on $\overline{\sf U}_{\sf E}$, where $\overline{\sf U}_{\sf E}$ is the lower triangular maximal unipotent subgroup of ${\sf G}_{\sf E}$.

Suppose that $P={\sf P}_{\sf E_0}(\Q_p)$ so that $N$ is its unipotent radical, and that $\lambda_0 \in \Hom_E(V^\n, E)$ is the unique generator defined over ${\sf E}$ such that $\langle \lambda_0, z'^{-1}{\sf v}\rangle =1$. 
For every  $\sf V$-balanced algebraic character $\sf w$ of $\mathsf Z_{\sf E}$,
%such that ${\sf w}_p$ is trivial on $\prod_{\wp\mid p, \, \wp\not\in\tS}\RU_1(\rk_\wp)$, 
 let $\lambda_{\sf V, w}\in \Hom_{\dot{\sf G}(\Q)}({\sf E_w}\otimes{\sf V}, {\sf E})$ be as in Lemma \ref{lambda01}.
Then for all algebraic characters  $\chi_\infty$ 
of $\RU_1(\rk_\infty)$ of weight ${\sf w}^{-1}$, we have the archimedean modular symbol map 
\[
\widehat{\mathcal P}_\infty^{\lambda_{\sf V, w}} : \oH^0(\z_\C,K_{\mathsf Z,\infty}^\circ; \C_{\sf w_\infty}\otimes\chi_\infty) \times\left(\Pi_\infty' \otimes \mathrm M 
 (\dot{\mathsf H}(\R))^\vee\otimes \mathrm{O}(\dot{\mathsf G}(\R)/\dot{K}_\infty^\circ)\right) \to \C
\]
defined as in \eqref{AMS}. The commutative diagram \eqref{cdm} is then a consequence of the Ichino-Ikeda conjecture (Theorem \ref{iic}).

Let $\Pi_{0,\infty} = \pi_{0,\infty}\boxtimes \pi_{0,\infty}$ be the irreducible %regular algebraic 
tempered Casselman-Wallach  representation 
of ${\sf G}(\R)$ whose infinitesimal character equals that of the trivial representation and %corresponding to the case that $\mu$ and $\nu$ are zero, 
whose central character equals that of 
${\sf V}\otimes \Pi_\infty$. Define the cohomology group $\Pi_{0,\infty}'$ as in \eqref{Piinf'}, so that $\Pi_{0,\infty}' = \pi_{0,\infty}'\boxtimes \pi_{0,\infty}'$ as in \eqref{piinf'}. 
Let $\lambda_0' \in \Hom_{\sf U_E}({\sf V}, {\sf E})$ be the unique generator such that $\langle \lambda_0', {\sf v}\rangle =1$. Following 
\cite{LLS24}, with the fixed Whittaker functionals and $\lambda_0'$ we have the translation map 
\[
 \jmath_{\sf V}= (\jmath_\mu\otimes \jmath_\nu)\otimes (\jmath_{\mu^\vee}\otimes \jmath_{\nu^\vee}): \Pi_{0,\infty}' \to \Pi_{\infty}',
\]
where $\mu^\vee$ and $\nu^\vee$ denote the highest weights of ${\sf F}_\mu^\vee$ and 
${\sf F}_\nu^\vee$ respectively. 
%which is a $\rk_\infty^{\times,\natural}$-equivariant isomorphism. 
As a specialization of \eqref{AMS}, we have a map 
\[
 \widehat\CP_\infty: \oH^0(\z_\C,K_{\mathsf Z,\infty}^\circ; \varepsilon_\infty) \times\left(\Pi_{0,\infty}' \otimes \mathrm M 
 (\dot{\mathsf H}(\R))^\vee\otimes \mathrm{O}(\dot{\mathsf G}(\R)/\dot{K}_\infty^\circ)\right) \to \C
\]
for every $\varepsilon_\infty \in \widehat{\RU_1(\rk_\infty)^\natural}$.

Define  $\widehat{\CP}^\circ_\infty$ in \eqref{NAMS} to be the map such that for every $\varepsilon_\infty \in \widehat{\RU_1(\rk_\infty)^\natural}$, $\widehat{\CP}^\circ_\infty(\varepsilon_\infty, \,\cdot\,)$ equals  the composition of 
\begin{eqnarray*}
&&  \Pi_{\infty}' \otimes\mathrm M 
 (\dot{\mathsf H}(\R))^\vee\otimes \mathrm{O}(\dot{\mathsf G}(\R)/\dot{K}_\infty^\circ) \\
& \xrightarrow{(1, \, \jmath_{\sf V}^{-1}\otimes {\rm id}\otimes {\rm id})} & %\bigoplus_{\varepsilon_\infty\in \widehat{\rk_\infty^{\times,\natural}}}
\oH^0(\z_\C,K_{\mathsf Z,\infty}^\circ; \varepsilon_\infty) \times \left(\Pi_{0, \infty}' \otimes\mathrm M (\dot{\mathsf H}(\R))^\vee\otimes \mathrm{O}(\dot{\mathsf G}(\R)/\dot{K}_\infty^\circ)\right)  \\
 & \xrightarrow{\widehat\CP_\infty} & \C.
 \end{eqnarray*}

%Recall the modifying factor $\Upsilon_{\Pi_\infty'}$ at $\infty$ in \eqref{sgninf}. 
We have the following archimedean nonvanishing hypothesis and period relations.  

\begin{thml}  \label{nonvanggp}
(a) There is an element 
\[
 \widehat \phi^\circ_\infty\in \Pi_\infty'({\sf E})\otimes  \RD  
 (\dot{\mathsf H}(\R))^\vee\otimes \mathrm{O}(\dot{\mathsf G}(\R)/\dot{K}_\infty^\circ)
\]
such that $\widehat\CP_\infty^\circ(\varepsilon_\infty, \widehat\phi^\circ_\infty)\neq 0$ for all $\varepsilon_\infty\in \widehat{\rk_\infty^{\times,\natural}}$. 

\noindent (b)
Let $\lambda_{\sf V, w}\in \Hom_{\dot{\sf G}(\Q)}({\sf E_w}\otimes{\sf V}, {\sf E})$ be as in Lemma \ref{lambda01}, where $\sf w$ is a  $\sf V$-balanced algebraic character of $\mathsf Z_{\sf E}$. Then
\[
\widehat\CP^{\lambda_{\sf V, w}}_\infty(1, \,\cdot\,) = %\Upsilon_{\Pi_\infty'}\cdot 
\widehat\CP_\infty^\circ({\sf w}_\infty \chi_\infty, \,\cdot\,)
\]
for all algebraic characters $\chi_\infty$ of $\RU_1(\rk_\infty)$ of weight ${\sf w}^{-1}$. 
\end{thml}

\begin{proof}
Define $\Upsilon_{\pi_\infty'}(\chi_\infty)$ 
as in \eqref{MFinfrs}. Define $\Upsilon_{{\pi_\infty^\vee}'}(\chi_\infty^{-1})$ similarly, but with ${\rm i}$ replaced by $-{\rm i}$ (see \cite[Theorem 1.6]{JLS24}).
 By Theorem \ref{thm:sun} and Lemma \ref{lem:zhang}, it suffices to show that 
\be \label{sgneq}
\Upsilon_{\pi_\infty'}(\chi_\infty)\cdot \Upsilon_{{\pi_\infty^\vee}'}(\chi_\infty^{-1}) = 1. %\Upsilon_{\Pi_\infty'}.
\ee
It is easy to verify that 
\[
\begin{aligned}
\Upsilon_{\pi_\infty'}(\chi_\infty)\cdot \Upsilon_{{\pi_\infty^\vee}'}(\chi_\infty^{-1})=\ & {\rm i}^{-\sum_{\iota\in \CE_\rk}\sum^{n-1}_{i=1}(n-i)(\mu^\iota_i +\mu^{\bar\iota}_{n+1-i} + \nu^\iota_i + \nu^{\bar\iota}_{n-i})}\\
& \cdot (-1)^{\sum_{\iota\in\CE_\rk}\sum_{i>j,\, i+j\leq n}(\mu^\iota_i -\mu^{\bar\iota}_{n+1-i}+\nu^\iota_j-\nu^{\bar\iota}_{n-j})},
\end{aligned}
\]
where for every $\iota\in \CE_\rk$, $\bar\iota$ denotes the composition of
\[
\rk\xrightarrow{\iota}\overline\Q\xrightarrow{\text{complex conjugation}}\overline\Q.
\]
Note that $\mu$ and $\nu$ are pure of weight $0$ in the sense of \cite{Cl90}, namely $\mu^\iota_i + \mu^{\bar\iota}_{n+1-i}=0$, $i=1,2,\dots, n$ and 
$\nu^\iota_j + \nu^{\bar\iota}_{n-j}=0$, $j=1,2,\dots, n-1$
for all $\iota\in\CE_\rk$. Then \eqref{sgneq} follows immediately.
\end{proof}

Following Definition \ref{df:period}, 
define the Bessel periods 
\be \label{bessel-per}
\Omega_{\Pi}(\varepsilon_\infty):= \left(\widehat\CP_\infty^\circ(\varepsilon_\infty, \widehat\phi_\infty^\circ)\right)^{-1},\quad \varepsilon_\infty \in \widehat{\RU_1(\rk_\infty)^\natural}.
\ee

\subsection{Open orbit integrals and normalized refined period} 
 Assume that $\Pi_p'  \subset\CB_P(\Pi_p)$ is a nearly ordinary refinement of $\Pi_p$ defined over ${\sf E}$, which is also viewed a character of $P$ that descends to a character of $L=P/N$. Note that
 $P = z'^{-1} \overline{B} z'$, where  $\overline B$ is the  Borel subgroup of lower triangular matrices in $G$. Define a character
\[
\kappa = \Pi_p' \circ {\rm Ad}(z'^{-1}): \overline B \to {\sf E}^\times,
\]
so that $\Pi_p = \pi_p\boxtimes \pi_p^\vee$ is isomorphic to  a quotient representation of 
\[
I(\tilde\kappa):=\Ind^{G}_{\overline B}(\tilde\kappa),\quad \text{where}\quad \tilde\kappa:=\kappa\otimes \delta_{\overline B}^{1/2}.
\]
Recall that by MVW involution,  $\kappa$ is of the form \eqref{MVW}.

\begin{comment}
Similar to \eqref{Lambdaint1} and \eqref{Lambdaint2}, we have the unnormalized Rankin-Selberg integral map
\[
\begin{array}{rcl}
\CP_{\wp}: \CX_\wp\times \left(I(\tilde\kappa_\wp)\otimes \RM(\dot\RG_\wp)\right) & \rightarrow &\C\cup\{\infty\},\\
  (\chi_\wp', f\otimes  \tau) &\mapsto & \int_{ (\dot\RU_\wp\times\dot\RU_\wp) \bs \dot\RG_\wp}\chi_\wp'(\jmath(g))\langle \lambda_\wp', g.f\rangle \od\!\bar\tau(g)
\end{array}
\]
where $\bar\tau$ denotes the quotient of $\tau$ by the product of the fixed measure on $\dot\RU_\wp$ in Section \ref{ssec:IMC}, 
\end{comment}
Similar to \eqref{Lambdaint2}, we have the open orbit integral map 
\[
\begin{array}{rcl}
\Lambda_{p}: \CX_p\times \left(I(\tilde\kappa)\otimes \RM(\dot G)\right)& \rightarrow &\C\cup\{\infty\},\\
   (\chi_p', f  \otimes  \tau') &\mapsto & \int_{ \dot G} \chi_p'(\jmath(g)) f(z' g)\od\!\tau'(g),
\end{array}
\]
defined by meromorphic continuation of absolutely convergent integrals. It naturally extends to a map
\[
\Lambda_p: \CX_p\times \left(\widehat{I(\tilde\kappa)}_{\p\mathrm{-sm}}\otimes \RM(\dot G)\right)\to \BC\cup\{\infty\}.
\]

For all $\wp\mid p$, let $\lambda_\wp'\in \Hom_{\RU_\wp\times\RU_\wp}(I(\tilde\kappa_\wp), \psi_{\RU_\wp}\boxtimes \psi_{\RU_\wp}^{-1})$ be given by the Jacquet integrals as in \eqref{jacint}. Fix the surjective $G$-homomorphism
\be \label{surjhomggp}
\xi_p: I(\tilde\kappa) \twoheadrightarrow \Pi_p, \quad f\mapsto (g\mapsto \langle \otimes_{\wp\mid p}\lambda_\wp', g.f\rangle, \quad g\in G),
\ee
which naturally extends to a surjective $G$-homomorphism 
\be \label{surjhom0ggp}
\xi_p: \widehat{I(\tilde\kappa)} \twoheadrightarrow \widehat{\Pi_p}.
\ee
The subspace $\widehat{I(\tilde\kappa)}_{z'}$ of the generalized functions supported in $\overline B z'$ is one-dimensional and the map
\eqref{surjhom0ggp} restricts to a $P$-isomorphism 
\[
\xi_p: \widehat{I(\tilde\kappa)}_{z'} \xrightarrow{\sim} \Pi_p'.
\]
Moreover, for every generator $\widehat f\otimes \tau'$ of $\widehat{I(\tilde\kappa)}_{z'} \otimes \RM(\dot G)$, $\Lambda_p(\,\cdot\, , \widehat f\otimes \tau')$ is a constant function on $\CX_p$ with values in $\C^\times$. 

Define the normalized refined period map to be the composition 
\begin{eqnarray*}
\widehat\CP_p^\circ &:& \CX_p \times \left(\Pi_p' \otimes \RM(\dot H \bs \dot G)\right)  \\
& \xrightarrow{\xi_p^{-1}} & \CX_p \times \left( \widehat{I(\tilde\kappa)}_{z'}\otimes \RM(\dot H\bs \dot G)\right) \\ 
& \xrightarrow{\tau\mapsto \, c_\tau \cdot \tau\otimes\tau} & \CX_p \times 
\left( \widehat{I(\tilde\kappa)}_{z'}\otimes \RM(\dot G)\right) \\
& \xrightarrow{\Lambda_p} & \C,
\end{eqnarray*}
where in the second arrow we use the identifications 
\[
\RM(\dot H\bs \dot G)=\RM(\GL_{n-1}(\rk_p))\quad\textrm{and}\quad \RM(\dot G)=\RM(\GL_{n-1}(\rk_p))\otimes \RM(\GL_{n-1}(\rk_p)),
\]
and $c_\tau$ is given by Lemma \ref{lem:zhang}. More precisely, if $\tau = \otimes_{\wp\mid p}\tau_\wp$ where $\tau_\wp\in \RM(\dot\RH_\wp \bs \dot \RG_\wp)$, then $c_\tau=\prod_{\wp\mid p}c_{\tau_\wp}$ for the constants $c_{\tau_\wp}$ as in Lemma \ref{lem:zhang}.

\subsection{Rational test vectors and modifying factors at $p$}
In what follows we define a rational test vector $\widehat\phi_p^\circ\in \Pi_p'({\sf E})\otimes \RD(\dot H \bs \dot G)$. 

\begin{leml} \label{lem:ggprat}
For every finite place $v$ of $\rk$ that splits in $\rk'$, the pairing \eqref{zhpairing} on $\pi_v \boxtimes \pi_v^\vee$ is rational with respect to the rational forms of the Whittaker models $\pi_v\subset \Ind^{\RG_{0,v}}_{\RU_v}\psi_{\RU_v}$ and $\pi_v^\vee\subset \Ind^{\RG_{0,v}}_{\RU_v}\psi_{\RU_v}^{-1}$ given by \eqref{autc}. 
\end{leml}

\begin{proof}
 Take $\varphi_v\in \pi_v$, $\varphi_v^\vee\in \pi_v^\vee$ and $\sigma\in \Aut(\C)$. Recall the element ${\bf t}_{\sigma, \ell}\in \RG_0(\rk_v)$ as in \eqref{tl}, where $\ell$ is the residue characteristic of $\rk_v$.  In fact 
 ${\bf t}_{\sigma, \ell}\in \RQ_v$. From \eqref{autc} we find that 
\[
\begin{aligned}
\la {}^\sigma\varphi_v, {}^\sigma \varphi_v^\vee\ra_v & = \int_{\RU_v\bs \RQ_v}\sigma(\varphi_v({\bf t}_{\sigma, \ell}\cdot g_v) \varphi_v^\vee({\bf t}_{\sigma, \ell}\cdot g_v)) \od\!g_v \\
& = \sigma(\la \varphi_v, \varphi_v^\vee\ra_v),
\end{aligned}
\]
which finishes the proof.
\end{proof}

For every $\wp\mid p$, define an action of $\Aut(\C/\Q(\kappa_\wp))$ on $I(\tilde\kappa_\wp)$ by
\[
{}^\sigma f(g) :=\kappa_\wp({\bf t}_{\sigma, p}, {\bf t}_{\sigma, p})\cdot\sigma(f(g)),\quad \sigma\in \Aut(\C/\Q(\kappa_\wp)),\ f\in I(\tilde\kappa_\wp), \ g\in \RG_\wp.
\]
Then by \eqref{gauss} we have that
\[
I(\tilde\kappa_\wp)^{\Aut(\C/\Q(\kappa_\wp))}=\{
f\in I(\tilde\kappa_\wp) \,:\, \text{$\omega_{\psi_\wp}(\kappa_\wp)\cdot f$ is $\Q(\kappa_\wp)$-valued}\},
\]
which is a $\Q(\kappa_\wp)$-form  of $I(\tilde\kappa_\wp)$, where  $\omega_{\psi_\wp}(\kappa_\wp)$ is in \eqref{omegapi'ggp}. By tensor products, we have a $\Q(\pi_p)$-form  of $\Pi_{p}$ as well as a $\Q(\kappa)$-form of  $I(\tilde\kappa)$. 

\begin{prpl} \label{pratggp}
One has that $\Q(\pi_p)\subset \Q(\kappa)$ and the map $\xi_p$ in \eqref{surjhomggp} is $\Q(\kappa)$-rational. 
\end{prpl}

\begin{proof}
    This follows from  Lemma \ref{lem:ggprat} and the proof of Proposition \ref{prat}.
\end{proof}

Let $\widehat\phi_p$ denote the generator of $\Pi_p' \otimes \RM(\dot H\bs \dot G)$ such that $\widehat\CP_p^\circ(\,\cdot\, ,\widehat\phi_p)$ equals the constant function $1$ on $\CX_p$. Write $\widehat\phi_p = \phi_p' \otimes \tau$ such that $\phi_p'\in \Pi_p'$
and $\tau\in \RD(\dot H\bs \dot G)$. 
Proposition \ref{pratggp} implies that
\[
\Omega_{\Pi_p'}^{-1}\cdot\phi_p' = \prod_{\wp\mid p}\omega_{\psi_p}(\kappa_\wp)^{-1}\cdot \phi_p'   \in \Pi_p'({\sf E}),
\]
where $\Omega_{\Pi_p'}$ is in \eqref{omegapi'ggp}. Define
\be \label{ptestggp}
\widehat\phi_p^\circ:= \Omega_{\Pi_p'}\cdot \widehat\phi_p =\Omega_{\Pi_p'}^{-1}\cdot\phi_p' \otimes \tau\in \Pi_p'({\sf E})\otimes \RD(\dot H\bs \dot G).
\ee
Then 
\[
\widehat\CP_p^\circ(\chi_p', \widehat\phi_p^\circ) = \Omega_{\Pi_p'}^{-1} \quad \text{for all }\chi_p'\in \CX_p. 
\]

Recall the modifying factor $\Upsilon_{\Pi_p'}(\chi_p)$ at $p$ in \eqref{MFpggp}. 

\begin{prpl} \label{ggp:st}
It holds that 
\[
\CP_p^\circ(\chi_p', \widehat\phi_p^\circ) =\Upsilon_{\Pi_p'}(\chi_p')\cdot \Omega_{\Pi_p'}^{-1}\quad \text{for all }\chi_p'\in \CX_p.
\]
Consequently $\Upsilon_{\Pi_p'}$ is algebraic on $\CX_p$. 
\end{prpl}

\begin{proof}
By Theorem \ref{thm:llss} and Lemma \ref{lem:zhang}, it is easy to see that the proposition follows from the equality
\[
 \prod_{i>j, \, i+j\leq n}(\kappa_{i,\wp}\kappa_{j,\wp}'\chi_\wp)(-1) \cdot \prod_{i>j, \, i+j\leq n} (\kappa_{n+1-i,\wp}^{-1} \kappa_{n-j,\wp}'^{-1}\chi_\wp^{-1})(-1) = \chi_{\pi_{n-1,\wp}}(-1)^n.
\]
By elementary calculations, the left hand side equals $\prod^{n-1}_{j=1}\kappa_{j,\wp}'(-1)^n$, which is clearly $\chi_{\pi_{n-1,\wp}}(-1)^n$.
\end{proof}

\subsection{$p$-adic L-functions and exceptional zeros}  Now we define an $E$-valued $p$-adic measure $\mathscr L_\Pi$ on $\CC_{\rk'}^-(p^\infty)$. 
Suppose that $Z_0$ and   $\varepsilon = \otimes_{v\nmid\infty p}\varepsilon_v$ in \eqref{ram} are both trivial.

\begin{itemize}
\item For  $v\nmid\infty p$, let
$
\phi_v^\circ\in \Pi_v({\sf E})\otimes\RD(\dot\RH_v\bs\dot\RG_v)
$
be as in Proposition \ref{lem:unr} such that
\[
\CP_v^\circ(\,\cdot\,, \phi_v^\circ)=1 \quad \textrm{on }\CX_v^{\rm un}.
\]

\item  Let $\lambda_0 \in  \Hom_E(V^\n, E)$ and $\widehat \phi^\circ_\infty\in \Pi_\infty'({\sf E})\otimes  \RD  
 (\dot{\mathsf H}(\R))^\vee\otimes \mathrm{O}(\dot{\mathsf G}(\R)/\dot{K}_\infty^\circ)$ be as in Section \ref{ssec:AMSMggp} such that 
 \[
 \widehat\CP_\infty^{\lambda_{\sf V, w}}(1,\widehat\phi^\circ_\infty)  = \Upsilon_{\Pi_\infty'}\cdot \widehat\CP_\infty^\circ({\sf w}_\infty\chi_\infty, \widehat\phi^\circ_\infty)
  = \frac{\Upsilon_{\Pi_\infty'}}{\Omega_\Pi({\sf w}_\infty\chi_\infty)}
 \]
 for all ${\sf V}$-balanced characters ${\sf w}$ and all algebraic characters $\chi_\infty$ of weight ${\sf w}^{-1}$, where  $\lambda_{\sf V, w}\in \Hom_{\dot{\sf G}_{\sf E}}(\sf E_{\sf w}\otimes {\sf V}, {\sf E})$ is the generator such that $\lambda_{\sf V, w}|_{V^\n} = \lambda_0$.

\item Let $ \widehat\phi_p^\circ \in \Pi_p'({\sf E}) \otimes \RD(\dot H\bs \dot G)$ be as \eqref{ptestggp}.
\end{itemize}
Using all the local test vectors, let
\[
\widehat\phi^\circ:= \widehat\phi_\infty^\circ \otimes \widehat\phi_p^\circ \otimes (\otimes_{v\nmid\infty p} \phi_v^\circ)
\in \mathscr{H}\otimes \RD(\dot{\sf G}, \dot\p),
\]
and following Definition \ref{df:padicL} let
\[
\mathscr{L}_\Pi:=   \mathcal L_{\varepsilon \otimes  \mathscr H } :=(\mathcal L_{\lambda_0\otimes \widehat \phi^\circ})|_{\con({\sf Z}^{\flat,\{0\}}, E)}.
\]

We restate  Theorem \ref{padicLggp0} below, which is now an immediate consequence of the above results and \eqref{pintp}.

\begin{thml} \label{padicLggp}
Let the notations and assumptions be as above. %If there is a ${\sf V}$-balanced automorphic character unramified outside $\infty p$, 
Then 
\[
\mathscr L_\Pi(\chi^\flat) = \frac{\Upsilon_{\Pi_\infty'}\cdot \Upsilon_{\Pi_p'}(\chi_p)\cdot \CL_\Pi(\chi)}{\Omega_{\Pi_p'}\cdot\Omega_\Pi({\sf w}_\infty \chi_\infty)}
\]
for all ${\sf V}$-balanced  automorphic characters $\chi=\otimes_\ell \chi_\ell : \RU_1(\rk)\bs \RU_1(\BA_{\rk})\to\BC^\times$ unramified outside $\infty p$, where ${\sf w}$ is the inverse of the weight of $\chi$.
\end{thml}

Similar to Proposition \ref{excep0rs}, we can also determine the exceptional zeros of $\mathscr L_\Pi$. 

\begin{prpl} \label{excep0ggp}
Under the assumptions in Theorem \ref{padicLggp}, $\chi^\flat$  is an exceptional zero of $\mathscr L_\Pi$ if and only if there exist  $\wp\mid p$ and $1\leq i<n$ such that 
\[
\chi_\wp = \kappa_{i,\wp}^{-1} \cdot \kappa_{n-i,\wp}'^{-1} \cdot\abs{\,\cdot\,}_\wp \quad \textrm{or} \quad  \kappa_{i+1,\wp}^{-1}\cdot \kappa_{n-i,\wp}'^{-1} \cdot \abs{\,\cdot\,}_\wp^{-1},
\]
where $\kappa_{i,\wp}$'s and $\kappa_{j,\wp}'$'s are as in \eqref{MFpggp}.
\end{prpl}

\begin{proof}
Recall that $\CL(\frac{1}{2}, \pi_\wp\times \chi_\wp)\in\C^\times$  because $\pi_\wp$ is tempered. Thus all the exceptional zeroes arise from the local $\gamma$-factors in \eqref{MFpggp}. The proposition follows easily
in view of Lemma \ref{lem:llss} and its proof.
\end{proof}

\section{Application III: Standard  L-functions  of symplectic type for $\GL_{2n}$}\label{sec:rs3}

In this section we retain the setup in Section \ref{ssec:stL} and construct the $p$-adic L-function $\mathscr L_\Pi$ in Theorem \ref{padicLst0} following 
Sections \ref{ssec:CLR}--\ref{ssec:PL}.
Then we determine the exceptional zeros of $\mathscr L_\Pi$.

\subsection{Friedberg-Jacquet integrals}  \label{ssec:FJI}
Let $\pi$ be an irreducible cuspidal automorphic representation of $\GL_{2n}(\A_\rk)$, and let $\eta: \rk^\times\bs \A_\rk^\times\to \C^\times$ be a Hecke character. Combining the 
works \cite{JS88, AS06, HS09}, it is known that the following 
statements are equivalent:
\begin{enumerate}
\item[(a)] the twisted exterior square L-function $\oL(s,  \pi, \wedge^2\otimes\eta^{-1})$ has a pole at $s=1$;
\item[(b)] $\pi$ is of symplectic type, namely it is the functorial transfer of an irreducible generic cuspidal
automorphic representation of ${\rm GSpin}_{2n+1}(\A_\rk)$ with central character $\eta$;
\item[(c)] $\pi$ has a nonzero $(\eta, \psi)$-Shalika integral. 
\end{enumerate}
If these equivalent statements hold, then $\pi^\vee\cong \pi\otimes \eta^{-1}$. 

Let us explain the statement (c). The Shalika subgroup of $\GL_{2n}$ is defined by
\[
\RS : = \Set{ \begin{bmatrix} g & gx \\  0 & g \end{bmatrix}  |  g \in \GL_n, \, x\in \RM_n },
\]
where $\RM_n$ denotes the space of $n\times n$ matrices. 
Recall the nontrivial additive character $\psi: \rk\bs \A_\rk \to \C^\times$ in \eqref{psi}. Let $\psi_\RS$ be the character 
\[
\psi_\RS: \RS(\rk) \bs \RS(\A_\rk) \to \C^\times,\quad \begin{bmatrix} g & gx \\  0 & g \end{bmatrix} \mapsto \eta(\det g) \psi( \tr\, x).
\]
Then we say that $\pi$ has a nonzero $(\eta, \psi)$-Shalika integral if $\eta^n$ equals the central character of $\pi$ and there exists $\varphi\in \pi$ such that
\[
\langle\lambda_\RS,\varphi\rangle: = \int_{(\RS(\rk) \A_\rk^\times)\bs \RS(\A_\rk)} \varphi(h)\cdot \psi_\RS^{-1}(h) \od\!h\neq 0,
\]
in which  case 
$
\lambda_\RS \in \Hom_{\RS(\A_\rk)}(\pi, \psi_\RS)
$
is a nonzero $\psi_\RS$-Shalika functional on $\pi$.

Assume that the above equivalent statements hold. Recall the algebraic group  ${\sf G} = ({\sf G}_{2n}\times {\sf G}_1)/{\sf G}_1'$ and the automorphic representation 
$\Pi=\pi\boxtimes \eta^{-1}$ of ${\sf G}(\A)$. We reformulate the Shalika functional as follows in order to fit the general theory in Section \ref{ssec:CLR}. Let $\RA_\RS\cong \GL_1$ be the center of $\RS$ and 
${\sf H}:={\sf S}/{\sf A_S}$, where ${\sf S}:=\Res_{\rk/\Q}\RS$ and ${\sf A_S} :=\Res_{\rk/\Q}\RA_\RS$. We identify ${\sf H}$
as a subgroup of ${\sf G}$ via 
\[
{\sf H}\to {\sf G},\quad \begin{bmatrix} g & gx\\ 0 & g\end{bmatrix}{\sf A_S}\mapsto \left(\begin{bmatrix} g & gx \\ 0 & g\end{bmatrix}, \det g\right){\sf G}_1'.
\]
Define the character 
\[
\psi_{\sf H} = \otimes_v \psi_{\RH_v}: {\sf H}(\Q)\bs {\sf H}(\A)\to\C^\times, \quad \begin{bmatrix} g & gx \\ 0 & g\end{bmatrix}{\sf A_S} \to \psi(\tr \, x),
\]
where $v$ runs over places of $\rk$, $\RH_v:=\RH(\rk_v)$, and $\RH:=\RS/\RA_\RS$.  Similar notations  such as $\RG_v$ and 
$\dot\RG_v$ will be used without explanation. 
Then we have the functional $\lambda_{\sf H}\in \Hom_{{\sf H}(\A)}(\Pi, \psi_{\sf H})$ as in Section \ref{ssec:CLR}. By the uniqueness of  twisted Shalika models \cite[Theorem A]{CS20} we have a factorization  
\[
\lambda_{\sf H}=\otimes_v \lambda_{\RH_v},  \quad\textrm{where}\ \ \lambda_{\RH_v}\in \Hom_{\RH_v}(\Pi_v,\psi_{\RH_v}). 
\]

Let  $\chi= \otimes_v\chi_v:\rk^\times\bs\A_\rk^\times\rightarrow \C^\times$ be a Hecke character. Following \cite{FJ93} we have the global Friedberg-Jacquet integral 
\[
\CP_\chi:  \Pi\otimes \RM(\dot{\sf G}(\Q)\bs \dot{\sf G}(\A))\to\chi^{-1}, \quad  f\otimes \tau\mapsto \int_{\dot{\sf G}(\Q)\bs \dot{\sf G}(\A)}\chi(\jmath(g))f(g)\od\!\tau(g),
\]
which converges absolutely. 

Let $\CX_v(\varepsilon_v)\subset \CX_v$ be as in Section \ref{ssec:RSI}, and define the normalized Friedberg-Jacquet integral
\[
\begin{aligned}
\CP^\circ_v: \CX_v\times \left(\Pi_v\otimes \RM(\dot\RH_v\bs \dot\RG_v)\right) & \to \C, \\
(\chi_v', f\otimes \tau)& \mapsto \frac{1}{\oL(\frac{1}{2}, \pi_v\otimes\chi_v')}\int_{\dot\RH_v\bs \dot\RG_v}
\chi_v'(\jmath(g))\langle \lambda_{\RH_v}, g.f\rangle\od\!\tau(g)
\end{aligned}
\]
by holomorphic continuation.  By \cite[Proposition 2.3]{FJ93},
\[
\CP_\chi = \oL(\frac{1}{2}, \pi\otimes\chi) \cdot \otimes_v\CP^\circ_v(\chi_v,\,\cdot\,).
\]

In the rest of this section, assume that $\pi$ is regular algebraic and $\chi$ is algebraic.  Let $v$ be a finite place of $\rk$ with residue characteristic $\ell$. Recall the cyclotomic character 
$\Aut(\C)\to \Z_\ell^\times$, $\sigma\mapsto t_{\sigma,\ell}$.  Put
\[
{\bf s}_{\sigma, \ell}: =\left( \begin{bmatrix} t^{-1}_{\sigma, \ell}\cdot 1_n & 0 \\ 0 &  1_n\end{bmatrix}, 1\right)\RG_{1,v}' \in \RG_v.
\]

Write $\eta=\otimes_v \eta_v
$ as usual. Then we have an action of $\Aut(\C/ \Q(\eta_v))$ on $\Ind^{\RG_v}_{\RH_v}\psi_{\RH_v}$ given by 
\be \label{autcSha}
{}^\sigma\varphi(g) := \sigma(\varphi({\bf s}_{\sigma, \ell}\cdot g)),\quad \varphi\in \Ind^{\RG_v}_{\RH_v}\psi_{\RH_v}, \  g\in \RG_v.
\ee
This defines the $\Q(\Pi_v)$-form of $\Pi_v \hookrightarrow \Ind^{\RG_v}_{\RH_v}\psi_{\RH_v}$ (see \cite{JST19}). The following result
is given by \cite[Theorem 3.1]{JST19}.

\begin{prpl} \label{napr:sha}
For every finite place $v$ of $\rk$, the linear functional
\[
\mathscr G_{\psi_v}(\chi_v)^n \cdot \CP_v^\circ(\chi_v, \,\cdot\,)
\]
is  defined over $\Q(\Pi_v, \chi_v)$. 
\end{prpl}

As before suppose that $ \Q(\Pi)\subset{\sf E}$. The $\Q(\Pi_v)$-form  of $\Pi_v$ induces an $\sf E$-form  of $\Pi_v$, to be denoted by $\Pi_v(\sf E)$. By \cite{FJ93} and Proposition \ref{napr:sha}, there is a family 
\[
\{\phi_v^\circ \in \Pi_v\otimes \RD(\dot\RH_v\bs \dot\RG_v)\}_{v\nmid \infty}
\]
such that 
\begin{itemize}
    \item for all $v\nmid \infty$, $\phi^\circ_v\in \Pi_v({\sf E})\otimes \RD(\dot\RU_v \bs \dot\RG_v)$ and $\CP^\circ_v(\cdot, \phi^\circ_v)$ takes a nonzero constant value $(\Omega_{\Pi_v}(\varepsilon_v))^{-1}$ on $\CX(\varepsilon_v)$, where $\Omega_{\Pi_v}(\varepsilon_v):=\mathscr G_{\psi_v}(\chi_v')^n$ for an arbitrary $\chi_v'\in \CX_v(\varepsilon_v)$;
    \item for all but finitely many $v\nmid\infty$, $\Omega_{\Pi_v}(\varepsilon_v)=1$ and $\phi_v^\circ$ is the fixed spherical vector used in the restricted tensor product $\Pi\otimes \RM(\dot\RH(\A_\rk)\bs \dot\RG(\A_\rk)) = \widehat\otimes'_v(\Pi_v\otimes\RM(\dot\RH_v\bs \dot\RG_v))$.
\end{itemize}

\subsection{Archimedean modular symbols} \label{ssec:AMSha}  Take $K_\infty=\mathsf A(\R)\cdot K_\infty'\subset \mathsf G(\R)$, where $\mathsf A$ is the largest central split  torus in $\mathsf G$ and $K_\infty'$ is the standard maximal compact subgroup. Take $\dot K_\infty':=K_\infty'\cap \dot{\mathsf G}(\R)$, which is the standard maximal compact subgroup of  $\dot{\sf G}(\R)$.  Recall that $\Q(\Pi)\subset{\sf E}$.  Then there is a geometrically irreducible algebraic representation ${\sf F}_\mu\boxtimes {\sf w}_\eta$ of ${\sf G}_{\sf E}$ %whose complexification is called  the coefficient system of $\Pi$, 
such that the total relative Lie algebra cohomology 
\[
\oH^\bullet(\g_\C, K_\infty^\circ; ({\sf F}_\mu^\vee\boxtimes {\sf w}_\eta^{-1})\otimes \Pi_\infty)\neq \{0\},
\]
where  ${\sf w}_\eta$ is the weight of $\eta^{-1}$ and $\mu\in (\Z^{2n})^{\CE_\rk}$ is as in Section \ref{ssec:RSGL}. Suppose that ${\sf V} ={\sf F}_\mu^\vee\boxtimes {\sf w}_\eta^{-1}$ (which is also viewed as  a representation of ${\sf G}(\Q)$), so that $V= E\otimes ({\sf F}_\mu^\vee\boxtimes {\sf w}_\eta^{-1})$ as a representation of $G={\sf G}(\Q_p)\subset {\sf G_E}(E)$.
Recall the ${\sf G}^\natural$-homomorphism \eqref{embcoh}, where  $\Phi$  is the family of all compact subsets of ${\sf G}(\Q)\bs \mathscr X$. We have the degree $i_0$ component $\Pi_\infty'$ of $\oH^\bullet(\g_\C, K_\infty^\circ; {\sf V}\otimes \Pi_\infty)$
and its ${\sf E}$-form $\Pi_\infty'({\sf E})$ in \eqref{Piinf'}, where $i_0 = \dim(\dot{\sf G}(\R)/\dot K_\infty^\circ)$. 

\begin{dfnl} \label{critst}
 (a) A character $\chi_\infty$ of $\rk_\infty^\times$ is said to be critical for $\pi$ if 
it is algebraic and $s=\frac{1}{2}$ is a pole of neither $\oL(s, \pi_\infty\otimes\chi_\infty)$ nor $\oL(1-s, \pi_\infty^\vee\otimes\chi_\infty^{-1})$.

\noindent (b) A Hecke character $\chi$ of $\rk^\times\bs\A_\rk^\times$ is said to be critical for $\pi$ if so is its archimedean component $\chi_\infty$.
\end{dfnl}

It is known that (see \cite{JST19, JLST24}) all $\sf V$-balanced characters of $\rk_\infty^\times$ are critical for $\pi$. Assume that $\chi_\infty$ is algebraic of weight ${\sf w}^{-1}$.  Then
it is $\sf V$-balanced if and only if 
\be \label{bal-Sha}
 \mu^\iota_{n+1} \leq {\sf w}_\iota \leq \mu^\iota_{n}\quad \text{ for all }\iota\in \CE_\rk.
\ee

%In the rest of this subsection, assume that ${\sf w}$ is ${\sf V}$-balanced. We first specify the parabolic subgroup $P$ of $G$, and the elements $\lambda_{\sf V, w}\in \Hom_{\dot{\sf G}_{\sf E}}({\sf E_w}\otimes{\sf V}, {\sf E})$ and $\lambda_0\in \Hom_E(V^\n, E)$ such that $\lambda_0 = \lambda_{\sf V, w}|_{V^\n}$.

Let $\overline{\sf Q}$ be the lower triangular parabolic subgroup of ${\sf G}$ of type $(n, n)$. Let 
%Then $\overline{\sf B} \gamma \dot{\sf G}$ is Zariski open in $\dot{\sf G}$, where 
\be \label{elet:gamma}
{\sf P}: = \gamma^{-1} \overline{\sf Q} \gamma,\quad \text{where}\quad \gamma: = \left( \begin{bmatrix} 1_n & 1_n \\  0 & 1_n\end{bmatrix}, 1\right){\sf G}_1'. 
\ee
Then ${\sf P}$ and $\dot{\sf G}$ are transversal, and we have that 
\[
\dot{\sf P}:={\sf P}\cap \dot{\sf G} = \set{ ( g, g )\dot{\sf G}_1' | g \in {\sf G}_n }\quad \text{and} \quad \dot{\sf P} = \dot{\sf H}.
\]

Let ${\sf B}$ be the upper  triangular Borel subgroup of ${\sf G}$. Similar to \cite{LLS24}, by using algebraic induction from ${\sf B}$ we realize ${\sf V}$ as a space of algebraic functions on ${\sf G}_{\sf E}$. Let 
${\sf v} \in {\sf V^{\overline{\sf U}}}$
be the unique  algebraic function in ${\sf V}$ which equals 1 on $\overline{\sf U}$, where $\overline{\sf U}$ is the lower triangular maximal unipotent subgroup of ${\sf G}$.

Suppose that $P ={\sf P}(\Q_p)$ so that $N$ is its unipotent radical, and that $\lambda_0 \in  \Hom_{E\otimes\dot\p}(V^\n, E)$ is the unique generator defined over $\sf E$ such that 
\[
\langle \lambda_0, \delta.{\sf v} \rangle =1,\quad \text{where}\quad 
\delta:=\left(\begin{bmatrix} 1_n & -w_n \\ 0 & w_n\end{bmatrix},1\right){\sf G}_1'.
\]
Recall that  ${\sf Z}={\sf G}_1$. 
For every $\sf V$-balanced algebraic character $\sf w$ of $\mathsf Z_{\sf E}$, let $\lambda_{\sf V, w}\in \Hom_{\dot{\sf G}(\Q)}({\sf E_w}\otimes{\sf V}, {\sf E})$ be as in Lemma \ref{lambda01}.  
Then  for all algebraic characters $\chi_\infty$ of $\rk_\infty^\times$ of weight ${\sf w}^{-1}$, 
we have the archimedean modular symbol map
\[
\widehat{\mathcal P}_\infty^{\lambda_{\sf V, w}} : \oH^0(\z_\C,K_{\mathsf Z,\infty}^\circ; \C_{\mathsf w_\infty}\otimes \chi_\infty) \times\left(\Pi_\infty' \otimes \mathrm M 
 (\dot{\mathsf H}(\R))^\vee\otimes \mathrm{O}(\dot{\mathsf G}(\R)/\dot{K}_\infty^\circ)\right) \to \C
\]
defined as in \eqref{AMS}.

Let $\Pi_{0,\infty}$ be the irreducible tempered Casselman-Wallach representation of ${\sf G}(\R)$ whose infinitesimal character equals that of the trivial representation and whose central character equals that of
$({\sf F}_\mu^\vee \boxtimes {\sf w}_\eta^{-1})\otimes\Pi_\infty$. Define the cohomology group $\Pi_{0,\infty}'$ as in \eqref{Piinf'}. Let $\lambda_0' \in \Hom_{\sf H}({\sf V}, {\sf E})$ be the unique generator such that 
\[
\langle \lambda_0', \delta'.{\sf v}\rangle=1,
\quad \text{where}\quad  
\delta':=\left(\begin{bmatrix} 1_n \\ & w_n\end{bmatrix},1\right){\sf G}_1'.
\]
Following 
\cite{JLST24}, with the fixed Shalika functionals and $\lambda_0'$ we have the translation map 
\[
 \jmath_{\mu}: \Pi_{0,\infty}' \to \Pi_{\infty}'
\]
which is a $\rk_\infty^{\times,\natural}$-equivariant isomorphism. 
As a specialization of \eqref{AMS}, we have a map 
\[
 \widehat\CP_\infty: \oH^0(\z_\C,K_{\mathsf Z,\infty}^\circ; \varepsilon_\infty) \times\left(\Pi_{0,\infty}' \otimes \mathrm M 
 (\dot{\mathsf H}(\R))^\vee\otimes \mathrm{O}(\dot{\mathsf G}(\R)/\dot{K}_\infty^\circ)\right) \to \C
\]
for every $\varepsilon_\infty \in \widehat{\rk_\infty^{\times,\natural}}$. Define $\widehat{\CP}^\circ_\infty$ in \eqref{NAMS} to be the map such that for every $\varepsilon_\infty \in \widehat{\rk_\infty^{\times,\natural}}$, $\widehat{\CP}^\circ_\infty(\varepsilon_\infty, \,\cdot\,)$ equals  the composition of 
\begin{eqnarray*}
&&  \Pi_{\infty}' \otimes\mathrm M 
 (\dot{\mathsf H}(\R))^\vee\otimes \mathrm{O}(\dot{\mathsf G}(\R)/\dot{K}_\infty^\circ) \\
& \xrightarrow{(1, \, \jmath_\mu^{-1}\otimes {\rm id}\otimes {\rm id})} & %\bigoplus_{\varepsilon_\infty\in \widehat{\rk_\infty^{\times,\natural}}}
\oH^0(\z_\C,K_{\mathsf Z,\infty}^\circ; \varepsilon_\infty) \times \left(\Pi_{0, \infty}' \otimes\mathrm M (\dot{\mathsf H}(\R))^\vee\otimes \mathrm{O}(\dot{\mathsf G}(\R)/\dot{K}_\infty^\circ)\right)  \\
 & \xrightarrow{\widehat\CP_\infty}& \C.
 \end{eqnarray*}

We have the following archimedean nonvanishing hypothesis and period relations, which are proved in \cite{Sun19, JST19, JLST24}.

\begin{thml}\label{thm:JLST}
(a) There is an element 
\[
 \widehat \phi^\circ_\infty\in \Pi_\infty'({\sf E})\otimes  \RD  
 (\dot{\mathsf H}(\R))^\vee\otimes \mathrm{O}(\dot{\mathsf G}(\R)/\dot{K}_\infty^\circ)
\]
such that $\widehat\CP_\infty^\circ(\varepsilon_\infty, \widehat\phi^\circ_\infty)\neq 0$ for all $\varepsilon_\infty\in \widehat{\rk_\infty^{\times,\natural}}$. 

\noindent (b)
Let $\lambda_{\sf V, w}\in \Hom_{\dot{\sf G}(\Q)}({\sf E_w}\otimes{\sf V}, {\sf E})$ be as in Lemma \ref{lambda01}, where $\sf w$ is a  $\sf V$-balanced algebraic character of $\mathsf Z_{\sf E}$. Then 
\[
\widehat\CP^{\lambda_{\sf V, w}}_\infty(1, \cdot) = \Upsilon_{\Pi_\infty'}(\chi_\infty)\cdot \widehat\CP_\infty^\circ({\sf w}_\infty \chi_\infty, \cdot)
\]
for all algebraic characters $\chi_\infty$ of $\rk_\infty^\times$ of weight ${\sf w}^{-1}$, where $\Upsilon_{\Pi_\infty'}(\chi_\infty)$ is in \eqref{MFinfst}.
\end{thml}

Following Definition \ref{df:period},  define the Shalika periods
\be\label{sha-per}
\Omega_{\Pi}(\varepsilon_\infty):=\left(\widehat{\mathcal P}_\infty^\circ(\varepsilon_\infty,  \widehat \phi_\infty^\circ)\right)^{-1},\quad \varepsilon_\infty\in \widehat{\rk_\infty^{\times,\natural}}.
\ee
% which are also called the Shalika periods of $\Pi$.

\subsection{Nearly ordinary refinement and Shalika models} \label{ssec:FJnoc}

Assume that $\Pi_p'\subset\CB_P(\Pi_p)$ is a nearly ordinary refinement of $\Pi_p$ defined over ${\sf E}$. 
As before  $\Pi_p$ is isomorphic to a quotient representation of 
$
\textrm{Ind}^G_P \,(\Pi_p'\cdot\delta_P^{1/2}).
$
Note that $P=\gamma^{-1}\overline{Q} \gamma$, where $\overline{Q}:=\overline{\sf Q}(\Q_p)$. Define a representation 
\be\label{defkappa}
\kappa:= \Pi_p' \circ {\rm Ad}(\gamma^{-1})
\ee
of $\overline Q$. Then $\Pi_p$ is also isomorphic to a quotient representation of 
\[
I(\tilde\kappa):=\Ind^G_{\overline{Q}}(\tilde\kappa),\quad \text{where}\quad \tilde\kappa:=\kappa\otimes\delta_{\overline Q}^{1/2}.
\]
Write $\tilde\kappa = \tilde\kappa_1 \otimes \tilde\kappa_2 \otimes \eta_p$, $\tilde\kappa_i = \otimes_{\wp\mid p}\tilde\kappa_{i,\wp}$ $(i=1,2)$ and $\eta_p = \otimes_{\wp \mid p}\eta_\wp$, where $\tilde\kappa_{i,\wp}$  is an irreducible admissible smooth representation of $\GL_{n, \wp}$. 

Let $\wp\mid p$ be a place of $\rk$. Similarly write
\[
I(\tilde\kappa_\wp): = \Ind^{\RG_\wp}_{\overline{\RQ}_\wp}(\tilde\kappa_\wp), \quad \text{where}\quad \tilde\kappa_\wp:=\tilde\kappa_{1,\wp}\otimes \tilde\kappa_{2,\wp} \otimes\eta_\wp^{-1}
\]
and $\overline{\RQ}_\wp$ is the lower triangular parabolic subgroup
of $\RG_\wp$ of type $(n,n)$. 

\begin{leml} \label{lem:hida}
 Assume that $\tilde\kappa_{1,\wp}$ is isomorphic to a quotient representation of 
\[
\sigma_1\times\cdots\times\sigma_k\qquad (k\geq 1)
\]
(the normalized smooth induction from the corresponding lower triangular parabolic subgroup to $\GL_{n, \wp}$), where $\sigma_i$ ($i=1,2,\dots, k$) is an irreducible admissible smooth representation of $\GL_{n_i, \wp}$ ($n_i>0$ with $n_1+\cdots+n_k=n$). Then 
\[
\frac{1}{n_k}{\rm ex}_p(\chi_{\sigma_k}) + \frac{n_k}{2} \leq -  \mu^\wp_n,
\]
where  $\mu^\wp_i$'s are defined as in Section \ref{ssec:PLEZ}.
\end{leml}

\begin{proof}
Write $\sigma':=\sigma_1\times\cdots\times\sigma_{k-1}$. Applying Hida's inequality (Proposition 
\ref{hidaineq}) for the lower triangular parabolic subgroup of $\GL_{2n, \wp}$ of type $(n-n_k, n_k, n)$ corresponding to the induction
$
\sigma' \times \sigma_k \times \tilde\kappa_{2,\wp}
$
and for the element
\[
\begin{bmatrix} p \cdot 1_{n-n_k}  \\ & 1_{n+n_k}\end{bmatrix},
\]
we obtain that 
\[
\sum^{k-1}_{i=1}{\rm ex}_p(\chi_{\sigma_i})+\frac{(n-n_k)(n+n_k)}{2} \geq -\sum^{n-n_k}_{j=1} \mu^\wp_j.
\]
%where $\mu^\wp_i$'s are defined as in Section \ref{ssec:PLEZ}. 
On the other hand, since $\Pi_p'$ is nearly ordinary we have that 
\[
\sum^{k}_{i=1}{\rm ex}_p(\chi_{\sigma_i})+\frac{n^2}{2} = -\sum^{n}_{j=1} \mu^\wp_j.
\]
It follows that
\[
\frac{1}{n_k}{\rm ex}_p(\chi_{\sigma_k}) + \frac{n_k}{2} \leq -\frac{1}{n_k}\sum^n_{j=n-n_k+1}\mu^\wp_j \leq -  \mu^\wp_n. 
\]
\end{proof}

   \begin{thml}\label{thm:sha} (a) The restriction map induces an isomorphism 
   \[
   \Hom_{\RH_\wp}(I(\tilde\kappa_\wp),\psi_{\RH_\wp})
   \cong \Hom_{\RH_\wp}(I(\tilde\kappa_\wp)^\circ, \psi_{\RH_\wp}),
   \]
   where $I(\tilde\kappa_\wp)^\circ\subset I(\tilde\kappa_\wp)$ is the subspace of the smooth functions  supported in $\overline\RQ_\wp\RH_\wp$.
 \noindent (b) There is a natural isomorphism 
  \[
  \Hom_{\RH_\wp}(I(\tilde\kappa_\wp)^\circ, \psi_{\RH_\wp})\cong \Hom_{\dot\RH_\wp}(\tilde\kappa_\wp, \C).
  \]
In particular $\tilde\kappa_{2,\wp}\cong \tilde\kappa_{1,\wp}^\vee\otimes \eta_\wp$ and the above spaces are one-dimensional. 
    \end{thml}
    
\begin{proof}
   The  space $\CQ_\wp:=\overline{\RQ}_\wp\bs \RG_\wp$ is naturally  a $\RG_\wp$-homegeneous totally connected space and we have that
\[
I(\tilde\kappa_\wp) = \Gamma(\CQ_\wp, \CF)
\]
for a certain $\RG_\wp$-homogeneous vector bundle $\CF$ over $\CQ_\wp$. Here and henceforth $\Gamma$  indicates the space of the Schwartz sections.  The right action of $\RH_\wp$ on $\CQ_\wp$ has a unique open orbit $\mathscr O:=\overline\RQ_\wp\cdot \RH_\wp$. Then
\[
I(\tilde\kappa_\wp)^\circ = \Gamma({\mathscr O}, \CF|_{\mathscr O}) \hookrightarrow I(\tilde\kappa_\wp), 
\]
where the embedding is the extension by zero homomorphism.  

 For every permutation $i_1, i_2,\dots, i_{2n}$ of $1,2,\dots, 2n$, write $w=(i_1, i_2,\dots,i_{2n})$ for the permutation matrix such that $wE_{k, l}w^{-1} = E_{i_k, i_l}$, where $E_{k, l}$ $(k, l=1,2,\dots 2n)$ is the matrix in $\GL_{2n}(\Z)$ whose $(k, l)$-entry is 1 and all other entries are zero.  Let $W_{2n}$ be the group of permutation matrices in $\GL_{2n}(\Z)$, and for $w\in W_{2n}$ write
 \[
   {\mathscr O}_w:= \overline\RQ_\wp \dot w \cdot \RH_\wp \subset \CQ_\wp, \quad \text{where}\quad \dot w := (w, 1){\sf G}_1' \in {\sf G}.
 \]
 
 Let $\RQ_\wp$ be the upper triangular parabolic subgroup of $\RG_\wp$ of type $(n, n)$. Then we have the Bruhat decomposition 
\[
\RG_\wp = \bigcup^n_{j=0} \overline\RQ_\wp \dot\gamma_j \RQ_\wp,  \quad\text{where}\quad \gamma_j: =\left[\begin{smallmatrix} & 1_{n-j} \\& & & 1_j \\ 1_j \\ & & 1_{n-j} \end{smallmatrix}\right].
\]
In the notation of permutations we have that 
\[
\gamma_j = (n+1, \dots, n+j, 1, \dots, n-j, n+j+1, \dots, 2n, n-j+1, \dots, n).
\]
Using the $\dot\RH_\wp$-orbits in the flag variety of 
$\dot\RG_\wp$, it is easy to show that 
\[
\CQ_\wp = {\mathscr O}\cup \bigcup^n_{j=1} \bigcup_{w\in \Omega_j} \mathscr O_w,
\]
where $\Omega_j\subset W_{2n}$ $(j=1,2,\dots, n)$ consists of the elements of the form 
\be \label{Omegaj}
w=(i_1, \dots, i_{2n}, n+j+1,\dots, 2n, n-j+1,\dotsm n)
\ee
such that $n+1,\dots, n+j$ and $1,\dots, n-j$  are subsequences of $i_1, \dots, i_n$. To prove (a), it suffices to show that 
\[
\Hom_{\RH_\wp}(\Gamma(\mathscr O_w, \CF|_{\mathscr O_w}), \psi_{\RH_\wp})=\{0\}
\]
for all $j=1,2\dots, n$ and $w\in \Omega_j$. 

Let $w$ be as in \eqref{Omegaj}. Write $\overline\RQ_{w,\wp}:= \overline\RQ_\wp \cap \dot w \RH_\wp \dot w^{-1}$, and define a character 
\[
\psi_{\RH_\wp}^w: \overline\RQ_{w,\wp}\to\C^\times,\quad g\mapsto \psi_{\RH_\wp}(\dot w^{-1}g\dot w).
 \]
 By Frobenius reciprocity and noting that $\RH_\wp$ is unimodular, to prove (a) we only need to show that
 \be \label{smallorb}
\Hom_{\overline\RQ_{w,\wp}}(\tilde\kappa_\wp\otimes \delta_{\overline\RQ_\wp}^{1/2} \otimes \delta_{\overline\RQ_{w,\wp}}^{-1}, \psi_{\RH_\wp}^w)=\{0\}.
 \ee
For the contrary, suppose that the hom space in \eqref{smallorb} is nonzero. Let $\RN_{\overline \RQ,\wp}$ be the unipotent radical of $\overline\RQ_\wp$ and $\RL_{\overline\RQ,\wp}:=\overline \RQ_\wp/ \RN_{\overline \RQ,\wp}$. Then $\psi^w_{\RH_\wp}$ is trivial on
$\RN_{\overline\RQ,\wp}\cap \overline\RQ_{w,\wp}$, which implies that $j\leq n-j$ and $w$ is of the form 
\[
w = (i_1, \dots, i_{n-j}, n-2j+1,\dots, n-j, n+j+1,\dots, 2n, n-j+1,\dots, n).
\]
From this it is easy to check that $\overline\RQ_{w,\wp}$ contains the torus 
\[
\RT_{j,\wp}:=\left\{\left( \begin{bmatrix} 1_{n-2j} \\ & a\cdot 1_{2j} \\ & &  1_n \end{bmatrix}, a^j\right)\RG_{1,\wp}' \,:\, a\in \rk_\wp^\times\right\},
 \]
 and the image of $\ker(\psi^w_{\RH,\wp})\subset \overline\RQ_{w,\wp}$ in $\oL_{\overline\RQ,\wp}$ contains the unipotent radical of the upper triangular maximal parabolic subgroup of  $\RL_{\overline\RQ,\wp}$ of type $(n-2j,2j, n)$.
 By the second adjointness theorem, $\tilde\kappa_{1,\wp}$ is isomorphic to a quotient representation 
$
\sigma_1 \times \sigma_2
$
as in Lemma \ref{lem:hida}, where $\sigma_1$ and $\sigma_2$ are irreducible admissible smooth representations of $\GL_{n-2j,\wp}$ and 
$\GL_{2j, \wp}$ respectively such that 
\[
\Hom_{\RT_{j,\wp}}\left( \left(\sigma_1\abs{\det(\cdot)}_\wp^j\otimes\sigma_2\abs{\det(\cdot)}_\wp^{j-n/2} \otimes \tilde\kappa_{2,\wp}\otimes\eta_\wp^{-1}\right)\otimes \delta^{1/2}_{\overline\RQ_\wp}\otimes \delta_{\overline\RQ_{w,\wp}}^{-1}, \C\right) \neq \{0\}.
\]
Then a direct calculation shows that 
$
\chi_{\sigma_2} =  \eta_\wp^j.
$
By \cite[Proposition 2.17]{JST19}, 
\[
\mu^\iota_1+\mu^\iota_{2n}=\mu^\iota_2+\mu^\iota_{2n-1}=\cdots = \mu^\iota_n+\mu^\iota_{n+1} \quad \text{for all $\iota\in \CE_\rk$}.
\]
It follows  that 
\[
\frac{1}{2j} {\rm ex}_p(\chi_{\sigma_2}) =  \frac{1}{2}{\rm ex}(\eta_\wp) = \frac{1}{2n} {\rm ex}(\chi_{\pi_\wp})= -\frac{1}{2n}\sum^{2n}_{i=1}\mu^\wp_i 
=-\frac{\mu^\wp_n+\mu^\wp_{n+1}}{2}\geq -\mu^\wp_n,
\]
which contradicts Lemma \ref{lem:hida}.

This proves that \eqref{smallorb} holds, hence finishes the proof of (a). Noting that $\overline\RQ_\wp \cap\RH_\wp= \dot\RH_\wp$,  (b) is obvious by Frobenius reciprocity. 
\end{proof}

\subsection{Open orbit integrals and normalized refined period} \label{ssec:FJopen}
To evaluate the modifying factor at $p$, similar to the method in the Rankin-Selberg case where we have applied Theorem \ref{thm:llss}, the idea is to
compare the Friedberg-Jacquet integral with an open orbit integral. We first formulate and prove a general result.

By Theorem \ref{thm:sha}, $\tilde\kappa_{2,\wp}\cong \tilde\kappa_{1,\wp}^\vee\otimes \eta_\wp$ and  $\Hom_{\RH_\wp}(I(\tilde\kappa_\wp), \psi_{\RH_\wp})$ is one-dimensional with a generator $\lambda_{\RH_\wp}'$ such that
\[
\langle \lambda_{\RH_\wp}', f\rangle = \int_{\dot \RH_\wp\bs \RH_\wp} \langle \lambda_{\tilde\kappa_\wp}, f(h)\rangle\psi_{\RH_\wp}^{-1}(h)\od\!h
\]
for all $f\in I(\tilde\kappa_\wp)^\circ$, where
$\lambda_{\tilde\kappa_\wp}$ is a generator of  $\Hom_{\dot\RH_\wp}(\tilde\kappa_\wp,\C)$ and  $\od\!h \in \RM(\dot\RH_\wp\bs\RH_\wp)$ is given by the product of self-dual Haar measures on $\rk_\wp$ with respect to $\psi_\wp$ ($\dot\RH_\wp\bs\RH_\wp$ is identified with the unipotent radical of $\RH_\wp$).

Then we have the unnormalized Friedberg-Jacquet integral 
\[
\begin{array}{rcl}
\CP_{\wp}: \CX_\wp\times \left(I(\tilde\kappa_\wp)\otimes \RM(\dot\RH_\wp\bs \dot\RG_\wp)\right) & \rightarrow &\C\cup\{\infty\},\\
  (\chi_\wp', f\otimes  \tau) &\mapsto & \int_{ \dot\RH_\wp\bs \dot\RG_\wp}\chi_\wp'(\jmath(g)) \langle \lambda_{\RH_\wp}', g.f\rangle \od\!\tau(g)
\end{array}
\]
defined by meromorphic continuation of absolutely convergent integrals. Define
\[
\begin{array}{rcl}
\Lambda_{\wp}: \CX_\wp\times \left(I(\tilde\kappa_\wp)\otimes \RM(\dot\RH_\wp\bs \dot\RG_\wp)\right) & \rightarrow &\C\cup\{\infty\},\\
  (\chi_\wp', f\otimes  \tau) &\mapsto &  \gamma\left(\frac{1}{2}, \tilde\kappa_{1,\wp} \otimes \chi_\wp', \psi_\wp\right) \cdot \CP_\wp(\chi_\wp', f\otimes\tau),
\end{array}
\]
which is meromorphic in $\chi_\wp'\in\CX_\wp$. 
%\[\Lambda_v(\chi_v', f\otimes\tau ) := \prod^n_{i=1} \gamma\left(\frac{1}{2}, \varrho_i\cdot \chi_v', \psi_v\right) \cdot \CP_v(\chi_v', f\otimes\tau),\]where $(\chi_v',f\otimes\tau)\in \CX_v\times \left(I(\varrho)\otimes \RM(\dot\RH_v\bs \dot\RG_v)\right)$.

We have the following comparison between Friedberg-Jacquet integral and an open orbit integral, the proof of which uses the functional equation of Godement-Jacquet integrals in \cite{GJ72}. 

\begin{thml} \label{thm:FJ-int} Let
\[
 I(\tilde\kappa_\wp)^\sharp:= \set{f\in I(\tilde\kappa_\wp) \,:\, {\rm supp}(f)\subset \overline\RQ_\wp \gamma \dot\RG_\wp}.
\]
 Then for every $\chi_\wp' \in \CX_\wp$ and $f\otimes\tau \in I(\tilde\kappa_\wp)^\sharp\otimes\RM(\dot\RH_\wp\bs \dot\RG_\wp)$, 
\be \label{eq:FJZ}
\Lambda_\wp(\chi_\wp', f\otimes\tau) = \int_{\dot\RH_\wp\bs \dot\RG_\wp} \chi_\wp'(\jmath(g)) \langle \lambda_{\tilde\kappa_\wp},  f(\gamma g)\rangle\od\!\tau(g).
\ee
Moreover the integral \eqref{eq:FJZ} is algebraic in $\chi_\wp' \in \CX_\wp$.
\end{thml}

\begin{proof} Note that $I(\tilde\kappa_\wp)^\sharp\subset I(\tilde\kappa_\wp)^\circ$
since $\overline\RQ_\wp\gamma\dot\RG_\wp\subset \overline\RQ_\wp\RH_\wp$. 
Assume that $f\in I(\tilde\kappa_\wp)^\sharp$. Then 
\[
 \langle \lambda'_{\RH_\wp}, f \rangle = \int_{x\in \RM_{n, \wp}}\left\langle \lambda_{\tilde\kappa_\wp}, f\left(\begin{bmatrix} 1_n & x \\ 0 & 1_n\end{bmatrix}\right)\right \rangle \psi_\wp^{-1}(\tr \, x) \od\! x,
\]
where  $\RM_{n,\wp}:=\RM_n(\rk_\wp)$ and $\od\!x$ is the product of self-dual Haar measures on $\rk_\wp$ with respect to $\psi_\wp$. Here and below, for convenience we use the same notations for  elements of 
$\RS_\wp$ and $\GL_{n,\wp}\times \GL_{n,\wp}$ and their images in $\RH_\wp$ and $\dot\RG_\wp$ respectively, which should not cause any confusion.

Since ${\rm supp}(g.f)\subset \overline\RQ_\wp\RH_\wp$ for all $g\in \dot\RG_\wp$, when $\Re(\chi_\wp')$ is sufficiently large (here $\Re(\chi_\wp')$ is the real number such that $\abs{\chi_\wp'(a)}=\abs{a}_\wp^{\Re(\chi_\wp')}$, $a\in \rk_\wp^\times$) we have that 
\begin{eqnarray*}
 & &\CP_\wp(\chi_\wp', f\otimes \tau)   \\
 &= & \int_{\GL_{n,\wp}} \int_{\RM_{n,\wp}} \left \langle \lambda_{\tilde\kappa_\wp},  f  \left(  \begin{bmatrix} 1_n &  x \\ 0 & 1_n\end{bmatrix} \begin{bmatrix} g & 0 \\ 0 & 1_n\end{bmatrix} \right) \right\rangle \psi_\wp^{-1}(\tr\, x)\od\!x  \, \chi_\wp'(\det g) \od\! \tau(g) \\
 &\stackrel{x\mapsto gx}{=} &  \int_{\GL_{n,\wp}}  \int_{\RM_{n,\wp}}  \left\langle \lambda_{\tilde\kappa_\wp}, \tilde\kappa_{1,\wp}(g). f\left(\begin{bmatrix} 1_n & x \\ 0 & 1_n\end{bmatrix}\right) \right\rangle\psi_\wp^{-1}(\tr (gx))\od\!x \, \chi_\wp'(\det g)\abs{\det g}_\wp^n\od\!\tau(g).
\end{eqnarray*}
Here and below we use the identification $\RM(\dot\RH_\wp\bs\dot\RG_\wp)= \RM(\GL_{n,\wp})$.
The support condition on $f$ implies that the function
\be \label{FUN}
(g, x) \mapsto   \left\langle \lambda_{\tilde\kappa_\wp}, \tilde\kappa_{1,\wp}(g). f\left(\begin{bmatrix} 1_n & x \\ 0 & 1_n\end{bmatrix}\right) \right\rangle,\quad (g,x)\in \GL_{n,\wp}\times \RM_{n,\wp}, 
\ee
lies in the space ${\rm MC}(\tilde\kappa_{1,\wp})\otimes \CS(\RM_{n,\wp})$, where ${\rm MC}(\tilde\kappa_{1,\wp})$ denotes the space spanned by the matrix coefficients of $\tilde\kappa_{1,\wp}$. More precisely,
\[
{\rm MC}(\tilde\kappa_{1,\wp}):={\rm Span}\{ \varphi_{u, u^\vee} \,:\, u\in \tilde\kappa_{1,\wp}, u^\vee\in \tilde\kappa_{1,\wp}^\vee \},
\]
where $\varphi_{u, u^\vee}(g):= \langle \tilde\kappa_{1,\wp}(g)u, u^\vee\rangle$, $g\in \GL_{n,\wp}$.
  Then the above inner integral is the Fourier transform of the function \eqref{FUN}  in the variable $x$ using $\psi_\wp^{-1}$, evaluated at $(g, g)\in \GL_{n,\wp}\times \RM_{n,\wp}$.  Hence the inner integral, as a function in $g\in \GL_{n,\wp}$, is the pull-back through the diagonal embedding $\GL_{n,\wp}\rightarrow \GL_{n,\wp}\times \RM_{n,\wp}$ of a function in ${\rm MC}(\tilde\kappa_{1,\wp})\otimes \CS(\RM_{n,\wp})$.

Thus $\CP_\wp(\chi_\wp',f\otimes\tau)$ is a Godement-Jacquet integral (\cite{GJ72}) for the representation $\tilde\kappa_{1,\wp}\otimes\chi_\wp'$ of $\GL_{n,\wp}$. By the functional equation of Godement-Jacquet integrals  and the uniqueness of meromorphic continuation, for $-\Re(\chi_\wp')$ sufficiently large we have that
\[
\begin{aligned}
&   \, \Lambda_\wp(\chi_\wp', f\otimes\tau)\\
 = & \,  \gamma(\frac{1}{2}, \tilde\kappa_{1,\wp}\otimes\chi_\wp', \psi_\wp) \cdot \CP_\wp(\chi_\wp',f\otimes\tau)\\
  = &  \int_{\GL_{n,\wp}}  \left\langle \lambda_{\tilde\kappa_\wp}, \tilde\kappa_{1,\wp}(g^{-1}). f\left(\begin{bmatrix} 1_n & g \\ 0 & 1_n\end{bmatrix}\right) \right\rangle \chi_\wp'^{-1}(\det g)\od\!\tau(g) \\
  = &  \int_{\GL_{n,\wp}}  \left\langle \lambda_{\tilde\kappa_\wp},  f\left(\begin{bmatrix} g^{-1} & 1_n \\ 0 & 1_n\end{bmatrix}\right) \right\rangle \chi_\wp'^{-1}(\det g)\od\!\tau(g) \\
= & \int_{\GL_{n,\wp}}  \left\langle \lambda_{\tilde\kappa_\wp},  f\left(\begin{bmatrix} g & 1_n \\ 0 & 1_n\end{bmatrix}\right) \right\rangle \chi_\wp'(\det g)\od\!\tau(g)\\
= & \int_{\dot\RH_\wp\bs \dot\RG_\wp} \langle \lambda_{\tilde\kappa_\wp}, f(\gamma g)\rangle  \chi_\wp'(\jmath(g))  \od\!\tau(g).
\end{aligned}
\] 
This proves that \eqref{eq:FJZ} holds for $-\Re(\chi_\wp')$ sufficiently large, hence for all $\chi_\wp'\in\CX_\wp$ by the uniqueness of meromorphic continuation. 
\end{proof}

By tensor products, the maps  $\CP_\wp$ and $\Lambda_\wp$ for  $\wp\mid p$ yield two maps
 \[
\CP_p,\, \Lambda_p:   \CX_p\times \left({I(\tilde \kappa)}\otimes \RM(\dot H\bs \dot G)\right)\rightarrow \C\cup\{\infty\},
\]
which naturally extend to maps 
\[\CP_p,\, \Lambda_p:   \CX_p\times \left(\widehat{I(\tilde \kappa)}_{\p\mathrm{-sm}}\otimes \RM(\dot H\bs \dot G)\right)\rightarrow \C\cup\{\infty\}.\]
%\[\CX_v\times \left(\widehat{I(\varrho)}_{\p_v\mathrm{-sm}}\otimes \RM(\dot\RH_v\bs \dot\RG_v)\right)\rightarrow \C\cup\{\infty\}\]where $\p_v$ denotes the Lie algebra of $\mathsf P(\rk_v)$. the maps $\CP_\wp$ and $\Lambda_\wp$ yield respectively maps 

Recall the Shalika functional $\lambda_{\RH_\wp}$ on $\Pi_\wp$.
Write 
\be\label{surjhom0sha}
\xi_p: I(\tilde\kappa)\rightarrow \Pi_p=\otimes_{\wp\mid p} \Pi_\wp
\ee
for the $G$-homomorphism 
whose composition with $\otimes_{\wp\mid p}\lambda_{\RH_\wp}$ equals $\otimes_{\wp\mid p}\lambda_{\RH_\wp}'$. It is surjective and naturally extends to a surjective $G$-homomorphism 
\be\label{surjhomsha}
\xi_p: \widehat{I(\tilde\kappa)}\twoheadrightarrow \widehat{\Pi_p}. 
\ee
Note that $\widehat{I(\tilde\kappa)}$ is identified with a space of generalized functions on $G$ with values in the formal completion of 
$\kappa$.
Denote by $\widehat{I(\tilde\kappa)}_\gamma\subset \widehat{I(\tilde\kappa)}_{\p\mathrm{-sm}}$
the subspace of the generalized functions  supported in  $\overline{ Q} \gamma$. Then  the map $\xi_p$ in \eqref{surjhomsha} restricts to a $P$-isomorphism
\[
  \xi_p: \widehat{I(\tilde\kappa)}_\gamma\xrightarrow{\sim}\Pi_p'. 
\]
%Moreover, for every generator $\widehat f\otimes \tau$ of $\widehat{I(\tilde\kappa)}_z\otimes \RM( \dot H\bs \dot G)$,  it follows from Theorem \ref{eq:FJZ} that $\Lambda_p(\,\cdot\, ,\widehat f\otimes \tau)$ is a constant function on $\CX_p$ with values in $\C^\times$. 

Now we define the normalized refined period map to be the composition 
\[
   \widehat\CP_p^\circ: \CX_p \times \left(\Pi_p'\otimes \RM(\dot H \bs \dot G)\right) \xrightarrow{\xi_p^{-1}} \CX_p \times \left(\widehat{I(\tilde\kappa)}_\gamma\otimes \RM(\dot H \bs \dot G)\right)\xrightarrow{\Lambda_p}\C.
\]

\subsection{Rational test vectors and modifying factors at $p$}  \label{ssec13.7}

In what follows we define a rational test vector $\widehat\phi_p^\circ\in \Pi_p'({\sf E})\otimes \RD(\dot H \bs \dot G)$.  
Recall the $\Q(\Pi_\wp)$-form of  $\Pi_\wp$ given by \eqref{autcSha}. By tensor product, we have a $\Q(\Pi_p)$-form  of $\Pi_{p}$. Denote by $\kappa({\sf E})$ the ${\sf E}$-form of $\kappa$ induced from $\Pi_p'({\sf E})$ via ${\rm Ad}(\gamma^{-1})$ (see \eqref{defkappa}).  Define an $\Aut(\C/ {\sf E})$-action on $I(\tilde\kappa)$ by 
\[
{}^\sigma f(g) = \chi_{\kappa_{1}}(t_{\sigma, p}^{-1})\cdot \sigma(f(g)),\quad \sigma\in \Aut(\C/{\sf E}), \  f\in I(\tilde\kappa), \ g\in G.
\]
Then by \eqref{gauss} we have that 
\[
  I(\tilde\kappa)^{\Aut(\C/ {\sf E})}=\{f\in I(\tilde\kappa)\,:\, \textrm{$\prod_{\wp\mid p}\mathscr{G}_{\psi_\wp}(\chi_{\kappa_{1,\wp}})^{-1}\cdot f$ is 
$\kappa({\sf E})$-valued}
  \},
\]
which is an ${\sf E}$-form  of $I(\tilde\kappa_\wp)$.

The same proof as that of Proposition \ref{prat} proves the following result.

\begin{prpl}   \label{pratSha}
The surjective homomorphism $c_p^{-n^2} \xi_p: I(\tilde\kappa) \twoheadrightarrow \Pi_p$ is ${\sf E}$-rational, where $\xi_p$ is in \eqref{surjhom0sha}.
\end{prpl}

Let $\widehat\phi_p$ be an element  of $\Pi_p'\otimes \RM(\dot H \bs \dot G)$
such that $\widehat\CP_p^\circ(\,\cdot\, ,\widehat\phi_p)$ equals the constant function $1$ on $\CX_p$.  Write $\widehat\phi_p = \phi_p' \otimes  \tau$ such that $\phi_p'\in \Pi_p'$
and $\tau\in \RD(\dot H\bs \dot G)$.  By Proposition \ref{pratSha}, 
\[
\Omega_{\Pi_p'}^{-1}\cdot \phi_p'\in \Pi_p'({\sf E}),
\]
where $\Omega_{\Pi_p'}$ is in \eqref{omegapi'st}.
Define
\be \label{ptestst}
\widehat\phi_p^\circ: = \Omega_{\Pi_p'}^{-1}\cdot \widehat\phi_p\in 
\Pi_p'({\sf E})\otimes \RD(\dot H \bs \dot G).
\ee
Then 
\[
\widehat\CP_p^\circ(\chi_p', \widehat\phi_p^\circ) = \Omega_{\Pi_p'}^{-1}
\quad \text{for all }\chi_p' \in \CX_p.
\]
From \eqref{MFpst} it is clear that
\[
\Upsilon_{\Pi_p'}(\chi_p') =  \frac{1}{ \prod_{\wp\mid p}\left(\oL(\frac{1}{2}, \pi_\wp\otimes\chi_\wp') \cdot\gamma(\frac{1}{2}, \tilde\kappa_{1,\wp}\otimes\chi_\wp', \psi_\wp)\right)},\quad \chi_p'=\otimes_{\wp\mid p}\chi_\wp'\in\CX_p.
\]
Then the following is immediate from Theorem \ref{thm:FJ-int}.

\begin{prpl} \label{prp:stSha}
It holds that 
\[
\CP^\circ_p(\chi_p', \widehat\phi_p^\circ)= \Upsilon_{\Pi_p'}(\chi_p')\cdot  \Omega_{\Pi_p'}^{-1}\quad \text{for all } \chi_p'\in \CX_p.
\]
Consequently $\Upsilon_{\Pi_p'}$ is algebraic on $\CX_p$.
\end{prpl}

By Proposition \ref{prp:stSha}, $\Upsilon_{\Pi_p'}(\chi_p)$ is the modifying factor at $p$ as in Definition \ref{df:MFp}. It is again consistent with the conjecture given by Coates and Perrin-Riou in \cite{CPR89, Co89}.

\subsection{$p$-adic L-functions and exceptional zeros}
Now we  construct standard $p$-adic L-function $\mathscr L_\Pi$ by combining the previous results. Suppose that $Z_0$ is trivial and recall  $\varepsilon = \otimes_{v\nmid\infty p}\varepsilon_v$ in \eqref{ram}.

\begin{itemize}
    \item For  $v\nmid \infty p$, let  
$\phi^\circ_v\in \Pi_v({\sf E})\otimes \RD(\dot\RH_v \bs \dot\RG_v)$ be as in Section \ref{ssec:FJI}   such that 
\[
\CP^\circ_v(\chi_v', \phi^\circ_v) = \frac{1}{\mathscr G_{\psi_v}(\chi_v')^n},\quad \chi_v'\in \CX(\varepsilon_v).
\]
\item  Let $\lambda_0 \in  \Hom_E(V^\n, E)$ and $\widehat \phi^\circ_\infty\in \Pi_\infty'({\sf E})\otimes  \RD  
 (\dot{\mathsf H}(\R))^\vee\otimes \mathrm{O}(\dot{\mathsf G}(\R)/\dot{K}_\infty^\circ)$ be as in Section \ref{ssec:AMSha} such that 
 \[
 \widehat\CP_\infty^{\lambda_{\sf V, w}}(1,\widehat\phi^\circ_\infty)  = \Upsilon_{\Pi_\infty'}(\chi_\infty) \cdot \widehat\CP_\infty^\circ({\sf w}_\infty\chi_\infty, \widehat\phi^\circ_\infty)
  = \frac{\Upsilon_{\Pi_\infty'}(\chi_\infty)}{\Omega_\Pi({\sf w}_\infty\chi_\infty)}
 \]
for all ${\sf V}$-balanced characters ${\sf w}$ and all algebraic characters $\chi_\infty$ of weight ${\sf w}^{-1}$, where  $\lambda_{\sf V, w}\in \Hom_{\dot{\sf G}_{\sf E}}(\sf E_{\sf w}\otimes {\sf V}, {\sf E})$ is the generator such that $ \lambda_{\sf V, w}|_{V^\n}=\lambda_0 $.
 \item Let $\widehat\phi_p^\circ\in  \Pi_p'({\sf E})\otimes \RD(\dot H \bs \dot G)$ be as in \eqref{ptestst}.
\end{itemize}
Using all the local test vectors, let
\[
\widehat\phi^\circ:= \widehat\phi_\infty^\circ \otimes \widehat\phi_p^\circ \otimes (\otimes_{v\nmid\infty p} \phi_v^\circ)
\in \mathscr{H}\otimes \RD(\dot{\sf G}, \dot\p),
\]
and following Definition \ref{df:padicL} let
\[
\mathscr{L}_\Pi:=   \mathcal L_{\varepsilon \otimes  \mathscr H } :=(\mathcal L_{\lambda_0\otimes \widehat \phi^\circ})|_{\con({\sf Z}, E)(\varepsilon)}.
\]

We restate Theorem \ref{padicLst0} below, which is now an immediate consequence of the above results and \eqref{pintp}.

\begin{thml} \label{padicLst}
Let the notations and assumptions be as above. Then 
\[
\mathscr L_\Pi(\chi^\flat)  = \frac{\Upsilon_{\Pi_\infty'}(\chi_\infty) \cdot \Upsilon_{\Pi_p'}(\chi_p)\cdot \oL(\frac{1}{2}, \pi\otimes \chi)}{\mathscr G_\psi(\chi^{(p)})^n \cdot \Omega_{\Pi_p'}\cdot \Omega_\Pi({\sf w}_\infty\chi_\infty)}
\]
  for all ${\sf V}$-balanced Hecke characters $\chi=\otimes_\ell \chi_\ell  \in \CX(\varepsilon)$, where ${\sf w}$ is the inverse of the weight of $\chi$.
\end{thml}

We end this section by applying Theorem \ref{padicLst} to determine the exceptional zeros of $\mathscr L_\Pi$. The strategy extends that of Proposition \ref{excep0rs} for the Borel case.

\begin{prpl} \label{excep0st}
Under the assumptions in  Theorem \ref{padicLst}, $\chi^\flat$ is an exceptional zero of $\mathscr L_\Pi$ if and only if there exists  $\wp\mid p$ such that 
\[
\oL\left(s, \pi_\wp\otimes\chi_\wp\right)\cdot \oL\left(s+\frac{1}{2}, \kappa_{1,\wp}^\vee\otimes\chi_\wp^{-1}\right)
\]
has a pole at $s=\frac{1}{2}$,
where $\kappa_{1,\wp}$ is as in \eqref{MFpst}.
%In this case $s=\frac{1}{2}$ is the unique critical place. 
\end{prpl}

\begin{proof} By \eqref{MFpst}, it suffices to show that 
\[
\oL\left(\frac{1}{2}, \tilde\kappa_{1,\wp}\otimes\chi_\wp\right)\quad \text{is finite for all $\wp\mid p$}.
\quad   
\]
%For  smooth admissible representations $\varrho_1, \varrho_2, \dots, \varrho_l$ of general linear groups over $\rk_\wp$, denote by $\varrho_1\times \cdots \times \varrho_l$the normalized smooth induction from the corresponding lower triangular parabolic subgroup. 
Since $\tilde\kappa_{1,\wp}$ is irreducible generic, it is isomorphic to a module of the form
\[
\sigma_1\times\cdots\times\sigma_k
\]
as in Lemma \ref{lem:hida}, where $\sigma_i$ ($i=1,2,\dots, k$) is an irreducible admissible smooth representation of $\GL_{n_i, \wp}$ whose restriction to $\{g\in \GL_{n_i, \wp}\,:\,\det g\in \CO_\wp\}$ is square-integrable.  %($n_i>0$ with  $n_1+\cdots+n_k=n$). 
Let $i=1,2,\dots, k$. 
It follows from Lemma \ref{lem:hida} that 
\be \label{maxexp}
\frac{1}{n_i}{\rm }{\rm ex}_p(\chi_{\sigma_i}) + \frac{n_i}{2} \leq -\mu^\wp_n.
\ee
If 
$\sigma_i\otimes \chi_\wp$ is not an unramified twist of the Steinberg representation, then $\oL(s, \sigma_i \otimes \chi_\wp)$ is the constant function $1$.
Otherwise, write 
\[
\sigma_i\otimes\chi_\wp =  {\rm St}\otimes \xi_i \quad (\text{${\rm St}$ is the Steinberg representation)},
\]
where $\xi_i$ is an unramified character of $\rk_\wp^\times$. Then
$\chi_{\sigma_i}  \chi_\wp^{n_i} = \xi_i^{n_i}$ and 
\[
\oL(s, \sigma_i\otimes \chi_\wp) = \oL\left(s+\frac{n_i-1}{2}, \xi_i\right)
\]
(\cf \cite[(0.1)]{JPSS83}). By \eqref{bal-Sha} and \eqref{maxexp},
\[
{\rm ex}_p(\xi_i) -\frac{n_i}{2} ={\sf w}_\wp+\frac{1}{n_i}{\rm ex}_p(\chi_{\sigma_i})-\frac{n_i}{2} \leq {\sf w}_\wp - \mu_n^\wp - n_i \leq -n_i <0,
\]
where ${\sf w}_\wp$ is as in the proof of Lemma \ref{lem:llss}. Hence $\oL(\frac{1}{2}, \sigma_i\otimes\chi_\wp)$ is finite, which finishes the proof.
\end{proof}

\section*{Acknowledgement}

We thank Yifeng Liu for several helpful discussions and for informing us of the papers \cite{Liu23, Liu24, LTX24}. We thank
Daniel Disegni and Wei Zhang for reading an earlier version of this paper and sending their paper \cite{DZ24}. 
We thank Dihua Jiang and Fangyang Tian for discussions related to Shalika models, and thank Liang Xiao for comments on Emerton's completed cohomology. We also thank Chris Williams for corrections concerning critical characters in an earlier version of this paper and for clarifying the results in \cite{BDGJW22, Wi23}.

D. Liu was supported in part by National Key R \& D Program of China (No. 2022YFA1005300)  and Zhejiang Provincial Natural Science Foundation of China (No. LZ22A010006).  B. Sun was supported in part by  National Key R \& D Program of China (No. 2022YFA1005300 and 2020YFA0712600) and New Cornerstone Investigator Program.

\end{document}